\documentclass[11pt]{amsart}
\usepackage[mathscr]{euscript}
\usepackage{marginnote,bm,amscd}
\usepackage{amssymb}
\usepackage[foot]{amsaddr}
\usepackage{color}
\usepackage{graphicx}
\usepackage{hyperref}
\usepackage{varwidth}
\usepackage{xr}
\usepackage[most]{tcolorbox}

\tcbset{colback=yellow!10!white, colframe=red!50!black, 
        highlight math style= {enhanced, %<-- needed for the ?remember? options
            colframe=red,colback=red!10!white,boxsep=0pt}
        }

\newtheorem{thm}{Theorem}[section]
\newtheorem*{thm*}{Theorem}
\newtheorem{prop}[thm]{Proposition}
\newtheorem*{prop*}{Proposition}
\newtheorem{lem}[thm]{Lemma}
\newtheorem*{lem*}{Lemma}

\newtheorem{exercise}{Exercise}
\newtheorem{theorem}{Theorem}[section]
\newtheorem{lemma}[theorem]{Lemma}

\newtheorem{proposition}[theorem]{Proposition}
\newtheorem{corollary}[theorem]{Corollary}
\newtheorem{definition}[theorem]{Definition}
\newtheorem{remark}[theorem]{Remark}

\newtheorem{FACT}[theorem]{Fact}
\newtheorem{AIM}[theorem]{Aim}

\newcommand\jwf[1]{\textcolor{red}{#1}}
\renewcommand\jwf[1]{}

\numberwithin{equation}{section}

\newcommand{\R}{\mathbf{R}}  % The real numbers.

\newcommand{\av}{\text{av}}

\newcommand{\norm}[1]{\left\Vert#1\right\Vert}
\newcommand{\abs}[1]{\left\vert#1\right\vert}

\newcommand{\ssc}{\text{sc}}

\renewcommand{\epsilon}{\varepsilon}

\newcommand{\wh}{\widehat}

\newcommand{\wt}{\widetilde}
\newcommand{\wtilde}{\widetilde}

\newcommand{\cl}{\operatorname{cl}}

\providecommand{\ker}[1]{$\text{ker}\ {#1}$}

\newcommand{\N}{{\mathbb N}}

\newcommand{\rup}{\rotatebox[origin=c]{180}{{$\Upsilon$}}}

\makeindex

\title[Polyfold Constructions]{Polyfold and SFT Notes II:\\
Local-Local M-Polyfold Constructions}
\author{ J. W. Fish \and H. Hofer}

\email{hofer@ias.edu}
\email{joel.fish@umb.edu}

\date{\today}                                           % Activate to display a given date or no date

\begin{document}
\maketitle
\tableofcontents

\noindent These notes are the beginning chapters from the upcoming book \cite{FH-poly} 
\vspace{0.5cm}
\begin{center}
{\bf J. W. Fish and H. Hofer, \\
Polyfold Constructions: Tools, Techniques, and Functors}\\
%Theory and Methods for Concrete Polyfold Constructions}\\
\end{center}
\vspace{0.5cm}
\noindent which we make  available for the upcoming workshop\\

\begin{center}
{\bf Workshop on Symplectic Field Theory IX:}\\
\vspace{0.5cm}
{\bf POLYFOLDS FOR SFT}\\
\vspace{0.5cm}
Augsburg, Germany\\
 Monday, 27 August 2018 - Friday, 31 August 2018\\
\vspace{0.5cm}
A Precourse takes place on the preceding weekend:\\

Saturday, 25 August 2018 - Sunday, 26 August 2018
\end{center} 
\vspace{9cm}
\pagebreak

%\input{TODO.tex}
%\input{S5-1.tex}
%\part*{Introduction}
%\input{intro.tex}
%\input{intro2.tex}

\section*{Introduction}

We aim to use polyfold theory to develop a a Fredholm Theory for SFT we
  need to construct  so-called M-polyfolds and strong bundles over them.
These arise as ambient spaces for a nonlinear Fredholm theory in the
  polyfold context.
The relevant  global M-polyfolds are build using a ``LEGO''-type system,
  \cite{Lego}, from smaller building blocks. The construction of  the
  ``LEGO"-pieces is described in the local-local theory, which is the
  subject of the current paper.

In a follow up paper  we shall show how  these pieces can be plugged
  together (fibered product constructions) to carry out more complex
  constructions.
This will rely on the abstract theory contained in the first part of this
  series, \\

\noindent {\bf J. W. Fish  and H. Hofer, Polyfold and SFT Notes I:\\
A Primer on Polyfolds and Construction tools.}\\

There we have provided an abstract theory which guarantees that
  local-local constructions, satisfying some properties, can be plugged
  together to produce local spaces with desired properties, i.e. a theory
  which guides the passage from the local-local to the local theory. 
This setup allows to recycle analysis in a controllable (i.e. checkable)
  fashion and to add novel construction which then
  automatically work with the other ingredients.  
A global theory is obtained from  interacting local pieces. This will be
  described in the next paper.

\jwf{This is a test}

\vspace{8cm}
\pagebreak

We first recall some concrete sc-smoothness results from the paper
  \cite{HWZ8.7}.
This paper contains a plethora of results useful in concrete
  constructions.
After this we shall study the situation of maps defined on a nodal Riemann
  surface as well as on the family of varying cylinders obtained by gluing.
This is done  in Section \ref{SECX1}. We refer to this as the {\bf nodal
  construction}.
In Section \ref{SECX2} a similar construction is carried out in the
  context where we also have a periodic orbit. More precisely we construct
  a M-polyfold structure describing how maps defined on finite cylinders
  decompose in the presence of periodic orbits when the cylinders
  get infinitely long and approaches  a nodal disk pair, while its image
  tries to approximate a cylinder over a periodic orbit.
This will be referred to as the {\bf  stretching near a periodic orbit
  construction}.
In Section \ref{SECX3} we summarize some classical constructions.

%%%%%%%%%%%%%%%%%%%%%%%%%%%%%%%%%%%%%%%%%%%%%%%%%%%%%%%%%%%%%%%%%%%%%%%%%%%%%%%%
%%%%%%%%%%                           SECTION                           %%%%%%%%%
%%%%%%%%%%%%%%%%%%%%%%%%%%%%%%%%%%%%%%%%%%%%%%%%%%%%%%%%%%%%%%%%%%%%%%%%%%%%%%%%
%
\section{Sc-Smoothness Results}
  \label{SEC_sc_smoothness_results}
The following discussion follows closely some of the topics in \cite{HWZ8.7}.

%%%%%%%%%%%%%%%%%%%%%%%%%%%%%%%%%%%%%%%%%%%%%%%%%%%%%%%%%%%%%%%%%%%%%%%%%%%%%%%%
%%%%%%%%%%                         SUB-SECTION                         %%%%%%%%%
%%%%%%%%%%%%%%%%%%%%%%%%%%%%%%%%%%%%%%%%%%%%%%%%%%%%%%%%%%%%%%%%%%%%%%%%%%%%%%%%
%
\subsection{Smoothness Versus Sc-Smoothness}
  \label{SEC_sc_smoothness_vs_smoothness}
It is important to know the relationship between the concepts of classical
  smoothness and sc-smoothness.
The proofs of the following results can be found  in  \cite{HWZ8.7}.
The first theorem gives implicitly an alternative definition of
  sc-smoothness.

%%%%%%%%%%%%%%%%%%%%%%%%%%%%%%%%%%%%%%%%%%%%%%%%%%%%%%%%%%%%%%%%%%%%%%%%%%%%%%%%
%%%%%%%%%%                           THEOREM                           %%%%%%%%%
%%%%%                                                                       %%%%
\begin{theorem}[{Proposition 2.1, \cite{HWZ8.7}}] 
  \label{x1}\index{sc-differentiability}
Let $U$ be a relatively open subset of a partial  quadrant in  a sc-Banach
  space $E$ and let $F$ be another $\ssc$-Banach space.
Then a  $\ssc^0$-map $f\colon U\to F$ is of class $\ssc^1$  if and only if
  the following conditions hold true.
\begin{itemize}\label{sc-100}
  \item[(\em{1})]\
  For every $m\geq 1$,  the induced map $f\colon U_m\to  F_{m-1}$ is of
    class $C^1$.
  In particular, the derivative  
  $$
    df\colon U_m\to  \mathscr{L}(E_m,F_{m-1}),\quad x\mapsto df(x)
    $$
  is a  continuous map.
  \item[(\em{2})]\ 
  For every  $m\geq 1$ and every $x\in U_m$,  the bounded  linear operator
    $df(x)\colon E_m\to F_{m-1}$ has an extension to a bounded  linear
    operator $Df(x)\colon E_{m-1}\to F_{m-1}$. In addition, the map
  \begin{equation*}
    U_m\oplus  E_{m-1}\to  F_{m-1}, \quad  (x,h)\mapsto  Df(x)h
    \end{equation*}
  is continuous.
  \end{itemize}
\end{theorem}
%%%%%                                                                       %%%%
%%%%%%%%%%                                                             %%%%%%%%%
%%%%%%%%%%%%%%%%%%%%%%%%%%%%%%%%%%%%%%%%%%%%%%%%%%%%%%%%%%%%%%%%%%%%%%%%%%%%%%%%
%

%%%%%%%%%%%%%%%%%%%%%%%%%%%%%%%%%%%%%%%%%%%%%%%%%%%%%%%%%%%%%%%%%%%%%%%%%%%%%%%%
%%%%%%%%%%                            REMARK                           %%%%%%%%%
%%%%%                                                                       %%%%
\begin{remark}
If $x\in U_\infty$ is  a smooth point in $U$  and $f\colon U\to F$ is a
  $\ssc^1$-map, then the linearization $$Df(x)\colon E\to F$$ is a
  sc-operator.
\end{remark}
%%%%%                                                                       %%%%
%%%%%%%%%%                                                             %%%%%%%%%
%%%%%%%%%%%%%%%%%%%%%%%%%%%%%%%%%%%%%%%%%%%%%%%%%%%%%%%%%%%%%%%%%%%%%%%%%%%%%%%%
%

A  consequence of Theorem \ref{x1} is the following result about lifting
  the indices\index{lifting the index}.

%%%%%%%%%%%%%%%%%%%%%%%%%%%%%%%%%%%%%%%%%%%%%%%%%%%%%%%%%%%%%%%%%%%%%%%%%%%%%%%%
%%%%%%%%%%                          PROPOSITION                        %%%%%%%%%
%%%%%                                                                       %%%%
\begin{proposition} [{\bf Proposition 2.2, \cite{HWZ8.7}}] 
  \label{sc_up}\index{sc-differentiability under lifts} 
  \hfill\\
Let  $U$  and $V$ be  relatively open subsets of partial quadrants in
  sc-Banach spaces, and let $f\colon U\rightarrow V$ be $\ssc^k$.
Then $f\colon U^1\rightarrow V^1$ is also $\ssc^k$.
\end{proposition}
%%%%%                                                                       %%%%
%%%%%%%%%%                                                             %%%%%%%%%
%%%%%%%%%%%%%%%%%%%%%%%%%%%%%%%%%%%%%%%%%%%%%%%%%%%%%%%%%%%%%%%%%%%%%%%%%%%%%%%%
%
If a map is $\ssc^k$ we can deduce some classical smoothness properties.

%%%%%%%%%%%%%%%%%%%%%%%%%%%%%%%%%%%%%%%%%%%%%%%%%%%%%%%%%%%%%%%%%%%%%%%%%%%%%%%%
%%%%%%%%%%                          PROPOSITION                        %%%%%%%%%
%%%%%                                                                       %%%%
\begin{proposition} [{Proposition 2.3, \cite{HWZ8.7}}] 
  \label{lower}\index{sc-smoothness versus $C^k$}
  \hfill
Let $U$ and $V$ be relatively op\-en subsets of partial quadrants in
  sc-Banach spaces.
If $f\colon U\to  V$ is $\ssc^k$, then for every $m\geq 0$,  the map
  $f\colon U_{m+k}\to  V_m$ is of class $C^k$.
Moreover,  $f\colon U_{m+l}\to  V_m$ is of class $C^l$ for every $0\leq
  l\leq k$.
\end{proposition}
%%%%%                                                                       %%%%
%%%%%%%%%%                                                             %%%%%%%%%
%%%%%%%%%%%%%%%%%%%%%%%%%%%%%%%%%%%%%%%%%%%%%%%%%%%%%%%%%%%%%%%%%%%%%%%%%%%%%%%%
%

%%%%%%%%%%%%%%%%%%%%%%%%%%%%%%%%%%%%%%%%%%%%%%%%%%%%%%%%%%%%%%%%%%%%%%%%%%%%%%%%
%%%%%%%%%%                            REMARK                           %%%%%%%%%
%%%%%                                                                       %%%%
\begin{remark}
We also note that a map $f:U\rightarrow F$ which is level-wise classically
  smooth is sc$^\infty$.
\end{remark}
%%%%%                                                                       %%%%
%%%%%%%%%%                                                             %%%%%%%%%
%%%%%%%%%%%%%%%%%%%%%%%%%%%%%%%%%%%%%%%%%%%%%%%%%%%%%%%%%%%%%%%%%%%%%%%%%%%%%%%%
%

The next result is very useful in proving that a given map between
  sc-Banach spaces is sc-smooth provided it has certain classical smoothness
  properties.

%%%%%%%%%%%%%%%%%%%%%%%%%%%%%%%%%%%%%%%%%%%%%%%%%%%%%%%%%%%%%%%%%%%%%%%%%%%%%%%%
%%%%%%%%%%                           THEOREM                           %%%%%%%%%
%%%%%                                                                       %%%%
\begin{theorem}[{Proposition 2.4, \cite{HWZ8.7}}] 
  \label{ABC-x}
  \hfill
Let $U$ be a relatively open subset of a partial  quadrant in  a sc-Banach
  space $E$ and let $F$ be another $\ssc$-Banach space.
Assume that for every $m\geq 0$ and $0\leq l\leq k$,  the map $f\colon
  U\rightarrow V$  induces  a map
  $$
    f\colon U_{m+l}\rightarrow  F_m,
    $$
  which is of  class $C^{l+1}$. Then $f$ is $\ssc^{k+1}.$
\end{theorem}
%%%%%                                                                       %%%%
%%%%%%%%%%                                                             %%%%%%%%%
%%%%%%%%%%%%%%%%%%%%%%%%%%%%%%%%%%%%%%%%%%%%%%%%%%%%%%%%%%%%%%%%%%%%%%%%%%%%%%%%
%
In the case that the target space $F= \R^N$, Theorem \ref{ABC-x} takes the
  following form.

%%%%%%%%%%%%%%%%%%%%%%%%%%%%%%%%%%%%%%%%%%%%%%%%%%%%%%%%%%%%%%%%%%%%%%%%%%%%%%%%
%%%%%%%%%%                          COROLLARY                           %%%%%%%%
%%%%%                                                                       %%%%
\begin{corollary} [{Corollary 2.5, \cite{HWZ8.7}}] \label{ABC-y}
  \hfill\\
Let $U$ be a  relatively op\-en subset of a partial quadrant in a
  sc-Banach space and  $f\colon U\to  \R^N$.
If for some $k$ and all $0\leq l \leq k$ the map $f\colon 
  U_l \to  \R^N$ belongs to  $C^{l +1}$, then $f$ is $\ssc^{k+1}$.
\end{corollary}
%%%%%                                                                       %%%%
%%%%%%%%%%                                                             %%%%%%%%%
%%%%%%%%%%%%%%%%%%%%%%%%%%%%%%%%%%%%%%%%%%%%%%%%%%%%%%%%%%%%%%%%%%%%%%%%%%%%%%%%
%

%%%%%%%%%%%%%%%%%%%%%%%%%%%%%%%%%%%%%%%%%%%%%%%%%%%%%%%%%%%%%%%%%%%%%%%%%%%%%%%%
%%%%%%%%%%                         SUB-SECTION                         %%%%%%%%%
%%%%%%%%%%%%%%%%%%%%%%%%%%%%%%%%%%%%%%%%%%%%%%%%%%%%%%%%%%%%%%%%%%%%%%%%%%%%%%%%
%
\subsection{The Fundamental Lemma} \label{FUND0}

The following results are taken from \cite{HWZ8.7}.
We begin by introducing several sc-Hilbert spaces.  
We denote by $L$ the sc-Hilbert space $L^2({\mathbb R}\times S^1, {\mathbb
  R}^N)$ equipped  the sc-structure $(L_m)_{m\in {\mathbb N}}$ defined by
  $L_m=H^{m,\delta_m}({\mathbb R} \times S^1,{\mathbb R}^N)$, where
  $(\delta_m)$ is a strictly increasing sequence starting at $\delta_0=0$.
That means that $L_m$ consists of all maps having partial derivatives up
  to order $m$  weighted by $e^{\delta_m |s|}$ belonging to $L^2$.

We  also introduce the sc-Hilbert spaces $F=H^{2,\delta_0}({\mathbb
  R}\times S^1,{\mathbb R}^N)$ with sc-structure, where level $m$
  corresponds to regularity $(m+2,\delta_m)$, with in this case
  $\delta_0>0$ and $(\delta_m)$ being a strictly increasing sequence.
Finally we introduce $E=H^{3,\delta_0}({\mathbb R}\times S^1,{\mathbb
  R}^N)$ with level $m$ corresponding to regularity $(m+3,\delta_m)$ and
  $(\delta_m)$  as in the $F$-case.
We shall  use the  so-called exponential gluing profile 
  $$
    \varphi (r)=e^{\frac{1}{r}}-e,  \quad \text{$r\in (0,1]$. }
    $$
The map $\varphi$ defines a diffeomorphism $(0,1]\rightarrow [0,\infty)$.
With the  nonzero complex number $a$ (gluing parameter) with $0<|a|<1$  we
  associate the gluing angle $\vartheta\in S^1$ and the gluing length $R$
  via  the formulae $$a=\abs{a}\cdot e^{2\pi i \vartheta}\quad
  \text{and}\quad R=\varphi (\abs{a}).$$
Note that $R\to \infty$ as $\abs{a}\to 0$.

The following two lemmata have many applications. 
In particular, they will be used to prove that  the transition maps
  between the local M-polyfolds which we shall construct later are
  sc-smooth.
The underlying result is Proposition 2.8 from \cite{HWZ8.7}. We have split
  the result into two parts.
The interested reader may consult the above-mentioned reference for
  proofs.  
We denote by  $B_{\frac{1}{2}}$  the manifold of  complex numbers $a$ with
  $|a|<1/2$.

%%%%%%%%%%%%%%%%%%%%%%%%%%%%%%%%%%%%%%%%%%%%%%%%%%%%%%%%%%%%%%%%%%%%%%%%%%%%%%%%
%%%%%%%%%%                            LEMMA                            %%%%%%%%%
%%%%%                                                                       %%%%
\begin{lemma} [Fundamental Lemma I]\label{FUNDAMENTAL -I}
  \hfill\\
The following two maps are sc-smooth, where $f:{\mathbb R}\rightarrow
  {\mathbb R}$ is a smooth map which is constant outside of a compact set
  satisfying $f(+\infty)=0$ with possibly $f(-\infty)\neq 0$.  
We shall consider maps 
  $$
    \Gamma_i:B_\frac{1}{2}\oplus G\to  G,\ \ i=1,2
    $$
  where $G=L$, $G=F$ or $G=E$ and which we shall introduce below. 
We abbreviate $R=\varphi(|a|)$ for $a\neq 0$ which is a function of
  nonzero $a\in B_{\frac{1}{2}}$.
\begin{figure}[h]
  \begin{center}
  \includegraphics[width=3.5cm]{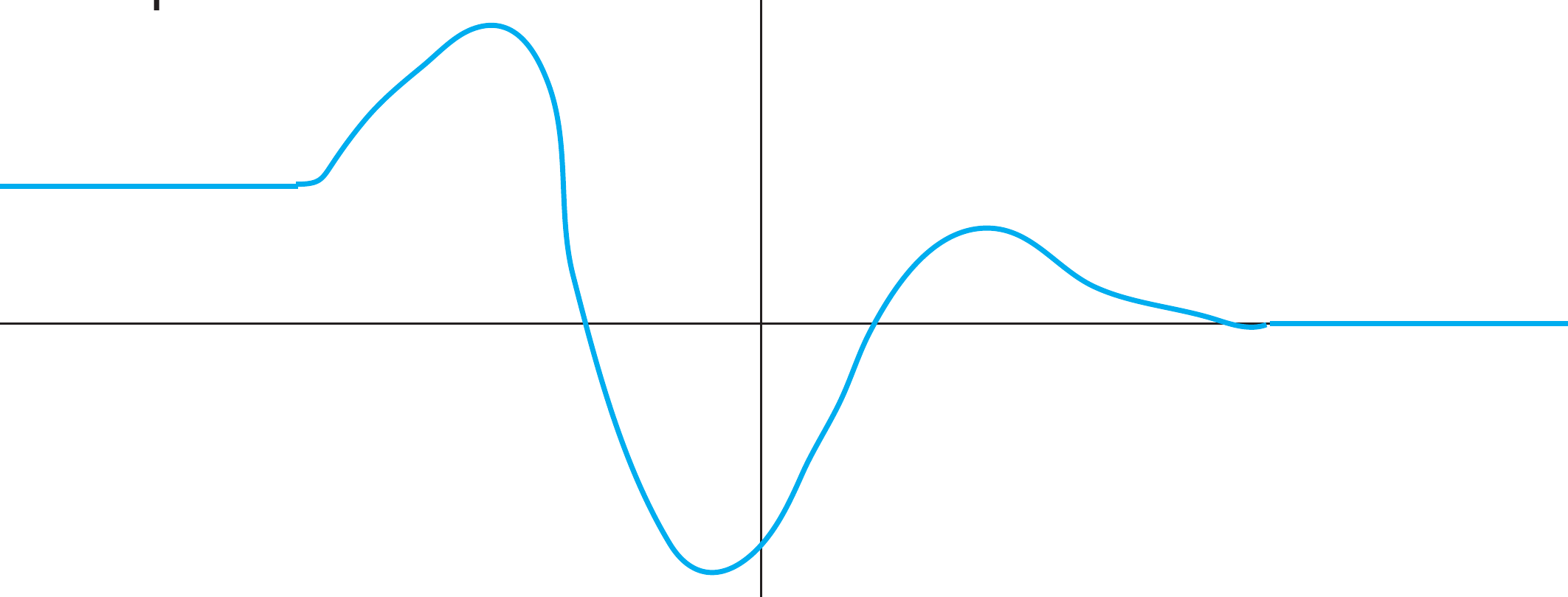}
  \end{center}
  \caption{The map $f$.}
  \end{figure}

\begin{itemize}
  \item[(1)]
  Define
  $$
    \Gamma_1(a,h)(s, t)=f\left(s -\frac{R}{2}\right)h (s, t)  
    $$
  if $a\neq 0$ and $\Gamma_1(0,h)=f(-\infty)h$ if $a=0$. \\
  \item[(2)]  Define 
  $$
    \Gamma_2(a,h)(s', t')=f\left(-s' -\frac{R}{2}\right)h(s',t')
    $$
  if $a\neq 0$ and $\Gamma_2(0,h)=f(-\infty)h$ if $a=0$.
  \end{itemize}

\end{lemma}
%%%%%                                                                       %%%%
%%%%%%%%%%                                                             %%%%%%%%%
%%%%%%%%%%%%%%%%%%%%%%%%%%%%%%%%%%%%%%%%%%%%%%%%%%%%%%%%%%%%%%%%%%%%%%%%%%%%%%%%
%

%%%%%%%%%%%%%%%%%%%%%%%%%%%%%%%%%%%%%%%%%%%%%%%%%%%%%%%%%%%%%%%%%%%%%%%%%%%%%%%%
%%%%%%%%%%                            LEMMA                            %%%%%%%%%
%%%%%                                                                       %%%%
\begin{lemma} [Fundamental Lemma II]\label{FUNDAMENTAL -II}
  \hfill
The following two maps are sc-smooth, where $g:{\mathbb R}\rightarrow
  {\mathbb R}$ is a smooth compactly supported map
  $$
    \Gamma_i:B_\frac{1}{2}\oplus G\to  G,\ \ i=3,4
    $$
  where $G=L$, $G=F$ or $G=E$. 
We shall use the function $R$ as before.
\begin{figure}[h]
  \begin{center}
  \includegraphics[width=3.5cm]{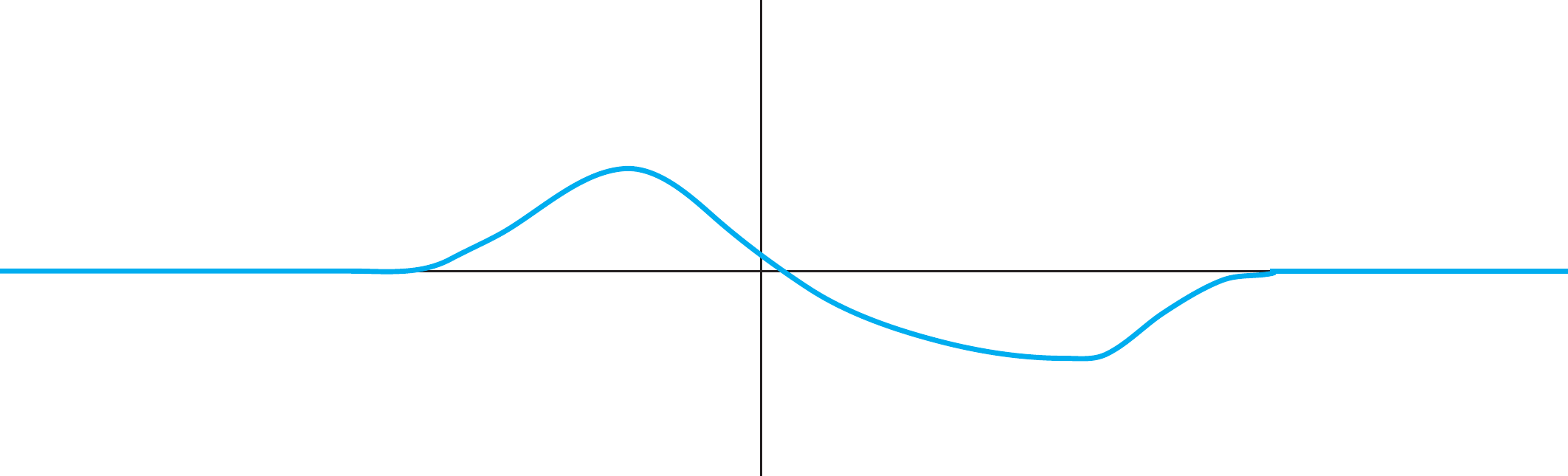}
  \end{center}
  \caption{The map $g$.}
  \end{figure}
\begin{itemize}
  \item[(1)] 
  Define  
  $$
    \Gamma_3(a,h)(s, t)=g\left(s -\frac{R}{2}\right)h(s-R,t-\vartheta)
    $$
  if  $a\neq 0$ and $\Gamma_2(0,h)=0$ if $a=0$.
  \item[(2)]  
  Define  
  $$
    \Gamma_4(a,h)(s', t')=g\left(-s' -\frac{R}{2}\right)h(s'+R,t'+\vartheta)
    $$
  if $a\neq 0$ and $\Gamma_4(0,h)=0$  if  $a=0$.
  \end{itemize}

\end{lemma}
%%%%%                                                                       %%%%
%%%%%%%%%%                                                             %%%%%%%%%
%%%%%%%%%%%%%%%%%%%%%%%%%%%%%%%%%%%%%%%%%%%%%%%%%%%%%%%%%%%%%%%%%%%%%%%%%%%%%%%%
%

Using these two results we shall be able to study several maps which will
  be important later on.
We denote by $\delta$ a strictly increasing sequence of real numbers
  $0\leq \delta_0<\delta_1<...$ and by $H^{3,\delta}_c({\mathbb R}^+\times
  S^1,{\mathbb R}^{N})$ the sc-Hilbert space which consists of maps $u$
  such that there exists a constant $c\in {\mathbb R}^N$ (called {\bf
  asymptotic constant}\index{asymptotic constant}) for which $u-c$ belongs
  to $H^{3,\delta_0}({\mathbb R}^+\times S^1,{\mathbb R}^N)$.  
The level $m$ consists of maps $u$ such that $u-c$ belongs to
  $H^{3+m,\delta_m}$.
Similarly we can define $H^{3,\delta}_c({\mathbb R}^-\times S^1,{\mathbb
  R}^N)$.
The maps of interest  and the corresponding results are given as follows.

%%%%%%%%%%%%%%%%%%%%%%%%%%%%%%%%%%%%%%%%%%%%%%%%%%%%%%%%%%%%%%%%%%%%%%%%%%%%%%%%
%%%%%%%%%%                          PROPOSITION                        %%%%%%%%%
%%%%%                                                                       %%%%
\begin{proposition}[{\bf M1}]
The map 
  $$
    H^{3,\delta}_c({\mathbb R}^+\times S^1,{\mathbb R}^{N})\to  {\mathbb
    R}^{N}, \quad u\mapsto c
    $$
  which associates to  $u$  its  asymptotic constant $c$ is ssc-smooth.

\end{proposition}
%%%%%                                                                       %%%%
%%%%%%%%%%                                                             %%%%%%%%%
%%%%%%%%%%%%%%%%%%%%%%%%%%%%%%%%%%%%%%%%%%%%%%%%%%%%%%%%%%%%%%%%%%%%%%%%%%%%%%%%
%
\begin{proof}
The proof is trivial.
\end{proof}%%%%%%%%%%%%%%%%%%%%%%%%%%%%%%%%%%%%%%%%%%%%%%%%            END PROOF

For an element $r^+\in  H^{3,\delta}({\mathbb R}^+\times S^1,{\mathbb
  R}^{N})$ and a gluing parameter $a\in B_{\frac{1}{2}}$ we define
  ${[r^+]}_0=0$ if $a=0$, and if $a\neq 0$ with $R=\varphi(|a|)$
  \begin{eqnarray}\label{EQw8.10}
    {[r^+]}_a=\int_{S^1} r^+(R/2,t)\cdot dt
    \end{eqnarray}
We can use the following definition if $u\in H^{3,\delta}_c({\mathbb
  R}^+\times S^1,{\mathbb R}^{N})$.
Namely $[u]_0$ is the asymptotic constant and otherwise, i.e. $a\neq 0$,
  we use the same integral definition as in (\ref{EQw8.10}).

%%%%%%%%%%%%%%%%%%%%%%%%%%%%%%%%%%%%%%%%%%%%%%%%%%%%%%%%%%%%%%%%%%%%%%%%%%%%%%%%
%%%%%%%%%%                          PROPOSITION                        %%%%%%%%%
%%%%%                                                                       %%%%
\begin{proposition} [{\bf M2}]
  \hfill\\
The map 
  $$
    B_\frac{1}{2}\times H^{3,\delta}({\mathbb R}^+\times S^1,{\mathbb
    R}^{N})\to  {\mathbb R}^{N},\quad (a,r^+)\mapsto  {[r^+]}_a
    $$
  is sc-smooth.   
In view of {\bf (M1)} the same holds for the map 
  $$
    B_\frac{1}{2}\times H_c^{3,\delta}({\mathbb R}^+\times S^1,{\mathbb
    R}^{N})\to  {\mathbb R}^{N},\quad (a,u)\mapsto  {[u]}_a.
    $$
\end{proposition}
%%%%%                                                                       %%%%
%%%%%%%%%%                                                             %%%%%%%%%
%%%%%%%%%%%%%%%%%%%%%%%%%%%%%%%%%%%%%%%%%%%%%%%%%%%%%%%%%%%%%%%%%%%%%%%%%%%%%%%%
%
For a proof see \cite{HWZ8.7},  Lemma 2.19. 
We also note that there is a version for $r^-$ as well.\\

%%%%%%%%%%%%%%%%%%%%%%%%%%%%%%%%%%%%%%%%%%%%%%%%%%%%%%%%%%%%%%%%%%%%%%%%%%%%%%%%
%%%%%%%%%%                          PROPOSITION                        %%%%%%%%%
%%%%%                                                                       %%%%
\begin{proposition} [{\bf M3}]
  \hfill\\
Let $f:{\mathbb R}\rightarrow {\mathbb R}$ be a smooth function which is
  constant outside of a compact set such that $f(+\infty)=0$.
However, $f(-\infty)$ may be nonzero. 
Define $f_0(s)= f(-\infty)$ and $f_a(s)=f(s-R/2)$ with $R=R(a)$. 
Then the map 
  $$ 
    B_\frac{1}{2}\times H^{3,\delta}({\mathbb R}^+\times S^1,{\mathbb
    R}^{N})\to  H^{3,\delta}({\mathbb R}^+\times S^1,{\mathbb R}^{N}),
    \quad (a,r^+)\mapsto  f_a(\cdot ){[r^+]}_a.
    $$
  is sc-smooth.
\end{proposition}
%%%%%                                                                       %%%%
%%%%%%%%%%                                                             %%%%%%%%%
%%%%%%%%%%%%%%%%%%%%%%%%%%%%%%%%%%%%%%%%%%%%%%%%%%%%%%%%%%%%%%%%%%%%%%%%%%%%%%%%
%
\noindent The proof is similar to  \cite{HWZ8.7}, Lemma 2.20. \\

\vspace{0.2cm}

%%%%%%%%%%%%%%%%%%%%%%%%%%%%%%%%%%%%%%%%%%%%%%%%%%%%%%%%%%%%%%%%%%%%%%%%%%%%%%%%
%%%%%%%%%%                          PROPOSITION                        %%%%%%%%%
%%%%%                                                                       %%%%
\begin{proposition} [{\bf M4}] 
  \hfill\\
Let $f:{\mathbb R}\rightarrow {\mathbb R}$ be a smooth function which is
  constant outside of a compact set such that
  $f(+\infty)=0$ be as in {\bf (M3)} and  define $f_0(s)= f(-\infty)$ and
  $f_a(s)=f(s-R/2)$ with $R=R(a)$.
Then the map $B_\frac{1}{2}\oplus H^{3,\delta}({\mathbb R}^+\times
  S^1,{\mathbb R}^{N})\to  H^{3,\delta}({\mathbb R}^+\times S^1,{\mathbb
  R}^{N}),$
  $$
    (a,r^+)\mapsto f_a(.) \cdot r^+
    $$
  is sc-smooth.
\end{proposition}
%%%%%                                                                       %%%%
%%%%%%%%%%                                                             %%%%%%%%%
%%%%%%%%%%%%%%%%%%%%%%%%%%%%%%%%%%%%%%%%%%%%%%%%%%%%%%%%%%%%%%%%%%%%%%%%%%%%%%%%
%
\noindent The proof is similar to  \cite{HWZ8.7}, Lemma 2.21.\\

\vspace{0.2cm}

%%%%%%%%%%%%%%%%%%%%%%%%%%%%%%%%%%%%%%%%%%%%%%%%%%%%%%%%%%%%%%%%%%%%%%%%%%%%%%%%
%%%%%%%%%%                          PROPOSITION                        %%%%%%%%%
%%%%%                                                                       %%%%
\begin{proposition} [{\bf M5}]
  \hfill\\
Let $g:{\mathbb R}\rightarrow {\mathbb R}$ be a smooth compactly supported
  map.
Then for a suitable $\varepsilon\in (0,1/2)$ the map $B_{\varepsilon}
  \oplus H^{3,\delta}({\mathbb R}^-\times S^1,{\mathbb R}^{N})\to
  H^{3,\delta}({\mathbb R}^+\times S^1,{\mathbb R}^{N})$,
  $$
    (a,r^-)\mapsto  g_a(.)\cdot  r^-(\cdot -R,\cdot -\vartheta),
    $$
  where for $s>R$ we define $ g_a(s)\cdot  r^-(s -R,t -\vartheta)=0$, is
  well-defined and  sc-smooth.
\end{proposition}
%%%%%                                                                       %%%%
%%%%%%%%%%                                                             %%%%%%%%%
%%%%%%%%%%%%%%%%%%%%%%%%%%%%%%%%%%%%%%%%%%%%%%%%%%%%%%%%%%%%%%%%%%%%%%%%%%%%%%%%
%
\noindent The proof is similar to \cite{HWZ8.7}, Lemma 2.22.

%%%%%%%%%%%%%%%%%%%%%%%%%%%%%%%%%%%%%%%%%%%%%%%%%%%%%%%%%%%%%%%%%%%%%%%%%%%%%%%%
%%%%%%%%%%                           SECTION                           %%%%%%%%%
%%%%%%%%%%%%%%%%%%%%%%%%%%%%%%%%%%%%%%%%%%%%%%%%%%%%%%%%%%%%%%%%%%%%%%%%%%%%%%%%
% 
\section{Nodal Constructions}
  \label{SECX1}\label{SEC_nodal_constructions}
Although the notions we need about nodal Riemann surfaces are standard,
  for the convenience of the reader, they are summarized
  in Appendix \jwf{[broken reference]} %\ref{APP5.1}.  

%%%%%%%%%%%%%%%%%%%%%%%%%%%%%%%%%%%%%%%%%%%%%%%%%%%%%%%%%%%%%%%%%%%%%%%%%%%%%%%%
%%%%%%%%%%                         SUB-SECTION                         %%%%%%%%%
%%%%%%%%%%%%%%%%%%%%%%%%%%%%%%%%%%%%%%%%%%%%%%%%%%%%%%%%%%%%%%%%%%%%%%%%%%%%%%%%
%
\subsection{Basic Construction}\label{SSEC2.1}

Before reading the following the reader should have a quick glance at
  Appendix \jwf{[broken reference because helmut didn't hand over all the
  files]} %\ref{APP5.1}.

%%%%%%%%%%%%%%%%%%%%%%%%%%%%%%%%%%%%%%%%%%%%%%%%%%%%%%%%%%%%%%%%%%%%%%%%%%%%%%%%
%%%%%%%%%%                         SUB-SECTION                         %%%%%%%%%
%%%%%%%%%%%%%%%%%%%%%%%%%%%%%%%%%%%%%%%%%%%%%%%%%%%%%%%%%%%%%%%%%%%%%%%%%%%%%%%%
%
\subsubsection{The Basic Idea}\label{BASICidea}
The constructions in Section \ref{SECX1} are concerned with a smooth
description of maps defined on an annulus type Riemann surface which
  decomposes into a nodal disk as the modulus tends to $\infty$. 
Modulo technicalities the following describes some of the ingredients.
Consider the two half-cylinders ${\mathbb R}^\pm\times S^1$. 
Given a number $R>0$ we can consider the subsets $[0,R]\times S^1\subset
  {\mathbb R}^+\times S^1$ and $[-R,0]\times S^1\subset {\mathbb R}^-\times
  S^1$. 
For a given $d\in S^1$ and $R>0$ we consider the set $Z_{(R,d)}$
  \index{$Z_{(R,d)}$} consisting of all $\{(s,t),(s',t')\}$ such that
  $(s,t)\in [0,R]\times S^1,\ (s',t')\in [-R,0]\times S^1,\ s=s'+R,\
  t=t'+d$.
We have two natural bijections
  \begin{eqnarray}
    &[0,R]\times S^1\leftarrow Z_{(R,d)}\rightarrow [-R,0]\times S^1&\\
    &(s,t)\leftarrow \{(s,t),(s',t')\}\rightarrow (s',t').&\nonumber
    \end{eqnarray}
As $R\rightarrow \infty$ we could view ${\mathbb R}^+\times S^1 \coprod
  {\mathbb R}^-\times S^1$ as  some kind of limit domain of $Z_{(R,d)}$.
We note that other limits  are possible, for example  we could also keep
  track of $d$ as $R\rightarrow\infty$.
Such a variant  will be important  for the periodic orbit case which will
  be studied later.
Fix a smooth map $\beta:{\mathbb R}\rightarrow [0,1]$ satisfying
  \begin{itemize}
    \item[(1)] $\beta(s)=1$ for $s\leq -1$.
    \item[(2)] $\beta'(s)<0$ for $s\in (-1,1)$.
    \item[(3)] $\beta(s)+\beta(-s)=1$ for all $s$.
    \end{itemize}
First consider pairs $(u^+,u^-)$ of continuous maps $u^\pm:{\mathbb
  R}^\pm\times S^1\rightarrow {\mathbb R}^N$ such that the two limits 
  $$
    \lim_{s\rightarrow\pm\infty} u^\pm(s,t)  =:u_{\pm\infty}
    $$
  exist in ${\mathbb R}^N$ uniformly in $t$, independent of $t$ and
  satisfy
  $$
    u_\infty= u_{-\infty}.
    $$
Let us refer to $u_{\pm\infty}$ as the (common) {\bf nodal
  value}\index{common nodal value} associated to $u^\pm$. Given $(R,d)$ we
  can define a continuous map
  $$
    \oplus(R,d,u^+,u^-):Z_{(R,d)}\rightarrow {\mathbb R}^N
    $$
  via
  \begin{eqnarray}
    &&\oplus(R,d,u^+,u^-)(\{(s,t),(s',t')\}) \\
    &=& \beta(s-R/2)\cdot u^+(s,t) + \beta(-s'-R/2) \cdot
    u^-(s',t').\nonumber
    \end{eqnarray}
This is a gluing construction, where we construct from two maps defined on
  complementary cylinders a map on a finite cylinder.
If $R>>0$ we note that the restriction of $\oplus(R,d,u^+,u^-)$ to the
  middle loop
  $$
    t\rightarrow \{(R/2,t),(-R/2,t-d)\}
    $$
  is almost a constant loop, very close to the common nodal value.
Given a map $w:Z_{(R,d)}\rightarrow {\mathbb R}^N$ we can construct two
  maps $w^\pm$ as follows. 
First we consider the mean value around the loop in the middle of
  $Z_{(R,d)}$
  $$
    \text{av}(w) =\int_{S^1} w(\{(R/2,t),(-R/2,t-d)\} )dt,
    $$
  and then define on ${\mathbb R}^\pm\times S^1$
  \begin{eqnarray}\label{qwas6.3}
    w^+(s,t) &=& \beta(s-R/2-2)\cdot w(\{(s,t),(s-R,t-d)\}) \\
    &&+ (1-\beta(s-R/2-2))\cdot \text{av}(w)\nonumber \\
    w^-(s',t')&=& \beta(-s'-R/2-2)\cdot w(\{(s,t),(s-R,t-d)\})\nonumber\\
    && + (1-\beta(-s'-R/2-2))\cdot \text{av}(w).\nonumber
    \end{eqnarray}
Define $f(w):=(R,d,w^+,w^-)$ and note that $(R,d)$ is a function of $w$,
  i.e. its domain parameter.
One easily verifies
  $$
  \oplus\circ f(w) = w,
  $$
  which immediately implies that $r:=f\circ \oplus$ satisfies $r\circ
  r=r$. 
Having $\oplus$-constructions in mind this should sound familiar, see
  Section \jwf{[broken reference]} %\ref{APP777}.
The current section is devoted to exploit the differential geometric
  content within the sc-smooth world of the above discussion.
For example we need to be precise about the regularity of the maps and one
  needs to suitably compactify the parameter set consisting of all $(R,d)$.
This will happen next.

%%%%%%%%%%%%%%%%%%%%%%%%%%%%%%%%%%%%%%%%%%%%%%%%%%%%%%%%%%%%%%%%%%%%%%%%%%%%%%%%
%%%%%%%%%%                         SUB-SECTION                         %%%%%%%%%
%%%%%%%%%%%%%%%%%%%%%%%%%%%%%%%%%%%%%%%%%%%%%%%%%%%%%%%%%%%%%%%%%%%%%%%%%%%%%%%%
%
\subsubsection{A Construction Functor}\label{SSEC2.1.1}
Denote by ${\mathcal D}=(D_x\sqcup D_y,\{x,y\})$ an unordered nodal disk
  pair, write ${\mathbb B}$ for  the complex manifold of associated  natural
  gluing parameters, which we recall consists of all formal expressions
  $r\cdot [\wh{x},\wh{y}]$ with $0\leq r< 1/4$.
Further denote by  $\varphi$  the exponential gluing profile. 
For $a\in {\mathbb B}$ we denote by $Z_a$ the glued Riemann surface.
A convenient definition is given in  Appendix \jwf{[broken reference]}
  %\ref{APP5.1},  which simplifies formulae.

For $\delta_0\in (0,\infty)$ we define by $H^{3,\delta_0}_c({\mathcal
  D},{\mathbb R}^N)$\index{$H^{3,\delta_0}_c({\mathcal D},{\mathbb R}^N)$}
  the Hilbert space, whose elements $(u^x,u^y)$ are maps of class
  $(3,\delta_0)$, see Appendix \jwf{[broken reference]} % \ref{SSSEXC16},
  with matching asymptotic constant, i.e. the values taken by the map at
  $x$ and $y$ are the same.

Given a strictly increasing sequence $\delta=(\delta_i)\subset {\mathbb
  R}$  with $\delta_0>0$  it holds that $H^{3,\delta_0}_c({\mathcal
  D},{\mathbb R}^N)$ has an sc-Hilbert space structure for which level $m$
  corresponds to regularity $(m+3,\delta_m)$.
Equipped with this structure we denote it by
  $H^{3,\delta}_c=H^{3,\delta}_c({\mathcal D},{\mathbb R}^N)$.
Then ${\mathbb B}\times H^{3,\delta}_c({\mathcal D},{\mathbb R}^N)$ is an
  ssc-manifold.
Denote by $H^3(Z_a,{\mathbb R}^N)$ the obvious Sobolev space.
We define (for the moment as a set) the disjoint union
  \begin{eqnarray}\label{plot007}
    &\ \ \ \ X^{3,\delta_0}_{{\mathcal D},\varphi}\equiv
    X^{3,\delta_0}_{{\mathcal D},\varphi}({\mathbb R}^N)=
    H^{3,\delta_0}_c({\mathcal D},{\mathbb
    R}^N)\coprod\left(\coprod_{0<|a|<1/4} H^3(Z_a,{\mathbb R}^N)\right)&
    \end{eqnarray}
  and denote by \index{$X^{3,\delta_0}_{{\mathcal D},\varphi}$}
  $$
    p_{\mathbb B}:X^{3,\delta_0}_{{\mathcal D},\varphi}\rightarrow
    {\mathbb B}
    $$
  the obvious map, extracting the domain parameter, where obviously
  ${\mathcal D}\rightarrow 0$.
Our goal is the construction of a natural M-polyfold structure on
  $X^{3,\delta_0}_{{\mathcal D},\varphi}$ which can be viewed as a sc-smooth
  completion of the space of maps on finite cylinders (or annuli), where the
  cylinders become infinitely long, i.e, decompose into a nodal disk pair.
What follows is a M-polyfold construction via the $\oplus$-method.
Pick any smooth map $\beta:{\mathbb R}\rightarrow [0,1]$ satisfying
  $\beta(s)=1$ for $s\leq -1$, $\beta(s)+\beta(-s)=1$, and $\beta'(s)<0$ for
  $s\in (-1,1)$ and define the plus-gluing
  $$
    \oplus: {\mathbb B}\times H^{3,\delta}_c\rightarrow
    X^{3,\delta_0}_{{\mathcal D},\varphi}
    $$
  as follows, where we use the models for gluing of Riemann surfaces in
  Appendix \jwf{[broken reference]} %\ref{APP5.1}.
If $a=0$ we put $\oplus(0,(u^x,u^y))=(u^x,u^y)$.  
For $a\neq 0$, $a=|a|\cdot [\wh{x},\wh{y}]$,  with $R=\varphi(|a|)$ and
  $h_{\wh{x}}$ and $h_{\wh{y}}$, see Appendix \jwf{[Broken reference]} 
  %\ref{APP5.1},  
  where $\{\wh{x},\wh{y}\}$ is a representative of $[\wh{x},\wh{y}]$,
  $\{z,z'\}\in Z_a$ so that   $h_{\wh{x}}(z)\cdot h_{\wh{y}}(z')=e^{-2\pi
  R}$, we define $\oplus(a,(u^x,u^y)):Z_a\rightarrow {\mathbb R}^N$ by
  \begin{eqnarray}\label{OPL1}
    \oplus(a,(u^x,u^y))((z,z'))&= & \beta(s_x(z)- R/2)\cdot u^x(z)\\
    && + \beta(s_y(z')-R/2))\cdot u^y(z'), \nonumber
    \end{eqnarray}
  where  $s_x(z)=-\frac{1}{2\pi}\cdot \ln(|h_{\wh{x}}(z)|)$ and
  $s_y(z')=-\frac{1}{2\pi}\cdot \ln(|h_{\wh{y}}(z')|)$.
We note that $\beta(s_x(z)- R/2)$ is a function of $a$ and $z$,  with
  $R=R(a)$.
The same holds for  $\beta(s_y(z')- R/2)$.  
Therefore we define 
  \begin{eqnarray}\label{EQNY2}\label{EQ_beta_a_xy}
    \beta^x_a(z)= \beta(s_x(z)- R/2)\ \ \text{and}\ \
    \beta^y_a(z')=\beta(s_y(z')- R/2)
    \end{eqnarray}
  and rewrite (\ref{OPL1}) as
  \begin{eqnarray}
    \oplus(a,(u^x,u^y))((z,z'))=  \beta^x_a(z)\cdot u^x(z) +
    \beta^y_a(z')\cdot u^y(z').
    \end{eqnarray}

%%%%%%%%%%%%%%%%%%%%%%%%%%%%%%%%%%%%%%%%%%%%%%%%%%%%%%%%%%%%%%%%%%%%%%%%%%%%%%%%
%%%%%%%%%%                            REMARK                           %%%%%%%%%
%%%%%                                                                       %%%%
\begin{remark}
  \hfill\\
We note that near the two boundary components there are suitable
  concentric annuli such that $\oplus$ acts as the identity over these
  annuli.
More precisely $\oplus(a,(u^x,u^y))(z,z')=u^x(z)$ for $z$ near $\partial
  D_x$ and $\oplus(a,(u^x,u^y))(z,z')=u^y(z')$ for $z'$ near $\partial D_y$.
\end{remark}
%%%%%                                                                       %%%%
%%%%%%%%%%                                                             %%%%%%%%%
%%%%%%%%%%%%%%%%%%%%%%%%%%%%%%%%%%%%%%%%%%%%%%%%%%%%%%%%%%%%%%%%%%%%%%%%%%%%%%%%
%

We note that we have the commutative diagram
  $$
    \begin{CD}
    {\mathbb B}\times H^{3,\delta}_c({\mathcal D},{\mathbb R}^N)@>\oplus
    >> X^{3,\delta_0}_{{\mathcal D},\varphi}({\mathbb R}^N)\\
    @V \text{pr}_1 VV  @V p_{\mathbb B}VV\\
    {\mathbb B} @= {\mathbb B}
    \end{CD}
    $$
The following theorem shows that $\oplus:{\mathbb B}\times
  H^{3,\delta}_c({\mathcal D},{\mathbb R})\rightarrow
  X^{3,\delta_0}_{{\mathcal D},\varphi}({\mathbb R}^N)$ is a
  $\oplus$-polyfold construction.
In addition we shall establish additionally properties which are useful.

%%%%%%%%%%%%%%%%%%%%%%%%%%%%%%%%%%%%%%%%%%%%%%%%%%%%%%%%%%%%%%%%%%%%%%%%%%%%%%%%
%%%%%%%%%%                           THEOREM                           %%%%%%%%%
%%%%%                                                                       %%%%
\begin{theorem} \label{thm1}
For every natural number $N\geq 1$ and  strictly increasing weight
  sequence $\delta$ starting with $\delta_0>0$, the set
  $X^{3,\delta_0}_{{\mathcal D},\varphi}({\mathbb R}^N)$ has a (uniquely
  defined) metrizable topology ${\mathcal T}$, as well as uniquely defined
  M-polyfold structure characterized by the  requirement that there exists
  a map $f:X^{3,\delta_0}_{{\mathcal D},\varphi}({\mathbb R}^N)\rightarrow
  {\mathbb B}\times H^{3,\delta_0}_c({\mathcal D},{\mathbb R}^N)$
  preserving the fibers over ${\mathbb B}$, i.e. we have the commutative
  diagram
  $$
    \begin{CD}
    X^{3,\delta_0}_{{\mathcal D},\varphi}({\mathbb R}^N)  @> f >> {\mathbb
    B}\times H^{3,\delta_0}_c({\mathcal D},{\mathbb R}^N)\\
    @V p VV @V pr_1VV\\
    {\mathbb B} @= {\mathbb B}
    \end{CD}
    $$
  such that 
  \begin{itemize}
  \item[(1)] 
  $\oplus\circ f =Id$.
  \item[(2)] 
  $f\circ \oplus$ as a map ${\mathbb B}\times H^{3,\delta}_c({\mathcal
  D},{\mathbb R}^N)\rightarrow {\mathbb B}\times H^{3,\delta}_c({\mathcal
  D},{\mathbb R}^N)$ is  sc-smooth.
  \end{itemize}
The M-polyfold structure (associated to the weight sequence $\delta$) on
  the set $X=X^{3,\delta_0}_{{\mathcal D},\varphi}({\mathbb R}^N)$ is
  denoted by $X^{3,\delta}_{{\mathcal D},\varphi}({\mathbb R}^N)$ and  has
  then the following additional properties, where we abbreviate the spaces
  by $X^{3,\delta}$ and $H^{3,\delta}_c$.
We note that the above is a more precise statement of the fact that
  $\oplus$ defines a $\oplus$-polyfold construction, f.e. the existence of a
  global $f$.
\begin{itemize}
  \item[(3)] 
  The M-polyfold structure on $X^{3,\delta}$ does not depend on the choice
  of $\beta$ nor on the choice of $f$ with the stated properties.
  \item[(4)] 
  A map $h:Y\rightarrow X^{3,\delta}$, where $Y$ is a M-polyfold, is
  sc-smooth if and only if $f\circ h:Y\rightarrow {\mathbb B}\times
  H^{3,\delta}_c$ is sc-smooth.
  \item[(5)] 
  A map $k:X^{3,\delta}\rightarrow Y$, where $Y$ is a M-polyfold is
  sc-smooth if and only if $k\circ \oplus :{\mathbb B}\times
  H^{3,\delta}_c\rightarrow Y$ is sc-smooth. In particular $f$ and
  $\oplus$ are sc-smooth.
  \end{itemize}

\end{theorem}
%%%%%                                                                       %%%%
%%%%%%%%%%                                                             %%%%%%%%%
%%%%%%%%%%%%%%%%%%%%%%%%%%%%%%%%%%%%%%%%%%%%%%%%%%%%%%%%%%%%%%%%%%%%%%%%%%%%%%%%
%
\begin{proof}
The main point is to show that $\oplus:{\mathbb B}\times
  H^{3,\delta}_c({\mathcal D},{\mathbb R}^N)\rightarrow
  X^{3,\delta_0}_{{\mathcal D},\varphi}({\mathbb R}^N) $ is
  $\oplus$-polyfold construction, where we can take a global $f$.
In view of this we have to construct $f$ and show the independence of the
  choice of $\beta$.
We define 
  \begin{eqnarray}
    f(u^x,u^y)=(0,(u^x,u^y)).
    \end{eqnarray}
If $a\in {\mathbb B}\setminus\{0\}$, say $a=|a|\cdot [\wh{x},\wh{y}]$, we
  are given a map $u:Z_a\rightarrow {\mathbb R}^N$.
Define $\text{av}(u)$, the average over the middle loop, with the help of
  a middle loop map  $\sigma_a:S^1\rightarrow Z_a$ introduced in Appendix
  \jwf{[broken reference]} \jwf{[broken reference]} %\ref{APP5.1},
  by 
  $$
    \text{av}(u) =\int_{S^1} u(\sigma_a(t)) dt.
    $$
We note that  the domain parameter is a function of $u$, i.e. $a=a(u)$.
This average does not depend on the choice of the middle loop. 
If $u=(u^x,u^y):Z_0\rightarrow {\mathbb R}^N$ we define
  $$
    \text{av}(u)= u^x(x)=u^y(y),
    $$
  which is the nodal value. 
In the following it happens very often that given a map on $Z_a$ we have
  to construct associated maps on $D_x$ and $D_y$ or vice a versa.
The elements of $Z_a$ are written $\{z,z'\}$ and satisfy
  $h_{\wh{x}}(z)\cdot h_{\wh{y}}(z')=e^{-2\pi R}$.
The elements of $D_x$ and $D_y$ are written $z$ and $z'$. 
There are, of course, elements $z$ and $z'$ which do not occur as an
  un-odered pair $\{z,z'\}$ in $Z_a$. As a consequence of this fact the
  formulae, which we shall write down, sometimes involve ingredients which
  might not be defined.
However, such occurrences always involve products where one of the
  well-defined expressions is $0$.
Hence our convention is that a non-defined value times a defined value
  zero takes the defined value zero.
With this in mind define if $a\neq 0$
  \begin{eqnarray}\label{Hf}
    f(u)=(a,(\eta^x,\eta^y))
    \end{eqnarray}as follows,
  where  $R=\varphi(|a|)$, and $\{z,z'\}\in Z_a$
  \begin{eqnarray*}                                                       %% EQN
    \eta^x(z) &=&\beta\left(s_x(z)- R/2 -2\right)\cdot u(\{z,z'\}) \\
    && +\big(1-\beta(s_x(z)- R/2-2)\big)\cdot \text{av}(u)\\
    &=:& \beta_{a,-2}^x(z)\cdot u(\{z,z'\}) + (1-\beta_{a,-2}^x(z))\cdot
    \text{av}(u)
    \end{eqnarray*}
  and 
  \begin{eqnarray*}                                                       %% EQN
    \eta^y(z') & =&\beta\left(s_y(z')- R/2-2\right)\cdot u(\{z,z'\})\\
    && + \left(1-\beta(s_y(z')- R/2-2)\right)\cdot\text{av}(u)\\
    &=:&\beta^y_{a,-2}(z')\cdot u(\{z,z'\}) +(1-\beta^y_{a,-2}(z'))\cdot
    \text{av}(u).
    \end{eqnarray*}
We also note that the pair has matching asymptotic constants, and for
  fixed $a$ the map $H^3(Z_a,{\mathbb R}^N)\rightarrow {\mathbb B}\times
  H^{3,\delta_0}_c({\mathcal D},{\mathbb R}^N)$ given by
  \begin{equation*}                                                       %% EQN
    u\mapsto \big(a,\eta^x(u),\eta^y(u)\big)   
    \end{equation*}
  is linear and obviously an sc-operator. It is elementary to establish that
  \begin{equation*}                                                      
    \oplus\circ f = Id.
    \end{equation*}
Indeed, using $\beta^x_a\cdot\beta^x_{a,-2}=\beta^x_a$ and similarly for
  the $y$-expression, we compute, observing that
  $\beta^x_a(z)+\beta^y_a(z')=1$ for $\{z,z'\}\in Z_a$
  \begin{eqnarray*}
    &&(\oplus\circ f(u))(\{z,z'\})\\
    &=& \oplus(a,\eta^x,\eta^y)(\{z,z'\})\\
    &=&\beta^x_a(z) \cdot \eta^x(z) +\beta^y_a(z')\cdot \eta^y(z') \\
    &=&\beta^x_a(z) \cdot \left(\beta_{a,-2}^x(z)\cdot u(\{z,z'\}) +
    (1-\beta_{a,-2}^x(z))\cdot \text{av}(u)\right)\\
    &&+\beta^y_a(z')\cdot\left(\beta^y_{a,-2}(z')\cdot u(\{z,z'\})
    +(1-\beta^y_{a,-2}(z'))\cdot \text{av}(u)\right)\\
    &=&\beta^x_a(z)\cdot u(z,z')+\beta^y_a(z')\cdot u(\{z,z'\})\\
    &=&u(\{z,z'\}).
   \end{eqnarray*}
Next we  show that \(f\circ \oplus\) is sc-smooth.
To that end, we write 
  $$
    (a,(\eta^x,\eta^y))=f(\oplus(a,(\xi^x,\xi^y))),
    $$
  so that $\eta^x$ is given as follows.
  \begin{eqnarray}\label{EQNH4}
    \eta^x(z)
    &=&\beta^x_{a,-2}(z)\cdot\left(\beta^x_a(z)\cdot\xi^x(z)+
    \beta^y_a(z')\cdot\xi^y(z')\right)\\
    && +(1-\beta^x_{a,-2}(z))\cdot
    \text{av}(\oplus(a,(\xi^x,\xi^y)))\nonumber\\
    &=& \beta^x_a(z)\cdot \xi^x(z) +
    (\beta^x_{a,-2}(z)\cdot\beta^y_a(z'))\cdot \xi^y(z')\nonumber\\
    &&+ (1/2)\cdot (1-\beta^x_{a,-2}(z))\cdot \left(
    \text{av}_a(\xi^x)+\text{av}_a(\xi^y)\right)\nonumber
    \end{eqnarray}
Here 
  $$
    \text{av}_a(\xi^x) =\int_{S^1} \xi^x(\sigma^x_a(t))\cdot dt\ \
    \text{and}\ \ \text{av}_a(\xi^y)=
    \int_{S^1} \xi^y(\sigma^y_a(t))\cdot dt,
    $$
  using {\bf $a$-loops}, for the definition  see the end of Subsection
  \jwf{[broken reference]} %\ref{APP5.1}.
Note that these averages do not depend on which $a$-loops were picked.
Similarly
  \begin{eqnarray}\label{EQNH10}
    \eta^y(z')&=& \beta^y_a(z')\cdot \xi^y(z')+(\beta^y_{a,-2}(z')\cdot
    \beta^x_a(z))\cdot \xi^x(z)\\
    &&(1/2)\cdot (1-\beta^y_{a,-2}(z'))\cdot
    (\av_a(\xi^y)+\text{av}_a(\xi^x)).\nonumber
    \end{eqnarray}
We note that $(0,(\xi^x,\xi^y))\rightarrow (\xi^x,\xi^y)$.
We need to show that the map
  \begin{eqnarray}\label{EQNR3}
    & {\mathbb B}\times H^{3,\delta}_c({\mathcal D},{\mathbb R}^N)\rightarrow
    H^{3,\delta}_c({\mathcal D},{\mathbb R}^N):(a,(\xi^x,\xi^y)) \rightarrow
    (\eta^x,\eta^y)&
    \end{eqnarray}
  is sc-smooth. 
By definition, the sc-Hilbert space $H^{3,\delta}_c({\mathcal D},{\mathbb
  R}^N)$ is sc-isomorphic to the codimension $N$ subspace $E$ of
  $H^{3,\delta}_c({\mathbb R}^+\times S^1,{\mathbb R}^N)\times
  H^{3,\delta}_c({\mathbb R}^-\times S^1,{\mathbb R}^N)$, consisting of
  elements with matching asymptotic constants. 
This isomorphism is given by the map
  \begin{align*}                                                          %% EQN
    &\Sigma: H_c^{3, \delta}(\mathcal{D}, \mathbb{R}^N) \to E
    \\
    &\Sigma(\xi^x,\xi^y)=
    (\xi^x\circ\sigma^+_{\wh{x}}\, ,\, \xi^y\circ\sigma^-_{\wh{y}}),
    \end{align*}
    where we fix some decorations $\wh{x}$ and $\wh{y}$, and we employ the
    functions
  \begin{align}                                                           %% EQN
    \sigma_{\widehat{x}}^+(s,t) &=
    h_{\widehat{x}}^{-1}\big(e^{-2\pi(s+it)}\big)\label{EQ_sig_xhat}
    \\ 
    \sigma_{\widehat{y}}^-(s',t') &=
    h_{\widehat{y}}^{-1}\big(e^{2\pi(s+it)}\big)\label{EQ_sig_yhat}
    \end{align}
  provided in Appendix \ref{SEC_holo_polar_coords}.
Hence, after conjugation with \(\Sigma\), the map in (\ref{EQNR3}) defines
  a map ${\mathbb B}\times E\rightarrow E$ and it suffices to show that it
  is sc-smooth.
It is also convenient to replace ${\mathbb B}$ by 
  ${\mathbb B}_{\mathbb C}=\{a\in {\mathbb C}\ |\ |a|<1/4\}$.  
We then need to re-express equations (\ref{EQNH4}) and (\ref{EQNH10}) in
  the E-setting (that is, via the conjugation by \(\Sigma\)) but we abuse
  notation by using the same symbols to denote the functions before and
  after conjugation.
In other words, we shall write \(\eta^x(s,t) = \eta^x\circ
  \sigma_{\widehat{x}}^+(s,t)\), and similarly for \(\eta^y\).
In order to proceed, we recall that \(\eta^x\) and \(\eta^y\) involve
  terms of the form \(\beta_a^x(z)\) and \(\beta_a^y(z')\), which are
  defined in equation (\ref{EQ_beta_a_xy}), and hence making use of the
  fact that
  \begin{align*}                                                          %% EQN
    z = \sigma_{\widehat{x}}^+(s,t) \qquad\text{and}\qquad z' =
    \sigma_{\widehat{y}}^-(s',t'),
    \end{align*}
  we find that
  \begin{align*}                                                          %% EQN
    s_x(\sigma_{\widehat{x}}^+(s,t)) 
    &= 
    -\frac{1}{2\pi} \ln \big|h_{\widehat{x}} \big( h_{\widehat{x}}^{-1}(
    e^{-2\pi(s+it)})\big)\big|\\
    &=
    s
    \end{align*}
  and 
  \begin{align*}                                                          %% EQN
    s_y(\sigma_{\widehat{y}}^-(s',t')) 
    &= 
    -\frac{1}{2\pi} \ln \big|h_{\widehat{y}} \big( h_{\widehat{y}}^{-1}(
    e^{2\pi(s'+it')})\big)\big|\\
    &=
    -s',
    \end{align*}
  so that 
  \begin{align*}                                                          %% EQN
    \beta_a^x(z)
    &=
    \beta_a^x\circ \sigma_{\widehat{x}}^+(s,t)
    \\
    &=
    \beta\big(s_x(\sigma_{\widehat{x}}^+(s,t)) -
    \textstyle{\frac{1}{2}}R\big)
    \\
    &=
    \beta(s - \textstyle{\frac{1}{2}}R)
    \end{align*}
  and
  \begin{align*}                                                          %% EQN
    \beta_a^y(z')
    &=
    \beta_a^y\circ \sigma_{\widehat{y}}^-(s',t')
    \\
    &=
    \beta\big(s_x(\sigma_{\widehat{y}}^+(s',t')) -
    \textstyle{\frac{1}{2}}R\big)
    \\
    &=
    \beta(-s' - \textstyle{\frac{1}{2}}R)
    \end{align*}
Next, given $(a,(\xi^x,\xi^y))\in {\mathbb B}_{\mathbb C}\times E$, with
  $R=\varphi(|a|)$, $a=|a|\cdot e^{-2\pi i\theta}$, we introduce the
  abbreviations
  \begin{align*}                                                          %% EQN
    \text{av}_a(\xi^x)&=\int_{S^1} \xi^x(R/2,t)\cdot dt\\
    \text{av}_a(\xi^y)&=\int_{S^1} \xi^y(-R/2,t)\cdot dt\\
    \text{av}_a(\xi^x,\xi^y)&=\frac{1}{2}\cdot
      \left(\text{av}_a(\xi^x)+\text{av}_a(\xi^y)\right),
    \end{align*}
  and define
  \begin{eqnarray}\label{EQN1}
    \sigma(s)=\beta(s-2)\cdot
    (1-\beta(s)).
    \end{eqnarray}
We also write $(\xi^x,\xi^y)=(c+r^x,c+r^y)$ with 
  \begin{align*}                                                          %% EQN
    c&\in \mathbb{R}^N\\
    r^x&\in H^{3, \delta_0}(\mathbb{R}^+\times S^1 , \mathbb{R}^N) 
    \\
    r^y &\in H^{3, \delta_0}(\mathbb{R}^-\times S^1 , \mathbb{R}^N).
    \end{align*}
Finally, recalling that 
  \begin{align*}                                                          %% EQN
    \beta(s)= \beta(s-2)\cdot \beta(s), \qquad s =
    s'+R\qquad\text{and}\qquad t = t'+\theta,
    \end{align*}
  we can re-express equations (\ref{EQNH4}) and (\ref{EQNH10}) in the
  E-setting as
  \begin{align}
    \eta^x(s,t)
    &
    = c + \beta\left(s-{\textstyle \frac{1}{2}}R\right) \cdot
    r^x(s,t) +\sigma\left(s-{\textstyle \frac{1}{2}}R\right) \cdot
    r^y(s-R,t-\theta) \\
    &
    \qquad + \left(1-\beta(s-{\textstyle \frac{1}{2}}R-2)\right)\cdot
    \text{av}_a(r^x, r^y)&\nonumber\\
    \eta^y(s',t')
    &
    = c + \beta\left(-s'+{\textstyle \frac{1}{2}}R\right) \cdot
    r^y(s,t) +\sigma\left(-s'-{\textstyle \frac{1}{2}}R\right) \cdot
    r^x(R-s',t'-\theta)\nonumber\\
    &
    \qquad + \left(1-\beta(-s'+{\textstyle \frac{1}{2}}R+2)\right)\cdot
    \text{av}_a(r^x, r^y).&\nonumber
    \end{align}
\jwf{ [I think there are typos in the above, which should instead be as
  follows]
\textcolor{red}{
\begin{align*}
    \eta^y(s',t')
    &
    = c + \beta\left(-s'-{\textstyle \frac{1}{2}}R\right) \cdot
    r^y(s',t') +\sigma\left(-s'-{\textstyle \frac{1}{2}}R\right) \cdot
    r^x(s'+R,t'+\theta)\nonumber\\
    &
    \qquad + \left(1-\beta(-s'+{\textstyle \frac{1}{2}}R+2)\right)\cdot
    \text{av}_a(r^x, r^y).&\nonumber
\end{align*}
}
}
%With these preliminaries established, and recalling that \(s = s'+R\),
%We start  with the data $(a,(\xi^x,\xi^y))\in {\mathbb B}_{\mathbb
%  C}\times E$, $R=\varphi(|a|)$, $a=|a|\cdot e^{-2\pi i\theta}$, and
%  introduce the abbreviations
%  \begin{eqnarray*}
%    & \text{av}_a(\xi^x)=\int_{S^1} \xi^x(R/2,t)\cdot dt,\ \
%    \text{av}_a(\xi^y)=\int_{S^1} \xi^y(-R/2,t)\cdot dt,\\
%    &\text{av}_a(\xi^x,\xi^y)=\frac{1}{2}\cdot
%    \left(\text{av}_a(\xi^x)+\text{av}_a(\xi^y)\right).&
%    \end{eqnarray*}
%Note that $\beta(s)= \beta(s-2)\cdot \beta(s)$ and  introduce the smooth
%  map $\sigma$  derived from $\beta$ by
%  \begin{eqnarray}\label{EQN1}
%    \sigma(s)=\beta(s-2)\cdot
%    (1-\beta(s)).
%    \end{eqnarray}
%Finally we shall write $(\xi^x,\xi^y)=(c+r^x,c+r^y)$, where $c$ is the
%  common asymptotic constant, and $r^x,r^y$ are the exponentially decaying
%  part. 
%We obtain the following formulae for $\eta^x$ and $\eta^y$

It suffices by symmetry to establish the sc-smoothness of the map
  ${\mathbb B}_{\mathbb C}\times E\rightarrow H^{3,\delta}_c({\mathbb
  R}^+\times S^1,{\mathbb R}^N):(a,(\xi^x,\xi^y))\rightarrow \eta^x$.
This map is the sum of several simpler maps which involve the following
  linear sc-operators defining sc-smooth maps:
\begin{itemize}
  \item[(i)] 
  Extracting the asymptotic constant 
  $$
    E\rightarrow {\mathbb R}^N\subset H^{3,\delta}_c({\mathbb R}^+\times
    S^1,{\mathbb R}^N):(\xi^x,\xi^y)\rightarrow c.
    $$
  \item[(ii)] 
  Extracting the exponential decaying part 
  \begin{eqnarray*}
    &E\rightarrow H^{3,\delta}({\mathbb R}^+\times S^1,{\mathbb
    R}^N):(\xi^x,\xi^y)\rightarrow r^x&\\
    &E\rightarrow H^{3,\delta}({\mathbb R}^-\times S^1,{\mathbb
    R}^N):(\xi^x,\xi^y)\rightarrow r^y.&
    \end{eqnarray*}
  \item[(iii)] Extracting averages 
  \begin{eqnarray*}
    &H^{3,\delta}({\mathbb R}^+\times S^1,{\mathbb R}^N)\rightarrow
    {\mathbb R}^N:r^y\rightarrow
    \int_{S^1} r^x(R/2,t)\cdot dt&\\
    &H^{3,\delta}({\mathbb R}^-\times S^1,{\mathbb R}^N)\rightarrow
    {\mathbb R}^N:r^y\rightarrow
    \int_{S^1} r^y(R/2,t)\cdot dt.
    \end{eqnarray*}
  \end{itemize}
As a consequence we see immediately that the map
  $$
    {\mathbb B}_{\mathbb C}\times E\rightarrow {\mathbb R}^N:
    (\xi^x,\xi^y)\rightarrow c+ \text{av}_a(r^x,r^y)
    $$
  associating the asymptotic constant of $\eta^x$ is sc-smooth.  
To complete the proof one only needs to establish the sc-smoothness of the
  following maps:
  \begin{itemize}
    \item[(1)] 
    ${\mathbb B}_{\mathbb C}\times H^{3,\delta}({\mathbb R}^+\times
    S^1,{\mathbb R}^N)\rightarrow H^{3,\delta}({\mathbb R}^+\times
    S^1,{\mathbb R}^N):r^x\rightarrow [(s,t)\rightarrow
    \beta\left(s-{\textstyle \frac{1}{2}}R\right) \cdot r^x(s,t)]$.
    \item[(2)] 
    ${\mathbb B}_{\mathbb C}\times H^{3,\delta}({\mathbb R}^-\times
    S^1,{\mathbb R}^N)\rightarrow H^{3,\delta}({\mathbb R}^+\times
    S^1,{\mathbb R}^N):r^y\rightarrow [(s,t)\rightarrow
    \sigma\left(s-{\textstyle \frac{1}{2}}R\right) \cdot
    r^y(s-R,t-\theta)]$.
    \item[(3)] 
    ${\mathbb B}_{\mathbb C}\times E\rightarrow  H^{3,\delta}({\mathbb
    R}^+\times S^1,{\mathbb R}^N):$\\ $ (a,(\xi^x,\xi^y))\rightarrow
    [(s,t)\rightarrow \beta(s-{\textstyle \frac{1}{2}}R-2)\cdot
    \text{av}_a(r^x,r^y)]$
    \end{itemize}
The sc-smoothness of these maps follows from a direct application of
  Proposition 2.8 and Proposition 2.17 in \cite{HWZ8.7}, or the results in
  Subsection \ref{FUND0}.
We conclude that indeed, \(f\circ \oplus\) is sc-smooth.

Assume we have two constructions using the smooth cut-off functions
  $\beta_1$ and $\beta_2$.
We need to show that $Id: X^1\rightarrow X^2$ and $Id:X^2\rightarrow X^1$
  are sc-smooth where $X^i$ is the obvious abbreviation. 
The two cases are, of course, treated similarly, and we provide details
  for the first one.
Associated to the two constructions we have the maps $\oplus^1,\oplus^2$
  as well as$f^1, f^2$.
By construction $Id:X^1\rightarrow X^2$ is sc-smooth if and only if
  $f^2\circ Id\circ  \oplus^1=f^2\circ\oplus^1$ is sc-smooth.
The new expressions analogue to (\ref{EQN1}) are similarly as in the
  $f\circ \oplus$-case, namely with $f^2$ being the analogue choice,
  \begin{eqnarray}
    \sigma(s)=\beta_2(s-2)\cdot
    (1-\beta_1(s)).
    \end{eqnarray}
Again applications of Proposition 2.8 in \cite{HWZ8.7} and Proposition
  2.17 in \cite{HWZ8.7} lead to the desired result.
Alternatively we can use the results from Subsection \ref{FUND0}.
\end{proof}%%%%%%%%%%%%%%%%%%%%%%%%%%%%%%%%%%%%%%%%%%%%%%%%            END PROOF

In order to simplify notation define $X(N)= X^{3,\delta}_{{\mathcal
  D},\varphi}({\mathbb R}^N)$, so that  $X(N)$ and $X(M)$ are the
  M-polyfolds associated to ${\mathbb R}^N$ and ${\mathbb R}^M$. 
Given a smooth map $h:{\mathbb R}^N\rightarrow {\mathbb R}^N$ define
  $h_\ast: X(N)\rightarrow X(M):u\rightarrow h\circ u$.

%%%%%%%%%%%%%%%%%%%%%%%%%%%%%%%%%%%%%%%%%%%%%%%%%%%%%%%%%%%%%%%%%%%%%%%%%%%%%%%%
%%%%%%%%%%                          PROPOSITION                        %%%%%%%%%
%%%%%                                                                       %%%%
\begin{proposition}\label{prop2}
The map $h_\ast:X(N)\rightarrow X(M)$ is sc-smooth.
\end{proposition}
%%%%%                                                                       %%%%
%%%%%%%%%%                                                             %%%%%%%%%
%%%%%%%%%%%%%%%%%%%%%%%%%%%%%%%%%%%%%%%%%%%%%%%%%%%%%%%%%%%%%%%%%%%%%%%%%%%%%%%%
%
\begin{proof} 
The map $h$ defines an ssc-smooth map ${\mathbb B}\times
  H^{3,\delta_0}_c({\mathcal D},{\mathbb R}^N)\rightarrow {\mathbb B}\times
  H^{3,\delta_0}_c({\mathcal D},{\mathbb R}^M)$ by $(a,u)\rightarrow
  (a,h\circ u)$, which follows from the level-wise classical results in
  \cite{El} on Fr\'echet differentiability.
In particular, the map is also sc-smooth and  we denote it by $h_\sharp$. 
Note that we have the commutative diagram, when we use for both situations
  the same cut-off and the previously given explicit example for $f$.
\begin{eqnarray}\label{DIAG1}
  \begin{CD}
  {\mathbb B}\times H^{3,\delta_0}_c({\mathcal D},{\mathbb R}^N) @>
  h_\sharp>> {\mathbb B}\times H^{3,\delta_0}_c({\mathcal D},{\mathbb
  R}^M)\\
  @A f^N AA         @V \oplus^M VV\\
  X(N) @> h_\ast >>   X(M).
  \end{CD}
\end{eqnarray}
Hence $h_\ast$ is the composition of sc-smooth maps and the result follows
  from the chain rule.
\end{proof}%%%%%%%%%%%%%%%%%%%%%%%%%%%%%%%%%%%%%%%%%%%%%%%%            END PROOF
%

%%%%%%%%%%%%%%%%%%%%%%%%%%%%%%%%%%%%%%%%%%%%%%%%%%%%%%%%%%%%%%%%%%%%%%%%%%%%%%%%
%%%%%%%%%%                            REMARK                           %%%%%%%%%
%%%%%                                                                       %%%%
\begin{remark}
The proof of Proposition \ref{prop2}  illustrates how we can verify if a
  map $h_{\ast}$ between `bad' spaces, i.e.  the M-polyfolds constructed by
  the $\oplus$-method, is sc-smooth.
Indeed, it can be checked by  studying the sc-smoothness of a map
  $h_{\sharp}$ between better spaces, and as is often the case such maps
  $h_{\ast}$ are classical contexts and have know properties.
In our case $h_{\sharp}$ is ssc-smooth and consequently als sc-smooth. 
The chain rule then implies sc-smoothness for $h_{\ast}$. 
These type of arguments, i.e. just writing down the right diagram and
  employing the chain rule will occur frequently.
\end{remark}
%%%%%                                                                       %%%%
%%%%%%%%%%                                                             %%%%%%%%%
%%%%%%%%%%%%%%%%%%%%%%%%%%%%%%%%%%%%%%%%%%%%%%%%%%%%%%%%%%%%%%%%%%%%%%%%%%%%%%%%
%

%%%%%%%%%%%%%%%%%%%%%%%%%%%%%%%%%%%%%%%%%%%%%%%%%%%%%%%%%%%%%%%%%%%%%%%%%%%%%%%%
%%%%%%%%%%                         SUB-SECTION                         %%%%%%%%%
%%%%%%%%%%%%%%%%%%%%%%%%%%%%%%%%%%%%%%%%%%%%%%%%%%%%%%%%%%%%%%%%%%%%%%%%%%%%%%%%
%
\subsubsection{Extension to Manifolds}
If $U$ is an open subset of ${\mathbb R}^N$, the subset of all $u\in X(N)$
  with image in $U$ is open and therefore has a M-polyfold structure.
Denote by $\mathfrak{M}$ the category of smooth manifolds without boundary
  and smooth maps between them.
Note that each connected component of such a manifold has a proper
  embedding into some ${\mathbb R}^N$.
From Theorem \ref{thm1} and Proposition \ref{prop2} one deduces
  immediately the following result.

%%%%%%%%%%%%%%%%%%%%%%%%%%%%%%%%%%%%%%%%%%%%%%%%%%%%%%%%%%%%%%%%%%%%%%%%%%%%%%%%
%%%%%%%%%%                           THEOREM                           %%%%%%%%%
%%%%%                                                                       %%%%
\begin{theorem} \label{THM1.3}
Assume that ${\mathcal D}=(D_x\sqcup D_y,\{x,y\})$ is an un-ordered disk
  pair, $\varphi$ the exponential gluing profile, and $\delta$ an increasing
  sequence of weights starting at $\delta_0>0$.
Abbreviate $X(N):=X^{3,\delta}_{{\mathcal D},\varphi}({\mathbb R}^N)$.
The functorial construction, which associates to $N$ the M-polyfold $X(N)$
  and to a smooth map $h:{\mathbb R}^N\rightarrow {\mathbb R}^M$
  the sc-smooth map $h_\ast:X(N)\rightarrow X(M)$ has a unique extension
  to the category $\mathfrak{M}$ characterized uniquely by the following
  properties.
\begin{itemize}
  \item[(1)] 
  If $Q=\coprod Q_\lambda$, where the $Q_\lambda$ are the connected
  components, then $X(Q)=\coprod X(Q_\lambda)$.
  Moreover if $Q={\mathbb R}^N$ then $X(Q)=X(N)$.
  \item[(2)] 
  If $Q$ properly embeds into some ${\mathbb R}^N$ then, with
  $\phi:Q\rightarrow {\mathbb R}^N$ being such a smooth embedding  as a set
  $$
    X(Q)=\{u:Z_a\rightarrow Q\ | \phi\circ u\in X(N)\}.
    $$
  \item[(3)] 
  The map $X(Q)\rightarrow X(N): u\rightarrow \phi\circ u$, where
    $\phi:Q\rightarrow {\mathbb R}^N$ for some $N$  is a proper smooth
    embedding, is an  sc-smooth embedding of M-polyfolds.
  \end{itemize}
The properties {\em (1)}, {\em(2)}, and {\em(3)} uniquely characterize the
  sc-smooth structure on $X(Q)$.
For the M-polyfold structure on $X(Q)$ the following properties hold.
\begin{itemize}
\item[(4)] 
For an open neighborhood $U$ of $\phi(Q)$ and a smooth map
  $R:U\rightarrow U$ with $R(U)=\phi(Q)$ and $R\circ R=R$ the map from the
  open set $X(U)\rightarrow X(Q): u\rightarrow \phi^{-1}\circ R\circ u$
  is sc-smooth.
\item[(5)] 
The obvious map $p:X(Q)\rightarrow {\mathbb B}$ is sc-smooth and $p$ has
  the submersion property, see Definition \ref{I-DEF_submersion_property}
  from \cite{FH-notes-I}.
\end{itemize}
\end{theorem}
%%%%%                                                                       %%%%
%%%%%%%%%%                                                             %%%%%%%%%
%%%%%%%%%%%%%%%%%%%%%%%%%%%%%%%%%%%%%%%%%%%%%%%%%%%%%%%%%%%%%%%%%%%%%%%%%%%%%%%%
%
\begin{proof}
One can apply Proposition
  \ref{I-PROP_extension_m_poly_construction_functors} from
  \cite{FH-notes-I}.

\jwf{[This text is kind of fucked up]}

in view of the previous discussions and it is clear that (1)--(3) will
  hold.
Since the situation is rather concrete and this is the first application
  of Proposition \ref{I-PROP_extension_m_poly_construction_functors} from
  \cite{FH-notes-I}
  we carry out the ideas which were used in its proof just to illustrate the
  procedure.

Property {(3)} says that   the subset $Y$  of $X(N)$ consisting of all
  $v\in X(N)$ with image in $\phi(Q)$ is a sub-M-polyfold and that for the
  induced M-polyfold structure the map $X(Q)\rightarrow Y:u\rightarrow
  \phi\circ u$ is an sc-diffeomorphism.
Since $\phi(Q)\subset {\mathbb R}^N$ is properly embedded we find a smooth
  map $R:{\mathbb R}^N\rightarrow {\mathbb R}^N$ and an open neighborhood
  $U$ of $\phi(Q)$ such that $(R|U)\circ (R|U)=R|U$ and $R(U)=\phi(Q)$.
Since the collection $X_U(N)$ of all $v\in X(N)$ with image in $U$  is
  open, we  see that $R|U$ defines an sc-smooth retraction
  $X_U(N)\rightarrow X_U(N)$ with image being the set $\Sigma_\phi$ of all
  $v\in X(N)$ with image in $\phi(Q)$.
Then by definition the map $X(Q)\rightarrow \Sigma_\phi$ is a bijection
  and we equip $X(Q)$ with the M-polyfold structure which makes it an
  sc-diffeomorphism.
This M-polyfold structure on $X(Q)$ might depend on the proper embedding
  $\phi$, and we denote it for the moment by $X_\phi(Q)$.
If $\psi:Q\rightarrow {\mathbb R}^M$ is a proper embedding, we find smooth
  maps $A:{\mathbb R}^N\rightarrow {\mathbb R}^M$
  and $B:{\mathbb R}^M\rightarrow {\mathbb R}^N$  such that
  $\psi\circ\phi^{-1} = A$ on $\phi(Q)$
  and $\phi\circ\psi^{-1}=B$ on $\psi(Q)$. 
The map ${(\psi\circ\phi^{-1})}_\ast: \Sigma_\phi\rightarrow \Sigma_\psi$
  is the restriction of an sc-smooth map and therefore sc-smooth. 
The same holds for ${(\phi\circ\psi^{-1})}_\ast.$
Hence ${(\psi\circ\phi^{-1})}_\ast $ is an sc-diffeomorphism.
We have the commutative diagram
  $$
    \begin{CD}
    X_\phi(Q)  @> Id >>X_{\psi}(Q)\\
    @V\phi_\ast VV  @V \psi VV\\
    \Sigma_\phi @>{(\psi\circ\phi^{-1})}_\ast>> \Sigma_\psi.
    \end{CD}
    $$
Hence $Id$ is sc-smooth and by reversing roles the same holds for the
  inverse.
Consequently the M-polyfold structure on $X(Q)$ does not depend on the
  choice of the proper embedding.
It is straight forward with the given definition of the M-polyfold
  structure, that a smooth map $f:Q\rightarrow P$ induces an sc-smooth map
  $f_\ast:X(Q)\rightarrow X(P)$.   
This completes the proof of showing that the functor $X$ has an extension
  to manifolds and it is clear that the properties uniquely determine this
  extension.
We have also verified {(4)}.

In order to prove {(5)} note that $p_{\mathbb B}:X(N)\rightarrow {\mathbb
  B}$ has the submersion property.
Take the usual sc-smooth maps $\oplus:{\mathbb B}\times
  H^{3,\delta}_c\rightarrow X(N)$ and $f:X(N)\rightarrow {\mathbb B}\times
  H^{3,\delta}_c$. We write $f(u) =(a(u),\bar{f}(u))$, where $a=a(u)$ is the
  domain parameter for $u$. The map $\bar{f}$ is sc-smooth into
  $H^{3,\delta}_c$.
We define the sc-smooth map  $\rho: X(N)\times {\mathbb B} \rightarrow
  X(N)\times {\mathbb B}$ by
  \begin{eqnarray}\label{sub-restZZ}
    \rho(u,b) =   (\oplus (b,\bar{f}(u)),b)
    \end{eqnarray}
Then 
  \begin{eqnarray*}
    \rho\circ\rho(u,b)&=& \rho(\oplus (b,\bar{f}(u)),b)\\
    &=&(\oplus (b, \bar{f}(\oplus (b,\bar{f}(u)))),b)\\
    &=& (\oplus\circ f\circ \oplus(b,\bar{f}(u)),b)\\
    &=& (\oplus(b,\bar{f}(u)),b)\\
    &=&\rho(u,b).
    \end{eqnarray*}
We note that $\rho(X(N)\times{\mathbb B})=\{(w,b)\ |\
  a(w)=b\}=:\text{Gr}(p_{\mathbb B})$.
If $Q\subset {\mathbb R}^N$ is a properly embedded submanifold take an
  open neighborhood $U=U(Q)$ with a smooth retraction $R:U\rightarrow U$
  satisfying $R(U)=Q$.
If $(u,a)$ satisfies $a(u)=a$ and $u\in X(Q)\subset X(N)$ we can define
  for nearby data $(v,b)\in X(Q)\times{\mathbb B}$
  $$
    \wt{\rho}(v,b) = (R(\oplus(b,\bar{f}(u))),b).
    $$
Then 
  \begin{eqnarray*}
    \wt{\rho}\circ\wt{\rho}(v,b)&=&\wt{\rho} (R(\oplus(b,\bar{f}(u))),b)\\
    &=&\wt{\rho}(v,b).
    \end{eqnarray*}
Hence for $X(Q)$ the submersion property holds.
\end{proof}%%%%%%%%%%%%%%%%%%%%%%%%%%%%%%%%%%%%%%%%%%%%%%%%            END PROOF

As a corollary of the previous result we note the following assertion which 
  follows from the observation that the definition (\ref{sub-restZZ}) of
  $\rho$ which involves a proper choice of $\bar{f}$ is the identity on
  boundary annuli.
This is important when we implement in the local constructions the idea of
  submersive $\oplus$-constructions with restrictions.

%%%%%%%%%%%%%%%%%%%%%%%%%%%%%%%%%%%%%%%%%%%%%%%%%%%%%%%%%%%%%%%%%%%%%%%%%%%%%%%%
%%%%%%%%%%                          COROLLARY                           %%%%%%%%
%%%%%                                                                       %%%%
\begin{corollary}
Let ${\mathcal D}$ be an un-ordered disk pair the construction functor
  $X=X_{{\mathcal D},\varphi}^{3,\delta}$ has the following properties.
For every $N$ we have that the submersive $\oplus$-construction.
$$
  {\mathbb B}_{\mathcal D}\times H^{3,\delta}_c({\mathcal D},{\mathbb
  R}^N)\xrightarrow{\oplus} X(N)\xrightarrow{p} {\mathbb B}_{\mathcal D}
  $$
Given compact  concentric  boundary annuli for ${\mathcal D}$, i.e.
  $A_x\subset D_x$ and $A_y\subset D_y$, where
  $\sigma_{\wh{x}}^+([0,R_x]\times S^1)=A_x$ for some $R_x\in [0,20]$ and
  similarly for $A_y$ define by
  $$
    p_x:X\rightarrow H^3(A,{\mathbb R}^N)\ \text{and}\ \
    p_y:X(N)\rightarrow H^3(A_y,{\mathbb R}^N)
    $$
  the sc-smooth restriction maps. 
Then $(\oplus,\{p_x,,p_y\},\bar{a})$ is a submersive $\oplus$-construction
  with resrictions.
\end{corollary}
%%%%%                                                                       %%%%
%%%%%%%%%%                                                             %%%%%%%%%
%%%%%%%%%%%%%%%%%%%%%%%%%%%%%%%%%%%%%%%%%%%%%%%%%%%%%%%%%%%%%%%%%%%%%%%%%%%%%%%%
%
\begin{proof}
One just has in the construction of the $\rho$'s to pick the $\bar{f}$
  carefully.
\end{proof}%%%%%%%%%%%%%%%%%%%%%%%%%%%%%%%%%%%%%%%%%%%%%%%%            END PROOF

We have given before sufficient and necessary criteria for a map into
  $X(N)$ or for a map defined on $X(N)$ to be sc-smooth.
The canonical construction given in the previous theorem allows to give a
  similar criterion for $X(Q)$, which follows immediately from the criterion
  in the special and the construction.

%%%%%%%%%%%%%%%%%%%%%%%%%%%%%%%%%%%%%%%%%%%%%%%%%%%%%%%%%%%%%%%%%%%%%%%%%%%%%%%%
%%%%%%%%%%                          PROPOSITION                        %%%%%%%%%
%%%%%                                                                       %%%%
\begin{proposition}\label{PROP1.4}
Let ${\mathcal D}$, $\varphi$, and $\delta$ be given and abbreviate
  $X(Q)=X^{3,\delta}_{{\mathcal D},\varphi}(Q)$. Assume that $Q$ is
  connected.
\begin{itemize}
  \item[(1)] 
  Let $Y$ be a M-polyfold and $A:Y\rightarrow X(Q)$ be a map. 
  Then $A$ is sc-smooth if and only if for one  proper smooth embedding
    $\phi:Q\rightarrow {\mathbb R}^N$ the map $Y\rightarrow {\mathbb
    B}\times H^{3,\delta}({\mathcal D},{\mathbb R}^N)$ defined by
    $$
      y\rightarrow f^N( \phi \circ A(y))
      $$
  is sc-smooth. 
  Here $f^N$ is the map constructed in Theorem \ref{thm1} occurring in the
    definition of a M-polyfold structure:
  $$
    X^{3,\delta}_{{\mathcal D},\varphi}({\mathbb R}^N)\rightarrow {\mathbb
    B}\times H^{3,\delta}_{c}({\mathcal D},{\mathbb R}^N).
    $$
  \item[(2)] 
  Let $Y$ be a M-polyfold. A map $B:X(Q)\rightarrow Y$ is sc-smooth if and
    only if for one smooth proper embedding $\phi:Q\rightarrow {\mathbb
    R}^N$ and tubular neighborhood $U=U(\phi(Q))$ which via
    $r:U\rightarrow U$ smoothly retracts to $\phi(Q)$ the composition
    $$
      {\mathcal U}\rightarrow Y: (a,(u^x,u^y))\rightarrow B(\phi^{-1}\circ r
      \circ \oplus(a,(u^x,u^y)))
      $$
    is sc-smooth.  
  Here ${\mathcal U}$ is the open subset of ${\mathbb B}\times
    H^{3,\delta}_c({\mathcal D},{\mathbb R}^N)$ consisting of all
    $(a,(u^x,u^y))$ so that $\oplus(a,(u^x,u^y))$ is a map having image in
    $U$.
  \end{itemize}
\end{proposition}
%%%%%                                                                       %%%%
%%%%%%%%%%                                                             %%%%%%%%%
%%%%%%%%%%%%%%%%%%%%%%%%%%%%%%%%%%%%%%%%%%%%%%%%%%%%%%%%%%%%%%%%%%%%%%%%%%%%%%%%
%

We need the following result from Subsection \jwf{[broken reference]} %\ref{APPA3}
  which will apply to $p:X(Q)\rightarrow {\mathbb B}$.

%%%%%%%%%%%%%%%%%%%%%%%%%%%%%%%%%%%%%%%%%%%%%%%%%%%%%%%%%%%%%%%%%%%%%%%%%%%%%%%%
%%%%%%%%%%                          PROPOSITION                        %%%%%%%%%
%%%%%                                                                       %%%%
\begin{proposition}
Assume that $h:X\rightarrow H$ is an sc-smooth map between M-polyfolds
  $X$, and it has  the submersion property, and $k:W\rightarrow H$ is an
  sc-smooth map.
Then the fibered product with projection $pr_2:
  X{_{h}\times_k}W\rightarrow W$ defines a M-polyfold and $pr_2$ has the
  submersion property.
\end{proposition}
%%%%%                                                                       %%%%
%%%%%%%%%%                                                             %%%%%%%%%
%%%%%%%%%%%%%%%%%%%%%%%%%%%%%%%%%%%%%%%%%%%%%%%%%%%%%%%%%%%%%%%%%%%%%%%%%%%%%%%%
%

We end the subsection with a useful remark, see also Remark
  \ref{REMARK2.11}.

%%%%%%%%%%%%%%%%%%%%%%%%%%%%%%%%%%%%%%%%%%%%%%%%%%%%%%%%%%%%%%%%%%%%%%%%%%%%%%%%
%%%%%%%%%%                            REMARK                           %%%%%%%%%
%%%%%                                                                       %%%%
\begin{remark}\label{REMdotX}
The construction of $X(Q)= X^{3,\delta}_{{\mathcal D},\varphi}(Q)$ depends
  on the gluing profile $\varphi$ and the weight sequence $\delta$.
We also have shown that $p_{\mathbb B}: X(q)\rightarrow {\mathbb B}$, the
  extraction of the domain gluing parameter, is submersive.
The subset $\dot{X}^{3,\delta}_{{\mathcal D},\varphi}(Q)$ defined by
  $$
    \dot{X}^{3,\delta}_{{\mathcal D},\varphi}(Q)=p_{\mathbb
    B}^{-1}({\mathbb B}\setminus\{0\})
    $$
  is open and has an induced M-polyfold structure. 
It is not difficult to show that the structure on
  $\dot{X}^{3,\delta}_{{\mathcal D},\varphi}(Q)$ does not depend on
  $\varphi$ and $\delta$.  
The level $m$ consists of maps of regularity $H^{3+m}$.
\end{remark}
%%%%%                                                                       %%%%
%%%%%%%%%%                                                             %%%%%%%%%
%%%%%%%%%%%%%%%%%%%%%%%%%%%%%%%%%%%%%%%%%%%%%%%%%%%%%%%%%%%%%%%%%%%%%%%%%%%%%%%%
%

%%%%%%%%%%%%%%%%%%%%%%%%%%%%%%%%%%%%%%%%%%%%%%%%%%%%%%%%%%%%%%%%%%%%%%%%%%%%%%%%
%%%%%%%%%%                         SUB-SECTION                         %%%%%%%%%
%%%%%%%%%%%%%%%%%%%%%%%%%%%%%%%%%%%%%%%%%%%%%%%%%%%%%%%%%%%%%%%%%%%%%%%%%%%%%%%%
%
\subsection{Group Action}\label{SEC_group_action}
Let $G_{\mathcal D}$ be the group of holomorphic isomorphisms of the
  un-ordered disk pair ${\mathcal D}=(D_x\sqcup D_y,\{x,y\})$.
We abbreviate $H^{3,\delta}_{{\mathcal D},c}= H^{3,\delta}_c({\mathcal
  D},{\mathbb R}^N)$ and define an action
  \begin{eqnarray}\label{EQN3}
    G_{\mathcal D}\times H^{3,\delta}_{{\mathcal D},c}\rightarrow
    H^{3,\delta}_{{\mathcal D},c}
    \end{eqnarray}
  by $g\ast (u^x,u^y)= (u^x,u^y)\circ g^{-1}$. 
Along the lines of the result in \cite{HWZ8.7}, Proposition 4.2, but much
  easier since we only have to consider rotations and reflection we can
  prove the following proposition.

%%%%%%%%%%%%%%%%%%%%%%%%%%%%%%%%%%%%%%%%%%%%%%%%%%%%%%%%%%%%%%%%%%%%%%%%%%%%%%%%
%%%%%%%%%%                          PROPOSITION                        %%%%%%%%%
%%%%%                                                                       %%%%
\begin{proposition}\label{PROP2.6}
  \hfill\\
The group action in (\ref{EQN3}) is sc-smooth.
\end{proposition}
%%%%%                                                                       %%%%
%%%%%%%%%%                                                             %%%%%%%%%
%%%%%%%%%%%%%%%%%%%%%%%%%%%%%%%%%%%%%%%%%%%%%%%%%%%%%%%%%%%%%%%%%%%%%%%%%%%%%%%%
%

We note that $G_{\mathcal D}$ acts on the set of natural gluing parameters
  so that $ g_a :Z_a\rightarrow Z_{g\ast a}$, which also includes the case
  $a=0$. 
We define
  $$
    G_{\mathcal D}\times X(N)\rightarrow X(N): (g,u)\rightarrow g\ast u:=
    u\circ g_{a(u)}^{-1}.
    $$
One easily verifies that we have the commutative diagram
  $$
    \begin{CD}
    G_{\mathcal D}\times H^{3,\delta}_{{\mathcal D},c} @>>>
    H^{3,\delta}_{{\mathcal D},c}\\
    @V Id\times \oplus VV     @A f AA\\
    G_{\mathcal D}\times X(N)   @>>>   X(N)
    \end{CD}
    $$
  where the horizontal maps are the obvious projections.
By the definition of the M-polyfold structure this precisely means that
  the action of $G_{\mathcal D}$ is sc-smooth. 
Since we can identify $X(Q)$ with a subset of some $X(N)$ the associated
  group action just restricts.
Hence we obtain the following result.

%%%%%%%%%%%%%%%%%%%%%%%%%%%%%%%%%%%%%%%%%%%%%%%%%%%%%%%%%%%%%%%%%%%%%%%%%%%%%%%%
%%%%%%%%%%                           THEOREM                           %%%%%%%%%
%%%%%                                                                       %%%%
\begin{theorem}
Given an unordered disk pair ${\mathcal D}$, the gluing profile $\varphi$,
  an increasing sequence $\delta$ and a smooth manifold $Q$ without
  boundary.
Then the natural action
  $$
    G_{\mathcal D}\times X^{3,\delta}_{{\mathcal D},\varphi}(Q)\rightarrow
    X^{3,\delta}_{{\mathcal D},\varphi}(Q)
    $$
  is sc-smooth.
\end{theorem}
%%%%%                                                                       %%%%
%%%%%%%%%%                                                             %%%%%%%%%
%%%%%%%%%%%%%%%%%%%%%%%%%%%%%%%%%%%%%%%%%%%%%%%%%%%%%%%%%%%%%%%%%%%%%%%%%%%%%%%%
%

%%%%%%%%%%%%%%%%%%%%%%%%%%%%%%%%%%%%%%%%%%%%%%%%%%%%%%%%%%%%%%%%%%%%%%%%%%%%%%%%
%%%%%%%%%%                         SUB-SECTION                         %%%%%%%%%
%%%%%%%%%%%%%%%%%%%%%%%%%%%%%%%%%%%%%%%%%%%%%%%%%%%%%%%%%%%%%%%%%%%%%%%%%%%%%%%%
%
\subsection{A Variation and Strong Bundles}
  \label{VARIATION1}\label{SEC_variation_sb}
The main goal is to construct certain strong bundles which requires some
  preparation.
Principally, however, this can be viewed as a generalization of the
  previous M-polyfold constructions.

%%%%%%%%%%%%%%%%%%%%%%%%%%%%%%%%%%%%%%%%%%%%%%%%%%%%%%%%%%%%%%%%%%%%%%%%%%%%%%%%
%%%%%%%%%%                         SUB-SECTION                         %%%%%%%%%
%%%%%%%%%%%%%%%%%%%%%%%%%%%%%%%%%%%%%%%%%%%%%%%%%%%%%%%%%%%%%%%%%%%%%%%%%%%%%%%%
%
\subsubsection{A Variation}
This is a slight modification of the construction in the previous
  subsection.
We introduce it for the purpose of constructing strong bundles later on
  and therefore we impose a different regularity assumption.
We assume ${\mathcal D}$, $\varphi$, and $\delta$ are as before. 
We shall write ${\mathbb K}$ for ${\mathbb R}$ or ${\mathbb C}$.
We denote by $H^{2,\delta}({\mathcal D},{\mathbb K}^N)$ the sc-Hilbert
  space over the field ${\mathbb K}$, consisting of maps of class
  $(2,\delta_0)$ with vanishing asymptotic limits, i.e. the nodal values
  are $0$.
The $m$-th level is defined by regularity $(2+m,\delta_m)$. 
Then we define the set
  \begin{eqnarray}\label{WEQN3.18}
    &X^{2,\delta_0}_{{\mathcal D},\varphi,0}({\mathbb K}^N)=\left(
    \{0\}\times H^{2,\delta_0}\right)\coprod
    \left( \coprod_{0<|a|<1/4} H^2(Z_a,{\mathbb K}^N)\right)&
    \end{eqnarray}
As before we take a cut-off function $\beta$ with the properties specified
  before and define
  $$
    \wh{\oplus}: {\mathbb B}\times H^{2,\delta}\rightarrow
    X^{2,\delta_0}_{{\mathcal D},\varphi,0}({\mathbb K}^N)
    \index{$X^{2,\delta_0}_{{\mathcal D},\varphi,0}({\mathbb K}^N)$}
    $$
  by the same formula already used for $\oplus$. 
We use the notation $\wh{\oplus}$ since the regularity is different and
  since due to the vanishing of asymptotic constants the map $f$ in the
  $\oplus$-context has to be replaced by a $\wh{f}$ which has a different
  structure.
The main result about this modified construction (\ref{WEQN3.18}) is the
  following theorem whose details are left to the reader with the exception
  that we exhibit a suitable $\wh{f}$.
The proof is otherwise completely analogue to Theorem \ref{thm1}.
We denote by $H^{2,\delta}$ the sc-Hilbert space of maps with the obvious
  regularity vanshing over the two nodal points.
Its levels are defined by regularity $(2+m,\delta_m)$.

%%%%%%%%%%%%%%%%%%%%%%%%%%%%%%%%%%%%%%%%%%%%%%%%%%%%%%%%%%%%%%%%%%%%%%%%%%%%%%%%
%%%%%%%%%%                           THEOREM                           %%%%%%%%%
%%%%%                                                                       %%%%
\begin{theorem}\label{THM1.7}
For every natural number $N\geq 1$ the set  $X^{2,\delta_0}_{{\mathcal
  D},\varphi,0}({\mathbb K}^N)$ has a uniquely defined metrizable topology
  ${\mathcal T}$, as well as uniquely defined  M-polyfold structure
  characterized by the  requirement that there exists a map
  $\wh{f}:X^{2,\delta_0}_{{\mathcal D},\varphi,0}({\mathbb K}^N)\rightarrow
  {\mathbb B}\times H^{2,\delta_0}({\mathcal D},{\mathbb R}^N)$ preserving
  the fibers over ${\mathbb B}$, i.e. we have the commutative diagram
  $$
    \begin{CD}
    X^{2,\delta_0}_{{\mathcal D},\varphi,0}({\mathbb K}^N)  @> \wh{f} >>
    {\mathbb B}\times H^{2,\delta_0}({\mathcal D},{\mathbb K}^N)\\
    @V \wh{p}_{\mathbb B} VV @V pr_1VV\\
    {\mathbb B} @= {\mathbb B}
    \end{CD}
    $$
  such that 
  \begin{itemize}
  \item[(1)] 
  $\wh{\oplus}\circ \wh{f} =Id$.
  \item[(2)] 
  $\wh{f}\circ \wh{\oplus}$ as a map ${\mathbb B}\times
  H^{2,\delta}({\mathcal D},{\mathbb K}^N)\rightarrow
  {\mathbb B}\times H^{2,\delta}({\mathcal D},{\mathbb K}^N)$ is
  sc-smooth.
  \end{itemize}
The M-polyfold structure on the set $X^{2,\delta_0}_{{\mathcal
  D},\varphi,0}({\mathbb K}^N)$ is denoted by $X^{2,\delta}_{{\mathcal
  D},\varphi,0}({\mathbb K}^N)$ and  has then the following additional
  properties, where we abbreviate the space by $X_0$ and $H^{2,\delta}$.
\begin{itemize}
  \item[(3)] 
  The maps $\wh{\oplus}$ and $\wh{f}$ are continuous.
  \item[(4)] 
  The M-polyfold structure on $X_0$ does not depend on the choice of
  $\beta$ nor on the choice of $\wh{f}$ with the stated properties.
  \item[(5)] 
  A map $h:Y\rightarrow X_0$, where $Y$ is a M-polyfold, is sc-smooth if
  and only if $\wh{f}\circ h:Y\rightarrow {\mathbb B}\times H^{2,\delta}$
  is sc-smooth.
  \item[(6)] 
  A map $k:X_0\rightarrow Y$, where $Y$ is a M-polyfold is sc-smooth if
    and only if $k\circ \wh{\oplus} :{\mathbb B}\times
    H^{2,\delta}\rightarrow Y$ is sc-smooth.
  In particular $\wh{f}$ and $\wh{\oplus}$ are sc-smooth.
  \end{itemize}
\end{theorem}
%%%%%                                                                       %%%%
%%%%%%%%%%                                                             %%%%%%%%%
%%%%%%%%%%%%%%%%%%%%%%%%%%%%%%%%%%%%%%%%%%%%%%%%%%%%%%%%%%%%%%%%%%%%%%%%%%%%%%%%
%
\begin{proof}
We give the formula for $\wh{f}$ with respect to the model, where the maps
  are defined on standard half-cylinders.
A suitable $\wh{f}$ is defined by
  $\wh{f}(a,u^x,u^y)=(a,\wh{\eta}^x,\wh{\eta}^y)$ where with
  $R=\varphi(|a|)$
  \begin{align*}                                                          %% EQN
    \wh{\eta}^x([s,t]) 
    &=\beta(s-\textstyle{\frac{1}{2}}R-2)\cdot u([s,t])    
    \end{align*}
  and 
  \begin{align*}                                                          %% EQN
    \wh{\eta}^y([s',t']')
    & =\beta\left(-s'- {\textstyle \frac{1}{2}}R-2\right)\cdot
      u\left(\left[s',t'\right]'\right)
    \end{align*}
The proof follows the same line as the proof of Theorem \ref{thm1} and relies
  on the results Proposition 2.8 and Proposition 2.17 in \cite{HWZ8.7}.
\end{proof}%%%%%%%%%%%%%%%%%%%%%%%%%%%%%%%%%%%%%%%%%%%%%%%%            END PROOF

%%%%%%%%%%%%%%%%%%%%%%%%%%%%%%%%%%%%%%%%%%%%%%%%%%%%%%%%%%%%%%%%%%%%%%%%%%%%%%%%
%%%%%%%%%%                          DEFINITION                         %%%%%%%%%
%%%%%                                                                       %%%%
\begin{definition}
We denote by $f^{\wh{\oplus}}$ the specific $\wh{f}$ introduced in Theorem
  \ref{THM1.7}.
\end{definition}
%%%%%                                                                       %%%%
%%%%%%%%%%                                                             %%%%%%%%%
%%%%%%%%%%%%%%%%%%%%%%%%%%%%%%%%%%%%%%%%%%%%%%%%%%%%%%%%%%%%%%%%%%%%%%%%%%%%%%%%
%

We have an obvious map $\wh{p}:X^{2,\delta}_{{\mathcal
  D},\varphi,0}({\mathbb K}^N)\rightarrow {\mathbb B}$ and as in the
  previous case this map has the submersion property.
We also have an sc-smooth action by $G_{\mathcal D}$. 
All these properties can be proved along the lines as in the previous case. 
The following theorem summarizes the results.

%%%%%%%%%%%%%%%%%%%%%%%%%%%%%%%%%%%%%%%%%%%%%%%%%%%%%%%%%%%%%%%%%%%%%%%%%%%%%%%%
%%%%%%%%%%                           THEOREM                           %%%%%%%%%
%%%%%                                                                       %%%%
\begin{theorem}\label{THM1.7X} 
The following holds.
\begin{itemize}
  \item[(1)] 
  The sc-smooth map $\wh{p}:X^{2,\delta}_{{\mathcal D},\varphi,0}({\mathbb
  K}^N)\rightarrow {\mathbb B}$ has the submersion property.
  \item[(2)] 
  The natural group action by $G_{\mathcal D}$ is sc-smooth.
  \end{itemize}
\end{theorem}
%%%%%                                                                       %%%%
%%%%%%%%%%                                                             %%%%%%%%%
%%%%%%%%%%%%%%%%%%%%%%%%%%%%%%%%%%%%%%%%%%%%%%%%%%%%%%%%%%%%%%%%%%%%%%%%%%%%%%%%
%

%%%%%%%%%%%%%%%%%%%%%%%%%%%%%%%%%%%%%%%%%%%%%%%%%%%%%%%%%%%%%%%%%%%%%%%%%%%%%%%%
%%%%%%%%%%                            REMARK                           %%%%%%%%%
%%%%%                                                                       %%%%
\begin{remark}\label{REMARK9.11}
We also note that we could work here with a different regularity and for
  example define the set
  $$
    X^{3,\delta_0}_{{\mathcal D},\varphi,0}({\mathbb K}^N)=\left(
    \{0\}\times H^{3,\delta_0}\right)\coprod
    \left( \coprod_{0<|a|<1/4} H^3(Z_a,{\mathbb K}^N)\right).
    $$
We shall equip this set with a M-polyfold structure.
When we compare this with the definition (\ref{plot007}) we see that in
  the case ${\mathbb R}$ the M-polyfold structures on
  $X^{3,\delta_0}_{{\mathcal D}}({\mathbb R}^N)$ induces the same structure
  on the open dense subspace $\coprod_{0<|a|<1/4} H^3(Z_a,{\mathbb K}^N)$.
With other words the resulting M-polyfolds $X^{3,\delta}_{{\mathcal
  D},\varphi,0}({\mathbb R}^N)$ and $X^{3,\delta}_{{\mathcal
  D},\varphi}({\mathbb R}^N)$ can be viewed as different smooth completions
  of the same underlying  sc-manifold.
There is another remark concerning this fact later on, see Remark
  \ref{REMARK2.11}.
\end{remark}
%%%%%                                                                       %%%%
%%%%%%%%%%                                                             %%%%%%%%%
%%%%%%%%%%%%%%%%%%%%%%%%%%%%%%%%%%%%%%%%%%%%%%%%%%%%%%%%%%%%%%%%%%%%%%%%%%%%%%%%
%

%%%%%%%%%%%%%%%%%%%%%%%%%%%%%%%%%%%%%%%%%%%%%%%%%%%%%%%%%%%%%%%%%%%%%%%%%%%%%%%%
%%%%%%%%%%                          EXERCISE                           %%%%%%%%%
%%%%%                                                                       %%%%
\begin{exercise}\label{EXERC100}
Construct a M-polyfold structure on $X^{3,\delta_0}_{{\mathcal
  D},\varphi,0}({\mathbb K}^N)$, the set defined in Remark \ref{REMARK9.11},
  using the $\oplus$-method with a map defined
  on ${\mathbb B}\times H^{3,\delta}$. 
This results in a construction which associates to $N$ a M-polyfold
  $X^{3,\delta}_{{\mathcal D},\varphi,0}({\mathbb K}^N)$.
Also show that $N\rightarrow X^{3,\delta}_{{\mathcal
  D},\varphi,0}({\mathbb R}^N)$ together with the smooth maps $h:({\mathbb
  R}^N,0)\rightarrow ({\mathbb R}^L,0)$  as morphisms can be viewed as
  construction functor.
\end{exercise}
%%%%%                                                                       %%%%
%%%%%%%%%%                                                             %%%%%%%%%
%%%%%%%%%%%%%%%%%%%%%%%%%%%%%%%%%%%%%%%%%%%%%%%%%%%%%%%%%%%%%%%%%%%%%%%%%%%%%%%%
%

Here is another variation.  

%%%%%%%%%%%%%%%%%%%%%%%%%%%%%%%%%%%%%%%%%%%%%%%%%%%%%%%%%%%%%%%%%%%%%%%%%%%%%%%%
%%%%%%%%%%                            REMARK                           %%%%%%%%%
%%%%%                                                                       %%%%
\begin{remark}\label{REMARK2.11}
Consider the set $\dot{X}_{\mathcal D}^3({\mathbb R}^N)$ consisting of all
  maps $w:Z_a\rightarrow {\mathbb R}^N$ of class $H^3$ with $a\in {\mathbb
  B}\setminus\{0\}$.
We can identify this with an open subset of $X^{3,\delta}_{{\mathcal
  D},\varphi}({\mathbb R}^N)$ or as an open subset of
  $X^{3,\delta}_{{\mathcal D},\varphi,0}({\mathbb R}^N)$.
We leave it as an exercise that the induced structure would be in both
  cases the same and that it would not depend on the choice of $\delta$.
The filtration on our new space gives level $m$ as Sobolev regularity
  $H^{m+3}$. 
The independence on $\delta$ follows from the fact that we can work with
  maps in $H^{3,\delta}$ which vanish near the nodal point.
As a consequence the M-polyfold structure on $\dot{X}_{\mathcal
  D}^3({\mathbb R}^N)$ is given by the surjective map
  $$
    \oplus: ({\mathbb B}\setminus\{0\})\times H^{3,\delta}({\mathcal
    D},{\mathbb R}^N)\rightarrow \dot{X}_{\mathcal D}^3({\mathbb R}^N)
    $$
  defined by $\oplus(a,h^x,h^y)(\{z,z'\}) =\beta^x(a,z)\cdot h^x(z) +
  \beta^y(a,z')\cdot h^y(z')$ which admits a map $f$ such that
  $f\circ \oplus$ is sc-smooth such that $\oplus\circ f=Id$. 
Going a step further one can define a natural  sc-manifold structure on
  the set $\dot{X}_{\mathcal D}^3({\mathbb R}^N)$ so that the identity map
  between the same set, but either equipped with the induced M-polyfold
  structure or the natural sc-manifold structure is a sc-diffeomorphism.
We leave the details to he reader, see the next exercise.
\end{remark}
%%%%%                                                                       %%%%
%%%%%%%%%%                                                             %%%%%%%%%
%%%%%%%%%%%%%%%%%%%%%%%%%%%%%%%%%%%%%%%%%%%%%%%%%%%%%%%%%%%%%%%%%%%%%%%%%%%%%%%%
%

%%%%%%%%%%%%%%%%%%%%%%%%%%%%%%%%%%%%%%%%%%%%%%%%%%%%%%%%%%%%%%%%%%%%%%%%%%%%%%%%
%%%%%%%%%%                          EXERCISE                           %%%%%%%%%
%%%%%                                                                       %%%%
\begin{exercise}\label{EXERC101}
  \hfill\\
We have two M-polyfold constructions by the $\oplus$-method defining
  M-polyfolds $X^{3,\delta}_{{\mathcal D},\varphi}({\mathbb R}^N)$ and
  $X^{3,\delta}_{{\mathcal D},\varphi,0}({\mathbb R}^N)$.
Both M-polyfolds come a submersive map to ${\mathbb B}$ and we denote by
  $\dot{X}^{3,\delta}_{{\mathcal D},\varphi}({\mathbb R}^N)$ and
  $\dot{X}^{3,\delta}_{{\mathcal D},\varphi,0}({\mathbb R}^N)$ the
  pre-images of ${\mathbb B}\setminus\{0\}$.  
Show that the following.
\begin{itemize}
  \item[(1)] 
  As sets $\dot{X}^{3,\delta_0}_{{\mathcal D},\varphi}({\mathbb
    R}^N)=\dot{X}^{3,\delta_0}_{{\mathcal D},\varphi,0}({\mathbb R}^N)$.
  Also show that these sets do not depend on the choice of $\delta_0$. 
  \item[(2)] 
  The identity map $I: \dot{X}^{3,\delta}_{{\mathcal D},\varphi}({\mathbb
    R}^N)\rightarrow \dot{X}^{3,\delta}_{{\mathcal D},\varphi,0}({\mathbb
    R}^N):u\rightarrow u$ is a sc-diffeomorphism.  
  \item[(3)] 
  Show that the M-polyfold structure on the (identical) spaces
    $\dot{X}^{3,\delta}_{{\mathcal D},\varphi}({\mathbb R}^N)$ and $
    \dot{X}^{3,\delta}_{{\mathcal D},\varphi,0}({\mathbb R}^N)$ does not
    depend on $\delta$ and $\varphi$, and denote it by
    $\dot{X}^3_{\mathcal D}({\mathbb R}^N)$.
  \item[(4)] 
  Show that the M-polyfold $\dot{X}^3_{\mathcal D}({\mathbb R}^N)$ admits
    a sc-manifold atlas compatible with the existing M-polyfold structure.
  Moreover $\dot{X}^3_{\mathcal D}({\mathbb R}^N)$ is sc-diffeomorphic to
    $({\mathbb B}\setminus\{0\})\times H^3([0,1]\times S^1,{\mathbb R}^N)$,
    where the latter has the obvious sc-manifold structure.
  The sc-diffeomorphism can be picked compatible with $p_{\mathcal B}$ and
    $\text{pr}_1$.
  \end{itemize}
\end{exercise}
%%%%%                                                                       %%%%
%%%%%%%%%%                                                             %%%%%%%%%
%%%%%%%%%%%%%%%%%%%%%%%%%%%%%%%%%%%%%%%%%%%%%%%%%%%%%%%%%%%%%%%%%%%%%%%%%%%%%%%%
%

%%%%%%%%%%%%%%%%%%%%%%%%%%%%%%%%%%%%%%%%%%%%%%%%%%%%%%%%%%%%%%%%%%%%%%%%%%%%%%%%
%%%%%%%%%%                         SUB-SECTION                         %%%%%%%%%
%%%%%%%%%%%%%%%%%%%%%%%%%%%%%%%%%%%%%%%%%%%%%%%%%%%%%%%%%%%%%%%%%%%%%%%%%%%%%%%%
%
\subsubsection{A Strong Bundle Construction}\label{SSEc7.3.2}
Denote by ${\mathbb K}$ either ${\mathbb R}$ or ${\mathbb C}$. 
For natural numbers $N,L\geq 1$ we consider the pull-back diagram, using
  the underlying gluing parameters,
  $$
  \begin{CD}
  @.    X^{2,\delta}_{{\mathcal D},\varphi,0}({\mathbb K}^L)\\
  @. @V \wh{p} VV\\
  X^{3,\delta}_{{\mathcal D},\varphi}({\mathbb R}^N)@> p>>  {\mathbb B}
  \end{CD}
  $$
  and use it to define a strong bundle which also comes with an sc-smooth
  map having the submersion property.
We define $W^{3,2,\delta}_{{\mathcal D},\varphi}({\mathbb R}^N\times
  {\mathbb K}^L)$\index{$W^{3,2,\delta}_{{\mathcal D},\varphi}({\mathbb
  R}^N\times {\mathbb K}^L)$} to consist of all pairs $(u,w)\in
  X^{3,\delta}_{{\mathcal D},\varphi}({\mathbb R}^N)\times
  X^{2,\delta}_{{\mathcal D},\varphi,0}({\mathbb K}^L)$ satisfying
  $p(u)=\wh{p}(w)$ and introduce the following maps
  \begin{eqnarray}
    &p_X:W^{3,2,\delta}_{{\mathcal D},\varphi}({\mathbb R}^N\times
    {\mathbb K}^L)\rightarrow
    X^{3,\delta}_{{\mathcal D},\varphi}({\mathbb R}^N):p_W(u,w)=u&\\
    &p_{\mathbb B}:W^{3,2,\delta}_{{\mathcal D},\varphi}({\mathbb
    R}^N\times {\mathbb K}^L)\rightarrow {\mathbb B}:p_{\mathbb
    B}(u,w):=p(u)&\nonumber
    \end{eqnarray}
These maps fit into the commutative diagram
  $$
    \begin{CD}
    W^{3,2,\delta}_{{\mathcal D},\varphi}({\mathbb R}^N\times {\mathbb
    K}^L)@>p_X>> X^{3,\delta}_{{\mathcal D},\varphi}({\mathbb R}^N)\\
    @V p_{\mathbb B} VV   @V p VV\\
    {\mathbb B} @= {\mathbb B}
    \end{CD}
    $$

%%%%%%%%%%%%%%%%%%%%%%%%%%%%%%%%%%%%%%%%%%%%%%%%%%%%%%%%%%%%%%%%%%%%%%%%%%%%%%%%
%%%%%%%%%%                           THEOREM                           %%%%%%%%%
%%%%%                                                                       %%%%
\begin{theorem}
The following holds.
\begin{itemize}
  \item[(1)] 
  $p_X:W^{3,2,\delta}_{{\mathcal D},\varphi}({\mathbb R}^N\times {\mathbb
  K}^L)\rightarrow X^{3,\delta}_{{\mathcal D},\varphi}({\mathbb R}^N)$ has
  naturally the structure of a strong bundle.
  \item[(2)] 
  $p_{\mathbb B}$ has the submersion property.
  \end{itemize}
\end{theorem}
%%%%%                                                                       %%%%
%%%%%%%%%%                                                             %%%%%%%%%
%%%%%%%%%%%%%%%%%%%%%%%%%%%%%%%%%%%%%%%%%%%%%%%%%%%%%%%%%%%%%%%%%%%%%%%%%%%%%%%%
%
\begin{proof}
Consider the submersive product
  $$
    X^{3,\delta}_{{\mathcal D},\varphi}({\mathbb R}^N)\times
    X^{2,\delta}_{{\mathcal D},\varphi,0}({\mathbb K}^L)\xrightarrow{p\times
    \wh{p}} {\mathbb B}\times {\mathbb B}.
    $$
First we note  that $X^{3,\delta}_{{\mathcal D},\varphi}({\mathbb
  R}^N)\times X^{2,\delta}_{{\mathcal D},\varphi,0}({\mathbb K}^L)$ has an
  obvious strong bundle structure over $X^{3,\delta}_{{\mathcal
  D},\varphi}({\mathbb R}^N)$ by allowing a shift by $i\in \{0,1\}$ in the
  second factor.
We pull back by the diagonal $\Delta:{\mathbb B}\rightarrow {\mathbb
  B}\times {\mathbb B}$.
Then $\Delta(p\times\wh{p})$ is a strong bundle with submersive map to
  ${\mathbb B}$.
Clearly $ \Delta(p\times\wh{p})\rightarrow X^{3,\delta}_{{\mathcal
  D},\varphi}({\mathbb R}^N)$ is a strong bundle by construction and  can
  naturally be identified with $p_X$.
\end{proof}%%%%%%%%%%%%%%%%%%%%%%%%%%%%%%%%%%%%%%%%%%%%%%%%            END PROOF
%

%%%%%%%%%%%%%%%%%%%%%%%%%%%%%%%%%%%%%%%%%%%%%%%%%%%%%%%%%%%%%%%%%%%%%%%%%%%%%%%%
%%%%%%%%%%                            REMARK                           %%%%%%%%%
%%%%%                                                                       %%%%
\begin{remark}
The elements in $W^{3,2,\delta}_{{\mathcal D},\varphi}({\mathbb R}^N\times
  {\mathbb K}^L)$ van be viewed as
  $$
    \langle u,h\rangle:Z_a\rightarrow {\mathbb R}^N\times {\mathbb K}^L
    $$
  with mixed regularity. 
With the data being clear we abbreviate 
\begin{eqnarray*}
  &W(N,L):= W^{3,2,\delta}_{{\mathcal D},\varphi}({\mathbb R}^N\times
  {\mathbb K}^L)&\\
  & p_X:W(N,L)\rightarrow X(N),\ \ p_{\mathbb B}:W(N,L)\rightarrow
  {\mathbb B}.&
  \end{eqnarray*}
\end{remark}
%%%%%                                                                       %%%%
%%%%%%%%%%                                                             %%%%%%%%%
%%%%%%%%%%%%%%%%%%%%%%%%%%%%%%%%%%%%%%%%%%%%%%%%%%%%%%%%%%%%%%%%%%%%%%%%%%%%%%%%
%

Using the results from Theorems \ref{thm1} and \ref{THM1.7} we can give a
  sufficient and necessary criterion for a map from a strong bundle into
  $W(N,L)$ or for a map defined on the latter being a strong bundle map.

%%%%%%%%%%%%%%%%%%%%%%%%%%%%%%%%%%%%%%%%%%%%%%%%%%%%%%%%%%%%%%%%%%%%%%%%%%%%%%%%
%%%%%%%%%%                           THEOREM                           %%%%%%%%%
%%%%%                                                                       %%%%
\begin{theorem}\label{THM1.11}
With $X(N)$ and $W(N,L)$ associated to the data ${\mathcal D}$, $\varphi$,
  and $\delta$ the following holds true.
\begin{itemize}
  \item[(1)] 
  Let $E\rightarrow Y$ be a strong bundle and $\Phi:E\rightarrow W(N,L)$ a
    map, linear on the fibers,  covering a map $\phi:Y\rightarrow X(N)$.
  Then $\Phi$ is an sc-smooth strong bundle map provided $(f\times_{\mathbb
    B} \wh{f})\circ \Phi$ is an sc-smooth strong bundle map into the
    ssc-smooth strong bundle ${\mathbb B}\times H^{3,\delta}_c({\mathcal
    D}\times {\mathbb R}^N)\times H^{2,\delta}({\mathcal D}\times {\mathbb
    K}^L)$.
  Here
    $$
      (f\times_{\mathbb B} \wh{f})(u,h)= (a,(u^x,u^y),(\eta^x,\eta^y)),
      $$
    where $f(u)=(a,(u^x,u^y))$ and $\wh{f}(h)=(a,(\eta^x,\eta^y))$.
  \item[(2)]  
  Let $E\rightarrow Y$ be a strong bundle and $\Phi:W(N,L)\rightarrow Y$ a
    map, linear on the fibers,  covering a map $\phi:X(N)\rightarrow Y$.
  Then $\Phi$ is an sc-smooth strong bundle map provided $\Phi\circ
    (\oplus\times_{\mathbb B} \wh{\oplus})$ is an sc-smooth strong bundle map
    defined on ${\mathbb B}\times H^{3,\delta}_c({\mathcal D}\times {\mathbb
    R}^N)$.
  Here 
    $$
      (\oplus\times_{\mathbb B} \wh{\oplus})(a,(u^x,u^y),(\eta^x,\eta^y))
      =(\oplus(a,(u^x,u^y)),\wh{\oplus}(a,(\eta^x,\eta^y))).
      $$
  \end{itemize}
\end{theorem}
%%%%%                                                                       %%%%
%%%%%%%%%%                                                             %%%%%%%%%
%%%%%%%%%%%%%%%%%%%%%%%%%%%%%%%%%%%%%%%%%%%%%%%%%%%%%%%%%%%%%%%%%%%%%%%%%%%%%%%%
%

We are interested in smooth bundle maps $H:{\mathbb R}^N\times {\mathbb
  K}^K\rightarrow {\mathbb R}^M\times {\mathbb K}^L$, which are linear in
  the fibers and cover a smooth map $h$. 
Given such $H$, on the level of sets $H_\ast: W(N,K)\rightarrow W(M,L)$
  given by $H_\ast\langle u,v\rangle = H\circ \langle u,v\rangle$ is
  well-defined and covers the sc-smooth map $h_\ast:X(N)\rightarrow X(M)$.

%%%%%%%%%%%%%%%%%%%%%%%%%%%%%%%%%%%%%%%%%%%%%%%%%%%%%%%%%%%%%%%%%%%%%%%%%%%%%%%%
%%%%%%%%%%                          PROPOSITION                        %%%%%%%%%
%%%%%                                                                       %%%%
\begin{proposition}\label{propx9.19}
For every smooth map  $H:{\mathbb R}^N\times {\mathbb K}^K\rightarrow
  {\mathbb R}^M\times {\mathbb K}^L$, which is linear in the fibers and
  covers a smooth map $h$ the map $H_\ast$ is a  sc-smooth strong bundle
  covering the sc-smooth map $h_\ast$ and fitting into the commutative
  diagram
  $$
    \begin{CD}
    W(N,K) @>H_\ast >>  W(M,L)\\
    @V p_X VV  @V p_X VV\\
    X(N) @> h_\ast >> X(M).
    \end{CD}
    $$
\end{proposition}
%%%%%                                                                       %%%%
%%%%%%%%%%                                                             %%%%%%%%%
%%%%%%%%%%%%%%%%%%%%%%%%%%%%%%%%%%%%%%%%%%%%%%%%%%%%%%%%%%%%%%%%%%%%%%%%%%%%%%%%
%
\begin{proof}
The proof is very similar to the proof of Proposition \ref{prop2}.  
By the definition of the structures involved we only need to consider the
  map
  $$
    (\oplus^M\times_{\mathbb B} \wh{\oplus}^L)\circ H_\sharp\circ
    (f^N\times_{\mathbb B} \wh{f}^K)
    $$
  which for the current situation corresponds to the diagram (\ref{DIAG1})
  in the proof of Proposition  \ref{prop2}. 
The vertical maps are sc-smooth essentially by definition of the smooth
  sc-structures, and one only needs to show that the map in the classical
  context $H_\sharp$ is an sc-smooth strong bundle map.
Of course, it is even an ssc-smooth strong bundle map, which follows from
  the right interpretation of the results in \cite{El}.
\end{proof}%%%%%%%%%%%%%%%%%%%%%%%%%%%%%%%%%%%%%%%%%%%%%%%%            END PROOF
%

%%%%%%%%%%%%%%%%%%%%%%%%%%%%%%%%%%%%%%%%%%%%%%%%%%%%%%%%%%%%%%%%%%%%%%%%%%%%%%%%
%%%%%%%%%%                          EXERCISE                           %%%%%%%%%
%%%%%                                                                       %%%%
\begin{exercise} \label{EXERC102}
Fill in the details of the proof of Proposition \ref{propx9.19}.
\end{exercise}
%%%%%                                                                       %%%%
%%%%%%%%%%                                                             %%%%%%%%%
%%%%%%%%%%%%%%%%%%%%%%%%%%%%%%%%%%%%%%%%%%%%%%%%%%%%%%%%%%%%%%%%%%%%%%%%%%%%%%%%
%

When we constructed the functor $N\rightarrow X(N)$ we showed that it has
  a canonical extension to $\mathfrak{M}$.
Having the additional functor $(N,L)\rightarrow W(N,L)$, which `fibers'
  over $X$, we are interested in naturally extending our construction to
  smooth ${\mathbb K}$-vector bundles $p_Q: E\rightarrow Q$, where $Q$ is an
  object in $\mathfrak{M}$.
Assuming $Q$ is connected we can embed $p_Q$ properly (this is a
  requirement on the base) into a suitable ${\mathbb R}^N\times {\mathbb
  K}^L$ so that the map is fiber-wise ${\mathbb K}$-linear and fits into the
  commutative diagram
  $$
    \begin{CD}
    E @>\Phi>>  {\mathbb R}^N\times {\mathbb K}^L\\
    @V p_Q VV   @V pr_1 VV\\
    Q@>\phi >>   {\mathbb R}^N.
    \end{CD}
    $$
Note that for a suitable open neighborhood $U=U(Q)\subset {\mathbb R}^N$
  which admits a smooth retraction onto $Q$ , say $r:U\rightarrow U$ there
  exists a canonical lift to a smooth bundle retraction of $U\times {\mathbb
  K}$ to to $\Phi(E) \rightarrow\phi(Q)$, covering $r$, by taking orthogonal
  projections in the fibers using the standard complex inner product on
  ${\mathbb K}^L$.
As a consequence of the previous discussion, using precisely the method
  from the proof of Theorem \ref{THM1.3}, we can define for
  $p_Q:E\rightarrow Q$ the strong bundle ${(p_Q)}_\ast:W(E) \rightarrow
  X(Q)$.
This gives immediate the following result, where $\mathfrak{V}_{\mathbb
  K}\mathfrak{M}$ is the category of  ${\mathbb K}$-vector bundles over
  objects in $\mathfrak{M}$, with smooth ${\mathbb K}$-vector bundle maps
  between them.

%%%%%%%%%%%%%%%%%%%%%%%%%%%%%%%%%%%%%%%%%%%%%%%%%%%%%%%%%%%%%%%%%%%%%%%%%%%%%%%%
%%%%%%%%%%                           THEOREM                           %%%%%%%%%
%%%%%                                                                       %%%%
\begin{theorem}\label{THM9.20}
Assume that ${\mathcal D}=(D_x\sqcup D_y,\{x,y\})$ is an un-ordered disk
  pair, $\varphi$ the exponential gluing profile, and $\delta$ an increasing
  sequence of weights starting at $\delta_0>0$. 
Abbreviate $X(N):=X^{3,\delta}_{{\mathcal D},\varphi}({\mathbb R}^N)$ and
  $W(N,L)=W^{3,2,\delta}_{{\mathcal D},\varphi}({\mathbb R}^N\times {\mathbb
  K}^L)$ so that $W(N,L)\rightarrow X(N)$ is a strong bundle.
The functorial construction of the latter has a unique extension to the
  category $\mathfrak{V}_{\mathbb K}\mathfrak{M}$ characterized uniquely by
  the following properties where we consider $p_Q:E\rightarrow Q$.
\begin{itemize} 
  \item[(1)] 
  If $Q=\coprod Q_\lambda$, where the $Q_\lambda$ are the connected
    components, writing $E_\lambda=E|Q_\lambda$ then $W(E)=\coprod
    W(E_\lambda)$.
  Moreover if $E={\mathbb R}^N\times {\mathbb K}^L$ then $W(E)=W(N,L)$.
  \item[(2)] 
  If $Q$ properly embeds into some ${\mathbb R}^N$ then, with
    $\Phi:E\rightarrow {\mathbb R}^N\times {\mathbb K}^L$ being a complex
    vector bundle embedding as a set
    $$
    W(E)=\{ w:Z_a\rightarrow E\ |  \Phi\circ w\in W(N,L)\}.
    $$
  \item[(3)] 
  The map $W(E)\rightarrow W(N,L): u\rightarrow \Phi\circ u$, where
    $\Phi:E\rightarrow {\mathbb R}^N\times {\mathbb K}^L$  is a proper
    smooth bundle embedding is an  sc-smooth embedding of M-polyfolds.
  \end{itemize}
The properties {\em(1)}, {\em(2)}, and {\em(3)} uniquely characterize the
  sc-smooth structure on $X(Q)$.
For the M-polyfold structure on $X(Q)$ the following properties hold.
\begin{itemize}
  \item[(4)] 
  For an open neighborhood $U$ of $\phi(Q)$ and a smooth complex bundle
    map $R:U\times {\mathbb K}^L\rightarrow U\times {\mathbb K}^L$ covering
    $r:U\rightarrow U$ with $r(U)=\phi(Q)$ and $R\circ R=R$ the map from the
    open set $W(U\times {\mathbb K}^L)\rightarrow X(E): w\rightarrow
    \Phi^{-1}\circ R\circ w$ is sc-smooth strong bundle map.
  \item[(5)] 
  The obvious map $p_{\mathbb B}:W(E)\rightarrow {\mathbb B}$ is sc-smooth
    and $p_{\mathbb B}$ has the submersion property.
  \end{itemize}
\end{theorem}
%%%%%                                                                       %%%%
%%%%%%%%%%                                                             %%%%%%%%%
%%%%%%%%%%%%%%%%%%%%%%%%%%%%%%%%%%%%%%%%%%%%%%%%%%%%%%%%%%%%%%%%%%%%%%%%%%%%%%%%
%
In the spirit of Proposition \ref{PROP1.4} one can characterize
  sc-smoothness of bundle maps into $W(E)$ or out of $W(E)$.
This again follows from the construction of $W(E)\rightarrow X(Q)$ and the
  smoothness characterization in Theorem \ref{THM1.11}.
This immediately gives the following result.

%%%%%%%%%%%%%%%%%%%%%%%%%%%%%%%%%%%%%%%%%%%%%%%%%%%%%%%%%%%%%%%%%%%%%%%%%%%%%%%%
%%%%%%%%%%                          PROPOSITION                        %%%%%%%%%
%%%%%                                                                       %%%%
\begin{proposition}\label{EQNN9.21}
Let ${\mathcal D}$, $\varphi$, and $\delta$ be given and abbreviate
  $X(Q)=X^{3,\delta}_{{\mathcal D},\varphi}(Q)$ and
  $W(E)=W^{3,2,\delta}_{{\mathcal D},\varphi}(E)$.
Assume that $Q$ is connected. 
\begin{itemize}
\item[(1)] 
Let $K\rightarrow Y$ be a strong bundle over a M-polyfold 
  and $\wh{A}:K\rightarrow W(E)$ a map covering $A:Y\rightarrow X(Q)$. 
Then $\wh{A}$ is an sc-smooth strong bundle if and only if for one proper
  smooth complex bundle embedding $\Phi:E\rightarrow {\mathbb R}^N\times
  {\mathbb K}^L$ the map $K\rightarrow {\mathbb B}\times
  H^{3,\delta}_c({\mathcal D},{\mathbb R}^N)\times H^{2,\delta}({\mathcal
  D},{\mathbb K}^L)$ defined by
  $$
  k \rightarrow (f^N\times_{\mathbb B} \wh{f}^L)( \Phi \circ \wh{A}(k))
  $$
  is an sc-smooth strong bundle map.
\item[(2)] 
Let $K\rightarrow Y$ be a strong bundle over a M-polyfold. 
A map $\wh{B}:W(E)\rightarrow K$ is an sc-smooth strong bundle map if and
  only if for one proper smooth complex bundle embedding $\Phi:E\rightarrow
  {\mathbb R}^N\times {\mathbb K}^L$ and tubular neighborhood
  $U=U(\phi(Q))$ which via $r:U\rightarrow U$ smoothly retracts to
  $\phi(Q)$ the composition $\wh{\mathcal U}\rightarrow K$ defined by
  \begin{eqnarray*}
    &(a,(u^x,u^y),(h^x,h^y))\longrightarrow &\\
    &\wh{B}(\Phi^{-1}\circ R \circ 
    ((\oplus\times_{\mathbb B} \wh{\oplus})(a,(u^x,u^y),(h^x,h^y))))&
    \end{eqnarray*}
  is an sc-smooth strong bundle map. 
Here $\wh{\mathcal U}$ is the open subset of ${\mathbb B}\times
  H^{3,\delta}_c({\mathcal D},{\mathbb R}^N)\times H^{2,\delta}({\mathcal
  D},{\mathbb K}^L)$ consisting of all elements $(a,(u^x,u^y),(h^x,h^y)))$
  so that $\oplus\times_{\mathbb B}\wh{\oplus} (a,(u^x,u^y),(h^x,h^y)))$ has
  the image in $U\times {\mathbb K}^L$.
\end{itemize}
\end{proposition}
%%%%%                                                                       %%%%
%%%%%%%%%%                                                             %%%%%%%%%
%%%%%%%%%%%%%%%%%%%%%%%%%%%%%%%%%%%%%%%%%%%%%%%%%%%%%%%%%%%%%%%%%%%%%%%%%%%%%%%%
%
In practice, the above criteria give expressions which can be dealt with
  using results in \cite{HWZ8.7}.

%%%%%%%%%%%%%%%%%%%%%%%%%%%%%%%%%%%%%%%%%%%%%%%%%%%%%%%%%%%%%%%%%%%%%%%%%%%%%%%%
%%%%%%%%%%                         SUB-SECTION                         %%%%%%%%%
%%%%%%%%%%%%%%%%%%%%%%%%%%%%%%%%%%%%%%%%%%%%%%%%%%%%%%%%%%%%%%%%%%%%%%%%%%%%%%%%
%
\subsubsection{A Strong Bundle with $(0,1)$-Forms}
Let ${\mathcal D}$ be an unordered disk pair, $\varphi$ the exponential
  gluing profile, and $\delta$ a strictly increasing sequence of weights
  starting at $\delta_0>0$.
As a target manifold we consider the case of an almost complex manifold
  $(Q,J)$ without boundary.
Then $(TQ,J)\rightarrow Q$ can be considered as a complex vector bundle
  and we obtain the strong bundle
  $$
    W^{3,2,\delta}_{{\mathcal D},\varphi}(TQ,J)\rightarrow
    X^{3,\delta}_{{\mathcal D},\varphi}(Q).
    $$
On every $Z_a$ with $0<|a|<1/4$ we have the canonical vector field $v_a$
  defined by
  \begin{eqnarray}\label{vavector}
    v_a(\sigma^a_{\wh{x}}(s,t)) =\frac{\partial
    \sigma^a_{\wh{x}}}{\partial s}(s,t)\ \ \text{for}\ (s,t)\in
    [0,\varphi(|a|)]\times S^1.
    \end{eqnarray}
It does not depend on the choice of $\wh{x}$. 
If $a=0$ we can define $v_0$ on $D_x\setminus\{x\}$ and
  $D_y\setminus\{y\}$ using $h_{\wh{x}}$ and $h_{\wh{y}}$ as follows
  \begin{eqnarray*}
    v_a(\sigma^+_{\wh{x}}(s,t)) &=&\frac{\partial
    \sigma^+_{\wh{x}}}{\partial s}(s,t)\ \text{for}\ (s,t)\in {\mathbb
    R}^+\times S^1\\
    v_a(\sigma^-_{\wh{y}}(s',t')&= &\frac{\partial
    \sigma^-_{\wh{y}}}{\partial s'}(s',t')\ \text{for}\ (s',t')\in
    {\mathbb R}^-\times S^1.
    \end{eqnarray*}
We define a new strong bundle over $X^{3,\delta}_{{\mathcal
  D},\varphi}(Q)$
  \begin{eqnarray}\label{EQN6}
    \Omega_{{\mathcal D},\varphi}^{3,2,\delta}(Q,J)\rightarrow
    X^{3,\delta}_{{\mathcal D},\varphi}(Q)
    \index{$\Omega_{{\mathcal D},\varphi}^{3,2,\delta}(Q,J)\rightarrow
    X^{3,\delta}_{{\mathcal D},\varphi}(Q)$}
    \end{eqnarray}
  as follows. 
The elements of $\Omega_{{\mathcal D},\varphi}^{3,2,\delta}(Q,J)$ are
  pairs $(u,\xi)$, where $u:Z_a\rightarrow Q$ belongs to
  $X^{3,\delta}_{{\mathcal D},\varphi}(Q)$ and $\xi$ is a map of class $H^2$
  such that $\xi(z):T_zZ_a\rightarrow (T_{u(z)}Q,J)$ is complex anti-linear. 
Using the vector fields $v_a$ we obtain a map 
  \begin{eqnarray}\label{WEQN9.22}
    \Omega_{{\mathcal D},\varphi}^{3,2,\delta}(Q,J)\rightarrow
    W^{3,2,\delta}_{{\mathcal D},\varphi}(TQ,J)
    \end{eqnarray}
  defined by
  $$
    (u,\xi)\rightarrow (u,\xi\circ v_{a(u)}).
    $$
This map is ${\mathbb C}$-linear in the fibers and a bijection. 
The map covers the identity and  equips (\ref{EQN6}) with the strong
  bundle structure making the above bijection a strong bundle isomorphism.

%%%%%%%%%%%%%%%%%%%%%%%%%%%%%%%%%%%%%%%%%%%%%%%%%%%%%%%%%%%%%%%%%%%%%%%%%%%%%%%%
%%%%%%%%%%                           THEOREM                           %%%%%%%%%
%%%%%                                                                       %%%%
\begin{theorem}\label{THMx7.22}
Given an ordered disk pair ${\mathcal D}$, the exponential gluing profile
  $\varphi$, and a strictly increasing sequence of weights $\delta$ starting
  at $\delta_0>0$ there exists a natural construction of a strong bundle
  over a M-polyfold which associates to a almost complex manifold $(Q,J)$
  without boundary and associated complex bundle $(TQ,J)\rightarrow Q$ a
  strong bundle
  $$
    \Omega^{3,2,\delta}_{{\mathcal D},\varphi}(Q,J)\rightarrow
    X^{3,\delta}_{{\mathcal D},\varphi}(Q).
    $$
The underlying sets are defined in (\ref{EQN6}). 
The identification (\ref{WEQN9.22}) defines the strong bundle structure in
  terms of the strong bundle
  $$
    W^{3,2,\delta}_{{\mathcal D},\varphi}(TQ,J)\rightarrow
    X^{3,\delta}_{{\mathcal D},\varphi}(Q),
    $$
  which is obtained via the extension result Theorem \ref{THM9.20} for
  construction functors.
Via these identifications sc-smooth questions concerning maps from or into
  $\Omega^{3,2,\delta}_{{\mathcal D},\varphi}(Q,J)$ can be decided via
  Proposition \ref{EQNN9.21}.
\end{theorem}
%%%%%                                                                       %%%%
%%%%%%%%%%                                                             %%%%%%%%%
%%%%%%%%%%%%%%%%%%%%%%%%%%%%%%%%%%%%%%%%%%%%%%%%%%%%%%%%%%%%%%%%%%%%%%%%%%%%%%%%
%

%%%%%%%%%%%%%%%%%%%%%%%%%%%%%%%%%%%%%%%%%%%%%%%%%%%%%%%%%%%%%%%%%%%%%%%%%%%%%%%%
%%%%%%%%%%                         SUB-SECTION                         %%%%%%%%%
%%%%%%%%%%%%%%%%%%%%%%%%%%%%%%%%%%%%%%%%%%%%%%%%%%%%%%%%%%%%%%%%%%%%%%%%%%%%%%%%
%
\subsection{Some Useful Remarks}\label{USEFul0}
For the analysis it is sometimes useful to have isomorphic models for
  $X^{3,\delta}_{{\mathcal D},\varphi}(Q)$ or $X^{3,\delta}_{{\mathcal
  D},\varphi}({\mathbb R}^N)$ which rather than the abstract glued domain
  use those arising in the Subsection \ref{BASICidea}.
In this case we start with the infinite half-cylinders ${\mathbb
  R}^\pm\times S^1$ and  take complex gluing parameters $a=|a|\cdot e^{2\pi
  i\theta}$ with $|a|<1/4$.
The gluing with $a=0$ results in the disjoint union of the half-cylinders. 
If $a\neq 0$ we define $Z_a^0$ with $R=\varphi(|a|)$ by 
  \begin{eqnarray}
    \ \ \ Z^0_a&=&\{\{(s,t),(s',t')\}\ |\ t,t'\in S^1,\ s\in [0,R],\ s'\in
    [-R,0],\\
    &&\ \ \ \ \ s=s'+R,\ t=t'+\theta\}.\nonumber
    \end{eqnarray}
If $a\neq 0$ we have two sets of natural coordinates on $Z^0_a$ given by
  the bijections
  \begin{eqnarray}
    &[0,R]\times S^1\leftarrow Z^0_a\rightarrow [-R,0]\times S^1&\\
    &(s,t)\leftarrow \{(s,t),(s',t')\}\rightarrow (s',t').&\nonumber
    \end{eqnarray}
There is a unique smooth structure on $Z^0_a$ making both maps
  diffeomorphisms.
Consider the set obtained by taking the union of all Hilbert spaces
  $H^3(Z^0_a,{\mathbb R}^N)$  for $0<|a|<1/4$ together with the set of pairs
  $(\wt{u}^+,\wt{u}^-)\in H^{3,\delta_0}_c({\mathbb R}^+\times S^1,{\mathbb
  R}^N)\times H^{3,\delta_0}_c({\mathbb R}^-\times S^1,{\mathbb R}^N)$ with
  matching asymptotic constants.
Let us denote this set by $X^{3,\delta_0}_{\varphi}({\mathbb R}^N)$.
Given an unordered disk pair ${\mathcal D}=(D_x\sqcup D_y,\{x,y\})$ we
  pick an ordering of $\{x,y\}$ and obtain, say the ordered nodal pair
  $(x,y)$.
Then we fix  decorations resulting in $(\wh{x}_0,\wh{y}_0)$ defining
  $[\wh{x}_0,\wh{y}_0]$.
The map
  $$
  \mathfrak{a}:\{a\in {\mathbb C}\ |\ |a|<1/4\}\rightarrow {\mathbb
  B}_{\mathcal D}:|a|\cdot e^{2\pi i\theta }\rightarrow |a|\cdot [e^{2\pi
  i\theta}\cdot \wh{x}_0,\wh{y}_0]
  $$
  defines a biholmorphic map by the definition of the structure on
  ${\mathbb B}_{\mathcal D}$. 
With $(\wh{x}_0,\wh{y}_0)$ in place we take the uniquely determined
  biholomorphic maps $h_{\wh{x}_0}:(D_x,x)\rightarrow ({\mathbb D},0)$ and
  $h_{\wh{y}_0}:(D_y,y)\rightarrow ({\mathbb D},0)$ such that
  $Th_{\wh{x}_0}\wh{x}_0= {\mathbb R}$ and similarly $Th_{\wh{y}_0}\wh{y}_0=
  {\mathbb R}$.
We note that $h_{e^{2\pi i\theta}\cdot \wh{x}_0}= e^{-2\pi i\theta}\cdot
  h_{\wh{x}_0}$.
Associated to $h_{e^{2\pi i\theta}\cdot\wh{x}_0}$ we take positive
  holomorphic polar coordinates denoted by
  $$
    \sigma_{\theta}^+(s,t) =h_{e^{2\pi i\theta}\cdot
    \wh{x}_0}^{-1}(e^{-2\pi (s+it)}),
    $$
  and negative holomorphic polar coordinates associated to $h_{\wh{y}_0}$
  defined by
  $$
    \sigma^-(s',t')= h_{\wh{y}_0}^{-1}(e^{2\pi (s+it)}).
    $$
Using these maps we can define for $a\in {\mathbb B}$ with $|a|<1/4$ the
  following maps.
If $a=0$ and consequently $\mathfrak{a}(a)=0$ we define 
  $$
  \phi_0:({\mathbb R}^+\times S^1)\sqcup ({\mathbb R}^-\times
  S^1)\rightarrow D_x\sqcup D_y
  $$
  by $\phi_0(s,t) = \sigma^+(s,t)$ and $\phi_0(s',t')=\sigma^-(s',t')$. 
If $0<|a|<1/4$ with $\mathfrak{a}=\mathfrak{a}(a)$ we define
  $Z^0_a\rightarrow Z_{\mathfrak{a}}$ by 
  $$
    \phi_a(\{(s,t),(s',t')\}) = \{\sigma_\theta^+(s,t),\sigma^-(s',t')\}.
    $$
Using these maps we can define a bijection from an M-polyfold to a set
  $$
    X^{3,\delta}_{{\mathcal D},\varphi}({\mathbb R}^N) \rightarrow
    X^{3,\delta_0}_{\varphi}({\mathbb R}^N)
    $$
  by mapping $\wt{u}:Z_{\mathfrak{a}}\rightarrow {\mathbb R}^N$ to
  $\wt{u}\circ\phi_a$.
There exists a unique M-polyfold structure on the set making this
  bijection a sc-diffeomorphism.
It is an easy exercise using the results about actions by diffeomorphism
  that making different choices in the construction defines the same
  M-polyfold structure on the set.
This also holds if we replace ${\mathbb R}^N$ by $Q$.  
We can, of course, fit strong bundles into this context and finally look
  at the Cauchy-Riemann section, which is more explicit for the coordinates
  coming from polar coordinates.
We shall use this in the next subsection.

%%%%%%%%%%%%%%%%%%%%%%%%%%%%%%%%%%%%%%%%%%%%%%%%%%%%%%%%%%%%%%%%%%%%%%%%%%%%%%%%
%%%%%%%%%%                         SUB-SECTION                         %%%%%%%%%
%%%%%%%%%%%%%%%%%%%%%%%%%%%%%%%%%%%%%%%%%%%%%%%%%%%%%%%%%%%%%%%%%%%%%%%%%%%%%%%%
%
\subsection{Sc-Smoothness of the  CR-Operator}  \label{CROPP}
Assume that ${\mathcal D}=(D_x\sqcup D_y,\{x,y\})$ is an un-ordered disk
  pair and $(Q,J)$ a manifold without boundary equipped with a smooth almost
  complex structure.
With the exponential gluing profile $\varphi$ and the usual weight
  sequence $\delta$ we obtain the strong bundle
  $p_X:\Omega^{3,2,\delta}_{{\mathcal D},\varphi}(Q,J)\rightarrow
  X^{3,\delta}_{{\mathcal D},\varphi}(Q)$.
For simplicity of notation we abbreviate it by
  \begin{eqnarray}
    p:\Omega\rightarrow X.
    \end{eqnarray}
Denoting the almost complex structure on $Z_a$ by $j_a$, the CR-section
  $\bar{\partial}:=\bar{\partial}_{J}$ \index{$\bar{\partial}_{J}$} of $p$
  given by
  \begin{eqnarray}
    X\rightarrow\Omega: u\rightarrow \left(u,\frac{1}{2}\cdot\left
    [Tu+J\circ Tu\circ j_{a(u)}\right]\right)
    \end{eqnarray}
  is well-defined.

%%%%%%%%%%%%%%%%%%%%%%%%%%%%%%%%%%%%%%%%%%%%%%%%%%%%%%%%%%%%%%%%%%%%%%%%%%%%%%%%
%%%%%%%%%%                         SUB-SECTION                         %%%%%%%%%
%%%%%%%%%%%%%%%%%%%%%%%%%%%%%%%%%%%%%%%%%%%%%%%%%%%%%%%%%%%%%%%%%%%%%%%%%%%%%%%%
%
\subsubsection{Sc-Smoothness}
We shall establish the sc-smoothness of the CR-section as well as its
  regularizing property.
These are two of the four properties which define a so-called {\bf
  pre-Fredholm section} (Note that we do not have boundary conditions, so
  hat we cannot expect a Fredholm property yet.).

%%%%%%%%%%%%%%%%%%%%%%%%%%%%%%%%%%%%%%%%%%%%%%%%%%%%%%%%%%%%%%%%%%%%%%%%%%%%%%%%
%%%%%%%%%%                          PROPOSITION                        %%%%%%%%%
%%%%%                                                                       %%%%
\begin{proposition}[$\bar{\partial}$-Sc-Smoothness]\label{PROPI9.20}
 \hfill\\
The section $\bar{\partial}$ is sc-smooth.
\end{proposition}
%%%%%                                                                       %%%%
%%%%%%%%%%                                                             %%%%%%%%%
%%%%%%%%%%%%%%%%%%%%%%%%%%%%%%%%%%%%%%%%%%%%%%%%%%%%%%%%%%%%%%%%%%%%%%%%%%%%%%%%
%
\begin{proof}
The sc-smoothness of $\bar{\partial}$ is a property which can be studied
  around a given point $u$. 
From the definition of the strong bundle structure on $\Omega\rightarrow
  X$ it suffices to study $X\rightarrow W^{3,2,\delta}_{{\mathcal
  D},\varphi}(TQ):u\rightarrow (\bar{\partial} u)\circ v_{a(u)}$. 
Since $J:TQ\rightarrow TQ$ defines a smooth bundle map it suffices to show
  that the sections of $W^{3,2,\delta}_{{\mathcal D},\varphi}(TQ)\rightarrow
  X^{3,\delta}_{{\mathcal D},\varphi}(Q)$ given by
  \begin{eqnarray}\label{XEQN13}
    u\rightarrow Tu\circ v_{a(u)}\ \ \text{and}\ \ u\rightarrow Tu\circ
    j_{a(u)}\circ v_{a(u)}
    \end{eqnarray}
  are sc-smooth.
Note that at this point $J$ is irrelevant.  
By definition of the M-polyfold structures we may  assume that $Q$ is a
  properly embedded submanifold  in some ${\mathbb R}^N$
  which defines automatically a bundle embedding of $TQ$ into ${\mathbb
  R}^N\times {\mathbb R}^N\subset {\mathbb R}^N\times {\mathbb K}^N$.
In this case  the two sections in (\ref{XEQN13}) are restrictions of the
  sections of the models
  $$
    W(N,N)\rightarrow X(N)
    $$
  defined by the same formulae $u\rightarrow (u,Tu)$ and $u\rightarrow
  (u,Tu\circ j_{a(u)}\circ v_{a(u)})$.
By definition of $W(N,N)$  sc-smoothness follows from the sc-smoothness of
  the principal parts
  $$
  X^{3,\delta}_{{\mathcal D},\varphi}({\mathbb R}^N)\rightarrow
  X^{2,\delta}_{{\mathcal D},\varphi,0}({\mathbb K}^N).
  $$
Abbreviate the two maps, which we are studying, by $\mathsf{T}$ and
  $\mathsf{S}$, where $\mathsf{T}(u)= Tu\circ v_{a(u)}$ and
  $\mathsf{S}(u)=Tu\circ j_{u(a)}\circ v_{a(u)}$.
By the definition of the M-polyfold structures we need to show that maps
  $$
    {\mathbb B}\times H^{3,\delta}_c({\mathcal D},{\mathbb R}^N)\rightarrow
    H^{2,\delta}({\mathcal D},{\mathbb K}^N)
    $$
  defined by
  $$
    (\eta^x,\eta^y)\rightarrow \wh{f}\circ \mathsf{T}\circ
    \oplus(a,(\eta^x,\eta^y)),\ \ (\eta^x,\eta^y)\rightarrow \wh{f}\circ
    \mathsf{S}\circ \oplus(a,(\eta^x,\eta^y))
    $$
  are sc-smooth. 
We define $(\xi^x,\xi^y)=\wh{f}\circ \mathsf{T}\circ
  \oplus(a,(\eta^x,\eta^y))$.
Then 
  \begin{eqnarray*}
    &\xi^x(z) =\beta_{-2}(a,z)\cdot &\\
    &\cdot T([\{z,z'\})\rightarrow
    (\beta(z,a)\eta^x(z)+\beta(a,z')\cdot\eta^y(z'))])\circ
    v_{a(u)}(\{z,z'\})&
    \end{eqnarray*}
Due to the symmetry of the situation it suffices to show that the map
  ${\mathbb B}\times H^{3,\delta}_c({\mathcal D},{\mathbb R}^N)\rightarrow
  H^{2,\delta}((D_x,x),{\mathbb K}^N):(a,(\eta^x,\eta^y))\rightarrow
  \xi^x$ is sc-smooth.
The involved sc-Hilbert spaces are defined via isomorphisms to the
  corresponding sc-Hilbert spaces associated to ${\mathbb R}^\pm \times
  S^1$. 
As a consequence we need to show the sc-smoothness of certain expressions
  $$
    {\mathbb B}_{\mathbb C}\times E\rightarrow H^{2,\delta}({\mathbb
    R}^+\times S^1,{\mathbb K}^N),
    $$
  where 
  $$
  E\subset H^{3,\delta}_c({\mathbb R}^+\times S^1,{\mathbb R}^N)\times
    H^{3,\delta}_c({\mathbb R}^-\times S^1,{\mathbb R}^N)
    $$
  is the the codimension $N$ linear subspace consisting of pairs having
  matching asymptotic constants.
Using the same letters we rewrite the expression for $\xi^x$ in this
  context, were $R=\varphi(|a|)$ and $a=|a|\cdot e^{-2\pi i\theta}$.
We also decompose $\eta^x=c+r^x$ and $\eta^y=c+r^y$, where $c$ is the
  common asymptotic constant:
\begin{eqnarray}
  &&\xi^x(s,t)\\
  &=&\beta(s-R/2-2)\cdot \frac{\partial}{\partial
  t}\left(\beta(s-R/2)\eta^x(s,t)\right)\nonumber\\
  &&+\beta(s-R/2-2)\cdot \frac{\partial}{\partial
  t}\left((1-\beta(s-R/2))\eta^y(s-R,t-\theta)\right).\nonumber\\
  &=& \beta(s-R/2-2)\cdot \frac{\partial}{\partial
  t}\left(\beta(s-R/2)r^x(s,t)\right)\nonumber\\ &&+\beta(s-R/2-2)\cdot
  \frac{\partial}{\partial
  t}\left((1-\beta(s-R/2))r^y(s-R,t-\theta)\right)\nonumber 
  \end{eqnarray}
The operator $\partial/\partial t$ (and also $\partial/\partial s$) $
  H^{3,\delta}_c({\mathbb R}^+\times S^1,{\mathbb R}^N)\rightarrow
  H^{2,\delta}_0({\mathbb R }^+\times S^1,{\mathbb R}^N)$ is obviously an
  sc-operator and therefore sc-smooth.
For the sc-smoothness assertion we only have to consider the following
  maps:
\begin{itemize}
  \item[(1)] 
  ${\mathbb B}_{\mathcal D}\times H^{3,\delta}_c({\mathbb R}^+\times
  S^1,{\mathbb R }^N)\rightarrow H^{3,\delta}({\mathbb R}^+\times
  S^1,{\mathbb R}^N):\eta^x\rightarrow (\beta(s-R/2))\cdot r^x$.
  \item[(2)] 
    ${\mathbb B}_{\mathcal D}\times H^{3,\delta}_c({\mathbb R}^-\times
    S^1,{\mathbb R}^N)\rightarrow H^{3,\delta}({\mathbb R}^+\times
    S^1,{\mathbb R }^N):\eta^y\rightarrow ((1-\beta(s-R/2))r^y(s-R,t-\theta)$
  \item[(3)] 
  ${\mathbb B}_{\mathbb K}\times H^{2,\delta}_0({\mathbb R}^+\times
    S^1,{\mathbb R}^N)\rightarrow H^{2,\delta}({\mathbb R}^+\times
    S^1,{\mathbb K}^N): (a,e^x)\rightarrow \beta(s-R/2-2)\cdot e^x$.
  \end{itemize}
These expressions are all sc-smooth by the already previously used result
  the ``Fundamental  Lemma'' described in Subsection \ref{FUND0}. 
The discussion of $\mathsf{S}$ is similar with $\partial/\partial t$
  replaced by $\partial/\partial s$ and which involves some more terms, but
  all covered by the fundamental lemma.
\end{proof}%%%%%%%%%%%%%%%%%%%%%%%%%%%%%%%%%%%%%%%%%%%%%%%%            END PROOF
%

%%%%%%%%%%%%%%%%%%%%%%%%%%%%%%%%%%%%%%%%%%%%%%%%%%%%%%%%%%%%%%%%%%%%%%%%%%%%%%%%
%%%%%%%%%%                         SUB-SECTION                         %%%%%%%%%
%%%%%%%%%%%%%%%%%%%%%%%%%%%%%%%%%%%%%%%%%%%%%%%%%%%%%%%%%%%%%%%%%%%%%%%%%%%%%%%%
%
\subsection{More Examples and Exercises}\label{SSUB7.6}
First we shall derive some more examples of $\oplus$-constructions and
  later introduce a filled version, see Definition \jwf{[broken reference]}
  %\ref{def:filled version},
  of the nodal Cauchy-Riemann operator.

%%%%%%%%%%%%%%%%%%%%%%%%%%%%%%%%%%%%%%%%%%%%%%%%%%%%%%%%%%%%%%%%%%%%%%%%%%%%%%%%
%%%%%%%%%%                         SUB-SECTION                         %%%%%%%%%
%%%%%%%%%%%%%%%%%%%%%%%%%%%%%%%%%%%%%%%%%%%%%%%%%%%%%%%%%%%%%%%%%%%%%%%%%%%%%%%%
%
\subsubsection{The $\ominus$-Map and Associated Spaces}\index{$\ominus$}
Consider for an ordered nodal disk pair ${\mathcal D}=(D_x\sqcup D_y,
  (x,y))$ the sc-Hilbert space $H^{3,\delta}_c({\mathcal D},{\mathbb R}^N)$,
  which consists of pairs of maps $(\wt{u}^x,\wt{u}^y)$ having at $x$ and
  $y$ the common asymptotic limit.
Using ${\mathcal D}$ and a gluing parameter in ${\mathbb B}_{\mathcal D}$
  we define the following Riemann surfaces.
First of all we put $C_0=\emptyset$.  
If $0<|a|<1/4$ we take the disjoint union $ D_y\sqcup D_x$ and with
  $a=|a|\cdot [\wh{x},\wh{y}]$ we take a representative
  $(\wh{x},\wh{y})$ and identify the points $\sigma^+_{\wh{x}}(s,t)$ with
  $s\in [0,R]\times S^1$,  with the points $\sigma^-_{\wh{y}}(s-R,t)$.
We obtain the Riemann surface $C_a$ which is a Riemann sphere with
  distinguished points $x$ and $y$.
We note that we can identify naturally $Z_a$ with a subset of $C_a$
  provided $a\neq 0$.
Given $\wh{x}$ there is a canonical biholomorphic map
  $$
    \sigma_{\wh{x}}^{C_a}:{\mathbb R}\times S^1\rightarrow C_a
    $$
  which on $Z_a$ is $\sigma_{\wh{x}}^{a,+}$.  
We pick a cut-off model $\beta$  and assuming $a\neq 0$ we define with
  $R=\varphi(|a|)$ and $c\in {\mathbb R}^N$ the map
  $\Xi_{a,c}$ by
  \begin{eqnarray}
    &\Xi_{a,c}(\sigma_{\wh{x}}^{C_a}(s,t)) =-(1-\beta(s-R/2))\cdot c +
    \beta(s-R/2)\cdot c&\\
    & \Xi_{a,c}(x)=c\ \ \text{and}\ \ \Xi_{a,c}(y)=-c.&\nonumber
    \end{eqnarray}
We introduce the sc-Hilbert space $H^{3,\delta}_{ap}(C_a,{\mathbb
  R}^N)$\index{$H^{3,\delta}_{ap}(C_a,{\mathbb R}^N)$} as follows.  
If $a=0$ the space has precisely one element, namely the  unique map
  defined on $\emptyset$ with image in ${\mathbb R}^N$. If $a\neq 0$ we
  consider all maps $u$  defined on $C_a$
  such that for a suitable $c\in {\mathbb R}^N$
  $$
    (u-\Xi_{a,c})\circ \sigma^{C_a}_{\wh{x}}\in H^{3,\delta_0}({\mathbb
    R}\times S^1,{\mathbb R}^N).
    $$
The definition of the space does not depend on the cut-off model nor on
  the choice of decoration $\wh{x}$.
We can turn $H^{3,\delta}_{ap}(C_a,{\mathbb R}^N)$ into an sc-Hilbert
  space by requiring level $m$ to consist of objects $u$ for which
  $u-\Xi_{a,c}$ is of class $(m+3,\delta_m)$.
Finally we define the set 
  $$
    X^{\ominus,3,\delta}_{{\mathcal D},\varphi}\index{$
    X^{\ominus,3,\delta}_{{\mathcal D},\varphi}$}=\bigcup_{a\in {\mathbb
    B}_{\mathcal D}}H^{3,\delta}_{ap}(C_a, {\mathbb R}^N).
    $$
Next we define a surjective map 
  $$
    \ominus: {\mathbb B}_{\mathcal D}\times H^{3,\delta}_c({\mathcal
    D},{\mathbb R}^N)\rightarrow  X^{\ominus,3,\delta}_{{\mathcal
    D},\varphi}({\mathbb R}^N)
    $$
  by setting $\ominus(0,u^x,u^y)=0$ and for $a\neq 0$ we proceed as
  follows. 
We define
  $$
  \text{av}_a(u^x):=\int_{S^1} u^x\circ\sigma^+_{\wh{x}}(R/2,t)\cdot dt\
  \ \text{and} \ \  \text{av}_a(u^y):=\int_{S^1}
  u^y\circ\sigma^-_{\wh{y}}(-R/2,t)\cdot dt.
  $$
We also put 
  \begin{eqnarray}
    \text{av}_a(u^x,u^y):=\frac{1}{2}\cdot
    [\text{av}_a(u^x)+\text{av}_a(u^y)].
    \end{eqnarray}
Finally we define $\ominus(a,u^x,u^y)$ as follows.
  \begin{eqnarray}
    &&\ \ \ \ \ominus(a,u^x,u^y)\circ\sigma^{C_a}_{\wh{x}}(s,t) \\
    &=& -(1-\beta(s-R/2))\cdot
    (u^x\circ\sigma^+_{\wh{x}}(s,t)-\text{av}_a(u^x,u^y))\nonumber\\
    && +\beta(s-R/2)\cdot (u^y\circ
    \sigma^-_{\wh{y}}(s-R,t)-\text{av}_a(u^x,u^y)).\nonumber
    \end{eqnarray}

%%%%%%%%%%%%%%%%%%%%%%%%%%%%%%%%%%%%%%%%%%%%%%%%%%%%%%%%%%%%%%%%%%%%%%%%%%%%%%%%
%%%%%%%%%%                           THEOREM                           %%%%%%%%%
%%%%%                                                                       %%%%
\begin{theorem}\label{THMBB7.26}
The map $\ominus: {\mathbb B}_{\mathcal D}\times H^{3,\delta}_c({\mathcal
  D},{\mathbb R}^N)\rightarrow  X^{\ominus,3,\delta}_{{\mathcal
  D},\varphi}$ is surjective and an $\oplus$-construction. 
The canonical map $P^{\ominus}_{\mathbb B}:
  X^{\ominus,3,\delta}_{{\mathcal D},\varphi}({\mathbb R}^N)\rightarrow
  {\mathbb B}_{\mathcal D}$ which extracts the domain parameter is
  sc-smooth and submersive.
\end{theorem}
%%%%%                                                                       %%%%
%%%%%%%%%%                                                             %%%%%%%%%
%%%%%%%%%%%%%%%%%%%%%%%%%%%%%%%%%%%%%%%%%%%%%%%%%%%%%%%%%%%%%%%%%%%%%%%%%%%%%%%%
%

%%%%%%%%%%%%%%%%%%%%%%%%%%%%%%%%%%%%%%%%%%%%%%%%%%%%%%%%%%%%%%%%%%%%%%%%%%%%%%%%
%%%%%%%%%%                          EXERCISE                           %%%%%%%%%
%%%%%                                                                       %%%%
\begin{exercise}\label{EXERC107}
Prove Theorem \ref{THMBB7.26}.  
See also in \cite{HWZ8.7} the discussion of gluing and anti-gluing.
\end{exercise}
%%%%%                                                                       %%%%
%%%%%%%%%%                                                             %%%%%%%%%
%%%%%%%%%%%%%%%%%%%%%%%%%%%%%%%%%%%%%%%%%%%%%%%%%%%%%%%%%%%%%%%%%%%%%%%%%%%%%%%%
%

Recall the M-polyfold with submersive $\bar{a}$ defined in
  (\ref{plot007}), see also Theorem \ref{thm1},
  $$
    p_{\mathbb B}: X^{3,\delta}_{{\mathcal D},\varphi}({\mathbb
    R}^N)\rightarrow {\mathbb B}_{\mathcal D}.
    $$
Recall that the structure on $X^{3,\delta}_{{\mathcal D},\varphi}({\mathbb
  R}^N)$ was defined by a map
  $$
  \oplus:  {\mathbb B}_{\mathcal D}\times H^{3,\delta}_c({\mathcal
  D},{\mathbb R}^N)\rightarrow  X^{3,\delta}_{{\mathcal
  D},\varphi}({\mathbb R}^N).
  $$
In view of a following result and to contrast the latter construction to
  the previous one, let us define
  $$
    X^{\oplus,3,\delta}_{{\mathcal D},\varphi}({\mathbb
    R}^N):=X^{3,\delta}_{{\mathcal D},\varphi}({\mathbb R}^N) \ \
    \text{and}\ \ p_{\mathbb B}^\oplus:=p_{\mathbb B}.
    $$

%%%%%%%%%%%%%%%%%%%%%%%%%%%%%%%%%%%%%%%%%%%%%%%%%%%%%%%%%%%%%%%%%%%%%%%%%%%%%%%%
%%%%%%%%%%                           THEOREM                           %%%%%%%%%
%%%%%                                                                       %%%%
\begin{theorem}\label{THY7.267}
The following holds.
\begin{itemize}
  \item[(1)] 
  The subset
    $$
      X^{\oplus, 3,\delta}_{{\mathcal D},\varphi}({\mathbb
      R}^N){_{p_{\mathbb B}^\oplus}\times_{p_{\mathbb B}^\ominus}}
      X^{\ominus,3,\delta}_{{\mathcal D},\varphi}({\mathbb R}^N)\subset
      X^{\oplus, 3,\delta}_{{\mathcal D},\varphi}({\mathbb R}^N)\times
      X^{\ominus,3,\delta}_{{\mathcal D},\varphi}({\mathbb R}^N)
      $$
    consisting of all $(u,v)$ with $p^{\oplus}_{\mathbb
    B}(u)=p^{\ominus}_{\mathbb B}(v)$ is a sub-M-polyfold.
  \item[(2)] 
  The map
    \begin{eqnarray*}
      &{\mathbb B}_{\mathcal D}\times H^{3,\delta}_c({\mathcal D},{\mathbb
      R}^N)\rightarrow X^{\oplus, 3,\delta}_{{\mathcal D},\varphi}({\mathbb
      R}^N)\times  X^{\ominus,3,\delta}_{{\mathcal D},\varphi}({\mathbb
      R}^N)&\\
      &(a,(u^x,u^y))\rightarrow
      (\oplus(a,u^x,u^y),\ominus(a,u^x,u^y))\nonumber&
      \end{eqnarray*}
    is an sc-diffeomorphism onto $X^{\oplus, 3,\delta}_{{\mathcal
    D},\varphi}({\mathbb R}^N){_{p_{\mathbb B}^\oplus}\times_{p_{\mathbb
    B}^\ominus}} X^{\ominus,3,\delta}_{{\mathcal D},\varphi}({\mathbb
    R}^N)$.
  \item[(3)] 
    The M-polyfold $X^{\oplus, 3,\delta}_{{\mathcal D},\varphi}({\mathbb
    R}^N){_{p_{\mathbb B}^\oplus}\times_{p_{\mathbb B}^\ominus}}
    X^{\ominus,3,\delta}_{{\mathcal D},\varphi}({\mathbb R}^N)$ has a
    uniquely determined sc-manifold structure inducing the M-polyfold
    structure.
  \end{itemize}
\end{theorem}
%%%%%                                                                       %%%%
%%%%%%%%%%                                                             %%%%%%%%%
%%%%%%%%%%%%%%%%%%%%%%%%%%%%%%%%%%%%%%%%%%%%%%%%%%%%%%%%%%%%%%%%%%%%%%%%%%%%%%%%
%

%%%%%%%%%%%%%%%%%%%%%%%%%%%%%%%%%%%%%%%%%%%%%%%%%%%%%%%%%%%%%%%%%%%%%%%%%%%%%%%%
%%%%%%%%%%                          EXERCISE                           %%%%%%%%%
%%%%%                                                                       %%%%
\begin{exercise}\label{EXERC108}
Prove Theorem \ref{THY7.267}.  
See also in \cite{HWZ8.7} the discussion of gluing and anti-gluing.
\end{exercise}
%%%%%                                                                       %%%%
%%%%%%%%%%                                                             %%%%%%%%%
%%%%%%%%%%%%%%%%%%%%%%%%%%%%%%%%%%%%%%%%%%%%%%%%%%%%%%%%%%%%%%%%%%%%%%%%%%%%%%%%
%

\subsubsection{The $\wh{\ominus}$-Map and Associated Spaces}
  \index{$\wh{\ominus}$}
Using the variations described in Subsection \ref{VARIATION1} we can
  introduce some more examples.
The important construction shown to be an $\oplus$-construction is
  $$
  \wh{\oplus}: {\mathbb B}_{\mathcal D}\times H^{2,\delta}({\mathcal
  D},{\mathbb K}^N)\rightarrow X^{2,\delta}_{{\mathcal
  D},\varphi,0}({\mathbb K}^N),\index{$\wh{\oplus}$}
  $$
  see (\ref{WEQN3.18}). 
The domain extraction is also known to be submersive.  
There is an associated $\wh{\ominus}$-construction which has a somewhat
  easier form than that for $\ominus$.
For that reason we rename the above data as follows
  $$
    p_{\mathbb B}^{\wh{\oplus}}: X^{\wh{\oplus},2,\delta}_{{\mathcal
    D},\varphi,0}({\mathbb K}^N)\rightarrow {\mathbb B}_{\mathcal D}.
    $$
Using the previous definitions of the $C_a$ we can define
  $H^{2,\delta}(C_a,{\mathbb K}^N)$ to be the sc-Hilbert space consisting of
  all the maps $u:C_a\rightarrow {\mathbb K}^N$ of class $(2,\delta_0)$
  (with vanishing asymptotic limit).
The level $m$ consists of regularity $(m+2,\delta_m)$. 
We define
  \begin{eqnarray}
    X^{\wh{\ominus},2,\delta}_{{\mathcal D},\varphi,0}({\mathbb K}^N) :=
    \bigcup_{a\in {\mathbb B}_{\mathcal D}} H^{2,\delta}(C_a,{\mathbb
    K}^N)\index{$X^{\wh{\ominus},2,\delta}_{{\mathcal
    D},\varphi,0}({\mathbb K}^N)$}
    \end{eqnarray}
  and the domain parameter extraction
  \begin{eqnarray*}
    p^{\wh{\ominus}}_{{\mathbb B}}:X^{\wh{\ominus},2,\delta}_{{\mathcal
    D},\varphi,0}({\mathbb K}^N) \rightarrow {\mathbb B}_{\mathcal D}.
    \end{eqnarray*}
We observe that the fiber over $a=0$ consists precisely of the zero
  element.
We define 
  $$
  \wh{\ominus}:{\mathbb B}_{\mathcal D}\times H^{2,\delta}({\mathcal
  D},{\mathbb K}^N)\rightarrow X^{\wh{\ominus},2,\delta}_{{\mathcal
  D},\varphi,0}({\mathbb K}^N)
  $$
  by mapping $(0,u^x,u^y)$ to $0$ and for $a\neq 0$ with $R=\varphi(|a|)$
  and $a=|a|\cdot [\wh{x},\wh{y}]$
  \begin{eqnarray*}
    \wh{\ominus}(a,u^x,u^y)\circ \sigma^{C_a}_{\wh{x}}(s,t)
    &=& -(1-\beta(s-R/2))\cdot u^x\circ\sigma_{\wh{x}}^+(s,t) \\
    &&+\beta(s-R/2)\cdot u^y\circ \sigma^-_{\wh{y}}(s-R,t)\cdot dt.
    \end{eqnarray*}
The following theorem is related to the gluing and anti-gluing discussion
  in \cite{HWZ8.7}.

%%%%%%%%%%%%%%%%%%%%%%%%%%%%%%%%%%%%%%%%%%%%%%%%%%%%%%%%%%%%%%%%%%%%%%%%%%%%%%%%
%%%%%%%%%%                           THEOREM                           %%%%%%%%%
%%%%%                                                                       %%%%
\begin{theorem}\label{THY7.2671}
The following holds.
\begin{itemize}
  \item[(1)] 
  The map $\wh{\ominus}$ is a M-polyfold construction.
  \item[(2)] 
  For the M-polyfold structure on $X^{\wh{\ominus},2,\delta}_{{\mathcal
    D},\varphi,0}({\mathbb K}^N)$ the map $p^{\wh{\ominus}}_{\mathbb B}$ is
    sc-smooth and submersive.
  \item[(3)] 
  The subset
    $$
      X^{\wh{\oplus}, 2,\delta}_{{\mathcal D},\varphi,0}({\mathbb
      K}^N){_{p_{\mathbb B}^{\wh{\oplus}}}\times_{p_{\mathbb
      B}^{\wh{\ominus}}}} X^{\wh{\ominus},2,\delta}_{{\mathcal
      D},\varphi,0}({\mathbb K}^N)\subset X^{\wh{\oplus},
      2,\delta}_{{\mathcal D},\varphi,0}({\mathbb K}^N)\times
      X^{\wh{\ominus},2,\delta}_{{\mathcal D},\varphi,0}({\mathbb K}^N)
      $$
    consisting of all $(u,v)$ with $p^{\wh{\oplus}}_{\mathbb
    B}(u)=p^{\wh{\ominus}}_{\mathbb B}(v)$ is a sub-M-polyfold.
  \item[(4)] 
  The map
    \begin{eqnarray*}
      &{\mathbb B}_{\mathcal D}\times H^{2,\delta}({\mathcal D},{\mathbb
      K}^N)\rightarrow X^{\wh{\oplus}, 2,\delta}_{{\mathcal
      D},\varphi,0}({\mathbb K}^N)\times
      X^{\wh{\ominus},2,\delta}_{{\mathcal D},\varphi,0}({\mathbb K}^N)&\\
      &(a,(u^x,u^y))\rightarrow (\wh{\oplus}(a,u^x,u^y),
      \wh{\ominus}(a,u^x,u^y))\nonumber&
      \end{eqnarray*}
    is an sc-diffeomorphism onto $X^{\wh{\oplus}, 2,\delta}_{{\mathcal
    D},\varphi,0}({\mathbb K}^N){_{p_{\mathbb
    B}^{\wh{\oplus}}}\times_{p_{\mathbb B}^{\wh{\ominus}}}}
    X^{\wh{\ominus},2,\delta}_{{\mathcal D},\varphi,0}({\mathbb K}^N)$.
  \item[(5)] 
  The M-polyfold  $X^{\wh{\oplus}, 2,\delta}_{{\mathcal
  D},\varphi,0}({\mathbb K}^N){_{p_{\mathbb
  B}^{\wh{\oplus}}}\times_{p_{\mathbb B}^{\wh{\ominus}}}}
  X^{\wh{\ominus},2,\delta}_{{\mathcal D},\varphi,0}({\mathbb K}^N)$ has a
  uniquely determined sc-manifold structure inducing the M-polyfold
  structure.
  \end{itemize}
\end{theorem}
%%%%%                                                                       %%%%
%%%%%%%%%%                                                             %%%%%%%%%
%%%%%%%%%%%%%%%%%%%%%%%%%%%%%%%%%%%%%%%%%%%%%%%%%%%%%%%%%%%%%%%%%%%%%%%%%%%%%%%%
%

%%%%%%%%%%%%%%%%%%%%%%%%%%%%%%%%%%%%%%%%%%%%%%%%%%%%%%%%%%%%%%%%%%%%%%%%%%%%%%%%
%%%%%%%%%%                          EXERCISE                           %%%%%%%%%
%%%%%                                                                       %%%%
\begin{exercise}\label{EXERC109}
Prove Theorem \ref{THY7.2671}.  
See also in \cite{HWZ8.7} the discussion of the hat-version of gluing and
  anti-gluing.
\end{exercise}
%%%%%                                                                       %%%%
%%%%%%%%%%                                                             %%%%%%%%%
%%%%%%%%%%%%%%%%%%%%%%%%%%%%%%%%%%%%%%%%%%%%%%%%%%%%%%%%%%%%%%%%%%%%%%%%%%%%%%%%
%

%%%%%%%%%%%%%%%%%%%%%%%%%%%%%%%%%%%%%%%%%%%%%%%%%%%%%%%%%%%%%%%%%%%%%%%%%%%%%%%%
%%%%%%%%%%                           SECTION                           %%%%%%%%%
%%%%%%%%%%%%%%%%%%%%%%%%%%%%%%%%%%%%%%%%%%%%%%%%%%%%%%%%%%%%%%%%%%%%%%%%%%%%%%%%
%
\section{Periodic Orbit Case}\label{SECX2}
We introduce the local models describing the stretching of a map near the
  cylinder over a periodic orbit.

%%%%%%%%%%%%%%%%%%%%%%%%%%%%%%%%%%%%%%%%%%%%%%%%%%%%%%%%%%%%%%%%%%%%%%%%%%%%%%%%
%%%%%%%%%%                         SUB-SECTION                         %%%%%%%%%
%%%%%%%%%%%%%%%%%%%%%%%%%%%%%%%%%%%%%%%%%%%%%%%%%%%%%%%%%%%%%%%%%%%%%%%%%%%%%%%%
%
\subsection{The Basic Results}
This subsection describes the basic result and we begin with the
  underlying idea.

%%%%%%%%%%%%%%%%%%%%%%%%%%%%%%%%%%%%%%%%%%%%%%%%%%%%%%%%%%%%%%%%%%%%%%%%%%%%%%%%
%%%%%%%%%%                         SUB-SECTION                         %%%%%%%%%
%%%%%%%%%%%%%%%%%%%%%%%%%%%%%%%%%%%%%%%%%%%%%%%%%%%%%%%%%%%%%%%%%%%%%%%%%%%%%%%%
%
\subsubsection{The Basic Idea}
This is a more complicated situation than the nodal case and we give our
  heuristics in a simple case.  
We consider maps into ${\mathbb R}\times {\mathbb R}^N$ and assume that we
  are given a smooth embedding
  $$
    \gamma:S^1\rightarrow {\mathbb R}^N.
    $$
We denote by $[\gamma]$ the collection of all $t\rightarrow \gamma(t+d)$,
  where $d\in S^1$. 
This defines a $S^1$-family of preferred parameterizations of the
  submanifold $\gamma(S^1)$.
Then ${\mathbb R}\times \gamma(S^1)$ is an infinite cylinder in ${\mathbb
  R}\times {\mathbb R}^N$, see Figure \ref{FIG10000}.

%%%%%%%%%%%%%%%%%%%%%%%%%%%%%%%%%%%%%%%%%%%%%%%%%%%%%%%%%%%%%%%%%%%%%%%%%%%%%%%%
%%%%%%%%%%                            FIGURE                           %%%%%%%%%
%%%%%                                                                       %%%%
\begin{figure}[h]
\begin{center}
\includegraphics[width=6.5cm]{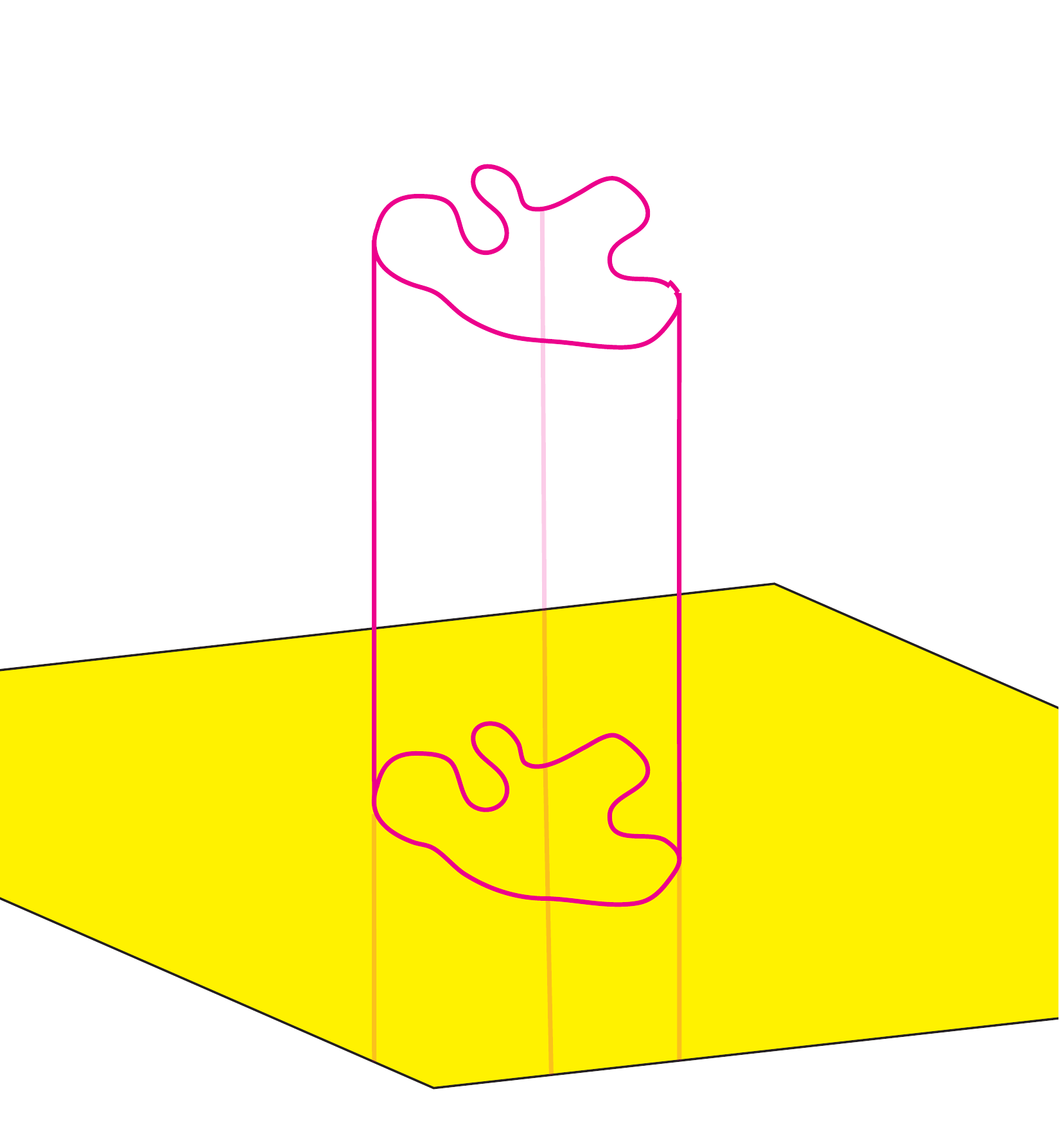}
\end{center}
\caption{The cylinder over $[\gamma]$.}\label{FIG10000}
\end{figure}
%%%%%                                                                       %%%%
%%%%%%%%%%                                                             %%%%%%%%%
%%%%%%%%%%%%%%%%%%%%%%%%%%%%%%%%%%%%%%%%%%%%%%%%%%%%%%%%%%%%%%%%%%%%%%%%%%%%%%%%

The embedded $\gamma(S^1)$ comes with preferred parameterizations
  $[\gamma]$, and due to its linear structure
  ${\mathbb R}$ has several preferred parameterizations as well. 
For example given a number $T>0$ we can consider the following preferred
  parametrization of the infinite cylinder
  \begin{eqnarray}
    &{\mathbb R}\times S^1: (s,t)\rightarrow (Ts+c,\gamma(t+d)),&\\
    &\text{where}\ c\in {\mathbb R}\ \text{and}\ \ d\in S^1.&\nonumber
    \end{eqnarray}
There are modifications of the above which are being used later on. 
For example for the above we can derive some preferred parameterization
  covering the cylinder $k$-fold, obtained from the above by considering the
  preferred parameterizations
  $$
    (s,t)\rightarrow (Ts+c,\gamma(kt+d)).
    $$
We note that $\bm{\gamma}:=([\gamma],T,k)$ captures the needed
  information.

Next we describe a basic idea in the case $([\gamma],T,1)$.
We start by considering tuples $(\wt{u}^+,\wt{u}^-)$  where
  $\wt{u}^\pm:{\mathbb R}^\pm\times S^1\rightarrow {\mathbb R}\times
  {\mathbb R}^N$. 
Moreover these maps are continuous and have the following form
\begin{eqnarray}
  &\wt{u}^+(s,t)= (Ts+c^+, \gamma(t+d^+) + r^+(s,t))&\\
  &\wt{u}^-(s',t')=(Ts'+c^-,\gamma(t+d^-)+r^-(s,t)).&\nonumber
  \end{eqnarray}
Here $r^\pm(s,t)\rightarrow 0$ as $s\rightarrow \pm\infty$, $c^\pm\in
  {\mathbb R}$ and $d^\pm\in S^1$.
We see that the maps $\wt{u}^\pm$ approximate the cylinder associated to
  $[\gamma]$ as $s\rightarrow \pm\infty$, respecting in an approximate
  sense the preferred parameterizations. 
Assume for the moment $r^\pm=0$ and a large number $S>>0$ is given. 
We would like to construct from $\wt{u}^\pm$ which in some sense
  approximate the cylinder at their ends,  a map on a long finite
  cylinder, which approximates the cylinder in its middle part.
The whole process should produce no unnecessary wrinkles, i.e. in some
  sense should be as efficient as possible.
To do so we  glue $\wt{u}^+$ and the shifted $S\ast \wt{u}^-$,  where we
  add $S$ to the first factor.
In order to avoid wrinkles the gluing parameter for the domain has to be
  picked carefully.
Our maps are  $\wt{u}^+(s,t)= (Ts+c^+,\gamma(t+d^+))$ and $S\ast
  \wt{u}^-(s',t')=(Ts'+c^-+S,\gamma(t'+d^-))$.
There is only one way to glue these maps in a way which avoids wrinkles.
Namely we define $R$ by 
  $$
  TR = S+ c^--c^+\ \ \text{and}\ \  d=d^--d^+.
  $$
Then we define the glued map $\wt{w}$ on $Z_{(R,d)}$ by the usual formula 
  \begin{eqnarray}
    && \wt{w}(\{(s,t),(s',t')\})\\
    & =&\beta(s-R/2)\cdot \wt{u}^+(s,t) +\beta(-s'-R/2)\cdot
    S\ast\wt{u}^-(s',t').\nonumber
    \end{eqnarray}
We compute 
  \begin{eqnarray*}
    && \wt{w}(\{(s,t),(s',t')\})\\
    &=& \beta(s-R/2)\cdot (Ts+c^+,\gamma(t+d^+))\\
    &&+\beta(-s'-R/2)\cdot (T(s-R)+c^-+S,\gamma(t-d+d^-))\\
    &=& (Ts+c^+,\gamma(t+d^+))
    \end{eqnarray*}
If $r^\pm$ is nonzero we obtain from $\wt{u}^\pm$ the constants $c^\pm$
  and $d=d^--d^+$ from the asymptotic behavior of $\wt{u}^\pm$.
Then for given $S>>0$ we can compute $R$ and use the gluing associated to
  $(R,d)$.
Hence the modified $\wt{w}$ is given by 
  $$
    \wt{w} = (Ts+c^-,\gamma(t+d^+))+\oplus(R,d,r^+,r^-).
    $$
This defines
\begin{eqnarray*}
  &&\bar{\oplus}(S, \wt{u}^+,\wt{u}^-)(\{(s,t),(s',t')\})\\
  &=& \beta(s-R/2)\cdot  \wt{u}^+(s,t)+\beta(-s'-R/2)\cdot
  \wt{u}^+(s',t').
  \end{eqnarray*}
Again, proper formulated, this will lead to an $\oplus$-construction.  
This time there is no global $f$ 
  which partially inverts $\bm{\bar{\oplus}}$, but one can get away with
  two such maps.
This time the  constructions are more subtle as in the nodal case.
The easier map is like in the nodal case, but the more 
 interesting map is obtained as follows.

We are given $(S,\wt{w})$, where  $\wt{w}:Z_{(R,d)}\rightarrow {\mathbb
  R}\times {\mathbb R}^N$ for $R>>0$ so that in addition the image of a
  middle-loop is close to the cylinder suitably parametrized, i.e.
  compatible with the distinguished coordinates.
From the domain of $\wt{w}$ we extract the parameter $(R,d)$. 
By a subtle averaging construction we can find a candidate for $d^+$ which
  together with $d$ determines $d^-$.
Using  the averaging construction we obtain a candidate for $c^+$ and with
  the help of $S$ and $R$ we obtain $c^-$.
Then we use a cut-off construction as in (\ref{qwas6.3}) to define an
  associated $\wt{w}^\pm$ obtained by transitioning
  from $\wt{w}$ to one of the distinguished parameterizations of the
  cylinder at infinity.
Recall that in (\ref{qwas6.3}) we transitioned to a constant value.
Our considerations in the following sections will also deal with the case
  that the cylinder associated to $[\gamma]$ is multiply-covered, i.e. a
  construction associated to $\bm{\gamma}=([\gamma],T,k)$.

%%%%%%%%%%%%%%%%%%%%%%%%%%%%%%%%%%%%%%%%%%%%%%%%%%%%%%%%%%%%%%%%%%%%%%%%%%%%%%%%
%%%%%%%%%%                            REMARK                           %%%%%%%%%
%%%%%                                                                       %%%%
\begin{remark}
There are two versions of the above, which are equally valid.  
Above we glued $\wt{u}^+$ and $S\ast\wt{u}^-$.
We can equally well also glue  $(-S/2)\ast\wt{u}^+$ and $(S/2)\ast
  \wt{u}^-$, or $(-S)\ast \wt{u}^+$ and $\wt{u}^-$.
Indeed, we shall utilize version 1 and version 3 in the case when we
  construct the M-polyfolds associated to symplectic cobordisms. 
Namely version 1 at the positive ends and version 3 at the negative ends.  
In the case of symplectizations one can take any of these versions. 
However, the fomulae for iterated constructions (more than two maps) get
  more cumbersome for version 2 and our choice is the version 1 gluing.
One should point out that  in this context one  deals with maps modulo the
  ${\mathbb R}$-action $\ast$, so that the different versions only produce
  different representatives, which are mod ${\mathbb R}$ identical.
\end{remark}
%%%%%                                                                       %%%%
%%%%%%%%%%                                                             %%%%%%%%%
%%%%%%%%%%%%%%%%%%%%%%%%%%%%%%%%%%%%%%%%%%%%%%%%%%%%%%%%%%%%%%%%%%%%%%%%%%%%%%%%
%

%%%%%%%%%%%%%%%%%%%%%%%%%%%%%%%%%%%%%%%%%%%%%%%%%%%%%%%%%%%%%%%%%%%%%%%%%%%%%%%%
%%%%%%%%%%                         SUB-SECTION                         %%%%%%%%%
%%%%%%%%%%%%%%%%%%%%%%%%%%%%%%%%%%%%%%%%%%%%%%%%%%%%%%%%%%%%%%%%%%%%%%%%%%%%%%%%
%
\subsubsection{Constructions around Periodic Orbits}
We begin with the  definition of a periodic orbit.

%%%%%%%%%%%%%%%%%%%%%%%%%%%%%%%%%%%%%%%%%%%%%%%%%%%%%%%%%%%%%%%%%%%%%%%%%%%%%%%%
%%%%%%%%%%                          DEFINITION                         %%%%%%%%%
%%%%%                                                                       %%%%
\begin{definition}
A {\bf periodic orbit}\index{periodic orbit} in ${\mathbb R}^N$ is a tuple
  $\boldsymbol{\gamma}=([\gamma],T,k)$, where $T$ is a positive real number,
  $k\geq 1$ a positive integer, $\gamma:S^1\rightarrow {\mathbb R}^N$ a
  smooth embedding, and $[\gamma]$ denotes the set of reparameterizations
  $t\rightarrow \gamma(t+\theta)$ with $\theta\in S^1$.
$T$ is called the {\bf period} and $k$ the {\bf covering number}. 
The number $T_0=T/k$ is called the {\bf minimal period}\index{minimal
  period}.
A {\bf weighted periodic orbit}\index{weighted periodic orbit}
  $\boldsymbol{\bar{\gamma}}$ (in ${\mathbb R}^N$) is a tuple
  $(\boldsymbol{\gamma},\delta)$, where $\delta={(\delta_i)}_{i=0}^\infty$
  is a strictly increasing sequence of weights
  $0<\delta_0<\delta_1<..<\delta_i<\delta_{i+1}$.  
There  are obvious generalization to a notion of periodic orbit in a
  smooth manifold $Q$, where we require that $\gamma:S^1\rightarrow Q$ is an
  embedding.
\end{definition}
%%%%%                                                                       %%%%
%%%%%%%%%%                                                             %%%%%%%%%
%%%%%%%%%%%%%%%%%%%%%%%%%%%%%%%%%%%%%%%%%%%%%%%%%%%%%%%%%%%%%%%%%%%%%%%%%%%%%%%%
%

There are several natural forgetful maps of interest to us, namely
  $\boldsymbol{\bar{\gamma}}\rightarrow \boldsymbol{\gamma}$,
  $\boldsymbol{\bar{\gamma}}\rightarrow \delta$ and
  $\boldsymbol{\bar{\gamma}}\rightarrow \delta_0$.  
In a moment we will define a so-called collection of standard maps,
  however before doing so it will be useful to recall the following
  preliminaries.
Given a disk \(D_x\) with \(x\in D_x\setminus \partial D_x\), we let
  \(\widehat{x}\subset T_x D_x\) denote an oriented line passing through
  \(0\in T_x D_x\), and call \(\widehat{x}\) a decoration.
The set of all decorations in \(T_x D_x\) is denoted \(\mathbb{S}_x\).
Given an ordered\footnote{The same definition holds for an unordered disk
  pair.} disk pair \(((D_x, D_y), (x,y))\), we say that \(\{\widehat{x},
  \widehat{y}\}\) and \(\{\widehat{x}', \widehat{y}'\}\) are equivalent
  provided \(\widehat{x}, \widehat{x}' \in \mathbb{S}_x\) and
  \(\widehat{y}, \widehat{y}' \in \mathbb{S}_y\) satisfy
  \begin{align*}                                                          %% EQN
    \widehat{x}' = e^{2\pi i\theta}\widehat{x}
    \qquad\text{and}\qquad\widehat{y}' = e^{-2\pi i\theta}\widehat{y},
    \end{align*}
  for some \(\theta\in S^1\). 
It can be shown that our notion of equivalency does indeed define an
  equivalence relation, and we let \([\widehat{x}, \widehat{y}]\) denote
  the equivalence class associated to \(\{\widehat{x}, \widehat{y}\}\).
We call such an \([\widehat{x}, \widehat{y}]\) a natural angle, and the
  set of natural angles is denoted by \(\mathbb{S}_{x,y}\).  Further details
  can be found in Appendix \ref{SEC_basic_notions}.
With these notions recalled, we are now prepared to define the collection
  of standard maps associated to a periodic orbit.

%%%%%%%%%%%%%%%%%%%%%%%%%%%%%%%%%%%%%%%%%%%%%%%%%%%%%%%%%%%%%%%%%%%%%%%%%%%%%%%%
%%%%%%%%%%                          DEFINITION                         %%%%%%%%%
%%%%%                                                                       %%%%
\begin{definition}\label{PPP10.3}
Let ${\mathcal D}=((D_x,D_y),(x,y))$ be an ordered disk pair.
Given a periodic orbit $\boldsymbol{\gamma}$ in ${\mathbb R}^N$, the
  associated collection of {\bf standard maps in ${\mathbb R}\times {\mathbb
  R}^N$} \index{standard maps} is  the set
  $\mathsf{S}_{\boldsymbol{\gamma}}$ consisting of tuples
  $(\wt{q}^x,[\wh{x},\wh{y}],\wt{q}^y)$, where
  \([\widehat{x},\widehat{y}]\in \mathbb{S}_{x,y}\) is a natural angle,
  and
  \begin{align*}                                                          %% EQN
    &\tilde{q}^x:D_x\setminus \{x\}\to \mathbb{R}\times \mathbb{R}^N
    \\
    &\tilde{q}^y:D_y\setminus \{y\}\to \mathbb{R}\times \mathbb{R}^N,
    \end{align*}
  satisfying the following conditions.
There exists \(\gamma\in [\gamma]\), \(c^x,c^y\in \mathbb{R}\), and
  representative \(\{\widehat{x},\widehat{y}\}\in [\widehat{x},
  \widehat{y}]\) for which the following holds.
  \begin{eqnarray}
    \wt{q}^x\circ \sigma_{\wh{x}}^+(s,t)&=&(Ts +c^x,\gamma(kt))\in
    {\mathbb R}\times {\mathbb R}^N\\
    \wt{q}^y\circ \sigma_{\wh{y}}^-(s',t')&=&(Ts'+c^y,\gamma(kt'))\in
    {\mathbb R}\times {\mathbb R}^N\nonumber;
    \end{eqnarray} 
  here \(\sigma_{\widehat{x}}^+\) and \(\sigma_{\widehat{y}}^-\) are as in
  equation (\ref{EQ_sig_xhat}) and equation (\ref{EQ_sig_yhat})
%  for a suitable representative $\gamma$ of $[\gamma]$ and suitable
%  constants $c^x,c^y\in {\mathbb R}$.
\end{definition}
%%%%%                                                                       %%%%
%%%%%%%%%%                                                             %%%%%%%%%
%%%%%%%%%%%%%%%%%%%%%%%%%%%%%%%%%%%%%%%%%%%%%%%%%%%%%%%%%%%%%%%%%%%%%%%%%%%%%%%%
%
Given a periodic orbit $\boldsymbol{\gamma}$ in ${\mathbb R}^N$ we have
  for every representative $\gamma$ in $[\gamma]$ a canonical map
  $$
    \phi_\gamma:{\mathbb R}\times {\mathbb S}_x\times {\mathbb R}\times
    {\mathbb S}_y\rightarrow \mathsf{S}_{\boldsymbol{\gamma}}
    $$
  defined by $\phi_\gamma(c^x,\wh{x},c^y,\wh{y})=
  (\wt{q}^\gamma_{c^x,\wh{x}},[\wh{x},\wh{y}],\wt{q}^\gamma_{c^y,\wh{y}})$,
  where
  \begin{eqnarray}\label{EQNC16}
    &\wt{q}^\gamma_{c^x,\wh{x}}\circ\sigma^+_{\wh{x}}(s,t)
    =(Ts+c^x,\gamma(kt))&\\
    &\wt{q}^\gamma_{c^y,\wh{y}}\circ\sigma^-_{\wh{y}}(s',t')=(Ts'
    +c^y,\gamma(kt')).&\nonumber
    \end{eqnarray}
The map $\phi_\gamma$ is $k:1$ and surjective. 
The cyclic group ${\mathbb Z}_k=\{0,...,k-1\}$ acts freely on
  $\wt{\Sigma}_{x,y}:= {\mathbb R}\times {\mathbb S}_x\times {\mathbb
  R}\times {\mathbb S}_y$ via
  $$
    j\ast (c^x,\wh{x},c^y,\wh{y})= (c^x,e^{2\pi i (j/k)}\cdot \wh{x},
    c^y,e^{-2\pi i(j/k)}\cdot \wh{y}).
    $$

%%%%%%%%%%%%%%%%%%%%%%%%%%%%%%%%%%%%%%%%%%%%%%%%%%%%%%%%%%%%%%%%%%%%%%%%%%%%%%%%
%%%%%%%%%%                          PROPOSITION                        %%%%%%%%%
%%%%%                                                                       %%%%
\begin{proposition}\label{PROPT8.84}
Given $\bm{\gamma}=([\gamma],T,k)$ the set $\mathsf{S}_{\bm{\gamma}}$
  admits a natural smooth manifold structure characterized by the property
  that for every $\gamma\in [\gamma]$, the map
  $\phi_\gamma:\wt{\Sigma}_{x,y}\rightarrow \mathsf{S}_\gamma$ is a local
  diffeomorphism.
The preimage of a point under $\phi_\gamma$ is a ${\mathbb Z}_k$-orbit
  and the quotient $\Sigma_{x,y}= \wt{\Sigma}_{x,y}/{\mathbb Z}_k$
  has a natural smooth manifold structure  obtained by the standard
  procedure of dividing out the smooth ${\mathbb Z}_k$-action. 
The induced map $\Sigma_{x,y}\rightarrow \mathsf{S}_\gamma$ is a bijection
  and the desired smooth manifold structure on $\mathsf{S}_{\bm{\gamma}}$
  is characterized by the requirement that this map is a diffeomorphism.
\end{proposition}
%%%%%                                                                       %%%%
%%%%%%%%%%                                                             %%%%%%%%%
%%%%%%%%%%%%%%%%%%%%%%%%%%%%%%%%%%%%%%%%%%%%%%%%%%%%%%%%%%%%%%%%%%%%%%%%%%%%%%%%
%
\begin{proof}
Assume that $\phi_{\gamma}(c^x,\wh{x},c^y,\wh{y})=\phi_\gamma(d^x,\wh{x}',
  d^y,\wh{y}')$.
We deduce that $[\wh{x},\wh{y}]=[\wh{x}',\wh{y}']$ and therefore we find
  $\tau\in [0,1)$ such that
  $$
  \wh{x}'=e^{2\pi i\tau}\cdot \wh{x}\ \ \text{and}\ \ \wh{y}'=e^{-2\pi
  i\tau}\cdot \wh{y}.
  $$
Moreover
  \begin{eqnarray*}
    (Ts+c^x,\gamma(kt))
    &=&\wt{q}^x_{c^x,\wh{x}}\circ\sigma^+_{\wh{x}}(s,t)\\
    &=&\wt{q}^x_{d^x,\wh{x}'}\circ\sigma^+_{\wh{x}}(s,t)\\
    &=&\wt{q}^x_{d^x,\wh{x}'}\circ\sigma^+_{\wh{x}'}(s,t-\tau)\\
    &=&(Ts+d^x,\gamma(k(t-\tau));)
    \end{eqnarray*}
  where we have used the previously established fact that
  \begin{align*}                                                          %% EQN
    h_{e^{2\pi i \tau} \widehat{x}} = e^{-2\pi i \tau} h_{\widehat{x}}
    \end{align*}
  and hence 
  \begin{align*}                                                          %% EQN
    \sigma_{\widehat{x}'}^+ (s,t) = h_{\widehat{x}}^{-1}
    (e^{-2\pi(s+i(t-\tau))}) = \sigma_{\widehat{x}}^+(s,t-\tau)).
    \end{align*}
This implies $k\tau=0$ mod $1$, and  since $\tau\in [0,1)$ that $\tau=j/k$
  for some $j\in \{0,...,k-1\}$.
Further $c^x=d^x$.  
Similarly we show that
  $$
    (Ts'+c^y,\gamma(kt')) = (Ts'+d^y,\gamma(k(t'-\tau))
    $$
  implying that $\tau=j/k$ as before. 
This discussion shows that the preimage of a point is a ${\mathbb
  Z}_k$-orbit.
The rest of the proof is obvious.
\end{proof}%%%%%%%%%%%%%%%%%%%%%%%%%%%%%%%%%%%%%%%%%%%%%%%%            END PROOF

Let $G$ be the automorphism group of the ordered ${\mathcal D}$. Then $G$
  is diffeomorphic to $S^1\times S^1$ and acts on
  $\mathsf{S}_{\boldsymbol{\gamma}}$ via
  $$
  g\ast (\wt{q}^x,[\wh{x},\wh{y}],\wt{q}^y) := (\wt{q}^x\circ
  g^{-1},[Tg\cdot \wh{x}, Tg\cdot \wh{y}],\wt{q}^y\circ g^{-1}).
  $$
This action is smooth since it corresponds under a $\Phi_\gamma$ to the
  smooth action
  $g\ast (c^x,\wh{x},c^y,\wh{y})=(c^x, Tg\cdot \wh{x},c^y, Tg\cdot
  \wh{y})$.

In a next step we introduce a set of maps which, as we shall show, has a
  natural ssc-manifold structure.

%%%%%%%%%%%%%%%%%%%%%%%%%%%%%%%%%%%%%%%%%%%%%%%%%%%%%%%%%%%%%%%%%%%%%%%%%%%%%%%%
%%%%%%%%%%                          DEFINITION                         %%%%%%%%%
%%%%%                                                                       %%%%
\begin{definition}
Consider a map $\wt{w}:{\mathbb R}^\pm\times S^1\rightarrow {\mathbb
  R}\times {\mathbb R}^{N}$ and a periodic orbit $\boldsymbol{\gamma}$ in
  ${\mathbb R}^N$.  
We say that $\wt{w}$ is of class
  $H^{m,\tau}_{\boldsymbol{\gamma}}(\mathbb{R}^\pm \times S^1,
  \mathbb{R}^N)$ provided that there exists $\gamma\in [\gamma]$ and $c\in
  {\mathbb R}$ with the property that the map defined by
  $$
    \wt{v}(s,t) =\wt{w}(s,t) - (Ts+c,\gamma(kt))
    $$
  belongs to $H^{m,\tau}({\mathbb R}^\pm\times S^1,{\mathbb R}^N)$.
\end{definition}
%%%%%                                                                       %%%%
%%%%%%%%%%                                                             %%%%%%%%%
%%%%%%%%%%%%%%%%%%%%%%%%%%%%%%%%%%%%%%%%%%%%%%%%%%%%%%%%%%%%%%%%%%%%%%%%%%%%%%%%
%

%%%%%%%%%%%%%%%%%%%%%%%%%%%%%%%%%%%%%%%%%%%%%%%%%%%%%%%%%%%%%%%%%%%%%%%%%%%%%%%%
%%%%%%%%%%                          DEFINITION                         %%%%%%%%%
%%%%%                                                                       %%%%
\begin{definition}\label{DEFn8.6X}
Let ${\mathcal D}$ be an ordered disk pair and
  $\boldsymbol{\bar{\gamma}}=(\boldsymbol{\gamma},\delta)$ a weighted
  periodic orbit in ${\mathbb R}^N$.
By  $Z_{\mathcal D}^{3,\delta_0}({\mathbb R}\times {\mathbb
  R}^N,\boldsymbol{{\gamma}})$ we denote the set which  consists of all
  $(\wt{u}^x,[\wh{x},\wh{y}],\wt{u}^y)$ of class $(3,\delta_0)$ converging
  to $\boldsymbol{\gamma}$ in a matching way.
This means that $\wt{u}^x$ is of class
  $H^{3,\delta_0}_{\boldsymbol{\gamma}+}$, $\wt{u}^y$ of class
  $H^{3,\delta_0}_{\boldsymbol{\gamma}-}$ and $\wt{u}^x$ and $\wt{u}^y$
  are $[\wh{x},\wh{y}]$-directionally matching.
(see  Appendix \jwf{[broken reference]} %\ref{APP_C}, 
Definition \jwf{[broken reference]} %\ref{DEF-E1} 
and Definition \jwf{[broken reference]} %\ref{DEF-E2}
). 
\end{definition}
%%%%%                                                                       %%%%
%%%%%%%%%%                                                             %%%%%%%%%
%%%%%%%%%%%%%%%%%%%%%%%%%%%%%%%%%%%%%%%%%%%%%%%%%%%%%%%%%%%%%%%%%%%%%%%%%%%%%%%%
%

By $H^{3,\delta}({\mathcal D},{\mathbb R}\times {\mathbb R}^N)$ we denote
  the sc-Hilbert space of maps of class $(3,\delta_0)$ with vanishing
  asymptotic limit.
We define the map
  \begin{eqnarray}\label{PP10.6}
    &\Psi:\mathsf{S}_{\boldsymbol{\gamma}}\times H^{3,\delta_0}({\mathcal
    D},{\mathbb R}\times {\mathbb R}^N)
    \rightarrow Z^{3,\delta_0}_{\mathcal D}({\mathbb R}\times {\mathbb
    R}^N,\boldsymbol{{\gamma}}):&\\
    &(\wt{q},(\wt{h}^x,\wt{h}^y))\rightarrow
    \wt{q}\dotplus(\wt{h}^x,\wt{h}^y),&\nonumber
    \end{eqnarray}
  where the expression $\wt{q}\dotplus(\wt{h}^x,\wt{h}^y)$ for
  $\wt{q}=(\wt{q}^x,[\wh{x},\wh{y}],\wt{q}^y)$ is defined by
  \begin{eqnarray*}
    \wt{q}\dotplus(\wt{h}^x,\wt{h}^y)
    =(\wt{q}^x+\wt{h}^x,[\wh{x},\wh{y}],\wt{q}^y+\wt{h}^y).
    \end{eqnarray*}
It is an easy exercise to verify the following lemma.

%%%%%%%%%%%%%%%%%%%%%%%%%%%%%%%%%%%%%%%%%%%%%%%%%%%%%%%%%%%%%%%%%%%%%%%%%%%%%%%%
%%%%%%%%%%                            LEMMA                            %%%%%%%%%
%%%%%                                                                       %%%%
\begin{lemma}
Let $\boldsymbol{\bar{\gamma}}$ be a weighted periodic orbit. 
Then the map $\Psi$ defined in \em{(\ref{PP10.6})} is a bijection. 
\end{lemma}
%%%%%                                                                       %%%%
%%%%%%%%%%                                                             %%%%%%%%%
%%%%%%%%%%%%%%%%%%%%%%%%%%%%%%%%%%%%%%%%%%%%%%%%%%%%%%%%%%%%%%%%%%%%%%%%%%%%%%%%
%
As a consequence we obtain the following result.

%%%%%%%%%%%%%%%%%%%%%%%%%%%%%%%%%%%%%%%%%%%%%%%%%%%%%%%%%%%%%%%%%%%%%%%%%%%%%%%%
%%%%%%%%%%                          PROPOSITION                        %%%%%%%%%
%%%%%                                                                       %%%%
\begin{proposition}\label{PROPOSITION3.5}
Let $\boldsymbol{\bar{\gamma}}=(\boldsymbol{\gamma},\delta)$ be a weighted
  periodic orbit in ${\mathbb R}^N$ and ${\mathcal D}$ be an ordered disk
  pair.
The set $Z^{3,\delta_0}_{\mathcal D}({\mathbb R}\times {\mathbb
  R}^N,\boldsymbol{{\gamma}})$ has a natural   ssc-manifold structure where
  the $m$-th level is given by regularity $(3+m,\delta_m)$.
This structure is characterized by the fact that the map $\Psi$ is a
  ssc-diffeomorphism.
We shall write $Z_{\mathcal D}({\mathbb R}\times {\mathbb
  R}^N,\boldsymbol{\bar{\gamma}})$ for the associated ssc-manifold.
\end{proposition}
%%%%%                                                                       %%%%
%%%%%%%%%%                                                             %%%%%%%%%
%%%%%%%%%%%%%%%%%%%%%%%%%%%%%%%%%%%%%%%%%%%%%%%%%%%%%%%%%%%%%%%%%%%%%%%%%%%%%%%%
%
The construction $Z_{\mathcal D}({\mathbb R}\times {\mathbb
  R}^N,\boldsymbol{\bar{\gamma}})$ is part of a construction functor.
Quite similar to a procedure described in Theorem \ref{THM1.3} we have a
  canonical extension to target manifolds $Q$ equipped with a weighted
  periodic orbit.
Namely, consider the category whose objects are pairs $({\mathbb
  R}^N,\boldsymbol{\bar{\gamma}})$, where
  $\boldsymbol{\bar{\gamma}}$ is weighted periodic orbit in ${\mathbb R}^N$. 
A morphism
  $$
    h:({\mathbb R}^N,\boldsymbol{\bar{\gamma}})\rightarrow ({\mathbb
    R}^{N'},\boldsymbol{\bar{\gamma}'})
    $$
  consists of a smooth map $h:{\mathbb R}^N\rightarrow {\mathbb R}^{N'}$,
  where  with $\boldsymbol{\bar{\gamma}}=(([\gamma],T,k),\delta)$ we have
  that $\boldsymbol{\bar{\gamma}'}=(([h\circ\gamma],T,k),\delta)$.  
Then with $\wt{h}=Id_{\mathbb R}\times h$ we obtain an induced map
  $$
    \wt{h}_\ast:Z_{\mathcal D}({\mathbb R}\times {\mathbb
    R}^N,\boldsymbol{\bar{\gamma}})\rightarrow Z_{\mathcal D}({\mathbb
    R}\times {\mathbb R}^{N'},\boldsymbol{\bar{\gamma}'}):
    \wt{u}\rightarrow \wt{h}\circ \wt{u}.
    $$
By considering the map on each level of regularity $(3+m,\delta_m)$ we
  obtain a map between Hilbert manifolds and classical smoothness belongs to
  the realm of \cite{El}. As a consequence we obtain the following result
  and the details of the proof are left the reader.

%%%%%%%%%%%%%%%%%%%%%%%%%%%%%%%%%%%%%%%%%%%%%%%%%%%%%%%%%%%%%%%%%%%%%%%%%%%%%%%%
%%%%%%%%%%                          PROPOSITION                        %%%%%%%%%
%%%%%                                                                       %%%%
\begin{proposition}
A morphism $h:({\mathbb R}^N,\boldsymbol{\bar{\gamma}})\rightarrow
  ({\mathbb R}^{N'},\boldsymbol{\bar{\gamma}'})$ induces an ssc-smooth map
  $\wt{h}_\ast:Z_{\mathcal D}({\mathbb R}\times {\mathbb
  R}^N,\boldsymbol{\bar{\gamma}})\rightarrow Z_{\mathcal D}({\mathbb
  R}\times {\mathbb R}^{N'},\boldsymbol{\bar{\gamma}'})$.
\end{proposition}
%%%%%                                                                       %%%%
%%%%%%%%%%                                                             %%%%%%%%%
%%%%%%%%%%%%%%%%%%%%%%%%%%%%%%%%%%%%%%%%%%%%%%%%%%%%%%%%%%%%%%%%%%%%%%%%%%%%%%%%
%
The automorphism group $G_{\mathcal D}$ of the ordered disk pair
  ${\mathcal D}$ acts on the disks by individual rotations and hereby
  defines an action on $Z_{\mathcal D}({\mathbb R}\times {\mathbb
  R}^N,\boldsymbol{\bar{\gamma}})$ via
  $$
    g\ast (\wt{u}^x,[\wh{x},\wh{y}],\wt{u}^y)= (\wt{u}^x\circ
    g^{-1},[Tg\cdot \wh{x},Tg\cdot \wh{y}],\wt{u}^y\circ g^{-1}).
    $$
Via the map $\Psi$ this action corresponds to the sc-smooth action on
  $\mathsf{S}_{\boldsymbol{\gamma}}\times H^{3,\delta}({\mathcal D},{\mathbb
  R}\times {\mathbb R}^N)$ defined by
  $$
    g\ast (\wt{q},(\wt{h}^x,\wt{h}^y)) = (g\ast \wt{q},(\wt{h}^x\circ
    g^{-1},\wt{h}^y\circ g^{-1})).
    $$
Here we use the sc-smoothness of the $G$-action on $H^{3,\delta}({\mathcal
  D},{\mathbb R}\times {\mathbb R}^N)$, see  Proposition \ref{PROP2.6} and
  note that our space is an invariant finite co-dimension subspace. 
Hence we have established the following result.

%%%%%%%%%%%%%%%%%%%%%%%%%%%%%%%%%%%%%%%%%%%%%%%%%%%%%%%%%%%%%%%%%%%%%%%%%%%%%%%%
%%%%%%%%%%                          PROPOSITION                        %%%%%%%%%
%%%%%                                                                       %%%%
\begin{proposition}
  \hfill\\
Let $\boldsymbol{\bar{\gamma}}$ be a weighted periodic orbit,  ${\mathcal
  D}$ be an ordered disk pair and assume that $Z_{\mathcal D}({\mathbb
  R}\times {\mathbb R}^N,\boldsymbol{\bar{\gamma}})$ is the associated
  ssc-manifold.
Then the action of the  automorphism group $G$ of ${\mathcal D}$ is
  sc-smooth.
The action is not(!) ssc-smooth.

\end{proposition}
%%%%%                                                                       %%%%
%%%%%%%%%%                                                             %%%%%%%%%
%%%%%%%%%%%%%%%%%%%%%%%%%%%%%%%%%%%%%%%%%%%%%%%%%%%%%%%%%%%%%%%%%%%%%%%%%%%%%%%%
%

%%%%%%%%%%%%%%%%%%%%%%%%%%%%%%%%%%%%%%%%%%%%%%%%%%%%%%%%%%%%%%%%%%%%%%%%%%%%%%%%
%%%%%%%%%%                            REMARK                           %%%%%%%%%
%%%%%                                                                       %%%%
\begin{remark}
If $\boldsymbol{\bar{\gamma}}$ is a weighted periodic orbit in a manifold
  $Q$ without boundary,  which can be properly embedded into some ${\mathbb
  R}^N$ we have a well-defined $Z_{\mathcal D}({\mathbb R}\times
  Q,\boldsymbol{\bar{\gamma}})$.
Since the maps involved are ssc-smooth the procedure from Theorem
  \ref{THM1.3} produces ssc-manifolds.
We leave the details for this classical case to the reader. 
\end{remark}
%%%%%                                                                       %%%%
%%%%%%%%%%                                                             %%%%%%%%%
%%%%%%%%%%%%%%%%%%%%%%%%%%%%%%%%%%%%%%%%%%%%%%%%%%%%%%%%%%%%%%%%%%%%%%%%%%%%%%%%
%

%%%%%%%%%%%%%%%%%%%%%%%%%%%%%%%%%%%%%%%%%%%%%%%%%%%%%%%%%%%%%%%%%%%%%%%%%%%%%%%%
%%%%%%%%%%                          EXERCISE                           %%%%%%%%%
%%%%%                                                                       %%%%
\begin{exercise}\label{EXERC100003}
Show that the construction $Z_{\mathcal D}$ which associates to $({\mathbb
  R}^N,\bm{\bar{\gamma}})$ the ssc-manifold $Z_{\mathcal D}({\mathbb
  R}\times {\mathbb R}^N,\bm{\bar{\gamma}})$ and to a morphism
  $h$ the ssc-smooth map $\wt{h}_{\ast}$ is a construction functor. 
Conclude using the ideas from Section 
  \jwf{[broken reference]} %\ref{CONST-F}
  that the construction has a natural extension to cover periodic orbits
  in manifolds, i.e. $(Q,\bm{\bar{\gamma}})$.
\end{exercise}
%%%%%                                                                       %%%%
%%%%%%%%%%                                                             %%%%%%%%%
%%%%%%%%%%%%%%%%%%%%%%%%%%%%%%%%%%%%%%%%%%%%%%%%%%%%%%%%%%%%%%%%%%%%%%%%%%%%%%%%
%

%%%%%%%%%%%%%%%%%%%%%%%%%%%%%%%%%%%%%%%%%%%%%%%%%%%%%%%%%%%%%%%%%%%%%%%%%%%%%%%%
%%%%%%%%%%                         SUB-SECTION                         %%%%%%%%%
%%%%%%%%%%%%%%%%%%%%%%%%%%%%%%%%%%%%%%%%%%%%%%%%%%%%%%%%%%%%%%%%%%%%%%%%%%%%%%%%
%
\subsubsection{The M-Polyfold $Y_{{\mathcal D},\varphi}({\mathbb
  R}\times{\mathbb R}^N,\boldsymbol{\bar{\gamma}})$}\label{SUBSUB}
Consider the ssc-manifold with boundary 
  $$
    \mathfrak{Z}:=[0,1)\times Z_{\mathcal D}({\mathbb R}\times {\mathbb
    R}^N,\boldsymbol{\bar{\gamma}})
    $$
  and recall that given $(\wt{u}^x,[\wh{x},\wh{y}],\wt{y}^y)$ in
  $Z_{\mathcal D}({\mathbb R}\times {\mathbb
  R}^N,\boldsymbol{\bar{\gamma}})$ we can write it uniquely as
  $(\wt{q}^x+\wt{h}^x,[\wh{x},\wh{y}],\wt{q}^y+\wt{h}^y)$,
  where $(\wt{q}^x,[\wh{x},\wh{y}],\wt{q}^y)$
  is a standard map. 
We can extract the asymptotic constants $c^x,c^y$ ssc-smoothly, i.e. the
  maps
  $$
    \bar{c}^x,\bar{c}^y: Z_{\mathcal D}({\mathbb R}\times {\mathbb
    R}^N,\boldsymbol{\bar{\gamma}})\rightarrow {\mathbb R}
    $$
  are ssc-smooth. 
We define an open neighborhood ${\mathcal V}$ of
  $\partial\mathfrak{Z}=\{0\}\times Z_{\mathcal D}({\mathbb R}\times
  {\mathbb R}^N,\boldsymbol{\bar{\gamma}})$ in $\mathfrak{Z}$  as follows.

%%%%%%%%%%%%%%%%%%%%%%%%%%%%%%%%%%%%%%%%%%%%%%%%%%%%%%%%%%%%%%%%%%%%%%%%%%%%%%%%
%%%%%%%%%%                          DEFINITION                         %%%%%%%%%
%%%%%                                                                       %%%%
\begin{definition} \label{DEFX3.7}
The open subset ${\mathcal V}$ of $\mathfrak{Z}$ consists of all tuples 
  $$
    (r,(\wt{u}^x,[\wh{x},\wh{y}],\wt{u}^y)) =:(r,\wt{u})
    $$
  such that either $r=0$, or in the case $r\in (0,1)$ the following holds.
  \begin{itemize}
    \item[(1)] 
    $\varphi(r)+c^y(\wt{u})-c^x(\wt{u})>0$.
    \item[(2)] 
    $\varphi^{-1}\left(\frac{1}{T}\cdot(\varphi(r)+c^y(\wt{u})-
    c^x(\wt{u}))\right)\in (0,1/4)$.
    \end{itemize}

\end{definition}
%%%%%                                                                       %%%%
%%%%%%%%%%                                                             %%%%%%%%%
%%%%%%%%%%%%%%%%%%%%%%%%%%%%%%%%%%%%%%%%%%%%%%%%%%%%%%%%%%%%%%%%%%%%%%%%%%%%%%%%
%
We note that ${\mathcal V}$ is a ssc-manifold with boundary. 
An important map is $\bar{r}:{\mathcal V}\rightarrow [0,1):(r,
  \wt{u})\rightarrow r.$
This is the restriction of the projection onto the first factor
  $[0,1)\times Z_{\mathcal D}({\mathbb R}\times {\mathbb
  R}^N,\boldsymbol{\bar{\gamma}})\rightarrow [0,1)$ onto an open subset.
Since the latter map is submersive the same holds for $\bar{r}$. It is
  clear that $\bar{r}$ is surjective. Hence we  obtain.

%%%%%%%%%%%%%%%%%%%%%%%%%%%%%%%%%%%%%%%%%%%%%%%%%%%%%%%%%%%%%%%%%%%%%%%%%%%%%%%%
%%%%%%%%%%                            LEMMA                            %%%%%%%%%
%%%%%                                                                       %%%%
\begin{lemma}
The map 
  $$
    \bar{r}:{\mathcal V}\rightarrow [0,1):(r, \wt{u})\rightarrow r.
    $$
  is a surjective and submersive ssc-smooth map.
\end{lemma}
%%%%%                                                                       %%%%
%%%%%%%%%%                                                             %%%%%%%%%
%%%%%%%%%%%%%%%%%%%%%%%%%%%%%%%%%%%%%%%%%%%%%%%%%%%%%%%%%%%%%%%%%%%%%%%%%%%%%%%%
%

Recall the M-polyfold $X^{3,\delta}_{{\mathcal D},\varphi}({\mathbb
  R}\times {\mathbb R}^N)$ with submersive map
  $$
    p_{\mathbb B} :X^{3,\delta}_{{\mathcal D},\varphi}({\mathbb R}\times
    {\mathbb R}^N)\rightarrow {\mathbb B}
    $$
  introduced in the previous Section \ref{SECX1}. 
We denote by $\dot{X}^{3,\delta}_{{\mathcal D},\varphi}({\mathbb R}\times
  {\mathbb R}^N)$ the preimage of ${\mathbb B}\setminus\{0\}$ under
  $p_{\mathbb B}$.
We have already discussed earlier that the set
  $\dot{X}^{3,\delta}_{{\mathcal D},\varphi}({\mathbb R}\times {\mathbb
  R}^N)$ does neither depend on $\varphi$ nor $\delta$ and the same is true
  for the induced M-polyfold structure as an open subset, see Remarks
  \ref{REMdotX} and \ref{REMARK2.11}.
For that reason we denote this M-polyfold, which also has a natural
  sc-manifold structure inducing the M-polyfold structure in question, by
  $\dot{X}_{\mathcal D}({\mathbb R}\times {\mathbb R}^N)$, see Exercise
  \ref{EXERC101}.

%%%%%%%%%%%%%%%%%%%%%%%%%%%%%%%%%%%%%%%%%%%%%%%%%%%%%%%%%%%%%%%%%%%%%%%%%%%%%%%%
%%%%%%%%%%                          DEFINITION                         %%%%%%%%%
%%%%%                                                                       %%%%
\begin{definition}
  \hfill\\
The M-polyfold  $\mathfrak{X}$ is by definition the open subset of
  $(0,1)\times {X}^{3,\delta}_{{\mathcal D},\varphi}({\mathbb R}\times
  {\mathbb R}^N)$ defined as 
  $$
    \mathfrak{X}=(0,1)\times \dot{X}_{\mathcal D}({\mathbb R}\times
    {\mathbb R}^N)
    $$
  and equipped with the induced M-polyfold structure. 
As we just noted before the latter comes, in fact,  from a uniquely
  natural  defined sc-manifold structure.
We note that the degeneracy index associated to $\mathfrak{X}$ vanishes
  identically.
\end{definition}
%%%%%                                                                       %%%%
%%%%%%%%%%                                                             %%%%%%%%%
%%%%%%%%%%%%%%%%%%%%%%%%%%%%%%%%%%%%%%%%%%%%%%%%%%%%%%%%%%%%%%%%%%%%%%%%%%%%%%%%
%
Next we define the set $Y^{3,\delta_0}_{{\mathcal D},\varphi}({\mathbb
  R}\times {\mathbb R}^N,\boldsymbol{\gamma})$ we are interested in, and
  which we shall equip with a M-polyfold structure by the $\oplus$-method.

%%%%%%%%%%%%%%%%%%%%%%%%%%%%%%%%%%%%%%%%%%%%%%%%%%%%%%%%%%%%%%%%%%%%%%%%%%%%%%%%
%%%%%%%%%%                          DEFINITION                         %%%%%%%%%
%%%%%                                                                       %%%%
\begin{definition}
  \hfill\\
Given an ordered disk pair ${\mathcal D}$ and a weighted periodic orbit
  $\boldsymbol{\bar{\gamma}}$ in ${\mathbb R}^N$
  \index{$Y^{3,\delta_0}_{{\mathcal D},\varphi}$} the set
  $Y^{3,\delta_0}_{{\mathcal D},\varphi}:=Y^{3,\delta_0}_{{\mathcal
  D},\varphi}({\mathbb R}\times {\mathbb R}^N,\boldsymbol{\gamma})$ is
  defined as the disjoint union
  $$
    Y^{3,\delta_0}_{{\mathcal D},\varphi}= \left(\{0\}\times
    Z^{3,\delta_0}_{\mathcal D}({\mathbb R}\times {\mathbb
    R}^N,\boldsymbol{\gamma})\right)\coprod \mathfrak{X} =
    \partial\mathfrak{Z}\coprod \mathfrak{X}.
    $$
\end{definition}
%%%%%                                                                       %%%%
%%%%%%%%%%                                                             %%%%%%%%%
%%%%%%%%%%%%%%%%%%%%%%%%%%%%%%%%%%%%%%%%%%%%%%%%%%%%%%%%%%%%%%%%%%%%%%%%%%%%%%%%
%

%%%%%%%%%%%%%%%%%%%%%%%%%%%%%%%%%%%%%%%%%%%%%%%%%%%%%%%%%%%%%%%%%%%%%%%%%%%%%%%%
%%%%%%%%%%                            REMARK                           %%%%%%%%%
%%%%%                                                                       %%%%
\begin{remark}\label{rem10.14}
We observe that the space is a union of a ssc-manifold
  $\partial\mathfrak{Z}$, provided we use the full $\delta$, and a
  sc-manifold $\mathfrak{X}$.
Our aim is to define a natural (up to a choice of the gluing profile
  $\varphi$) M-polyfold structure on
  $Y^{3,\delta_0}_{{\mathcal D},\varphi}$ in such a way that
  $\mathfrak{X}$ is an open and dense sub-M-polyfold with the induced
  structure being the original one.
Moreover, also $\partial\mathfrak{Z}$ will be a sub-M-polyfold and the
  induced structure will be the M-polyfold structure diffeomorphic to the
  sc-structure derived from  the original ssc-manifold structure.
\end{remark}
%%%%%                                                                       %%%%
%%%%%%%%%%                                                             %%%%%%%%%
%%%%%%%%%%%%%%%%%%%%%%%%%%%%%%%%%%%%%%%%%%%%%%%%%%%%%%%%%%%%%%%%%%%%%%%%%%%%%%%%
%
In order to define a M-polyfold structure on $ Y^{3,\delta_0}_{{\mathcal
  D},\varphi}$ we shall use the  $\oplus$-method.
This case will turn out to be somewhat more involved than the nodal case
  since we need to construct two maps defined on open subsets (for the
  quotient topology) of $Y^{3,\delta_0}_{{\mathcal D},\varphi}$ rather than
  the single $f$ in the nodal case.

The set  ${\mathcal V}$ introduced in Definition \ref{DEFX3.7} has
  naturally the structure of a ssc-manifold and we shall define a map
  \begin{eqnarray}
    \boldsymbol{\bar{\oplus}}:{\mathcal V}\rightarrow
    Y^{3,\delta_0}_{{\mathcal
    D},\varphi},\index{$\boldsymbol{\bar{\oplus}}:{\mathcal V}\rightarrow
    Y^{3,\delta_0}_{{\mathcal D},\varphi}$}
    \end{eqnarray}
  where the right-hand side for the moment is viewed as just a set.

%%%%%%%%%%%%%%%%%%%%%%%%%%%%%%%%%%%%%%%%%%%%%%%%%%%%%%%%%%%%%%%%%%%%%%%%%%%%%%%%
%%%%%%%%%%                          DEFINITION                         %%%%%%%%%
%%%%%                                                                       %%%%
\begin{definition}
The  map $\mathsf{a}:{\mathcal V}\rightarrow {\mathbb B}$ is defined  by
  associating to the tuple $(0,(\wt{u}^x,[\wh{x},\wt{y}],\wt{u}^y))$ the
  element $0$ and in the case $r\neq 0$ the element
  $a=|a|\cdot[\wh{x},\wh{y}]$, where
  $$
    T\cdot \varphi(|a|) =\varphi(r)+c^y-c^x.
    $$
Here $c^x=c^x(\wt{u}^x)$ and $c^y=c^y(\wt{u}^y)$.
\end{definition}
%%%%%                                                                       %%%%
%%%%%%%%%%                                                             %%%%%%%%%
%%%%%%%%%%%%%%%%%%%%%%%%%%%%%%%%%%%%%%%%%%%%%%%%%%%%%%%%%%%%%%%%%%%%%%%%%%%%%%%%
%
As a consequence of \cite{HWZ8.7}, Lemma 4.4, the following holds.

%%%%%%%%%%%%%%%%%%%%%%%%%%%%%%%%%%%%%%%%%%%%%%%%%%%%%%%%%%%%%%%%%%%%%%%%%%%%%%%%
%%%%%%%%%%                            LEMMA                            %%%%%%%%%
%%%%%                                                                       %%%%
\begin{lemma}\label{LEM3.14}
The map $\mathsf{a}:{\mathcal V}\rightarrow {\mathbb B} $ is ssc-smooth.
\end{lemma}
%%%%%                                                                       %%%%
%%%%%%%%%%                                                             %%%%%%%%%
%%%%%%%%%%%%%%%%%%%%%%%%%%%%%%%%%%%%%%%%%%%%%%%%%%%%%%%%%%%%%%%%%%%%%%%%%%%%%%%%
%
Now we begin with the crucial construction.

%%%%%%%%%%%%%%%%%%%%%%%%%%%%%%%%%%%%%%%%%%%%%%%%%%%%%%%%%%%%%%%%%%%%%%%%%%%%%%%%
%%%%%%%%%%                          DEFINITION                         %%%%%%%%%
%%%%%                                                                       %%%%
\begin{definition}\label{DEF3.15}
The map $\boldsymbol{\bar{\oplus}}:{\mathcal V}\rightarrow
  Y^{3,\delta_0}_{{\mathcal D},\varphi}({\mathbb R}\times {\mathbb
  R}^N,\boldsymbol{\gamma})$ is defined as follows.
We put, if $r=0$,
  $$
  \boldsymbol{\bar{\oplus}}(0,(\wt{u}^x,[\wh{x},\wh{y}],\wt{u}^y))=
  (0,\bar{\oplus}(0,(\wt{u}^x,[\wh{x},\wh{y}],\wt{u}^y)))=(0,(\wt{u}^x,
  [\wh{x},\wh{y}],\wt{u}^y)).
  $$
If $r\neq 0$ we define
  \begin{eqnarray}
    \boldsymbol{\bar{\oplus}}(r,(\wt{u}^x,[\wh{x},\wh{y}],\wt{u}^y))=
    (r,\bar{\oplus}(r,(\wt{u}^x,[\wh{x},\wh{y}],\wt{u}^y)))
    =    (r,\wt{w})
    \end{eqnarray}
  where, with $a=\mathsf{a}(r,(\wt{u}^x,[\wh{x},\wt{y}],\wt{u}^y))$, the
  map  $\wt{w}:Z_a\rightarrow {\mathbb R}\times {\mathbb R}^N$ is given by
  $$
    \wt{w}(z,z') =\beta^x_a(z)\cdot \wt{u}^x (z)+\beta^y_a(z')\cdot
    (\varphi(r)\ast \wt{u}^y(z')).
    $$
Above, $``\ast"$ denotes the additive ${\mathbb R}$-action on the first
  factor.
\end{definition}
%%%%%                                                                       %%%%
%%%%%%%%%%                                                             %%%%%%%%%
%%%%%%%%%%%%%%%%%%%%%%%%%%%%%%%%%%%%%%%%%%%%%%%%%%%%%%%%%%%%%%%%%%%%%%%%%%%%%%%%
%
Here is the main result in this subsection, which will be proved later. 

%%%%%%%%%%%%%%%%%%%%%%%%%%%%%%%%%%%%%%%%%%%%%%%%%%%%%%%%%%%%%%%%%%%%%%%%%%%%%%%%
%%%%%%%%%%                           THEOREM                           %%%%%%%%%
%%%%%                                                                       %%%%
\begin{theorem}\label{RRR}
Let ${\mathcal D}$ be an ordered disk pair and $\boldsymbol{\bar{\gamma}}$
  a weighted periodic orbit in ${\mathbb R}^N$. 
The map $\boldsymbol{\bar{\oplus}}:{\mathcal V}\rightarrow
  Y^{3,\delta_0}_{{\mathcal D},\varphi}({\mathbb R}\times {\mathbb
  R}^N,\boldsymbol{\gamma})$ is a M-polyfold construction by the
  $\oplus$-method fitting into the commutative diagram
  $$
    \begin{CD}
    {\mathcal V} @> \bm{\bar{\oplus}}>> Y^{3,\delta_0}_{{\mathcal
    D},\varphi}({\mathbb R}\times {\mathbb R}^N,\boldsymbol{\gamma})\\
    @V pr_1 VV    @V \bar{r} VV\\
    [0,1)\ \ \ \  @= [0,1),
    \end{CD}
    $$
  where the vertical arrows are the obvious extractions of the
  $r$-parameter.  
Moreover the following holds.   
\jwf{\marginnote{Check that these items have been proved later}}
\begin{itemize}
  \item[(1)] 
  For the defined M-polyfold structure an element $(r,\wt{u})$ has
  degeneracy index $1$ if $r=0$ and $0$ otherwise.
  \item[(2)] 
  The  M-polyfold structure is tame. 
  \end{itemize}
\end{theorem}
%%%%%                                                                       %%%%
%%%%%%%%%%                                                             %%%%%%%%%
%%%%%%%%%%%%%%%%%%%%%%%%%%%%%%%%%%%%%%%%%%%%%%%%%%%%%%%%%%%%%%%%%%%%%%%%%%%%%%%%
%
\begin{proof}
The proof of Theorem \ref{RRR} is given in Subsection \ref{SSEC10.6.1},
  but requires a considerable amount of preparation provided in the
  following sections.
\end{proof}%%%%%%%%%%%%%%%%%%%%%%%%%%%%%%%%%%%%%%%%%%%%%%%%            END PROOF

Recall that the set $Y^{3,\delta_0}_{{\mathcal D},\varphi}({\mathbb
  R}\times {\mathbb R}^N,\boldsymbol{\gamma})$ is the disjoint union
  of $\partial\mathfrak{Z}$ and $\mathfrak{X}$, where
  $\partial\mathfrak{Z}$ has already a ssc-manifold structure and
  $\mathfrak{X}$ a sc-manifold structure.  
The $\bm{\bar{\oplus}}$-structure will be related to the already existing
  structures. 
The following theorem summarizes the additional properties, where
  ${\mathcal T}$ is the quotient topology underlying the M-polyfold
  structure arising from Theorem \ref{RRR}. 

%%%%%%%%%%%%%%%%%%%%%%%%%%%%%%%%%%%%%%%%%%%%%%%%%%%%%%%%%%%%%%%%%%%%%%%%%%%%%%%%
%%%%%%%%%%                          DEFINITION                         %%%%%%%%%
%%%%%                                                                       %%%%
\begin{definition}
The set $Y^{3,\delta_0}_{{\mathcal D},\varphi}({\mathbb R}\times {\mathbb
  R}^N,\boldsymbol{\gamma})$ equipped with the M-polyfold structure
  resulting from Theorem \ref{RRRR} via
  $$
    \boldsymbol{\bar{\oplus}}:{\mathcal V}\rightarrow
    Y^{3,\delta_0}_{{\mathcal D},\varphi}({\mathbb R}\times {\mathbb
    R}^N,\boldsymbol{\gamma})
    $$
 is denoted by $Y^{3,\delta}_{{\mathcal D},\varphi}({\mathbb R}\times
 {\mathbb R}^N,\boldsymbol{\bar{\gamma}})$, or, more explicitly,
 $Y_{{\mathcal D},\varphi}^{3,\delta}({\mathbb R}\times {\mathbb
 R}^N,\boldsymbol{\gamma})$.
\end{definition}
%%%%%                                                                       %%%%
%%%%%%%%%%                                                             %%%%%%%%%
%%%%%%%%%%%%%%%%%%%%%%%%%%%%%%%%%%%%%%%%%%%%%%%%%%%%%%%%%%%%%%%%%%%%%%%%%%%%%%%%
%
Recall that the subsets $\partial\mathfrak{Z}$ and $\mathfrak{X}$ of
  $Y^{3,\delta}_{{\mathcal D},\varphi}({\mathbb R}\times {\mathbb
  R}^N,\boldsymbol{\bar{\gamma}})$ already possessed previously defined
  structures, see Remark \ref{rem10.14}.
The following Theorem \ref{RRRR} and Proposition \ref{PROPN3.15}  show
  that the new construction recovers the original structures in a suitable
  sense. 
In addition they show  that the M-polyfold structure does not depend on
  the choice of the cut-off model $\beta$.

%%%%%%%%%%%%%%%%%%%%%%%%%%%%%%%%%%%%%%%%%%%%%%%%%%%%%%%%%%%%%%%%%%%%%%%%%%%%%%%%
%%%%%%%%%%                           THEOREM                           %%%%%%%%%
%%%%%                                                                       %%%%
\begin{theorem}\label{RRRR}
The $\bm{\bar{\oplus}}$-polyfold structure on $Y^{3,\delta_0}_{{\mathcal
  D},\varphi}({\mathbb R}\times {\mathbb R}^N,\boldsymbol{\gamma})$ has the
  following properties.
\begin{itemize}
  \item[(1)] 
  The topologies induced by ${\mathcal T}$  on $\mathfrak{X}$ and
    $\partial\mathfrak{Z}$ are the original topologies for the already
    existing sc-manifold structure on $\mathfrak{X}$ and ssc-manifold
    structure on $\partial\mathfrak{Z}$.
  \item[(2)]
  The M-polyfold structure on $Y^{3,\delta_0}_{{\mathcal
    D},\varphi}({\mathbb R}\times {\mathbb R}^N,\boldsymbol{\gamma})$ does
    not(!) depend on the smooth $\beta:{\mathbb R}\rightarrow [0,1]$ which
    was taken in the definition of $\boldsymbol{\bar{\oplus}}$ as long as it
    satisfies the usual properties $\beta(s)+\beta(-s)=1$,
    $\beta(s)=1$ for $s\leq -1$, and $\beta'(s)<0$ for $s\in (-1,1)$.
  \end{itemize}
\end{theorem}
%%%%%                                                                       %%%%
%%%%%%%%%%                                                             %%%%%%%%%
%%%%%%%%%%%%%%%%%%%%%%%%%%%%%%%%%%%%%%%%%%%%%%%%%%%%%%%%%%%%%%%%%%%%%%%%%%%%%%%%
%
\begin{proof}
The proof of Theorem \ref{RRRR} is given in Subsection \ref{SSEC10.6.1}
  together with the proof of Theorem \ref{RRR}.
Both need some common preparation provided in the coming sections. 
\end{proof}%%%%%%%%%%%%%%%%%%%%%%%%%%%%%%%%%%%%%%%%%%%%%%%%            END PROOF

The following result will be a corollary of Theorem \ref{CORRX3.26} which
  will be stated and proved later.

%%%%%%%%%%%%%%%%%%%%%%%%%%%%%%%%%%%%%%%%%%%%%%%%%%%%%%%%%%%%%%%%%%%%%%%%%%%%%%%%
%%%%%%%%%%                          PROPOSITION                        %%%%%%%%%
%%%%%                                                                       %%%%
\begin{proposition}\label{PROPN3.15}
  \hfill\\
The following holds.
\begin{itemize}
  \item[(1)] 
  The subset $\mathfrak{X}$ of $Y_{{\mathcal D},\varphi}({\mathbb R}\times
    {\mathbb R}^N,\boldsymbol{\bar{\gamma}})$ is open and dense for the
    topology ${\mathcal T}$ and the induced M-polyfold structure is the
    original existing one.
  \item[(2)]
  The subset $\partial\mathfrak{Z}$ is closed for the topology ${\mathcal
    T}$ and carries the structure of a sub-M-polyfold which is the
    M-polyfold structure induced by the original ssc-manifold structure.
  \end{itemize}
\end{proposition}
%%%%%                                                                       %%%%
%%%%%%%%%%                                                             %%%%%%%%%
%%%%%%%%%%%%%%%%%%%%%%%%%%%%%%%%%%%%%%%%%%%%%%%%%%%%%%%%%%%%%%%%%%%%%%%%%%%%%%%%
%
\begin{proof}
Follows from Theorem \ref{CORRX3.26}. See also the end of Subsection
  \ref{SSEC10.6.1}.
\end{proof}%%%%%%%%%%%%%%%%%%%%%%%%%%%%%%%%%%%%%%%%%%%%%%%%            END PROOF

For later constructions we introduce several maps. 
The first one already occurred in the statement of Theorem \ref{RRR}.

%%%%%%%%%%%%%%%%%%%%%%%%%%%%%%%%%%%%%%%%%%%%%%%%%%%%%%%%%%%%%%%%%%%%%%%%%%%%%%%%
%%%%%%%%%%                          DEFINITION                         %%%%%%%%%
%%%%%                                                                       %%%%
\begin{definition}
We have the following canonical maps:
\begin{itemize}
  \item[(1)]  
  $\bar{r}:Y_{{\mathcal D},\varphi}({\mathbb R}\times {\mathbb
    R}^N,\boldsymbol{\bar{\gamma}})\rightarrow [0,1)$
    extracts the parameter $r$. 
  \item[(2)] 
  $\bar{a}_d:Y_{{\mathcal D},\varphi}({\mathbb R}\times {\mathbb
  R}^N,\boldsymbol{\bar{\gamma}})\rightarrow [0,1/4)\times {\mathbb S}$
  extracts the decorated domain gluing parameter $(|a|,[\wh{x},\wh{y}])$.
  \item[(3)] 
  $\bar{a}:Y_{{\mathcal D},\varphi}({\mathbb R}\times {\mathbb
    R}^N,\boldsymbol{\bar{\gamma}})\rightarrow {\mathbb B}$ extracts the
    domain gluing parameter $a =|a| \cdot [\wh{x},\wh{y}] =: m\circ
    \bar{a}_d(|a|,[\wh{x},\wh{y}])$, where $m(b,[\wh{x},\wh{y}])=b\cdot
    [\wh{x},\wh{y}]$.
  \end{itemize}
\end{definition}
%%%%%                                                                       %%%%
%%%%%%%%%%                                                             %%%%%%%%%
%%%%%%%%%%%%%%%%%%%%%%%%%%%%%%%%%%%%%%%%%%%%%%%%%%%%%%%%%%%%%%%%%%%%%%%%%%%%%%%%
%

The following proposition summarizes the properties of the three maps and
  follows from Proposition \ref{PPROPRRR} stated later on.

%%%%%%%%%%%%%%%%%%%%%%%%%%%%%%%%%%%%%%%%%%%%%%%%%%%%%%%%%%%%%%%%%%%%%%%%%%%%%%%%
%%%%%%%%%%                          PROPOSITION                        %%%%%%%%%
%%%%%                                                                       %%%%
\begin{proposition}\label{PROPRRR}
The following holds.
\begin{itemize}
  \item[(1)] 
  For the previously defined M-polyfold structure on $Y_{{\mathcal
    D},\varphi}({\mathbb R}\times {\mathbb R}^N,\boldsymbol{\bar{\gamma}})$
    the natural maps $\bar{r}$, $\bar{a}_d$, and $\bar{a}$ are sc-smooth.
  \item[(2)] 
  The map $\bar{r}$ is submersive.
  \end{itemize}
\end{proposition}
%%%%%                                                                       %%%%
%%%%%%%%%%                                                             %%%%%%%%%
%%%%%%%%%%%%%%%%%%%%%%%%%%%%%%%%%%%%%%%%%%%%%%%%%%%%%%%%%%%%%%%%%%%%%%%%%%%%%%%%
%
\begin{proof}
Follows from Proposition \ref{PPROPRRR}.
\end{proof}%%%%%%%%%%%%%%%%%%%%%%%%%%%%%%%%%%%%%%%%%%%%%%%%            END PROOF
%

%%%%%%%%%%%%%%%%%%%%%%%%%%%%%%%%%%%%%%%%%%%%%%%%%%%%%%%%%%%%%%%%%%%%%%%%%%%%%%%%
%%%%%%%%%%                         SUB-SECTION                         %%%%%%%%%
%%%%%%%%%%%%%%%%%%%%%%%%%%%%%%%%%%%%%%%%%%%%%%%%%%%%%%%%%%%%%%%%%%%%%%%%%%%%%%%%
%
\subsubsection{Extensions to Manifolds}
In the construction of $Y_{{\mathcal D},\varphi}({\mathbb R}\times
  {\mathbb R}^N,\boldsymbol{\bar{\gamma}})$ we would like to replace
  $({\mathbb R}^N,\boldsymbol{\bar{\gamma}})$ by a periodic orbit in some
  manifold $Q$, in order to define a M-polyfold denoted by $Y_{{\mathcal
  D},\varphi}({\mathbb R}\times Q,\boldsymbol{\bar{\gamma}})$ In order to do
  we will proceed as in the nodal case.
Namely we show that the construction $Y_{{\mathcal D},\varphi}$ defines a
  functor on the category having pairs $({\mathbb R}^N,\bm{\bar{\gamma}})$
  as objects and suitable morphisms based on smooth maps between them.
Then the idea from Subsection \jwf{[broken reference]} % \ref{qsec5.1} 
  can be applied.

Consider first the category ${\mathcal P}_w^0$, whose objects are pairs
  $({\mathbb R}^N,\boldsymbol{\bar{\gamma}})$, where
  $\boldsymbol{\bar{\gamma}}$ is a weighted periodic orbit in ${\mathbb
  R}^N$.
A morphism $({\mathbb R}^N,\boldsymbol{\bar{\gamma}})\rightarrow ({\mathbb
  R}^{N'},\boldsymbol{\bar{\gamma}}')$ only exists provided
  $(\delta,T,k)=(\delta',T',k')$ and is given by a smooth map $h:{\mathbb
  R}^N\rightarrow {\mathbb R}^{N'}$ having the property that
  $h\circ\gamma\in [\gamma']$ for $\gamma\in [\gamma]$. 
Associated to $h$ we have the map $\wt{h}=Id_{\mathbb R}\times h:{\mathbb
  R}\times {\mathbb R}^N\rightarrow {\mathbb R}\times {\mathbb R}^{N'}$. We
  define
  $$
  \wt{h}_\ast: Y_{{\mathcal D},\varphi}({\mathbb R}\times{\mathbb
    R}^N,\boldsymbol{\bar{\gamma}})\rightarrow Y_{{\mathcal
    D},\varphi}({\mathbb R}\times{\mathbb
    R}^{N'},\boldsymbol{\bar{\gamma}}')
    $$
  by $(r,\wt{w})\rightarrow (r,\wt{h}\circ \wt{w})$ if $r\in (0,1)$ and
  $(0,(\wt{u}^x,[\wh{x},\wh{y}],\wt{u}^y))\rightarrow (0,(\wt{h}\circ\wt{u}^x,[\wh{x},\wh{y}],\wt{h}\circ\wt{u}^y))$. 
The following holds.  

%%%%%%%%%%%%%%%%%%%%%%%%%%%%%%%%%%%%%%%%%%%%%%%%%%%%%%%%%%%%%%%%%%%%%%%%%%%%%%%%
%%%%%%%%%%                          PROPOSITION                        %%%%%%%%%
%%%%%                                                                       %%%%
\begin{proposition}\label{PPP3.22}
Let $h:({\mathbb R}^N,\boldsymbol{\bar{\gamma}})\rightarrow ({\mathbb
  R}^{N'},\boldsymbol{\bar{\gamma}}')$ be a morphism.
Then the map
  $$
    \wt{h}_\ast: Y_{{\mathcal D},\varphi}({\mathbb R}\times{\mathbb
    R}^N,\boldsymbol{\bar{\gamma}})\rightarrow
    Y_{{\mathcal D},\varphi}({\mathbb R}\times{\mathbb
    R}^{N'},\boldsymbol{\bar{\gamma}}')
    $$
  is sc-smooth and $\wt{(h'\circ h)}_\ast =\wt{h}'_{\ast}\circ
  \wt{h}_\ast$ if $h'$ is another morphism composeable
  with $h$. 
Also $\wt{Id}_\ast = Id$.
\end{proposition}
%%%%%                                                                       %%%%
%%%%%%%%%%                                                             %%%%%%%%%
%%%%%%%%%%%%%%%%%%%%%%%%%%%%%%%%%%%%%%%%%%%%%%%%%%%%%%%%%%%%%%%%%%%%%%%%%%%%%%%%
%
\begin{proof}
This follows from  Proposition \ref{PROX3.49}.
\end{proof}%%%%%%%%%%%%%%%%%%%%%%%%%%%%%%%%%%%%%%%%%%%%%%%%            END PROOF

Now it is clear that we can use the procedure of the previous section to
  extend our construction to pairs $(Q,\boldsymbol{\bar{\gamma}})$.

%%%%%%%%%%%%%%%%%%%%%%%%%%%%%%%%%%%%%%%%%%%%%%%%%%%%%%%%%%%%%%%%%%%%%%%%%%%%%%%%
%%%%%%%%%%                          DEFINITION                         %%%%%%%%%
%%%%%                                                                       %%%%
\begin{definition}
Let $Q$ be a smooth connected manifold without boundary which allows a
  proper embedding into some ${\mathbb R}^N$.  
A periodic orbit in $Q$ is a tuple $\boldsymbol{\gamma}=([\gamma],T,k)$,
  where $\gamma:S^1\rightarrow Q$ is a smooth embedding and $[\gamma]$
  consists of all parameterizations of the form $t\rightarrow
  \gamma(t+\theta)$.
As in the ${\mathbb R}^N$-case we can define weighted periodic orbits
  $\boldsymbol{\bar{\gamma}}=(\boldsymbol{\gamma},\delta)$.
\end{definition}
%%%%%                                                                       %%%%
%%%%%%%%%%                                                             %%%%%%%%%
%%%%%%%%%%%%%%%%%%%%%%%%%%%%%%%%%%%%%%%%%%%%%%%%%%%%%%%%%%%%%%%%%%%%%%%%%%%%%%%%
%

We consider the category $\mathcal{P}_w$ whose objects are pairs
  $(Q,\boldsymbol{\bar{\gamma}})$ as just described.
Morphisms $(Q,\boldsymbol{\bar{\gamma}})\rightarrow
  (Q',\boldsymbol{\bar{\gamma}}')$ only exist if $\delta=\delta'$,
  $(T,k)=(T',k')$, and this case they are given by a smooth map
  $h:Q\rightarrow Q'$ such that for $\gamma\in [\gamma]$ the map
  $h\circ\gamma\in [\gamma']$.

%%%%%%%%%%%%%%%%%%%%%%%%%%%%%%%%%%%%%%%%%%%%%%%%%%%%%%%%%%%%%%%%%%%%%%%%%%%%%%%%
%%%%%%%%%%                           THEOREM                           %%%%%%%%%
%%%%%                                                                       %%%%
\begin{theorem}\label{TTT3.24}
The polyfold construction $Y_{{\mathcal D},\varphi}({\mathbb
  R}\times{\mathbb R}^N,\boldsymbol{\bar{\gamma}})$ for pairs $({\mathbb
  R}^N,\bm{\bar{\gamma}})$  and the morphisms between them extends to a
  functorial construction  defined on ${\mathcal P}_w$.
\end{theorem}
%%%%%                                                                       %%%%
%%%%%%%%%%                                                             %%%%%%%%%
%%%%%%%%%%%%%%%%%%%%%%%%%%%%%%%%%%%%%%%%%%%%%%%%%%%%%%%%%%%%%%%%%%%%%%%%%%%%%%%%
%
\begin{proof} 
The idea follows along the line of the construction functor "technology"
  used already in the nodal case, see also Section \jwf{[broken reference]}
  %\ref{CONST-F}.
We sketch the proof. 
$Y_{{\mathcal D},\varphi}({\mathbb R}\times Q,\boldsymbol{\bar{\gamma}})$
  is defined by taking a proper embedding $\phi:Q\rightarrow {\mathbb
  R}^N$ for a suitable $N$.
With $\boldsymbol{\bar{\gamma}}=(([\gamma],T,k),\delta)$ we define 
  $\boldsymbol{\bar{\gamma}'}=(([\gamma'],T,k),\delta)$, where
  $\gamma'=\phi\circ\gamma$.
We consider the subset $\Sigma_\phi$ of $Y_{{\mathcal D},\varphi}({\mathbb
  R}\times {\mathbb R}^N,\boldsymbol{\bar{\gamma}'})$ consisting of tuples
  $(r,\wt{u})$ with $\wt{u}$ having image in ${\mathbb R}\times \phi(Q)$.
The set $\Sigma_\phi$ is a sub-M-polyfold using the ideas described in
  Section \ref{SECX1}.
Then  $Y_{{\mathcal D},\varphi}({\mathbb R}\times
  Q,\boldsymbol{\bar{\gamma}})$ is the set of tuples $(r,\wt{v})$, where
  $\wt{v}$ is a map into ${\mathbb R}\times Q$, so that
  $(r,(Id\times\phi)\circ \wt{v})\in\Sigma_{\phi}$.
The M-polyfold structure is the one making it a sc-diffeomorphism. 
We leave the details to the reader since it follows closely the already
  explained construction functor idea, see the next exercise.
\end{proof}%%%%%%%%%%%%%%%%%%%%%%%%%%%%%%%%%%%%%%%%%%%%%%%%            END PROOF
%

%%%%%%%%%%%%%%%%%%%%%%%%%%%%%%%%%%%%%%%%%%%%%%%%%%%%%%%%%%%%%%%%%%%%%%%%%%%%%%%%
%%%%%%%%%%                          EXERCISE                           %%%%%%%%%
%%%%%                                                                       %%%%
\begin{exercise}\label{EXERC100003b} % this was multiply-defined
Carry out the proof of Theorem \ref{TTT3.24} in more detail assuming
  Proposition \ref{PPP3.22}.
\end{exercise}
%%%%%                                                                       %%%%
%%%%%%%%%%                                                             %%%%%%%%%
%%%%%%%%%%%%%%%%%%%%%%%%%%%%%%%%%%%%%%%%%%%%%%%%%%%%%%%%%%%%%%%%%%%%%%%%%%%%%%%%
%

In the following subsections we shall provide the basic constructions
  which we just presented.

%%%%%%%%%%%%%%%%%%%%%%%%%%%%%%%%%%%%%%%%%%%%%%%%%%%%%%%%%%%%%%%%%%%%%%%%%%%%%%%%
%%%%%%%%%%                         SUB-SECTION                         %%%%%%%%%
%%%%%%%%%%%%%%%%%%%%%%%%%%%%%%%%%%%%%%%%%%%%%%%%%%%%%%%%%%%%%%%%%%%%%%%%%%%%%%%%
%
\subsection{The Coretraction \texorpdfstring{$H$}{H}}\label{CORETRH}
We have to study 
\begin{eqnarray}\label{HKL1}
  \boldsymbol{\bar{\oplus}}:{\mathcal V}\rightarrow
  Y^{3,\delta_0}_{{\mathcal D},\varphi}({\mathbb R}\times {\mathbb
  R}^N,\boldsymbol{\gamma})
  \end{eqnarray}
  and show that it defines a M-polyfold structure on
  $Y^{3,\delta_0}_{{\mathcal D},\varphi}({\mathbb R}\times {\mathbb
  R}^N,\boldsymbol{\gamma})$ by the $\oplus$-method. 
As we have previously mentioned we need to construct two maps in order to
  show that $\bm{\bar{\oplus}}$ is an $\oplus$-construction. 
These maps denoted by $H$ and $K$ provided the local $H_y$ by restriction
  as they occur in Definition \jwf{[broken reference]} %\ref{DEFNR1.1}.
Since we also want to prove some additional properties we have to take
  special care.
The set $Y^{3,\delta_0}_{{\mathcal D},\varphi}$ is the disjoint union of
  $\partial\mathfrak{Z}$ and $\mathfrak{X}$, which already have certain
  sc-smooth structures and our aim is to find a M-polyfold structure on
  the whole set which induces on these subsets the already existing
  structures. 
We shall denote by ${\mathcal T}$  the quotient topology on
  $Y^{3,\delta_0}_{{\mathcal D},\varphi}({\mathbb R}\times {\mathbb
  R}^N,\boldsymbol{\gamma})$ associated to (\ref{HKL1}).

Recall from Section \ref{SECX1}, Remarks \ref{REMdotX} and
  \ref{REMARK2.11} the following fact which we shall use repeatedly.

\begin{FACT} 
\jwf{Check this ``fact'' envinronment}
The space 
  $$
    \dot{X}^3_{{\mathcal D}}({\mathbb R}\times {\mathbb R}^N):=\{\wt{w}\in
    {X}^3_{{\mathcal D}}\ |\ a(\wt{w})\in {\mathbb B}\setminus\{0\}\}
    $$
  has a M-polyfold structure characterized by
  $$
    \oplus: ({\mathbb B}\setminus \{0\})\times H^{3,\delta}({\mathcal
    D},{\mathbb R}\times {\mathbb R}^N)\rightarrow
    \dot{X}^3_{{\mathcal D}}({\mathbb R}\times {\mathbb R}^N)
    $$
  where 
  $$
    \oplus(a,\wt{h}^x,\wt{h}^y)(z,z')=\beta^x_a(z)\cdot \wt{h}^x(z)+
    \beta^y_a(z')\cdot \wt{h}^y(z').
    $$
The $\oplus$-method gives a well-defined M-polyfold structure on
  $\dot{X}^3_{{\mathcal D}}({\mathbb R}\times {\mathbb R}^N)$.
This structure has the property that the  map $\oplus$ is sc-smooth,
  surjective, and open, and admits an sc-smooth $f$ such that $\oplus\circ f
  =Id$.
This was previously discussed.  
As already previously mentioned we can also define a natural sc-manifold
  structure on $\dot{X}^3_{{\mathcal D}}({\mathbb R}\times {\mathbb R}^N)$,
  see Exercise \ref{EXERC101}.
More  details are also given in Subsection \jwf{[broken reference]} %\ref{SS9.1.2}. 
The identity map is an sc-diffeomorphism between the two structures. 
\end{FACT}

By construction $\mathfrak{X}=(0,1)\times \dot{X}^3_{{\mathcal
  D}}({\mathbb R}\times {\mathbb R}^N)$ and in view of the previous fact it
  has a M-polyfold structure coming from a sc-manifold structure. 
We always view $\mathfrak{X}$ as being equipped with this M-polyfold
  structure.
Clearly the structure on $\mathfrak{X}$ can also be obtained by the
  $\oplus$-method.  
The surjective map in this case is $Id_{(0,1)}\times \oplus:(0,1)\times (
  ({\mathbb B}\setminus \{0\})\times H^{3,\delta}({\mathcal D},{\mathbb
  R}\times {\mathbb R}^N))\rightarrow
  \mathfrak{X}$.

\begin{AIM}[{\bf 1}]
\jwf{[check this environment]}
We shall construct with ${\mathcal V}$ given in Definition \ref{DEFX3.7} a
  map $H: \mathfrak{X}\rightarrow   {\mathcal V}$\index{$H:
  \mathfrak{X}\rightarrow   {\mathcal V}$} preserving the fiber over
  $(0,1)$, i.e. fitting into the commutative diagram
  $$
    \begin{CD}
    \mathfrak{X} @> H>> {\mathcal V}\\
    @V pr_1 VV   @V \bar{r} VV\\
    (0,1) @>>>[0,1),
    \end{CD}
    $$
  such that $H$ is sc-smooth (for the already existing M-polyfold
  structure on $\mathfrak{X}$) and $\boldsymbol{\bar{\oplus}}\circ H =
  Id_{\mathfrak{X}}$, where $\boldsymbol{\bar{\oplus}}$ is introduced
  in Definition \ref{DEF3.15}.
\end{AIM}

We fix a representative $\gamma_0$ in $[\gamma]$, a decoration $\wh{x}_0$,
  and consider with $c^x=0$ the associated $\wt{q}_{0,\wh{x}_0}^{\gamma_0}$
  defined in (\ref{EQNC16}).
Given $(r,\wt{w})\in \mathfrak{X}$ let $a=a(\wt{w})$, $R=\varphi(|a|)$, 
  and write $a=|a|\cdot [\wh{x}_0,\wh{y}]$.  
We define $H(r,\wt{w})=(r,(\wt{u}^x,[\wh{x}_0,\wh{y}],\wt{u}^y))$,  where
  \begin{eqnarray}\label{eqnrr7.5}
    \wt{u}^x(z) &=&  \wt{q}^{\gamma_0}_{0,\wh{x}_0} (z)+
    \beta^x_{a,-2}(z)\cdot\left( \wt{w}(z,z')
    -\wt{q}^{\gamma_0}_{0,\wh{x}_0}(z)\right),
    \end{eqnarray}
  with $z=\sigma_{\wh{x}_0}^+(s,t)$.  
Then the asymptotic constant of $\wt{u}^x$ is $c^x=0$ and the
$\wh{x}_0$-directional limit is $\gamma_0(0)$.  
Define $c^y$, a function of $(r,\wt{w})$, by the equation
  $$
    T\cdot \varphi(|a|)=\varphi(r)+c^y
    $$
  and put
\begin{eqnarray}\label{eqnrr7.6}
  \ \ \ \ \ \ \ \ \ \ \wt{u}^y(z') &=&\wt{q}^{\gamma_0}_{c^y,\wh{y}}(z') +
  \beta^y_{a,-2}(z') \cdot  ((-\varphi(r))\ast \wt{w}(z,z')
  -\wt{q}^{\gamma_0}_{c^y,\wh{y}}(z')).
  \end{eqnarray}
Considering the maps $\wt{u}^x$ and $\wt{u}^y$ we see that $\wt{u}^x$ near
  the lower boundary is $\wt{w}$ and near the upper boundary
  $\wt{u}^y$ is equal to $\wt{w}$ but moved downward by $-\varphi(r)$.
Further the $\wh{y}$-directional asymptotic limit is $\gamma_0(0)$ and the
  asymptotic constant is $c^y$.  
Consequently $(r,(\wt{u}^x,[\wh{x}_0,\wh{y}],\wt{u}^y))\in
  (0,1)\times Z_{\mathcal D}({\mathbb R}\times {\mathbb
  R}^N,\boldsymbol{\bar{\gamma}})$.
We note that by construction $T\cdot \varphi(|a|)=\varphi(r)+c^y-c^x$
  (recall $c^x=0$), and $a\in {\mathbb B}\setminus\{0\}$,
  which implies that $\varphi(r)+c^y-c^x>0$ and
  $\varphi^{-1}\left(\frac{1}{T}\cdot (\varphi(r)+c^y-c^x)\right)\in
  (0,1/4)$ and hence $H(r,\wt{w})\in {\mathcal V}$.  
From a trivial computation it follows that
  \begin{eqnarray}\label{pqn10.12}
    \boldsymbol{\bar{\oplus}}\circ H(r,\wt{w})=(r,\wt{w}).
    \end{eqnarray}
We also note that $(\wt{q}^{\gamma_0}_{0,\wh{x}_0},
  [\wh{x}_0,\wh{y}],\wt{q}^{\gamma_0}_{c^y,\wh{y}})\in
  \mathsf{S}_{\boldsymbol{\gamma}}$ and shall prove the  following result.

%%%%%%%%%%%%%%%%%%%%%%%%%%%%%%%%%%%%%%%%%%%%%%%%%%%%%%%%%%%%%%%%%%%%%%%%%%%%%%%%
%%%%%%%%%%                            LEMMA                            %%%%%%%%%
%%%%%                                                                       %%%%
\begin{lemma}\label{LEMMN3.16}
The map $H:\mathfrak{X}\rightarrow {\mathcal V}$ has the following
  properties.
\begin{itemize}
  \item[(1)]  
  $\boldsymbol{\bar{\oplus}}\circ H= Id_{\mathfrak{X}}$. 
    In particular $\boldsymbol{\bar{\oplus}}: {\mathcal V}\rightarrow
    Y^{3,\delta_0}_{{\mathcal D},\varphi}({\mathbb R}\times {\mathbb
    R}^N,\boldsymbol{\gamma})$ is surjective.
  \item[(2)] 
  Given the (original) M-polyfold structure on $\mathfrak{X}$ the map
    $H:\mathfrak{X}\rightarrow {\mathcal V}$ is sc-smooth.
  \item[(3)] 
  The set $\dot{\mathcal V}=\boldsymbol{\bar{\oplus}}^{-1}(\mathfrak{X})$ is
    open and the map $H\circ\boldsymbol{\bar{\oplus}}: \dot{\mathcal
    V}\rightarrow {\mathcal V}$ is sc-smooth.  
  \item[(4)] 
  The map $\boldsymbol{\bar{\oplus}}:\dot{\mathcal V}\rightarrow
    \mathfrak{X}$ is sc-smooth, where $\mathfrak{X}$ is equipped with its
    original stucture.
  \end{itemize}
\end{lemma}
%%%%%                                                                       %%%%
%%%%%%%%%%                                                             %%%%%%%%%
%%%%%%%%%%%%%%%%%%%%%%%%%%%%%%%%%%%%%%%%%%%%%%%%%%%%%%%%%%%%%%%%%%%%%%%%%%%%%%%%
%
\begin{proof}
{(1)} The elements in $\mathfrak{X}$ are in the image by the existence of
  $H$ and the property displayed in (\ref{pqn10.12}).
The elements in $\partial\mathfrak{Z}$ belong trivially to the image of
  $\boldsymbol{\bar{\oplus}}$. 
Hence $\bm{\bar{\oplus}}$ is surjective.\\

\noindent{(2)}  
By the definition of the original  M-polyfold structure on $\mathfrak{X}$
  we just need to show that
  $$
    H\circ (Id_{(0,1)}\times \oplus): (0,1)\times (({\mathbb
    B}\setminus\{0\})\times H^{3,\delta}_0({\mathcal D},{\mathbb R}\times
    {\mathbb R}^N)) \rightarrow {\mathcal V}
    $$
  is sc-smooth.  
Given $(r,(a,(\wt{h}^x,\wt{h}^y)))$ with $r\in (0,1)$ and $a\in {\mathbb
  B}\setminus\{0\}$ the element
  $$
    (r,(\wt{u}^x,[\wh{x},\wh{y}],\wt{u}^y))=H\circ (Id_{(0,1)}\times
    \oplus)(r,(a,(\wt{h}^x,\wt{h}^y)))\in {\mathcal V}
    $$
  is obtained as follows
  \begin{eqnarray*}
    \wt{u}^x(z) 
    &=&\wt{q}^{\gamma_0}_{0,\wh{x}_0}(z)\\
    &&+ \beta^x_{a,-2}(z)\cdot \left( \beta^x_a(z)\cdot
    \wt{h}^x(z)+\beta^y_a(z')\cdot \wt{h}^y(z')
    -\wt{q}^{\gamma_0}_{0,\wh{x}_0}(z)\right)\\
    &=:& \wt{q}^x(z)+\wt{w}^x(z).
    \end{eqnarray*}
Moreover,
  \begin{eqnarray*}
    \wt{u}^y(z')
    &=&\wt{q}^{\gamma_0}_{c^y,\wh{y}}(z') + \beta_{a,-2}^y(z')\cdot\\
    &&\ \ \ \left((-\varphi(r))\ast\left(\beta^x_a(z)\cdot
    \wt{h}^x(z)+\beta^y_a(z')\cdot
    \wt{h}^y(z')\right)-\wt{q}^{\gamma_0}_{c^y,\wh{y}}(z')\right)\\
    &=:&\wt{q}^y(z')+\wt{w}^y(z').
    \end{eqnarray*}
We note that with $a=|a|\cdot [\wh{x}_0,\wh{y}]$ the element
  $(\wt{q}^x,[\wh{x}_0,\wh{y}],\wt{q}^y)$ belongs to
  $\mathsf{S}_{\boldsymbol{\gamma}}$.
Recall that we always take $c^x=0$ and $c^y$ is determined by $T\cdot
  \varphi(|a|)=\varphi(r)+c^y$.
From this we see that the map, after having fixed $\wh{x}_0$ (as we did),
  $(r,a)\rightarrow (\wt{q}^x,[\wh{x}_0,\wh{y}],\wt{q}^y)$ is smooth.
The pair $(\wt{u}^x,\wt{u}^y)$ depends sc-smoothly on the input data which
  follows from the Fundamental Lemma, in Subsection  \ref{FUND0},
  together with the sc-smooth dependence of $-\beta^x_{a,-2}\cdot
  \wt{q}^{\gamma_0}_{0,\wh{x}_0}$ and $-\beta^y_{a,-2}\cdot
  \wt{q}^{\gamma_0}_{c^y,\wh{y}}$.
It is, of course, important that $a\in {\mathbb B}\setminus\{0\}$ and
  $r\in (0,1)$.\\
 
\noindent{(3)} 
The set $\dot{\mathcal V}$ is obviously open.
We consider $(r,(\wt{u}^x,[\wh{x},\wh{y}],\wt{u}^y))\in \dot{\mathcal V}$,
  which implies $r\in (0,1)$. 
We can write $\wt{u}^x =\wt{q}^{\gamma_0}_{c^x,\wt{x}}+\wt{h}^x$ and
  $\wt{u}^y=\wt{q}^{\gamma_0}_{c^y,\wh{y}}+\wt{h}^y$. 
It suffices to show that the map
$$
  ((c^x,\wh{x},c^y,\wh{y}),(\wt{h}^x,\wt{h}^y))\rightarrow
  H\circ\boldsymbol{\bar{\oplus}}(r,(\wt{q}^{\gamma_0}_{c^x,\wt{x}}+
  \wt{h}^x,[\wh{x},\wh{y}],\wt{q}^{\gamma_0}_{c^y,\wh{y}}+\wt{h}^y))
  $$
  is sc-smooth.
From this data we obtain $a=|a|\cdot [\wh{x},\wh{y}]$, where $|a|$
  satisfies $T\cdot \varphi(|a|)=\varphi(r)+c^y-c^x$ and see that $a$
  depends sc-smoothly on the input. 
We obtain
  \begin{eqnarray*}
    \boldsymbol{\bar{\oplus}}(r,(\wt{u}^x,[\wh{x},\wh{y}],\wt{u}^y))(z,z')
    &=&\wt{q}^{\gamma_0,a}_{c^x,\wh{x}}(z,z') +
    \oplus_a(\wt{h}^x,\wt{h}^y)(z,z')
    \end{eqnarray*}
  which is defined on $Z_a$.  
For the construction of $H$ we have fixed a decoration $\wh{x}_0$ and we
  note that given $[\wh{x},\wh{y}]$ there exists a unique
  $\wh{y}_0=\wh{y}_0([\wh{x},\wh{y}])$ such that
  $[\wh{x},\wh{y}]=[\wh{x}_0,\wh{y}_0]$. 
Moreover $T\cdot \varphi(|a|)=\varphi(r)+c^y-c^x$.
For the construction we take $c^x_0=0$ and $c^y_0$ satisfies consequently 
  $$
    c^y_0=c^y-c^x.
    $$
Applying $H$ we obtain $(\wt{w}^x,[\wh{x},\wh{y}],\wt{w}^y)$ via
  \begin{eqnarray*}
    \wt{w}^x (z)&=&
    \wt{q}^{\gamma_0}_{0,\wh{x}_0}(z)+\beta^x_{a,-2}(z)\cdot
    (\wt{q}^{\gamma_0}_{c^x,\wh{x}}(z,z') +
    \oplus_a(\wt{h}^x,\wt{h}^y)(z,z')-\wt{q}^{\gamma_0}_{0,\wh{x}}(z))\\
    &=:&  \wt{q}^{\gamma_0}_{0,\wh{x}_0}(z)+\wt{k}^x
    \end{eqnarray*}
  and 
  \begin{eqnarray*}
    \wt{w}^y(z')&=& \wt{q}^{\gamma_0}_{c^y-c^x,\wh{y}_0}(z')
    +\beta_{a,-2}(z')\cdot \\
    &&((-\varphi(r))\ast (\wt{q}^{\gamma_0}_{c^x,\wh{x}}(z,z') +
    \oplus_a(\wt{h}^x,\wt{h}^y)(z,z')-
    \wt{q}^{\gamma_0}_{c^y-c^x,\wh{y}_0}(z')))\\
    &=&\wt{q}^{\gamma_0}_{c^y-c^x,\wh{y}_0}(z')+\wt{k}^y(z').
    \end{eqnarray*}
It is clear that the map 
\begin{eqnarray*}
  &(0,1)\times {\mathbb R}\times {\mathbb S}_x\times {\mathbb R}\times
  {\mathbb S}_y\rightarrow
  \mathsf{S}_{\boldsymbol{\gamma}}:&\\
  &(r,(c^x,\wh{x},c^y,\wh{y}))
  \rightarrow
  (\wt{q}^{\gamma_0}_{0,\wh{x}_0},[\wh{x}_0,\wh{y}_0],
  \wt{q}^{\gamma_0}_{c^y-c^x,\wh{y}_0})&
  \end{eqnarray*}
  is smooth. 
With $a=a(r,(c^x,\wh{x},c^y,\wh{y}))$ being a smooth map we see that
  \begin{eqnarray*}
    &(0,1)\times {\mathbb R}\times {\mathbb S}_x\times {\mathbb R}\times
    {\mathbb S}_y\times H^{3,\delta}_0\rightarrow H^{3,\delta}_0:&\\
    &(r,(c^x,\wh{x},c^y,\wh{y}), (\wt{h}^x,\wt{h}^y))\rightarrow
    (\wt{k}^x,\wt{k}^y)&
    \end{eqnarray*}
is sc-smooth in view of the results  in Subsection \ref{FUND0}.\\

\noindent{(4)} 
By the definition of the M-polyfold structure on $\mathfrak{X}$ we need to
  show that $f\circ \boldsymbol{\bar{\oplus}}:\dot{\mathcal V} \rightarrow
  (0,1)\times H^{3,\delta}_0({\mathcal D},{\mathbb R}\times {\mathbb R}^N)$
  is sc-smooth.
Recall that $f$ is being constructed in the proof of Theorem \ref{thm1},
  see (\ref{Hf}).
It is important in this argument that the occuring values $r$ are
  different from $0$.
We leave this argument to the reader. 
It can be, after some mild computation, again reduced to an application
  of the results in Subsection \ref{FUND0}.
\end{proof}%%%%%%%%%%%%%%%%%%%%%%%%%%%%%%%%%%%%%%%%%%%%%%%%            END PROOF

Let us draw some of the consequences of the previous result.  
Denote by ${\mathcal T}$ the finest topology so that
  $\boldsymbol{\bar{\oplus}}:{\mathcal V}\rightarrow
  Y^{3,\delta_0}_{{\mathcal D},\varphi}=\partial\mathfrak{Z} \coprod
  \mathfrak{X}$ is continuous, i.e. ${\mathcal T}$ is the quotient
  topology.
The map
  \begin{eqnarray}\label{rho-def}
    \rho:=H\circ\boldsymbol{\bar{\oplus}}:\dot{\mathcal V}\rightarrow
    {\mathcal V}
    \end{eqnarray}
  has image in $\dot{\mathcal V}$ and is an sc-smooth retraction.
Abbreviating $O_{\dot{\mathcal V}}:=\rho(\dot{\mathcal V})$ the map
  $$
    \boldsymbol{\bar{\oplus}}:O_{\dot{\mathcal V}}\rightarrow \mathfrak{X}
    $$
  is a bijection and defines a metrizable topology on $\mathfrak{X}$,
  denoted by ${\mathcal T}_{\mathfrak{X}}'$, and it  defines uniquely a
  (possibly new)  M-polyfold structure for $(\mathfrak{X},{\mathcal
  T}_{\mathfrak{X}}')$ for which this map is a sc-diffeomorphism.
We denote the set $\mathfrak{X}$, equipped with this M-polyfold structure
  and topology by $\mathfrak{X}'$.
Hence we obtain the  tautological result

%%%%%%%%%%%%%%%%%%%%%%%%%%%%%%%%%%%%%%%%%%%%%%%%%%%%%%%%%%%%%%%%%%%%%%%%%%%%%%%%
%%%%%%%%%%                            LEMMA                            %%%%%%%%%
%%%%%                                                                       %%%%
\begin{lemma}\label{LOP:}
The map $\boldsymbol{\bar{\oplus}}:O_{\dot{\mathcal V}}\rightarrow
  \mathfrak{X}'$ is a sc-diffeomorphism and $H:\mathfrak{X}'\rightarrow
  O_{\dot{\mathcal V}}$ is the inverse sc-diffeomorphism.
\end{lemma}
%%%%%                                                                       %%%%
%%%%%%%%%%                                                             %%%%%%%%%
%%%%%%%%%%%%%%%%%%%%%%%%%%%%%%%%%%%%%%%%%%%%%%%%%%%%%%%%%%%%%%%%%%%%%%%%%%%%%%%%
%
Since $\bm{\bar{\oplus}}^{-1}(\mathfrak{X})=\dot{\mathcal V}$ is open in
  ${\mathcal V}$ we see that $\mathfrak{X}\in {\mathcal T}$ is open.
Denote by ${\mathcal T}_{\mathfrak{X}}$ the restriction of ${\mathcal T}$
  to $\mathfrak{X}$.
Hence an element $U\in {\mathcal T}_{\mathfrak{X}}$ has the form
  $$
    U= \mathfrak{X}\cap V\ \ \text{where}\
    \boldsymbol{\bar{\oplus}}^{-1}(V)\ \text{is open},\ V\subset
    Y^{3,\delta_0}_{{\mathcal D},\varphi}.
    $$
Since $\mathfrak{X}\in {\mathcal T}$ the restriction of ${\mathcal T}$ to
  $\mathfrak{X}$ consists of all subsets $U$ of $\mathfrak{X}$ for which
  $\boldsymbol{\bar{\oplus}}^{-1}(U)$ is open in ${\mathcal V}$.

%%%%%%%%%%%%%%%%%%%%%%%%%%%%%%%%%%%%%%%%%%%%%%%%%%%%%%%%%%%%%%%%%%%%%%%%%%%%%%%%
%%%%%%%%%%                          PROPOSITION                        %%%%%%%%%
%%%%%                                                                       %%%%
\begin{proposition}\label{ABER}
As M-polyfolds $\mathfrak{X}=\mathfrak{X}'$.  
In particular this implies that the restriction of ${\mathcal T}$ to
  $\mathfrak{X}$ denoted by ${\mathcal T}_{\mathfrak{X}}$ satisfies
  ${\mathcal T}_{\mathfrak{X}}={\mathcal T}'_{\mathfrak{X}}$.
\end{proposition}
%%%%%                                                                       %%%%
%%%%%%%%%%                                                             %%%%%%%%%
%%%%%%%%%%%%%%%%%%%%%%%%%%%%%%%%%%%%%%%%%%%%%%%%%%%%%%%%%%%%%%%%%%%%%%%%%%%%%%%%
%
\begin{proof}
We shall derive the proposition via the following two lemmata.

%%%%%%%%%%%%%%%%%%%%%%%%%%%%%%%%%%%%%%%%%%%%%%%%%%%%%%%%%%%%%%%%%%%%%%%%%%%%%%%%
%%%%%%%%%%                            LEMMA                            %%%%%%%%%
%%%%%                                                                       %%%%
\begin{lemma}\label{LEM3.25}
The restriction ${\mathcal T}_{\mathfrak{X}}$  of ${\mathcal T}$ to
  $\mathfrak{X}$ is ${\mathcal T}_{\mathfrak{X}}'$.
\end{lemma}
%%%%%                                                                       %%%%
%%%%%%%%%%                                                             %%%%%%%%%
%%%%%%%%%%%%%%%%%%%%%%%%%%%%%%%%%%%%%%%%%%%%%%%%%%%%%%%%%%%%%%%%%%%%%%%%%%%%%%%%
%
\begin{proof}
Indeed, if  $U\in {\mathcal T}_{\mathfrak{X}}$ then
  $\boldsymbol{\bar{\oplus}}^{-1}(U)$ is open.
Pick $u\in U$ and consider $H(u)\in \boldsymbol{\bar{\oplus}}^{-1}(U)$ and
  note that $H(u)\in  O_{\dot{\mathcal V}}$.
We find an open neighborhood $W$  of $H(u)$ contained in
  $\boldsymbol{\bar{\oplus}}^{-1}(U)$.
Then $\boldsymbol{\bar{\oplus}}(W\cap O_{\dot{\mathcal V}})$ is open for
  ${\mathcal T}_{\mathfrak{X}}'$ and contained in $U$.  
This shows that $U$ can be written as union of elements of  ${\mathcal
  T}_{\mathfrak{X}}'$ and therefore $U\in  {\mathcal T}_{\mathfrak{X}}'$.

If $U\in  {\mathcal T}_{\mathfrak{X}}'$, since the map
  $\bm{\bar{\oplus}}:\dot{\mathcal V}\rightarrow \mathfrak{X}'$ is
  sc-smooth it is also continuous for ${\mathcal T}_{\mathfrak{X}}'$ and
  therefore $\bm{\bar{\oplus}}^{-1}(U)$ is open.
This shows hat $U\in {\mathcal T}_{\mathfrak{X}}$.  
\end{proof}%%%%%%%%%%%%%%%%%%%%%%%%%%%%%%%%%%%%%%%%%%%%%%%%            END PROOF
%

%%%%%%%%%%%%%%%%%%%%%%%%%%%%%%%%%%%%%%%%%%%%%%%%%%%%%%%%%%%%%%%%%%%%%%%%%%%%%%%%
%%%%%%%%%%                            LEMMA                            %%%%%%%%%
%%%%%                                                                       %%%%
\begin{lemma}\label{LEM3.26}
The identity maps $I: \mathfrak{X}\rightarrow \mathfrak{X}'$ and
  $J:\mathfrak{X}'\rightarrow \mathfrak{X}$ are sc-smooth.
\end{lemma}
%%%%%                                                                       %%%%
%%%%%%%%%%                                                             %%%%%%%%%
%%%%%%%%%%%%%%%%%%%%%%%%%%%%%%%%%%%%%%%%%%%%%%%%%%%%%%%%%%%%%%%%%%%%%%%%%%%%%%%%
%
\begin{proof} 
As we have established  $H:\mathfrak{X}' \rightarrow  O_{\dot{\mathcal
  V}}$ is an sc-diffeomorphism and $H:\mathfrak{X}\rightarrow \dot{\mathcal
  V}$ is sc-smooth. 
It follows that $I:\mathfrak{X}\rightarrow \mathfrak{X}'$ can be written
  as  the composition of sc-smooth maps
  $$
    \mathfrak{X}\xrightarrow{H} O_{\dot{\mathcal V}}\xrightarrow{H^{-1}}
    \mathfrak{X}'.
    $$
We can write $J$  as the composition
  $$
  \mathfrak{X}' \xrightarrow{H}\dot{\mathcal V}
  \xrightarrow{\boldsymbol{\bar{\oplus}}} \mathfrak{X}.
  $$
The first map sc-smooth by the definition of the M-polyfolds structure
  $\mathfrak{X}'$ and the second map  is sc-smooth using Lemma
  \ref{LEMMN3.16} {(4)}.
\end{proof}%%%%%%%%%%%%%%%%%%%%%%%%%%%%%%%%%%%%%%%%%%%%%%%%            END PROOF
This completes the proof of the Proposition \ref{ABER}.
\end{proof}%%%%%%%%%%%%%%%%%%%%%%%%%%%%%%%%%%%%%%%%%%%%%%%%            END PROOF

In view of this discussion and particularly Lemmata \ref{LEM3.25} and
  \ref{LEM3.26} we obtain the following result.

%%%%%%%%%%%%%%%%%%%%%%%%%%%%%%%%%%%%%%%%%%%%%%%%%%%%%%%%%%%%%%%%%%%%%%%%%%%%%%%%
%%%%%%%%%%                           THEOREM                           %%%%%%%%%
%%%%%                                                                       %%%%
\begin{theorem}\label{CORRX3.26}
With the map $\boldsymbol{\bar{\oplus}}:{\mathcal V}\rightarrow
  Y^{3,\delta_0}_{{\mathcal D},\varphi}({\mathbb R}\times {\mathbb
  R}^n,\boldsymbol{\gamma})$ as given in Definition \ref{DEF3.15} the
  following holds.
The restricted map
  $$
    \boldsymbol{\bar{\oplus}}:\dot{\mathcal V}\rightarrow \mathfrak{X}
    $$
  induces by the $\oplus$-method a M-polyfold structure on $\mathfrak{X}$
  together with a topology.
This topology is the same as ${\mathcal T}_{\mathfrak{X}}$ and the
  M-polyfold structure is the original one.
The map $H:\mathfrak{X}\rightarrow \dot{\mathcal V}$ constructed before
  Lemma \ref{LEMMN3.16} is sc-smooth for this M-polyfold structure and
  satisfies $\boldsymbol{\bar{\oplus}}\circ H=Id_{\mathfrak{X}}$.
\end{theorem}
%%%%%                                                                       %%%%
%%%%%%%%%%                                                             %%%%%%%%%
%%%%%%%%%%%%%%%%%%%%%%%%%%%%%%%%%%%%%%%%%%%%%%%%%%%%%%%%%%%%%%%%%%%%%%%%%%%%%%%%
%
We note that as a consequence of Theorem \ref{CORRX3.26} we can derive
  Proposition \ref{PROPN3.15}.

This concludes the easier part of the proof that $\bm{\bar{\oplus}}$ is an
  $\oplus$-construction.
It also shows that on $\mathfrak{X}$ the new construction rediscovers the
  already existing structure coming from the nodal $\oplus$-construction.

%%%%%%%%%%%%%%%%%%%%%%%%%%%%%%%%%%%%%%%%%%%%%%%%%%%%%%%%%%%%%%%%%%%%%%%%%%%%%%%%
%%%%%%%%%%                         SUB-SECTION                         %%%%%%%%%
%%%%%%%%%%%%%%%%%%%%%%%%%%%%%%%%%%%%%%%%%%%%%%%%%%%%%%%%%%%%%%%%%%%%%%%%%%%%%%%%
%
\subsection{Averaging}
As we have already mentioned several times we need to construct two maps
  partially inverting $\bm{\bar{\oplus}}$.
One of them, $H$, we constructed in the previous subsection. 
In order to construct $K$, we have to introduce an averaging construction,
  which finds for a map defined on a long cylinder, which  also
  approximates a cylinder over a periodic orbit, the right
  parameterization of the latter among the preferred parameterizations.

%%%%%%%%%%%%%%%%%%%%%%%%%%%%%%%%%%%%%%%%%%%%%%%%%%%%%%%%%%%%%%%%%%%%%%%%%%%%%%%%
%%%%%%%%%%                            FIGURE                           %%%%%%%%%
%%%%%                                                                       %%%%
\begin{figure}[h]
\begin{center}
\includegraphics[width=3.5cm]{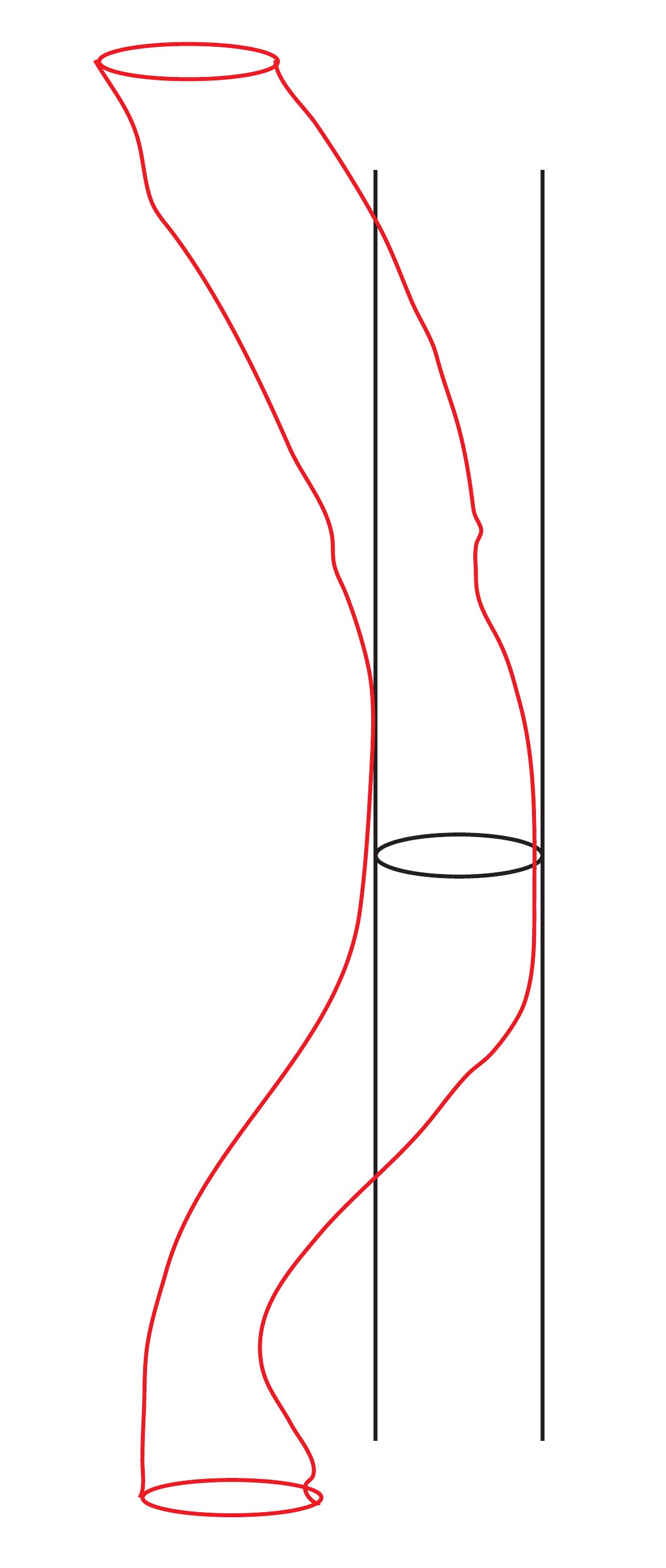}
\end{center}
\caption{Averaging of the approximation near the middle.}\label{FIG200}
\end{figure}
%%%%%                                                                       %%%%
%%%%%%%%%%                                                             %%%%%%%%%
%%%%%%%%%%%%%%%%%%%%%%%%%%%%%%%%%%%%%%%%%%%%%%%%%%%%%%%%%%%%%%%%%%%%%%%%%%%%%%%%

%%%%%%%%%%%%%%%%%%%%%%%%%%%%%%%%%%%%%%%%%%%%%%%%%%%%%%%%%%%%%%%%%%%%%%%%%%%%%%%%
%%%%%%%%%%                         SUB-SECTION                         %%%%%%%%%
%%%%%%%%%%%%%%%%%%%%%%%%%%%%%%%%%%%%%%%%%%%%%%%%%%%%%%%%%%%%%%%%%%%%%%%%%%%%%%%%
%
\subsubsection{The Averaging Map  $\mathsf{A}_{\Phi}$}\label{SSS3.2.3}
With $S^1={\mathbb R}/{\mathbb Z}$ we have the canonical periodic orbit
  $t\rightarrow t$ and its shifts $t\rightarrow t+c$.
Consider for the Abelian group $S^1$ the standard covering map 
  $$
    e:{\mathbb R}\rightarrow S^1:t\rightarrow t\ \text{mod}\ {\mathbb Z}.
    $$
In the following we shall consider (continuous) maps $q:S^1\rightarrow
  S^1$.
For such a $q$ take a continuous lift $\wt{q}:{\mathbb R}\rightarrow
  {\mathbb R}$ which then satisifies
  $$
    e\circ \wt{q}(t) =q(e(t))\ \text{for all}\ t\in {\mathbb R}.
    $$
We define
  $$
    \oint_{S^1} q(t)\cdot dt := \int_{[0,1]} \wt{q}(t) \cdot dt\
    \text{mod}\ {\mathbb Z}.
    $$
The definition does not depend on the choice of the lift. 
We shall also sometimes refer to this integral as the $S^1$-average or
  $\oint_{S^1}$-average of $q$.

%%%%%%%%%%%%%%%%%%%%%%%%%%%%%%%%%%%%%%%%%%%%%%%%%%%%%%%%%%%%%%%%%%%%%%%%%%%%%%%%
%%%%%%%%%%                          DEFINITION                         %%%%%%%%%
%%%%%                                                                       %%%%
\begin{definition}
With $\boldsymbol{{\gamma}}$ being a periodic orbit in ${\mathbb R}^N$, a
  {\bf good averaging coordinate}\index{good averaging coordinate}
  $\Phi:U\rightarrow S^1$ consists of an open neighborhood $U\subset
  {\mathbb R}^N$ of $\gamma(S^1)$ and a smooth map $\Phi:U\rightarrow S^1$
  such that  for  a suitable $\gamma_0\in [\gamma]$ it holds that
  $\Phi\circ\gamma_0(t)=t$ for $t\in S^1$. (Note that $\gamma_0$ is
  uniquely determined by $\Phi$.) 
\end{definition}
%%%%%                                                                       %%%%
%%%%%%%%%%                                                             %%%%%%%%%
%%%%%%%%%%%%%%%%%%%%%%%%%%%%%%%%%%%%%%%%%%%%%%%%%%%%%%%%%%%%%%%%%%%%%%%%%%%%%%%%
%
Given $\bm{\gamma}$ and $\Phi:U\rightarrow S^1$ assume that
  $w:Z_a\rightarrow {\mathbb R}^N$ is a continuous map, where $Z_a$ is
  obtained from the gluing parameter $a=|a|\cdot [\wh{x},\wh{y}]$ and the
  ordered disk pair $((D_x,D_y),(x,y))$.  
Next we need the notion of a `middle annulus' on $Z_a$.

%%%%%%%%%%%%%%%%%%%%%%%%%%%%%%%%%%%%%%%%%%%%%%%%%%%%%%%%%%%%%%%%%%%%%%%%%%%%%%%%
%%%%%%%%%%                          DEFINITION                         %%%%%%%%%
%%%%%                                                                       %%%%
\begin{definition}
Define $M^h_a$, the so-called {\bf middle annulus} \index{middle annulus}
  of $Z_a$ of width $2\cdot h$, for $h\in [0,15)$,  to consist of all
  $(z,z')\in Z_a$ such that $z=\sigma^{+,a}_{\wh{x}}(s,t)$ with $s\in
  [R/2-h,R/2+h]$, where $R=\varphi(|a|)$, and $t\in S^1$.
Since $|a|\in (0,1/4)$ it always holds that $M^h_a\subset Z_a$ provided
  $h\in [0,25)$.
For smaller $|a|$ it contains larger middle annuli.  
See also  Appendix \jwf{[broken reference]} %\ref{APP5.1}.
\end{definition}
%%%%%                                                                       %%%%
%%%%%%%%%%                                                             %%%%%%%%%
%%%%%%%%%%%%%%%%%%%%%%%%%%%%%%%%%%%%%%%%%%%%%%%%%%%%%%%%%%%%%%%%%%%%%%%%%%%%%%%%
%

We impose the following two properties on ${w}:Z_a\rightarrow {\mathbb
  R}^N$:
  \begin{itemize}
    \item[(1)]   
    ${w}(M^3_a)\subset U$.
    \item[(2)] 
    For a decoration $\wh{x}$ the map $S^1\rightarrow S^1: t\rightarrow
      \Phi\circ w\circ \sigma_{\wh{x}}^{+,a}(R/2,t)
      $ has degree $k$.   
    Of course, the degree does not depend on the decoration of $x$.
    \end{itemize}
The $\oint_{S^1}$-average of the map in (2) depends on the choice of
  $\wh{x}$, but we shall introduce later on  new expressions which do not,
  and which carriy nontrivial information. For this reason we need to
  understand fully certain dependencies.

%%%%%%%%%%%%%%%%%%%%%%%%%%%%%%%%%%%%%%%%%%%%%%%%%%%%%%%%%%%%%%%%%%%%%%%%%%%%%%%%
%%%%%%%%%%                            LEMMA                            %%%%%%%%%
%%%%%                                                                       %%%%
\begin{lemma}\label{THETAtau}
Let $\theta=e^{2\pi i\tau}$ for $\tau\in [0,1)$. 
Then mod ${\mathbb Z}$ we have the identity
  $$
  \oint_{S^1} \Phi\circ w\circ\sigma_{\theta\cdot\wh{x}}^{+,a}(R/2,t)\cdot
  dt = \oint_{S^1} \Phi\circ w\circ\sigma_{\wh{x}}^{+,a}(R/2,t)\cdot dt -
  k\tau.
  $$
\end{lemma}
%%%%%                                                                       %%%%
%%%%%%%%%%                                                             %%%%%%%%%
%%%%%%%%%%%%%%%%%%%%%%%%%%%%%%%%%%%%%%%%%%%%%%%%%%%%%%%%%%%%%%%%%%%%%%%%%%%%%%%%
%
\begin{proof}
We note that $\sigma^+_{\theta\cdot
  \wh{x}}(s,t+\tau)=\sigma_{\wh{x}}^+(s,t)$.
We define 
  $$
    \alpha(t):= \Phi\circ w\circ\sigma_{\theta\cdot\wh{x}}^{+,a}(R/2,t)
    $$
  and consequently
  \begin{eqnarray*}
    \alpha(t+\tau)=\Phi\circ
    w\circ\sigma_{\theta\cdot\wh{x}}^{+,a}(R/2,t+\tau)= \Phi\circ
    w\circ\sigma_{\wh{x}}^{+,a}(R/2,t).
    \end{eqnarray*} 
Then we compute mod ${\mathbb Z}$
  \begin{eqnarray*}
    &&\oint_{S^1} \alpha(t)\cdot dt =\int_0^1 \wt{\alpha}(t)\cdot dt \ \
    \text{mod}\ \ {\mathbb Z}\\
    &=& \int_{-\tau}^{1-\tau} \wt{\alpha}(t+\tau)\cdot dt \ \ \text{mod}\
    \ {\mathbb Z}\\
    &=&\int_{-\tau}^0 \wt{\alpha}(t+\tau)\cdot dt +
    \int_{0}^{1-\tau}\wt{\alpha}(t+\tau)\cdot dt \ \ \text{mod}\ \
    {\mathbb Z}\\
    &=&\int_{1-\tau}^1( \wt{\alpha}(t+\tau)-k)\cdot dt +
    \int_0^{1-\tau}\wt{\alpha}(t+\tau)\cdot dt \ \
    \text{mod}\ \ {\mathbb Z}\\
    &=&\int_0^1\wt{\alpha}(t+\tau)\cdot dt -k\cdot\tau \ \ \text{mod}\ \
    {\mathbb Z}\\
    &=&\oint_{S^1} \alpha(t+\tau) \cdot dt -k\cdot \tau \ \ \text{mod}\ \
    {\mathbb Z}.
    \end{eqnarray*}
\end{proof}%%%%%%%%%%%%%%%%%%%%%%%%%%%%%%%%%%%%%%%%%%%%%%%%            END PROOF

Given $\Phi$ there exists a unique $\gamma_0$ satisfying
  $\Phi\circ\gamma_0(t)=t$ mod $1$.
Every other $\gamma\in [\gamma_0]$ can be written as
  $\gamma(t)=\gamma_0(t+c)$ for some $c\in S^1$ and we abbreviate
  $\gamma_c(t)=\gamma_0(t+c)$ and compute

\begin{eqnarray}\label{I^4}
  &&\oint_{S^1} \Phi\circ \gamma_c(kt)\cdot dt\\
  &=& \oint_{S^1} \Phi\circ\gamma_0(kt+c)= \oint_{S^1}  (kt+c\ \text{mod}\
  {\mathbb Z})\cdot dt\nonumber\\
  &=&\int_{0}^1 (kt+c)\cdot dt\ \text{mod}\ {\mathbb Z}= c + k/2\
  \text{mod}\ {\mathbb Z}.\nonumber
  \end{eqnarray}
We extend the discussion from the previous Subsection \ref{CORETRH} to
  also deal with data of the form
  $(0,(\wt{w}^x,[\wh{x},\wh{y}],\wt{w}^y))\in \partial\mathfrak{Z}$.
Write $\boldsymbol{\gamma}=([\gamma],T,k)$.

%%%%%%%%%%%%%%%%%%%%%%%%%%%%%%%%%%%%%%%%%%%%%%%%%%%%%%%%%%%%%%%%%%%%%%%%%%%%%%%%
%%%%%%%%%%                          DEFINITION                         %%%%%%%%%
%%%%%                                                                       %%%%
\begin{definition}\label{KLM3.33}
Having fixed a good averaging coordinate $\Phi:U\rightarrow S^1$ for
  $\boldsymbol{\gamma}$, the subset ${\mathcal W}$  of
  $Y^{3,\delta_0}_{{\mathcal D},\varphi}({\mathbb R}\times {\mathbb
  R}^N,\boldsymbol{\gamma})$ consists of the following elements, where as
  usual $\wt{w}=(b,w)$.
\begin{itemize}
  \item[(1)] 
  All $(0,(\wt{w}^x,[\wh{x},\wh{y}],\wt{w}^y))$.
  \item[(2)]  
  All $(r,\wt{w})$ with $r\in (0,1)$ so that $w|M^3_a$ has the image in
  $U$ and it holds for a suitable decoration $\wh{x}$ that for every $s\in
  [R/2-3,R/2+3]$, $R=R(|a|)=R(|a(\wt{w})|)$, the map
  $$
  S^1\rightarrow S^1 :t\rightarrow \Phi\circ w\circ
  \sigma^{+,a}_{\wh{x}}(s,t)
  $$
  has degree $k$. 
  \end{itemize}
\end{definition}
%%%%%                                                                       %%%%
%%%%%%%%%%                                                             %%%%%%%%%
%%%%%%%%%%%%%%%%%%%%%%%%%%%%%%%%%%%%%%%%%%%%%%%%%%%%%%%%%%%%%%%%%%%%%%%%%%%%%%%%
%

%%%%%%%%%%%%%%%%%%%%%%%%%%%%%%%%%%%%%%%%%%%%%%%%%%%%%%%%%%%%%%%%%%%%%%%%%%%%%%%%
%%%%%%%%%%                            LEMMA                            %%%%%%%%%
%%%%%                                                                       %%%%
\begin{lemma}\label{LEMMA3.34}
The preimage $\boldsymbol{\bar{\oplus}}^{-1}({\mathcal W})$ is open in
  ${\mathcal V}$.
\end{lemma}
%%%%%                                                                       %%%%
%%%%%%%%%%                                                             %%%%%%%%%
%%%%%%%%%%%%%%%%%%%%%%%%%%%%%%%%%%%%%%%%%%%%%%%%%%%%%%%%%%%%%%%%%%%%%%%%%%%%%%%%
%
\begin{proof}
Obvious.
\end{proof}%%%%%%%%%%%%%%%%%%%%%%%%%%%%%%%%%%%%%%%%%%%%%%%%            END PROOF

Assume that $(r,\wt{w})\in {\mathcal W}$ with $r\in (0,1)$ and write
  $a=|a|\cdot [\wh{x},\wh{y}]$.
We pick a representative $(\wh{x},\wh{y})\in {\mathbb S}_x\times {\mathbb
  S}_y$  and take the $\oint_{S^1}$-average, denoted by
  $d_{\wh{x}}(r,\wt{w})$
  \begin{eqnarray}\label{I^5}
    d_{\wh{x}}(r,\wt{w}):=\oint_{S^1} \Phi\circ w\circ
    \sigma^{+,a}_{\wh{x}}(R/2,t)\cdot dt\in S^1,
    \end{eqnarray}
  which will depend on $\wh{x}$ as described in Lemma \ref{THETAtau}. 
Here $R=\varphi(|a|)$ and $a=|a|\cdot [\wh{x},\wh{y}]$.
Then we define numbers $c^x(r,\wt{w})$ and $c^y(r,\wt{w})$, where
  $\wt{w}=(b,w)$, by
  \begin{eqnarray}\label{I^6}
    &&c^x(r,\wt{w}) :=\int_{S^1} b\circ \sigma^{+,a}_{\wh{x}}(R/2,t)\cdot
    dt  - \frac{1}{2}\cdot T\cdot \varphi(|a|),\\
    &&c^y(r,\wt{w}):= c^x(r,\wt{w}) +T\cdot \varphi(|a|)-
    \varphi(r)\nonumber\\
    &&\ \ \ \ \ \ \ = \int_{S^1} b\circ \sigma^{+,a}_{\wh{x}}(R/2,t)\cdot
    dt +\frac{1}{2}\cdot T\cdot \varphi(|a|)-\varphi(r).\nonumber
    \end{eqnarray}
We note that $d_{\wh{x}}(r,\wt{w})$ depends on the choice of $\wh{x}$, but
  that the two real numbers defined in (\ref{I^6}) do not depend on this
  choice.
Note that at this stage, given $(r,\wt{w})\in {\mathcal W}$ we can extract
  from the domain of $\wt{w}$ the domain gluing parameter $a=|a|\cdot
  [\wh{x},\wh{y}]$ and the averages as described above.
From this data we shall be able to construct some kind of approximation of
  $\wt{w}$ in $\mathsf{S}_{\bm{\gamma}}$.

%%%%%%%%%%%%%%%%%%%%%%%%%%%%%%%%%%%%%%%%%%%%%%%%%%%%%%%%%%%%%%%%%%%%%%%%%%%%%%%%
%%%%%%%%%%                            LEMMA                            %%%%%%%%%
%%%%%                                                                       %%%%
\begin{lemma}
Given $(r,\wt{w})\in {\mathcal W}$ there exists a uniquely determined element $\wt{p}:=(\wt{p}^x,[\wh{x},\wh{y}],\wt{p}^y)\in \mathsf{S}_{\bm{\gamma}}$ 
characterized by the following properties, where we write $\wt{p}^x=(b^x,p^x)$ and $\wt{p}^y=(b^y,p^y)$, and $(\wh{x},\wh{y})$
is a representative of $[\wh{x},\wh{y}]$.
\begin{itemize}
\item[(1)] $c^x(\wt{p}) = c^x(r,\wt{w})$ and $c^y(\wt{p})=c^y(r,\wt{w})$.
\item[(2)] $\oint_{S^1} \Phi\circ p^x\circ \sigma^{+}_{\wh{x}}(s,t)\cdot dt=   d_{\wh{x}}(r,\wt{w})$ for all $s\in {\mathbb R}^+$.
\item[(3)]  $\oint_{S^1} \Phi\circ p^y\circ \sigma^{-}_{\wh{y}}(s',t')\cdot dt'  =  d_{\wh{x}}(r,\wt{w})$ for all
$s'\in {\mathbb R}^-$.
\end{itemize}

\end{lemma}
%%%%%                                                                       %%%%
%%%%%%%%%%                                                             %%%%%%%%%
%%%%%%%%%%%%%%%%%%%%%%%%%%%%%%%%%%%%%%%%%%%%%%%%%%%%%%%%%%%%%%%%%%%%%%%%%%%%%%%%
%
\begin{proof}
We need to prove existence and uniqueness of the element $\wt{p}$. 
Let us note that the expressions in (2) and (3) on the right and left have
  the same dependencies with respect to $\wh{x}$ and $\wh{y}$, respectively,
  and consequently the equations do not depend on the choices made.\\

\noindent{\bf Existence:}
Denote by $\gamma_0\in [\gamma]$ the representative satisfying
  $\Phi\circ\gamma_0(t)=t$ for $t\in S^1$.
Introduce $\gamma_c$ for $c\in S^1$ by $\gamma_c(t)=\gamma_0(t+c)$.
Define $\wt{p}^x =(b^x,p^x)$ and $\wt{p}^y=(b^y,p^y)$ by 
  \begin{eqnarray}
    &\ \ \ \ \ \ b^x\circ\sigma_{\wh{x}}^+(s,t) = Ts +c^x(r,\wt{w}),\ \
    b^y\circ\sigma_{\wh{y}}^-(s',t')= Ts' +c^y(r,\wt{w})&
    \end{eqnarray}
  and
  \begin{eqnarray}
    p^x\circ \sigma^+_{\wh{x}}(s,t) &=&
    \gamma_{[d_{\wh{x}}(r,\wt{w})-(k/2)]}(kt)\\
    p^y\circ
    \sigma^-_{\wh{y}}(s',t')&=&\gamma_{[d_{\wh{x}}(r,\wt{w})-(k/2)]}(kt').\nonumber
    \end{eqnarray}
With these definitions the element
  $\wt{p}=(\wt{p}^x,[\wh{x},\wh{y}],\wt{p}^y)$ in fact belongs to
  $\mathsf{S}_{\bm{\gamma}}$. 
Namely the $\wh{x}$-directional limit of $p^x\circ\sigma^+_{\wh{x}}(s,t)$
  is given by 
  $$
    \text{lim}_{s\rightarrow
    \infty}p^x\circ\sigma^+_{\wh{x}}(s,0)=\gamma_0(d_{\wh{x}}(r,\wt{w})-k/2)
    $$
  and for $p^y$ the $\wh{y}$-directional limit by
  $$
    \text{lim}_{s'\rightarrow-\infty} p^y\circ \sigma^-_{\wh{y}}(s',0)=
    \gamma_0(d_{\wh{x}}(r,\wt{w})-k/2).
    $$
Hence the data is $[\wh{x},\wh{y}]$-matching.
The real asymptotic constants satisfy
  $$
    c^x(\wt{p})=c^x(r,\wt{w})\ \text{and}\  c^y(\wt{p})=c^y(r,\wt{w}).
    $$
We compute
  \begin{eqnarray}
    &\oint_{S^1} \Phi\circ p^x\circ \sigma_{\wh{x}}^+(s,t)\cdot dt
    =(\int_0^1 [kt + d_{\wh{x}}(r,\wt{w})-(k/2)]\cdot dt)_{/{\mathbb Z}}
    = d_{\wh{x}}(r,\wt{w})&\nonumber
    \end{eqnarray}
  and similarly
  \begin{eqnarray*}
    &\oint_{S^1} \Phi\circ p^y\circ \sigma_{\wh{y}}^+(s',t')\cdot dt'
    = d_{\wh{x}}(r,\wt{w}).\nonumber
    \end{eqnarray*}

\noindent{\bf Uniqueness:} 
Assume that $\wt{p}_1$ is another element in $\mathsf{S}_{\bm{\gamma}}$
  and $[\wh{x}_1,\wh{y}_1]=[\wh{x},\wh{y}]$ such that
  $$
    c^x(\wt{p}_1)=c^x(r,\wt{w})\ \text{and}\ \ c^y(\wt{p}_1)=c^y(r,\wt{w})
    $$
  and 
  \begin{eqnarray}
    &\oint_{S^1} \Phi\circ p^x_1\circ\sigma_{\wh{x}_1}^+(s,t)\cdot dt
    =d_{\wh{x}_1}(r,\wt{w})&\\
    &\oint_{S^1} \Phi\circ p^y_1\circ\sigma_{\wh{y}_1}^-(s',t')\cdot dt'
    =d_{\wh{x}_1}(r,\wt{w}).&\nonumber
    \end{eqnarray}
Then $c^x(\wt{p})=c^x(\wt{p}_1)$ and $c^y(\wt{p})=c^y(\wt{p}_1)$. 
We can write
  $$
  p_1^x\circ \sigma_{\wh{x}_1}^+(s,t) =\gamma_0(kt+c)
  $$
  for a suitable $c\in [0,1)$ and similarly $p^x\circ
  \sigma_{\wh{x}}^+(s,t)=\gamma_0(kt+c_0)$.
We compute (mod ${\mathbb Z}$) with $\theta=e^{2\pi i\tau}$, $\tau\in
  [0,1)$, and $\wh{x}_1=\theta\cdot\wh{x}$ and $\wh{y}_1=\theta^{-1}\cdot
  \wh{y}$  the following
  \begin{eqnarray*}
    && k/2 +c\ \text{mod}\ {\mathbb Z}=\oint_{S^1} ( kt+c\ \text{mod}\ \
    {\mathbb Z})\cdot dt\\
    &=&\oint_{S^1}\Phi\circ p^x_1\circ\sigma_{\wh{x}_1}^+(s,t)\cdot dt\\
    & =&d_{\wh{x}_1}(r,\wt{w})= d_{\wh{x}}(r,\wt{w})-k\tau\ \text{mod}\
    {\mathbb Z}\\
    &=& \oint_{S^1} \Phi\circ p^x\circ \sigma_{\wh{x}}^+(s,t)\cdot dt
    -k\tau\ \text{mod}\ {\mathbb Z}\\
    &=& \oint_{S^1} (kt+c_0\ \text{mod}\ {\mathbb Z})\cdot dt - k\tau\
    \text{mod}\ {\mathbb Z}\\
    &=&k/2 +c_0-k\tau\ \text{mod}\ {\mathbb Z}.
    \end{eqnarray*}
Consequently $c=c_0-k\tau$ mod ${\mathbb Z}$.  
Hence 
  \begin{eqnarray*}
    && p^x_1\circ \sigma^+_{\wh{x}_1}(s,t)=\gamma_0(kt+c)\\
    &=&\gamma_0(kt+c_0-k\tau)=\gamma_0(k(t-\tau) +c_0)\\
    &=&p^x\circ\sigma^+_{\wh{x}}(s,t-\tau)=p^x\circ
    \sigma^+_{\wh{x}_1}(s,t).
    \end{eqnarray*}
This shows that $p^x_1=p^x$ and similarly we can show that $p^y_1=p^y$.
This completes the proof of uniqueness.
\end{proof}%%%%%%%%%%%%%%%%%%%%%%%%%%%%%%%%%%%%%%%%%%%%%%%%            END PROOF

In view of this lemma we can now give  the definition of the averaging map
  $\mathsf{A}_{\Phi}$.

%%%%%%%%%%%%%%%%%%%%%%%%%%%%%%%%%%%%%%%%%%%%%%%%%%%%%%%%%%%%%%%%%%%%%%%%%%%%%%%%
%%%%%%%%%%                          DEFINITION                         %%%%%%%%%
%%%%%                                                                       %%%%
\begin{definition} \label{ewr3.37}
Given a good averaging coordinate  $\Phi:U\rightarrow S^1$  associated to
  $\bm{\gamma}$ the {\bf averaging map}\index{averaging map}
  $$
    \mathsf{A}_{\Phi}\index{$\mathsf{A}_{\Phi}$}:{\mathcal W}\rightarrow
    \mathsf{S}_{\bm{\gamma}}
    $$
  is defined as follows.  
For the special elements with $r=0$ we set
  $$ 
    \mathsf{A}_{\Phi}(0,(\wt{q}^x,[\wh{x},\wh{y}],
    \wt{q}^y)\dotplus(\wt{h}^x,\wt{h}^y))
    =(\wt{q}^x,[\wh{x},\wh{y}],\wt{q}^y)
    $$
For elements $(r,\wt{w})$ with $r\neq 0$ we define
  $$
    \mathsf{A}(r,\wt{w})=\wt{p},
    $$
  with $c^x(\wt{p})=c^x(r,\wt{w})$, $c^y(\wt{p})=c^y(r,\wt{w})$ and
  $$
    d_{\wh{x}}(r,\wt{w})= \oint_{S^1}\Phi\circ
    p^x\circ\sigma^+_{\wh{x}}(s,t)\cdot dt
    =\oint_{S^1}\Phi\circ p^y\circ\sigma^+_{\wh{y}}(s',t')\cdot dt'.
    $$
\end{definition}
%%%%%                                                                       %%%%
%%%%%%%%%%                                                             %%%%%%%%%
%%%%%%%%%%%%%%%%%%%%%%%%%%%%%%%%%%%%%%%%%%%%%%%%%%%%%%%%%%%%%%%%%%%%%%%%%%%%%%%%
%

%%%%%%%%%%%%%%%%%%%%%%%%%%%%%%%%%%%%%%%%%%%%%%%%%%%%%%%%%%%%%%%%%%%%%%%%%%%%%%%%
%%%%%%%%%%                            REMARK                           %%%%%%%%%
%%%%%                                                                       %%%%
\begin{remark}
$A_\Phi(r,\wt{w})$ gives an element
  $\wt{p}=(\wt{p}^x,[\wh{x},\wh{y}],\wt{p}^y)$ in $\mathsf{S}_{\gamma}$,  so
  that the   map $\bar{\oplus}(r,(\wt{p}^x,[\wh{x},\wh{y}],\wt{p}^y))$, is
  in some sense the  best approximation of  $\wt{w}$ by a simple type of
  maps.
\end{remark}
%%%%%                                                                       %%%%
%%%%%%%%%%                                                             %%%%%%%%%
%%%%%%%%%%%%%%%%%%%%%%%%%%%%%%%%%%%%%%%%%%%%%%%%%%%%%%%%%%%%%%%%%%%%%%%%%%%%%%%%
%

%%%%%%%%%%%%%%%%%%%%%%%%%%%%%%%%%%%%%%%%%%%%%%%%%%%%%%%%%%%%%%%%%%%%%%%%%%%%%%%%
%%%%%%%%%%                          EXERCISE                           %%%%%%%%%
%%%%%                                                                       %%%%
\begin{exercise}\label{exc300}
Prove for $(r,\wt{q})\in [0,1)\times \mathsf{S}_{\bm{\gamma}}$ the
  identity
  $$
    \wt{q}=\mathsf{A}_{\Phi}\circ \bm{\bar{\oplus}}(r,\wt{q}).
    $$
\end{exercise}
%%%%%                                                                       %%%%
%%%%%%%%%%                                                             %%%%%%%%%
%%%%%%%%%%%%%%%%%%%%%%%%%%%%%%%%%%%%%%%%%%%%%%%%%%%%%%%%%%%%%%%%%%%%%%%%%%%%%%%%
%
At this point, as the basic achievement of this subsection, we have
  constructed the averaging map $\mathsf{A}_{\Phi}$. \\
\begin{tcolorbox}
  $$
  \text{Averaging map}\  \mathsf{A}_{\Phi}:{\mathcal W}\rightarrow
  \mathsf{S}_{\bm{\gamma}}\ \text{associated to}\ \Phi:U\rightarrow S^1.
    $$
  \end{tcolorbox}

%%%%%%%%%%%%%%%%%%%%%%%%%%%%%%%%%%%%%%%%%%%%%%%%%%%%%%%%%%%%%%%%%%%%%%%%%%%%%%%%
%%%%%%%%%%                         SUB-SECTION                         %%%%%%%%%
%%%%%%%%%%%%%%%%%%%%%%%%%%%%%%%%%%%%%%%%%%%%%%%%%%%%%%%%%%%%%%%%%%%%%%%%%%%%%%%%
%
\subsubsection{Sc-Smoothness Properties of $\mathsf{A}_{\Phi}$}
The main goal of this subsection is the study of the sc-smoothness
  properties of $\mathsf{A}_{\Phi}$.
To begin our considerations define  the smooth manifold 
  $$
    \mathsf{M}= {\mathbb R}\times {\mathbb S}_{{\mathcal D}}\times
    S^1\times {\mathbb R}.
    $$
Here ${\mathbb S}_{\mathcal D}$ is the smooth manifold diffeomorphic to
  ${\mathbb S}^1$ consisting of the set of  all $[\wh{x},\wh{y}]$.
We fix $\wh{x}_0$  so that every element $[\wh{x}',\wh{y}'] \in {\mathbb
  S}_{{\mathcal D}}$ can be uniquely written as $[\wh{x}_0,\wh{y}]$.
Associated to this choice $\wh{x}_0$ we define a map
  $$
    \mathsf{R}_{\wh{x}_0}:\mathsf{M}\rightarrow \mathsf{S}_{\bm{\gamma}}
    $$
  by associating to $(c^x,[\wh{x}_0,\wh{y}],d,c^y)$ the element
  $$
    \wt{p}:=\mathsf{R}_{\wh{x}_0}(c^x,[\wh{x}_0,\wh{y}],d,c^y),
    $$
  defined by $\wt{p}=(\wt{p}^x,[\wh{x}_0,\wh{y}],\wt{p}^y)$ satisfying
  with $\wt{p}^x=(b^x,p^x)$ and $\wt{p}^y=(b^y,p^y)$
  $$
    b^x\circ \sigma^+_{\wh{x}_0}(s,t)=Ts+c^x\ \text{and}\ \ b^y\circ
    \sigma^-_{\wh{y}}(s',t')=Ts'+c^y
    $$
  and moreover
  \begin{eqnarray*}
    &\oint_{S^1}\Phi\circ p^x\circ \sigma^+_{\wh{x}_0}(s,t) \cdot dt = d\
    \ \text{and}\ \ \oint_{S^1}\Phi\circ p^y\circ
    \sigma^-_{\wh{y}}(s',t')\cdot dt' =d.&
    \end{eqnarray*}

%%%%%%%%%%%%%%%%%%%%%%%%%%%%%%%%%%%%%%%%%%%%%%%%%%%%%%%%%%%%%%%%%%%%%%%%%%%%%%%%
%%%%%%%%%%                            LEMMA                            %%%%%%%%%
%%%%%                                                                       %%%%
\begin{lemma} \label{LEMM3.38}
The map $\mathsf{R}_{\wh{x}_0}$ is smooth.
\end{lemma}
%%%%%                                                                       %%%%
%%%%%%%%%%                                                             %%%%%%%%%
%%%%%%%%%%%%%%%%%%%%%%%%%%%%%%%%%%%%%%%%%%%%%%%%%%%%%%%%%%%%%%%%%%%%%%%%%%%%%%%%
%
\begin{proof}
Recall that $\Phi\circ\gamma_0(t)=t$.
We define for $d\in S^1$ the element $p^x_d$ by
  $$
    p^x_d \circ\sigma_{\wh{x}_0}(s,t) = \gamma_0(kt-(k/2)+d),
    $$
  so that $\oint_{S^1} \Phi\circ p^x_d\circ \sigma_{\wh{x}_0}(s,t)\cdot dt
  =d$.
We define $b^x_{c^x}$ by  $b^x_{c^x}\circ \sigma^+_{\wh{x}_0}(s,t) = Ts+c^x$
and $b^y_{c^y}$ by $b^y_{c^y}\circ \sigma^-_{\wh{y}}(s',t')=Ts'+c^y$. 
We note that the definition of the latter does not depend on the choice of
  $\wh{y}$.
The element $p^y_{\wh{y},d}$ is defined by
  $$
    p^y_{\wh{y},d}\circ\sigma^-_{\wh{y}}(s',t')= \gamma_0(kt'-(k/2)+d).  $$
With these definitions it follows that
  $$
    \mathsf{R}_{\wh{x}_0}(c^x,[\wh{x}_0,\wh{y}],d,c^y)
    =((b^x_{c^x},p^x_{d}),[\wh{x}_0,\wh{y}],(b^y_{c^y},p^y_{d,\wh{y}})).
    $$
We recall the maps $\phi_\gamma$ defined in (\ref{EQNC16}) which are
  smooth $k:1$ coverings.
We note that 
  $$
    \phi_{\gamma_0}(c^x,e^{-2\pi i (d-k/2)}\cdot\wh{x}_0,c^y,e^{2\pi i
    (d-k/2)}\cdot \wh{y})=
    \mathsf{R}_{\wh{x}_0}(c^x,[\wh{x}_0,\wh{y}],d,c^y),
    $$
  and the smoothness of the map
  $$
    (c^x,[\wh{x}_0,\wh{y}],d,c^y)\rightarrow (c^x,e^{-2\pi i
    (d-k/2)}\cdot\wh{x}_0,c^y,e^{2\pi i (d-k/2)}\cdot \wh{y})
    $$
  implies the desired result since $\phi_{\gamma_0}$ is a local
  diffeomorphism.
\end{proof}%%%%%%%%%%%%%%%%%%%%%%%%%%%%%%%%%%%%%%%%%%%%%%%%            END PROOF

Our aim is to show that $\boldsymbol{\bar{\oplus}}:{\mathcal V}\rightarrow
  Y^{3,\delta_0}_{{\mathcal D},\varphi}$ is an $\oplus$-polyfold
  construction and defines a M-polyfold structure on the target
  characterized by a list of properties. Since ${\mathcal W}$
  is a subset of $Y^{3,\delta_0}_{{\mathcal D},\varphi}$ which is open for
  the quotient topology ${\mathcal T}$, the sc-smoothnees of
  $\mathsf{A}_{\Phi}$ for the M-polyfold structure would be equivalent to
  the sc-smoothness of the map $A_\Phi\circ
  \bm{\bar{\oplus}}:\bm{\bar{\oplus}}^{-1}({\mathcal W})\rightarrow
  \mathsf{S}_{\boldsymbol{\gamma}}$, if in fact $\bm{\bar{\oplus}}$
  defines a M-polyfold structure.
In view of the previous lemma we can reduce this question to the study of
  the dependencies of $c^x,c^y$ and $d$.
The following result means that $\mathsf{A}_{\Phi}$ will be sc-smooth for
  the M-polyfold structure we are going to define on
  $Y^{3,\delta_0}_{{\mathcal D},\varphi}$.

%%%%%%%%%%%%%%%%%%%%%%%%%%%%%%%%%%%%%%%%%%%%%%%%%%%%%%%%%%%%%%%%%%%%%%%%%%%%%%%%
%%%%%%%%%%                          PROPOSITION                        %%%%%%%%%
%%%%%                                                                       %%%%
\begin{proposition} \label{PROPP3.33}
 The map 
  $$
    \mathsf{A}_{\Phi}\circ \boldsymbol{\bar{\oplus}}:
    \boldsymbol{\bar{\oplus}}^{-1}({\mathcal W})\rightarrow
    \mathsf{S}_{\boldsymbol{\gamma}},
    $$
  which is defined on an open subset of ${\mathcal V}$ containing
  $\partial\mathfrak{Z}$, is sc-smooth.
\end{proposition}
%%%%%                                                                       %%%%
%%%%%%%%%%                                                             %%%%%%%%%
%%%%%%%%%%%%%%%%%%%%%%%%%%%%%%%%%%%%%%%%%%%%%%%%%%%%%%%%%%%%%%%%%%%%%%%%%%%%%%%%
%
\begin{proof}
Given an element $(r,(\wt{u}^x,[\wh{x},\wh{y}],\wt{u}^y))\in
  \boldsymbol{\bar{\oplus}}^{-1}({\mathcal W})$ we define 
  $$
    \wt{w} = \boldsymbol{\bar{\oplus}}(r,(\wt{u}^x,[\wh{x},\wh{y}],\wt{u}^y)),
    $$
  where, if $r=0$, $\wt{w}$ stands for
  $(0,(\wt{u}^x,[\wh{x},\wh{y}],\wt{u}^y))$.
The procedure produces from $[\wh{x},\wh{y}]$, if $r\neq 0$ a glued
  surface $Z_a$ on which $\wtilde{w}$ is defined, where $a$ has angular
    part $[\wh{x},\wh{y}]$ and in the case $r=0$ keeps it as part of the
  data.
The element $(\wt{u}^x,[\wh{x},\wh{y}],\wt{u}^y)$ has the form 
  $$
    \wt{q}\dotplus\wt{h}:=
    (\wt{q}^x,[\wh{x},\wh{y}],\wt{q}^y)\dotplus(\wt{h}^x,\wt{h}^y)
    $$
  and the gluing parameter $a$ is smoothly computed from $(r,\wt{q})\in
  [0,1)\times \mathsf{S}_{\boldsymbol{\gamma}}$. 
Consider $\wt{Q}:=\bm{\bar{\oplus}}(r,(\wt{q}^x,[\wh{x},
  \wh{y}],\wt{q}^y))$, which equals $(\wt{q}^x,[\wh{x},\wh{y}],\wt{q}^y)$
  if $r=0$ and $\wt{H}=\oplus_a(\wt{h}^x,\wt{h}^y)$,
  where $a=\mathsf{a}(r,(\wt{u}^x,[\wh{x},\wh{y}],\wt{u}^y))$ is the
  associated gluing parameter, which by Lemma \ref{LEM3.14} is ssc-smooth.
Then
  $$
    \bm{\bar{\oplus}}(r,(\wt{u}^x,[\wh{x},\wh{y}],\wt{u}^y))
    =\bm{\bar{\oplus}}(r,(\wt{q}^x,[\wh{x},\wh{y}],\wt{q}^y))+
    \oplus_a(\wt{h}^x,\wt{h}^y)=\wt{Q}+\wt{H}.
    $$
Write $\wt{Q}=(B,Q)$ and $\wt{H}=(C,H)$ and fix a decoration $\wh{x}_0$.
With the definitions in (\ref{I^5}) and (\ref{I^6}) we obtain using
  $R=\varphi(|a|)$
  \begin{eqnarray}\label{ew3.12}
    &&c^x(r,\wt{w})= \int_{S^1}
    [B\circ\sigma^+_{\wh{x}_0}(R/2)+C\circ\sigma^+_{\wh{x}_0}(R/2,t)]\cdot
    dt-\frac{1}{2}\cdot T\cdot R\\
    &=& c^x(\wt{q}^x) +\int_{S^1}C\circ\sigma^+_{\wh{x}_0}(R/2,t)\cdot
    dt.\nonumber
    \end{eqnarray}
Here we have used that $\wt{Q}$ is the restriction of $\wt{q}^x$ to $Z_a$.   
Given $[\wh{x},\wh{y}]$ there exists a smooth map
  $[\wh{x},\wh{y}]\rightarrow \wh{y}_0([\wh{x},\wh{y}])$ such that
  $$
    [\wh{x},\wh{y}]=[\wh{x}_0,\wh{y}_0([\wh{x},\wh{y}])].
    $$
Next we compute, abbreviating $\wh{y}_0=\wh{y}_{0}([\wh{x},\wh{y}])$
  \begin{eqnarray}\label{ew3.13}
    &&c^y(r,\wt{w})= c^x(\wt{q}^x) +T\cdot R-\varphi(r)
    +\int_{S^1}C\circ\sigma^+_{\wh{x}_0}(R/2,t)\cdot dt\\
    &=&c^y(\wt{q}^y) +\int_{S^1}C\circ\sigma^+_{\wh{x}_0}(R/2,t)\cdot
    dt.\nonumber
    \end{eqnarray}
Further we compute
  \begin{eqnarray}\label{ew3.14}
    &&d_{\wh{x}_0}(r,\wt{w})=\oint_{S^1} \Phi\circ (Q+H)\circ
    \sigma^{+,a}_{\wh{x}_0}(R/2,t)\cdot dt.
    \end{eqnarray}
In view of Lemma \ref{LEMM3.38} it suffices to show that the maps
  \begin{eqnarray*}
    &(r,\wt{q}\dotplus\wt{h})\rightarrow \wh{y}_{0}([\wh{x},\wh{y}])&\\
    &(r,\wt{q}\dotplus\wt{h})\rightarrow c^x(r,\wt{w})&\\
    &(r,\wt{q}\dotplus\wt{h})\rightarrow c^y(r,\wt{w})&\\
    &(r,\wt{q}\dotplus\wt{h})\rightarrow d_{\wh{x}_0}(r,\wt{w}).&
    \end{eqnarray*}
  are sc-smooth. 
As a consequence  of our discussion so far this boils down to the
  sc-moothness of the following maps, where we use smoothness properties
  of $c^x(\wt{q}^x)$ and $c^y(\wt{q}^y)$.
  \begin{eqnarray}
    &(r,\wt{q}\dotplus\wt{h})\rightarrow
    \wh{y}_{0}([\wh{x},\wh{y}])&\label{xc1}\\
    &(r,\wt{q}\dotplus\wt{h})\rightarrow
    \int_{S^1}C\circ\sigma^+_{\wh{x}_0}(R/2,t)\cdot dt\label{xc2}&\\
    &(r,\wt{q}\dotplus\wt{h})\rightarrow
    d_{\wh{x}_0}(r,\wt{w})&\label{xc3}.
    \end{eqnarray}
The first map (\ref{xc1}) is ssc-smooth.  
The second map (\ref{xc2}) has, with $\wt{h}^x=(b^x,h^x)$ and
  $\wt{h}^y=(b^y,h^y)$ the form
  $$
    (r,\wt{q}\dotplus\wt{h})\rightarrow \int_{S^1} \frac{1}{2}\cdot
    [b^x(R/2,t) + b^y(R/2,t)]\cdot dt
    $$
  and is sc-smooth (not ssc-smooth!) by the definition of the
  ssc-structure on  $Z_{\mathcal D}({\mathbb R}\times {\mathbb
  R}^N,\boldsymbol{\bar{\gamma}})$, see Proposition \ref{PROPOSITION3.5}
  and before and Lemma \ref{LEM3.14}, for which the map
  $(r,\wt{q}\dotplus\wt{h})\rightarrow (\mathsf{a}(r,\wt{q}),\wt{h})$ is
  ssc-smooth, and the Fundamental Lemma in Subsection  \ref{FUND0}.
It remains to show that the third map (\ref{xc3}) is sc-smooth. 
We shall write the  map $(r,\wt{q}\dotplus\wt{h})\rightarrow
  d_{\wh{x}_0}(r,\wt{w})$ more explicitly. For this note that
  \begin{eqnarray}\label{EQR3.19}
    &&w\circ\sigma_{\wh{x}_0}(R/2,t)\\
    &&+ \frac{1}{2}\cdot
    [h^x\circ\sigma^+_{\wh{x}_0}(R/2,t)+h^y\circ\sigma^-_{\wh{y}_{0}([\wh{x},
    \wh{y}])}(-R/2,t)]\nonumber\\
    &=& \gamma_0(kt+m(\wt{q}))\nonumber\\
    && + \frac{1}{2}\cdot
    [h^x\circ\sigma^+_{\wh{x}_0}(R/2,t)+h^y\circ\sigma^-_{\wh{y}_{oo}}(-R/2,
    t+\vartheta([\wh{x},\wh{y}]))],
    \nonumber
    \end{eqnarray}
  where $q^+\circ\sigma^+_{\wh{x}_0}(s,t)=\gamma_0(m(\wt{q})+kt)$,
  $\wh{y}_{oo}$ is a fixed choice so that
  $$
    \sigma^-_{\wh{y}_{oo}}(s',t'+\vartheta([\wh{x},\wh{y}]))
    =\sigma^-_{\wh{y}_{0}([\wh{x},\wh{y}])}(s',t').
    $$
We note that $m(\wt{q})$ and $\vartheta([\wh{x},\wh{y}])$ depend
  ssc-smoothly on $(r,\wt{q}+\wt{h})$.
In view of the Fundamental Lemma the map
  \begin{eqnarray*}
    &{\mathbb B}\times H^{3,\delta}({\mathbb R}^+\times S^1,{\mathbb
    R}^N)\rightarrow H^{3,\delta}({\mathbb R}^+\times S^1,{\mathbb
    R}^N):&\\
    & (a,h\circ\sigma^+_{\wh{x}_0})\rightarrow 
    h\circ\sigma^+_{\wh{x}_0}(R/2+s,t)
    \end{eqnarray*}
  is sc-smooth. 
Similarly
  \begin{eqnarray*}
    &{\mathbb B}\times H^{3,\delta}({\mathbb R}^-\times S^1,{\mathbb R}^N
    )\rightarrow
    H^{3,\delta}({\mathbb R}^-\times S^1,{\mathbb R}^N):&\\
    &(a,h\circ\sigma^-_{\wh{y}_{oo}})\rightarrow 
    h\circ\sigma^-_{\wh{y}_{oo}}(-R/2+s,t-\theta).&
    \end{eqnarray*}
In the above definition, in the case $a=0$ the image is the zero-map.
We also note that the maps $H^{3,\delta}({\mathbb R}^\pm\times
  S^1,{\mathbb R}^N)\rightarrow C^0(S^1,{\mathbb R}^N)$
  $$
    h^\pm\rightarrow h^\pm(0,.),
    $$
  where $C^0$ is equipped with the standard sc-structure, i.e. level $m$
  corresponds to $C^m$,  are linear sc-operators.  
The map
  $$
    S^1\times C^0(S^1,{\mathbb R}^N)\rightarrow C^0(S^1,{\mathbb
    R}^N):(e,\sigma)\rightarrow [t\rightarrow \sigma(t+e)]
    $$
  is sc-smooth which is an easy exercise along the line of similar results
  in \cite{HWZ8.7}.
Revisiting (\ref{EQR3.19}) we deduce from the previous discussion the
  following facts.\\

\begin{FACT}
The following maps $\bm{\bar{\oplus}}^{-1}({\mathcal W})\rightarrow
  C^0(S^1,{\mathbb R}^N)$ are sc-smooth.
  \begin{itemize}
    \item[(1)] 
    $(r,\wt{q}\dotplus\wt{h})\rightarrow [t\rightarrow
    \gamma_0(t+m(\wt{q}))]$.  
    \item[(2)]
    $(r,\wt{q}\dotplus\wt{h})\rightarrow [t\rightarrow
    h^x\circ\sigma^+_{\wh{x}_0}(R/2,t)].$
    \item[(3)] 
    $(r,\wt{q}\dotplus\wt{h})\rightarrow
    h^y\circ\sigma^-_{\wh{y}_{oo}}([\wh{x},\wh{y}])(-R/2,t+
    \vartheta([\wh{x},\wh{y}]))]$.
    \end{itemize}
\end{FACT}
As a consequence the map
  $$
    \bm{\bar{\oplus}}^{-1}({\mathcal W})\rightarrow C^0(S^1,{\mathbb
    R}^N):(r,\wt{q}+\wt{h})\rightarrow [t\rightarrow
    w\circ\sigma_{\wh{x}_0}(R/2,t)]
    $$
is sc-smooth and takes image in $C^0(S^1,U)$. 
The map $C^0(S^1,U)\rightarrow C^0(S^1,S^1)$ defined by
  $$
    u\rightarrow \Phi\circ u
    $$
  is ssc-smooth.  
From this we obtain that 
  $$
    \bm{\bar{\oplus}}^{-1}({\mathcal W})\rightarrow
    C^0(S^1,S^1):(r,\wt{q}\dotplus\wt{h})\rightarrow [t\rightarrow \Phi\circ
    w\circ\sigma_{\wh{x}_0}(R/2,t)]
    $$
  is sc-smooth and by construction takes the values in the maps of degree
  $k$.
Taking the $\oint$-average is obviously sc-smooth. 
Hence we see that we can write the map
  $$
    (r,\wt{q}\dotplus \wt{h})\rightarrow d_{\wh{x}_0}(r,\wt{w})
    $$
  as a composition of sc-smooth and ssc-smooth maps and consequently by
  the chain rule this map is sc-smooth.
The proof of the proposition is complete.
\end{proof}%%%%%%%%%%%%%%%%%%%%%%%%%%%%%%%%%%%%%%%%%%%%%%%%            END PROOF
The result of this subsection can be summarized as follows.\\
\begin{tcolorbox}
$$
  \text{Sc-smoothness  of the  averaging map}\
  \mathsf{A}_{\Phi}:{\mathcal W}\rightarrow \mathsf{S}_{\bm{\gamma}}.
  $$
\end{tcolorbox}

%%%%%%%%%%%%%%%%%%%%%%%%%%%%%%%%%%%%%%%%%%%%%%%%%%%%%%%%%%%%%%%%%%%%%%%%%%%%%%%%
%%%%%%%%%%                         SUB-SECTION                         %%%%%%%%%
%%%%%%%%%%%%%%%%%%%%%%%%%%%%%%%%%%%%%%%%%%%%%%%%%%%%%%%%%%%%%%%%%%%%%%%%%%%%%%%%
%
\subsection{Comparing Averages}
The main goal of this subsection is to understand quantitatively the
  average.
Namely starting with $(r,q\dotplus\wt{h})$ we can compare $\wt{q}$ and
  $\wt{p}=\mathsf{A}_{\Phi}\circ\bm{\bar{\oplus}}(r,\wt{q}\dotplus\wt{h})$,
  where we recall that
  $\wt{q}=\mathsf{A}_{\Phi}\circ\bm{\bar{\oplus}}(r,\wt{q})$.
The main result in this subsection is Proposition \ref{COMPARE} and this
  result will be crucial to define the (second) map $K$ which partially
  inverts $\bm{\bar{\oplus}}$ in the (next) Subsection \ref{SSS3.2.4}.

We shall use some facts which were derived in the proof of the previous
  Proposition \ref{PROPP3.33}.
Abbreviate ${\mathcal U}=\bm{\bar{\oplus}}^{-1}({\mathcal W})$ and
  consider the sc-smooth map
  $$
    \mathsf{A}_{\Phi}\circ \bm{\bar{\oplus}}:{\mathcal U}\rightarrow
    \mathsf{S}_{\bm{\gamma}}, $$
  where $\mathsf{A}_{\Phi}$ is introduced in Definition \ref{ewr3.37}.  
A given element in ${\mathcal U}$
  $$
    (r,(\wt{u}^x,[\wh{x},\wh{y}],\wt{u}^y))
    $$
  can be written uniquely as
  $$
    (r,(\wt{u}^x,[\wh{x},\wh{y}],\wt{u}^y))=(r,\wt{q}\dotplus\wt{h}),
    $$
  where $\wt{q}=(\wt{q}^x,[\wh{x},\wh{y}],\wt{q}^y)\in
  \mathsf{S}_{\bm{\gamma}}$ and $\wt{h}=(\wt{h}^x,\wt{h}^y)$ has the usual
  decay properties. 
We denote by $a=\mathsf{a}(r,\wt{q})$ the associated gluing parameter,
  which we recall does not depend on $\wt{h}$, and abbreviate
  $\wt{p}=\mathsf{A}_{\Phi}\circ\bm{\bar{\oplus}}(r,\wt{q}\dotplus\wt{h})$.
We define the map
  $$
    \mathsf{D}:{\mathcal U}\rightarrow H^{3,\delta}({\mathcal D},{\mathbb
    R}\times {\mathbb R}^N)
    $$
  as follows
  \begin{eqnarray}\label{OPLIY}
    &\ \ \ \ \ \ \
    \mathsf{D}(r,\wt{q}\dotplus\wt{h})=\left[\begin{array}{cc}
    (\beta^x_{a,-2}\cdot (\wt{q}^x-\wt{p}^x),\beta^y_{a,-2}\cdot
    (\wt{q}^y-\wt{p}^y)) & r\neq 0\\
    (0,0)& r=0.
    \end{array}
    \right.&
    \end{eqnarray}
From the discussion in the previous subsection it is clear that the map
  $\mathsf{D}$ is sc-smooth when we restrict it to $\dot{\mathcal
  U}=\{(r,\wt{q}\dotplus\wt{h})\in {\mathcal U}\ |\ r\neq 0\}$.
However, the case near $r=0$ is subtle. 
If $(r,\wt{q}\dotplus\wt{h})$ converges to $(0,\wt{q}_0\dotplus\wt{h}_0)$
  one needs to show that the average $\wt{p}$ converges to $\wt{q}$ fast
  enough, so that cutting-off their difference (with  cut-off functions of
  increasing support), the resulting map is sc-smooth, see (\ref{OPLIY}).
Here is the precise statement.

%%%%%%%%%%%%%%%%%%%%%%%%%%%%%%%%%%%%%%%%%%%%%%%%%%%%%%%%%%%%%%%%%%%%%%%%%%%%%%%%
%%%%%%%%%%                          PROPOSITION                        %%%%%%%%%
%%%%%                                                                       %%%%
\begin{proposition}\label{COMPARE}
The map $\mathsf{D}:{\mathcal U}\rightarrow H^{3,\delta}({\mathcal
  D},{\mathbb R}\times {\mathbb R}^N)$ is sc-smooth.
\end{proposition}
%%%%%                                                                       %%%%
%%%%%%%%%%                                                             %%%%%%%%%
%%%%%%%%%%%%%%%%%%%%%%%%%%%%%%%%%%%%%%%%%%%%%%%%%%%%%%%%%%%%%%%%%%%%%%%%%%%%%%%%
%
The proof will follow from two lemmata.  
Abbreviate $\wt{k}=\mathsf{D}(r,\wt{q}\dotplus\wt{h})$ and 
  \begin{eqnarray}\label{ktildeX}
    \wt{k}=(\wt{k}^x,\wt{k}^y)\ \ \text{and}\ \  \wt{k}^x=(e^x,k^x),\
    \wt{k}^y=(e^y,k^y).
    \end{eqnarray}

%%%%%%%%%%%%%%%%%%%%%%%%%%%%%%%%%%%%%%%%%%%%%%%%%%%%%%%%%%%%%%%%%%%%%%%%%%%%%%%%
%%%%%%%%%%                            LEMMA                            %%%%%%%%%
%%%%%                                                                       %%%%
\begin{lemma}
  \hfill\\
The map ${\mathcal U}\rightarrow H^{3,\delta}({\mathcal D},{\mathbb
R}):(r,\wt{q}\dotplus\wt{h})\rightarrow (e^x,e^y) $ is sc-smooth.
\end{lemma}
%%%%%                                                                       %%%%
%%%%%%%%%%                                                             %%%%%%%%%
%%%%%%%%%%%%%%%%%%%%%%%%%%%%%%%%%%%%%%%%%%%%%%%%%%%%%%%%%%%%%%%%%%%%%%%%%%%%%%%%
%
\begin{proof}
It suffices to study the problem near an element with $r=0$.
From the definition of $\wt{p}$, using (\ref{ew3.12}) and (\ref{ew3.13})
  it follows that
  \begin{eqnarray}\label{ewr3.19}
    &&c^x(\wt{q})-c^x(\wt{p})\\
     &= &c^y(\wt{q})-c^y(\wt{p})\nonumber \\
    &=& -\frac{1}{2}\cdot \int_{S^1} [b^x\circ
    \sigma^+_{\wh{x}}(R/2,t)+b^y\circ\sigma^-_{\wh{y}}(-R/2,t)]\cdot
    dt,\nonumber
    \end{eqnarray}
  where $\wt{h}^x=(b^x,h^x)$ and $\wt{h}^y=(b^y,h^y)$ and
  $\wt{h}=(\wt{h}^x,\wt{h}^y)$.
We note that the choice of $(\wh{x},\wh{y})$ is indeed irrelevant.
Writing $\wt{k}=\mathsf{D}(r,\wt{q}\dotplus\wt{h})$ the first component
  $e^x$ of $\wt{k}^x$ is given by
  \begin{eqnarray}
    &&e^x\circ\sigma^+_{\wh{x}_0}(s,t)\\
    &=&-\frac{1}{2}\cdot \beta(s-R/2-2) \cdot \int_{S^1}[b^x\circ
    \sigma^+_{\wh{x}}(R/2,t)+b^y\circ\sigma^-_{\wh{y}}(-R/2,t)].\nonumber
    \end{eqnarray}
Here $b^x$ and $b^y$ are the first components of $\wt{h}^x$ and
  $\wt{h}^y$, respectively.  
There is a similar formula for $e^y$ and in view of the Fundamental Lemma,
  see Subsection \ref{FUND0},  the map
  $$
    {\mathcal U}\rightarrow H^{3,\delta}({\mathcal D},{\mathbb
    R}):(r,\wt{q}\dotplus\wt{h})\rightarrow (e^x,e^y)
    $$
  is sc-smooth. 
This completes the proof of the lemma.
\end{proof}%%%%%%%%%%%%%%%%%%%%%%%%%%%%%%%%%%%%%%%%%%%%%%%%            END PROOF

In order to  consider the  pair $(k^x,k^y)$, see (\ref{ktildeX}), we find
  appropriate formulae for them.  
Denote by $\gamma_0$ the representative satisfying
  $\Phi\circ\gamma_0(t)=t$.  
Then
 \begin{eqnarray}\label{ewr3.20}
   && k^x\circ\sigma^+_{\wh{x}_0}(s,t)\\
   &=&\beta(s-R/2-2)\cdot  (q^x-p^x)\circ
   \sigma_{\wh{x}_0}^+(s,t)\nonumber\\
   &=&\beta(s-R/2-2)\cdot
   \left[\gamma_0(kt-(k/2)+d_{\wh{x}_0}(\wt{q}^x))\right.\nonumber\\
   &&\ \ \ \ \ \ \ \ \ \
   -\left.\gamma_0(kt-(k/2)+d_{\wh{x}_0}(\wt{p}^x))\right].\nonumber
  \end{eqnarray}
Similarly
  \begin{eqnarray}\label{ewr3.20b} % This was multiply-defined
    && k^y\circ\sigma^-_{\wh{y}}(s',t')\\
    &=&\beta(|s'|-R/2-2)\cdot  (q^y-p^y)\circ
    \sigma_{\wh{y}}^-(s',t')\nonumber\\
    &=&\beta(|s'|-R/2-2)\cdot
    \left[\gamma_0(kt'-(k/2)+d_{\wh{x}_0}(\wt{q}^x))\right.\nonumber\\
    &&\ \ \ \ \ \ \ \ \ \
    -\left.\gamma_0(kt'-(k/2)+d_{\wh{x}_0}(\wt{p}^x))\right].\nonumber
    \end{eqnarray}
Here $d_{\wh{x}_0}(\wt{q}^x)=\oint_{S^1}\Phi\circ q^x\circ
  \sigma^+_{\wh{x}_0}(s,t)\cdot dt$,
  $(r,\wt{w})=\bm{\bar{\oplus}}(r,\wt{q}\dotplus\wt{h})$,
  $\wt{p}=\mathsf{A}_{\Phi}(r,\wt{w})$, and
  \begin{eqnarray*}
    d_{\wh{x}_0}(\wt{p}^x)
    &=&d_{\wh{x}_0}(r,\wt{w})\\
    &=&\oint_{S^1}\Phi\circ w\circ\sigma_{\wh{x}_0}(R/2,t)\cdot dt\nonumber\\
    &=&\oint_{S^1}\Phi\circ
    \left[\gamma_0\left(kt-(k/2)+d_{\wh{x}_0}(\wt{q}^x)\right)
    \phantom{\frac{1}{2}}\right.\\
    && \left. + \frac{1}{2}\cdot\left (h^x\circ \sigma^+_{\wh{x}_0}(R/2,t)
    +h^y\circ\sigma^-_{\wh{y}}(-R/2,t)\right)\right]\cdot dt.\nonumber
    \end{eqnarray*}
There is a similar formula for the $y$-part.
Consequently 
  \begin{eqnarray}\label{ewr3.21}
    &&  d_{\wh{x}_0}(\wt{p}^x)-d_{\wh{x}_0}(\wt{q}^x)\\
    &=&\oint_{S^1}\Phi\circ
    \left[\gamma_0\left(kt-(k/2)+d_{\wh{x}_0}(\wt{q}^x)\right)
    \phantom{\frac{1}{2}}\right.\nonumber\\
    && \left. + \frac{1}{2}\cdot\left (h^x\circ \sigma^+_{\wh{x}_0}(R/2,t)
    +h^y\circ\sigma^-_{\wh{y}}(-R/2,t)\right)\right]\cdot dt\nonumber\\
    &&-\oint_{S^1}\Phi\circ
    \gamma_0\left(kt-(k/2)+d_{\wh{x}_0}(\wt{q}^x)\right)\cdot
    dt.\nonumber
    \end{eqnarray}
Our aim is to show that the map
  $$
    {\mathcal U}\rightarrow H^{3,\delta}({\mathcal D},{\mathbb
    R}^N):(r,\wt{q}\dotplus\wt{h})\rightarrow (k^x,k^y)
    $$
  is sc-smooth.  
Once this is established the proof is complete. 
This study boils down to understanding the dependence of
  $d_{\wt{x}_0}(\wt{q})$ and $d_{\wh{x}_0}(\wt{p}^x)$ on
  $(r,\wt{q}\dotplus\wt{h}^y)$. 
It suffices to study $k^x$. 
We already know that we may assume that $r$ is small.  
Consider for two real numbers $d,d'$ and $t\in S^1$ the map
  $$
    (d,d',t)\rightarrow \gamma_0(t+d)-\gamma_0(t+d'),
    $$
  which we can rewrite as
  \begin{eqnarray*}
    \gamma_0(t+d)-\gamma_0(t+d')&=& (d-d')\cdot \left(\int_0^1
    \dot{\gamma}_0(t+\tau d+(1-\tau)d')d\tau\right)\\
    &=:&(d-d')\cdot  \mathsf{B}(t,d,d').
    \end{eqnarray*}
The map $\mathsf{B}: S^1\times {\mathbb R}\times {\mathbb R}\rightarrow
  {\mathbb R}^N$ is smooth.
Assume that $(r,\wt{q}\dotplus\wt{h})$ is near the element
  $(0,\wt{q}_0\dotplus\wt{h}_0)$.  
Since $d_{\wh{x}_0}(0,\wt{q}_0\dotplus\wt{h}_0)=d_{\wh{x}_0}(\wt{q}_0)$ it
  follows that $d_{\wh{x}_0}(\wt{p})$ as well as $d_{\wh{x}_0}(\wt{q})$
  are near $d_{\wh{x}_0}(\wt{q}_0)$. 
Since we already know that $\wt{q}\rightarrow d_{\wh{x}_0}(\wt{p}),\
  d_{\wh{x}_0}(\wt{q})\in S^1$  are sc-smooth we can take a local sc-smooth
  lift to ${\mathbb R}$ denoted by $\wt{d}_{\wh{x}_0}(\wt{p})$ and
  $\wt{d}_{\wh{x}_0}(\wt{q})$, respectively. Then, using $\mathsf{B}$ we
  can rewrite (\ref{ewr3.20}) as follows.
\begin{eqnarray}\label{ewr3.23}
  && k^x\circ\sigma^+_{\wh{x}_0}(s,t)\\
  &=&\beta(s-R/2-2)\cdot
  (\wt{d}_{\wh{x}_0}(\wt{q})-\wt{d}_{\wh{x}_0}(\wt{p})) \cdot \nonumber \\
  &&\ \ \ \cdot
  \mathsf{B}(kt-(k/2)+d_{\wh{x}_0}(\wt{q}),\wt{d}_{\wh{x}_0}(\wt{q}),
  \wt{d}_{\wh{x}_0}(\wt{p})).\nonumber
  \end{eqnarray}
First we consider the map
  \begin{eqnarray}\label{ewr3.26}
    &{\mathbb R}\times {\mathbb R}\times H^{m,\tau}({\mathbb R}^+\times
    S^1,{\mathbb R})\rightarrow H^{m,\tau}({\mathbb R}^+\times
    S^1,{\mathbb R}^N)&\\
    &(d,d',h)\rightarrow [(s,t)\rightarrow \mathsf{B}(kt-k/2+d,d,d')\cdot
    h(s,t)].&\nonumber
    \end{eqnarray}
We have the following result about classical differentiability which is
  well-known, \cite{El}, and a trivial consequence.

%%%%%%%%%%%%%%%%%%%%%%%%%%%%%%%%%%%%%%%%%%%%%%%%%%%%%%%%%%%%%%%%%%%%%%%%%%%%%%%%
%%%%%%%%%%                            LEMMA                            %%%%%%%%%
%%%%%                                                                       %%%%
\begin{lemma}\label{LEMMA3.43}
The map in (\ref{ewr3.26}) is classically $C^\infty$ for every choice of
  $m\geq 0$ and $\tau\geq 0$.
As a consequence this map (viewed in an sc-setting) ${\mathbb R}\times
  {\mathbb R}\times H^{3,\delta}({\mathbb R}^+\times S^1,{\mathbb
  R})\rightarrow H^{3,\delta}({\mathbb R}^+\times S^1,{\mathbb R}^N)$ is
  sc-smooth.
\end{lemma}
%%%%%                                                                       %%%%
%%%%%%%%%%                                                             %%%%%%%%%
%%%%%%%%%%%%%%%%%%%%%%%%%%%%%%%%%%%%%%%%%%%%%%%%%%%%%%%%%%%%%%%%%%%%%%%%%%%%%%%%
%
As a corollary, for $(r,\wt{q}\dotplus\wt{h})$ in a suitable neighborhood
  of $(0,\wt{q}_0\dotplus\wt{h}_0)$ contained in ${\mathcal U}$ and $v\in
  H^{3,\delta}({\mathbb R}^+\times S^1,{\mathbb R})$, the map
  \begin{eqnarray*}
    &{\mathcal O}(0,\wt{q}_0\dotplus\wt{h}_0)\times H^{3,\delta}({\mathbb
    R}^+\times S^1,{\mathbb R})\rightarrow
    H^{3,\delta}({\mathbb R}^+\times S^1,{\mathbb R}^N)&\\
    &((r,\wt{q}\dotplus\wt{h}),v)\rightarrow [(s,t)\rightarrow v(s,t)
    \cdot\mathsf{B}(kt-k/2+d_{\wh{x}_0}(\wt{q}),\wt{d}_{\wh{x}_0}(\wt{q})
    ,\wt{d}_{\wh{x}_0}(\wt{p}))]
    \end{eqnarray*}
  is sc-smooth.
In view of the discussion so far we only need to show the following lemma,
  since a similar argument will hold for $k^y$, to complete the proof of
  Proposition \ref{COMPARE}.

%%%%%%%%%%%%%%%%%%%%%%%%%%%%%%%%%%%%%%%%%%%%%%%%%%%%%%%%%%%%%%%%%%%%%%%%%%%%%%%%
%%%%%%%%%%                            LEMMA                            %%%%%%%%%
%%%%%                                                                       %%%%
\begin{lemma}
  \hfill\\
The map 
  \begin{eqnarray}\label{ewr3.24}
    &(r,\wt{q}\dotplus\wt{h})\rightarrow [(s,t)\rightarrow
    \beta(s-R/2-2)\cdot
    (\wt{d}_{\wh{x}_0}(\wt{q})-\wt{d}_{\wh{x}_0}(\wt{p}))]&
    \end{eqnarray}
  defined on a suitable open neighborhood of ${\mathcal
  O}(0,\wt{q}_0\dotplus\wt{h}_0)$ in\ \ ${\mathcal U}$ with image in
  $H^{3,\delta}({\mathbb R}^+\times S^1,{\mathbb R})$ is sc-smooth.
\end{lemma}
%%%%%                                                                       %%%%
%%%%%%%%%%                                                             %%%%%%%%%
%%%%%%%%%%%%%%%%%%%%%%%%%%%%%%%%%%%%%%%%%%%%%%%%%%%%%%%%%%%%%%%%%%%%%%%%%%%%%%%%
%
Then it follows from (\ref{ewr3.23}) and Lemma \ref{LEMMA3.43} that the
  map
  $$
    (r,\wt{q}\dotplus\wt{h})\rightarrow k^y
    $$
  is sc-smooth.
\begin{proof}
It suffices to verify the sc-smoothness of (\ref{ewr3.24}) near
  $(0,\wt{q}_0\dotplus\wt{h})$. 
In view of(\ref{ewr3.23}) we find that
  \begin{eqnarray}\label{ER3.33}
    && \wt{d}_{\wh{x}_0}(\wt{p}^x)-\wt{d}_{\wh{x}_0}(\wt{q}^x)=
    d_{\wh{x}_0}(\wt{p}^x)-d_{\wh{x}_0}(\wt{q}^x)\\
    &=&\oint_{S^1}\Phi\circ
    \left[\gamma_0\left(kt-(k/2)+d_{\wh{x}_0}(\wt{q}^x)\right)
    \phantom{\frac{1}{2}}\right.\nonumber\\
    && \left. + \frac{1}{2}\cdot\left (h^x\circ \sigma^+_{\wh{x}_0}(R/2,t)
    +h^y\circ\sigma^-_{\wh{y}}(-R/2,t)\right)\right]\cdot dt\nonumber\\
    &&-\oint_{S^1}\Phi\circ
    \gamma_0\left(kt-(k/2)+{d}_{\wh{x}_0}(\wt{q}^x)\right)\cdot
    dt.\nonumber\\
    &=&\oint_{S^1}\Phi\circ
    \left[\gamma_0\left(kt-(k/2)+\wt{d}_{\wh{x}_0}(\wt{q}^x)\right)
    \phantom{\frac{1}{2}}\right.\nonumber
    \nonumber\\
    &&\left. + \frac{1}{2}\cdot\left (h^x\circ \sigma^+_{\wh{x}_0}(R/2,t)
    +h^y\circ\sigma^-_{\wh{y}}(-R/2,t)\right)\right]\cdot dt\nonumber \\
    &&-\oint_{S^1}\Phi\circ
    \gamma_0\left(kt-(k/2)+\wt{d}_{\wh{x}_0}(\wt{q}^x)\right)\cdot
    dt.\nonumber
    \end{eqnarray}
We compute with $d\in {\mathbb R}$ or $d\in S^1$ and small $e\in {\mathbb R}^N$
  \begin{eqnarray*}
    && \Phi (\gamma_0(t+d)+e)- \Phi\circ\gamma_0(t+d)\\
    &=&  \Phi(\gamma_0(t+d)+e)-t-d\\
    &=& \left(\int_0^1 \frac{d}{d\tau}(\Phi(\gamma_0 (t+d) +\tau e)\cdot
    d\tau\right)e\ \text{mod}\ {\mathbb Z}\\
    &=&\mathsf{L}(t+d,e)\cdot e\ \ \text{mod}\ \ {\mathbb Z}.
    \end{eqnarray*}
There is no loss of generality assuming that $\mathsf{L}$ is a globally
  defined smooth map
  $$
    \mathsf{L}:S^1\times {\mathbb R}^N\rightarrow {\mathcal L}({\mathbb
    R}^N,{\mathbb R}).
    $$
We can rewrite (\ref{ER3.33}) as follows, where we abbreviate for $t\in
  S^1$
  $$
    e(t) = \frac{1}{2}\cdot\left (h^x\circ \sigma^+_{\wh{x}_0}(R/2,t)
    +h^y\circ\sigma^-_{\wh{y}}(-R/2,t)\right).  
    $$
  and obtain the following equality
  \begin{eqnarray}
    &&\wt{d}_{\wh{x}_0}(\wt{p}^x)-\wt{d}_{\wh{x}_0}(\wt{q}^x)\\
    &=&\int_{S^1} \mathsf{L}(kt-(k/2)+\wt{d}_{\wh{x}_0}(\wt{q}^x),e(t))
    (e(t))\cdot  dt.\nonumber
    \end{eqnarray}
Let us introduce the following abbreviations, where $(s,t)\in {\mathbb
  R}^+\times S^1$
  \begin{eqnarray}\label{I***0}
    &H^x(s,t)=\beta(s-R/2-2)\cdot h^x\circ\sigma_{\wh{x}_0}(s,t)&\\
    &H^y(s,t)=\beta(s-R/2-2)\cdot
    h^y\circ\sigma_{\wh{y}}(s-R,t),&\nonumber
    \end{eqnarray}
Fixing $\wh{y}_0$ we can write $H^y(s,t)=\beta(s-R/2-2)\cdot
  h^y(s-R,t-\vartheta(\wh{y}))$.
If $a=0$ we define $H^x$ and $H^y$ to be zero. 
By the Fundamental Lemma the maps
  $$
    {\mathbb B}\times H^{3,\delta}({\mathcal D},{\mathbb R}^N)\rightarrow
    H^{3,\delta}({\mathbb R}^+\times S^1,{\mathbb R}^N)
    $$
  defined by
  $$
    (a,h^x,h^y)\rightarrow H^x\ \ \text{and}\ \ \ (a,h^x,h^y)\rightarrow
    H^y
    $$
  are sc-smooth. 
Consequently 
  \begin{eqnarray}\label{I***1}
  (a,h^x,h^y)\rightarrow H:=\frac{1}{2}\cdot (H^x+H^y)\ \ \text{is
  sc-smooth}.
  \end{eqnarray}
The map
  \begin{eqnarray}
    &{\mathbb R}\times H^{3,\delta}({\mathbb R}^+\times S^1,{\mathbb
    R}^N)\rightarrow H^{3,\delta}({\mathbb R}^+\times S^1, {\mathbb R}^N)&\\
    & (d,H)\rightarrow [(s,t)\rightarrow
    E:=\mathsf{L}(kt-(k/2)+d,H(s,t))(H(s,t))]&\nonumber
    \end{eqnarray}
  is by classical theory ssc-smooth. 
Taking $d=d(\wt{q}^x)$ which is a ssc-smooth map of $\wt{q}\dotplus
  \wt{h}$ and using (\ref{I***0}) and (\ref{I***1}) we see that the map
\begin{eqnarray}\label{I***3}
  \wt{q}\dotplus \wt{h}\rightarrow (d(\wt{q}^x),H)\rightarrow E
  \end{eqnarray}
  is sc-smooth. 
We note that 
  $$
  \wt{d}_{\wh{x}_0}(\wt{p}^x)-\wt{d}_{\wh{x}_0}(\wt{q}^x)= \int_{S^1}
  E(R/2,t) dt.
  $$
As a consequence of the Fundamental Lemma the map
  $$
    {\mathbb B}\times H^{3,\delta}({\mathbb R}^+\times S^1,{\mathbb
    R}^N)\rightarrow H^{3,\delta}({\mathbb R}^+\times S^1,{\mathbb R}^N)
    $$
  which associates to $(a,E)$ the map $(s,t)\rightarrow
  \beta(s-R/2-2)\cdot \int_{S^1} E(R/2,t)\cdot dt$ is sc-smooth.
Composing the latter with (\ref{I***3}) we obtain the map defined in
  (\ref{ewr3.24}) and have proved, as the consequence
  of the chain rule, that this map is sc-smooth. 
The proof of the lemma is complete.
\end{proof}
We loosely summarize the findings in this subsection as follows.\\

\begin{tcolorbox}
As $(r,\wt{q}\dotplus\wt{h})$ converges to $(0,\wt{q})$ the element
  $\wt{p}=\mathsf{A}_{\Phi}\circ\bm{\bar{\oplus}}(r,\wt{q}\dotplus\wt{h})$
  converges rapidly to $\wt{q}$ due to the exponential decay properties of
  $\wt{h}$.  
\end{tcolorbox}

%%%%%%%%%%%%%%%%%%%%%%%%%%%%%%%%%%%%%%%%%%%%%%%%%%%%%%%%%%%%%%%%%%%%%%%%%%%%%%%%
%%%%%%%%%%                         SUB-SECTION                         %%%%%%%%%
%%%%%%%%%%%%%%%%%%%%%%%%%%%%%%%%%%%%%%%%%%%%%%%%%%%%%%%%%%%%%%%%%%%%%%%%%%%%%%%%
%
\subsection{The Coretraction \texorpdfstring{$K$}{K} }\label{SSS3.2.4}
The aim of this subsection is the following.\index{$K:{\mathcal
  W}\rightarrow {\mathcal V}$}
 
\begin{AIM}
With the set ${\mathcal W}$ defined as in the construction of
  $\mathsf{A}_{\Phi}$, we shall define a map $K:{\mathcal W}\rightarrow
  {\mathcal V}$, which preserves the $\bar{r}$-fibers
  such that $K\circ \bm{\bar{\oplus}}:\bm{\bar{\oplus}}^{-1}({\mathcal
  W})\rightarrow {\mathcal V}$ is sc-smooth
  and $\bm{\bar{\oplus}}\circ K=Id_{\mathcal W}$. 
\end{AIM}
\vspace{0.3cm}
\noindent We define $K$ as follows. 
If $r=0$ we put
  $$
    K(0,(\wt{u}^x,[\wh{x},\wh{y}],\wt{u}^y))=
    (0,(\wt{u}^x,[\wh{x},\wh{y}],\wt{u}^y)).
    $$
If $r\in (0,1)$ we define
  \begin{eqnarray}\label{EEE21}
    K(r,\wt{w})=(r,\mathsf{A}_\Phi(r,\wt{w})\dotplus(\wt{h}^x,\wt{h}^y)),
    \end{eqnarray}
  where, with
  $\mathsf{A}_\Phi(r,\wt{w})=(\wt{q}^x,[\wh{x},\wh{y}],\wt{q}^y)$
  \begin{eqnarray}\label{EEE22}
    \wt{h}^x(z) &=&\beta^x_{a,-2}(z)\cdot (\wt{w}(z,z') -\wt{q}^x(z))\\
    \wt{h}^y(z')&=&\beta^y_{a,-2}(z')\cdot ((-\varphi(r))\ast\wt{w}(z,z')
    -\wt{q}^y(z'))\nonumber
    \end{eqnarray}
Then 
  $$
    \begin{CD}
    {\mathcal W} @> K>> {\mathcal V}\\
    @V \bar{r} VV   @V \bar{r}VV\\
    [0,1) @= [0,1).
    \end{CD}
    $$
Recall that ${\mathcal U}$ is open in ${\mathcal V}$.  
The main result in this subsection is the following theorem.

%%%%%%%%%%%%%%%%%%%%%%%%%%%%%%%%%%%%%%%%%%%%%%%%%%%%%%%%%%%%%%%%%%%%%%%%%%%%%%%%
%%%%%%%%%%                           THEOREM                           %%%%%%%%%
%%%%%                                                                       %%%%
\begin{theorem}\label{CORRX3.27}
The map $K:{\mathcal W}\rightarrow {\mathcal V}$ preserves the
  $\bar{r}$-fibers and the following two properties.
\begin{itemize}
  \item[(1)] 
  $\boldsymbol{\bar{\oplus}}\circ K=Id_{\mathcal W}$.
  \item[(2)] 
  $K\circ \boldsymbol{\bar{\oplus}}: {\mathcal U}\rightarrow {\mathcal V}$
  is sc-smooth.
  \end{itemize}
\end{theorem}
%%%%%                                                                       %%%%
%%%%%%%%%%                                                             %%%%%%%%%
%%%%%%%%%%%%%%%%%%%%%%%%%%%%%%%%%%%%%%%%%%%%%%%%%%%%%%%%%%%%%%%%%%%%%%%%%%%%%%%%
%
\begin{proof}
The proof is a consequence of two lemmata.

%%%%%%%%%%%%%%%%%%%%%%%%%%%%%%%%%%%%%%%%%%%%%%%%%%%%%%%%%%%%%%%%%%%%%%%%%%%%%%%%
%%%%%%%%%%                            LEMMA                            %%%%%%%%%
%%%%%                                                                       %%%%
\begin{lemma}\label{LEMM3.26}
The map $K:{\mathcal W}\rightarrow {\mathcal V}$ satisfies
  $\boldsymbol{\bar{\oplus}}\circ K=Id_{\mathcal W}$.
\end{lemma}
%%%%%                                                                       %%%%
%%%%%%%%%%                                                             %%%%%%%%%
%%%%%%%%%%%%%%%%%%%%%%%%%%%%%%%%%%%%%%%%%%%%%%%%%%%%%%%%%%%%%%%%%%%%%%%%%%%%%%%%
%
\begin{proof}
Clearly $\boldsymbol{\bar{\oplus}}\circ
  K(0,(\wt{u}^x,[\wh{x},\wh{y}],\wt{u}^y))=(0,(\wt{u}^x,
  [\wh{x},\wh{y}],\wt{u}^y))$ and we may assume that the
  given element is $(r,\wt{w})$ with $r\in (0,1)$. 
We compute easily that also in this case, with $a=|a|\cdot[\wh{x},\wh{y}]$
  determined by $T\cdot\varphi(|a|)=\varphi(r) +
  c^y(\wt{q}^x)-c^x(\wt{q}^y)$ the following holds.
\begin{eqnarray*}
  &&\boldsymbol{\bar{\oplus}}\circ K(r,\wt{w})\\
  &=&
  \boldsymbol{\bar{\oplus}}(r,\mathsf{A}_\Phi(r,\wt{w})
  \dotplus(\wt{h}^x,\wt{h}^y))\\
  &=& (r,\bar{\oplus}_r((\wt{q}^x,[\wh{x},\wh{y}],\wt{q}^y)\dotplus
  (\wt{h}^x,\wt{h}^y)))\\
  &=&(r,{\oplus}_a(\wt{q}^x+\wt{h}^x,\varphi(r)\ast (\wt{q}^y+\wt{h}^y)))\\
  &=& (r,\beta^x(a,z)\wt{w} + \beta^y(a,z') \wt{w} ) =(r,\wt{w}).
  \end{eqnarray*}
The key point in the calculation being that the domain gluing parameter
  computed from the data is the original $a$.
\end{proof}%%%%%%%%%%%%%%%%%%%%%%%%%%%%%%%%%%%%%%%%%%%%%%%%            END PROOF
Abbreviate ${\mathcal U}=\bm{\bar{\oplus}}^{-1}({\mathcal W})$.

%%%%%%%%%%%%%%%%%%%%%%%%%%%%%%%%%%%%%%%%%%%%%%%%%%%%%%%%%%%%%%%%%%%%%%%%%%%%%%%%
%%%%%%%%%%                            LEMMA                            %%%%%%%%%
%%%%%                                                                       %%%%
\begin{lemma}\label{K-lemmar}
$K\circ \boldsymbol{\bar{\oplus}}: {\mathcal U}\rightarrow {\mathcal V}$
  is sc-smooth.
\end{lemma}
%%%%%                                                                       %%%%
%%%%%%%%%%                                                             %%%%%%%%%
%%%%%%%%%%%%%%%%%%%%%%%%%%%%%%%%%%%%%%%%%%%%%%%%%%%%%%%%%%%%%%%%%%%%%%%%%%%%%%%%
%
\begin{proof}
We need first to derive a formula for our composition.
We have that $K\circ
  \boldsymbol{\bar{\oplus}}(0,(\wt{u}^x,[\wh{x},\wh{y}],\wt{u}^y))=(0,
  (\wt{u}^x,[\wh{x},\wh{y}],\wt{u}^y))$
  and we may therefore assume that $r\in (0,1)$.
Hence we  start with $(r,(\wt{q}^x+\wt{h}^x,[\wh{x},\wh{y}],
  \wt{q}^y+\wt{h}^y))\in {\mathcal U}$, where the associated constants are
  $c^x(\wt{q}),c^y(\wt{q})$. 
Then we have that
  \begin{eqnarray*}
    &&K\circ
    \boldsymbol{\bar{\oplus}}(r,(\wt{q}\dotplus(\wt{h}^x,\wt{h}^y)))\\
    &=&K(r, \bar{\oplus}(r,\wt{q}) + \oplus_a(\wt{h}^x,\wt{h}^y))\\
    &=:&(r, \mathsf{A}_{\Phi}( \bar{\oplus}(r,\wt{q}) +
    \oplus_a(\wt{h}^x,\wt{h}^y)) + (\wt{k}^x,\wh{k}^y))\\
    &=& (r,\wt{p} \dotplus (\wt{k}^x,\wt{k}^y)),
    \end{eqnarray*}
  where $\wt{p}=\mathsf{A}_{\Phi}( \bar{\oplus}(r,\wt{q}) +
  \oplus_a(\wt{h}^x,\wt{h}^y))=\mathsf{A}_{\Phi}\circ\boldsymbol{\bar{\oplus}}
  (r,\wt{q}\dotplus(\wt{h}^x,\wt{h}^y)) $ depends sc-smoothly on the input
  data.
The pair $(\wt{k}^x,\wt{k}^y)\in H^{3,\delta}({\mathcal D},{\mathbb
  R}\times {\mathbb R}^N)$ is given by
  \begin{eqnarray}
    \wt{k}^x &=& \beta^x_{a,-2}(\wt{q}^x-\wt{p}^x)  + \beta^x_{a,-2}\cdot
    \oplus_a(\wt{h}^x,\wt{h}^y)\\
    \wt{k}^y &=& \beta^y_{a,-2}(\wt{q}^y-\wt{p}^y)  + \beta^y_{a,-2}\cdot
    \oplus_a(\wt{h}^x,\wt{h}^y).\nonumber
    \end{eqnarray}
We need to show that the map  $(r,\wt{q}\dotplus\wt{h})\rightarrow
  (\wt{k}^x,\wt{k}^y)$ is sc-smooth.
The maps $(r,\wt{q}\dotplus\wt{h})\rightarrow \beta^x_{a,-2}\cdot
  \oplus_a(\wt{h}^x,\wt{h}^y)$ and $(r,\wt{q}\dotplus\wt{h})\rightarrow
  \beta^y_{a,-2}\cdot  \oplus_a(\wt{h}^x,\wt{h}^y)$ are sc-smooth by the
  Fundamental Lemma. 
The map
  $$
    (r,(\wt{q}^x+\wt{h}^x,[\wh{x},\wh{y}],\wt{q}^y+\wt{h}^y))\rightarrow
    (\beta^x_{a,-2}(\wt{q}^x-\wt{p}^x),\beta^y_{a,-2}(\wt{q}^y-\wt{p}^y)
    $$
  are sc-smooth as proved in Proposition \ref {COMPARE}.  
This completes the proof of the lemma.
\end{proof}%%%%%%%%%%%%%%%%%%%%%%%%%%%%%%%%%%%%%%%%%%%%%%%%            END PROOF

\end{proof}%%%%%%%%%%%%%%%%%%%%%%%%%%%%%%%%%%%%%%%%%%%%%%%%            END PROOF

At this point we squared away all the technical work we need to establish
  the main results, which will essentially be a consequences of Theorem
  \ref{CORRX3.26} and Theorem \ref{CORRX3.27}.

%%%%%%%%%%%%%%%%%%%%%%%%%%%%%%%%%%%%%%%%%%%%%%%%%%%%%%%%%%%%%%%%%%%%%%%%%%%%%%%%
%%%%%%%%%%                         SUB-SECTION                         %%%%%%%%%
%%%%%%%%%%%%%%%%%%%%%%%%%%%%%%%%%%%%%%%%%%%%%%%%%%%%%%%%%%%%%%%%%%%%%%%%%%%%%%%%
%
\subsection{Proofs of the Basic Results}
We first summarize the facts which already have been established.

\begin{FACT}
We have an open subset ${\mathcal V}$ in $ \mathfrak{Z}$ and a surjective
  map $\boldsymbol{\bar{\oplus}}:{\mathcal V}\rightarrow
  Y^{3,\delta_0}_{{\mathcal D},\varphi}({\mathbb R}\times {\mathbb
  R}^N,\boldsymbol{\gamma})$.  
This map fits into the commutative diagram
  $$
    \begin{CD}
    {\mathcal V} @>\boldsymbol{\bar{\oplus}}>>  Y^{3,\delta_0}_{{\mathcal
    D},\varphi}=\partial\mathfrak{Z}\coprod \mathfrak{X}\\
    @V pr_1 VV @V \bar{r} VV\\
    [0,1)  @= [0,1)
    \end{CD}
    $$
We constructed a co-retraction $H:\mathfrak{X}\rightarrow {\mathcal V}$.
Defining $\dot{\mathcal V}=\bm{\bar{\oplus}}^{-1}(\mathfrak{X})$ we note
  that this is the open subset of ${\mathcal V}$ consisting of elements
  with $r\neq 0$.
The map $H$ has the following properties. 
\begin{itemize}
  \item[(1)] 
  $\boldsymbol{\bar{\oplus}}\circ H=Id_{\mathfrak{X}}$ 
  \item[(2)] 
  $H\circ \boldsymbol{\bar{\oplus}}:\dot{\mathcal V}\rightarrow {\mathcal
  V}$ is sc-smooth.
  \end{itemize}
Defining $O_{\dot{\mathcal V}}:=\rho(\dot{\mathcal V})$, with $\rho=H\circ
  \bm{\bar{\oplus}}$, the bijection
  $\boldsymbol{\bar{\oplus}}:O_{\dot{\mathcal V}}\rightarrow \mathfrak{X}$
  defines a M-polyfold structure on $\mathfrak{X}$ with inverse $H$. 
Indeed, as was shown, this  is the original M-polyfold structure and the
  induced topology is the original one as well, see Theorem \ref{CORRX3.26}.
Further we have constructed a  subset ${\mathcal W}$ of
  $Y^{3,\delta_0}_{{\mathcal D},\varphi}$ containing $\partial\mathfrak{Z}$
  so that ${\mathcal U}=\boldsymbol{\bar{\oplus}}^{-1}({\mathcal W})$ is
  an open subset of ${\mathcal V}$ containing $\partial\mathfrak{Z}\subset
  {\mathcal V}$ and a co-retraction $K:{\mathcal W}\rightarrow {\mathcal
  U}$ such that $\boldsymbol{\bar{\oplus}}\circ K =Id_{\mathcal W}$ and
  $K\circ \boldsymbol{\bar{\oplus}}:{\mathcal U}\rightarrow {\mathcal V}$
  is sc-smooth.  
Moreover, we have the commutative diagram
  $$
    \begin{CD}
    {\mathcal U} @> \bm{\bar{\oplus}} >> {\mathcal W}\\
    @V pr_1 VV  @V \bar{r} VV\\
    [0,1) @=   [0,1)
    \end{CD}
    $$
Define $\tau:=K\circ\bm{\bar{\oplus}}:{\mathcal U}\rightarrow {\mathcal
  V}$ and let $O_{\mathcal U}=\tau({\mathcal U})$. 
Then $\bm{\bar{\oplus}}:O_{\mathcal U}\rightarrow {\mathcal W}$ defines a
  M-polyfold structure on  ${\mathcal W}$ making it an sc-diffeomorphism.
By general theory the transition maps  associated to
  $\bm{\bar{\oplus}}:O_{\dot{\mathcal V}}\rightarrow \mathfrak{X}$ and
  $\bm{\bar{\oplus}}:O_{\mathcal U}\rightarrow {\mathcal W}$
  are sc-smooth.  
\end{FACT}

We have derived several properties of these constructions and are in the
  position to prove Theorem \ref{RRR} and Theorem \ref{RRRR}. 
Denote the quotient topology associated to $\bm{\bar{\oplus}}$ by
  ${\mathcal T}$.

%%%%%%%%%%%%%%%%%%%%%%%%%%%%%%%%%%%%%%%%%%%%%%%%%%%%%%%%%%%%%%%%%%%%%%%%%%%%%%%%
%%%%%%%%%%                         SUB-SECTION                         %%%%%%%%%
%%%%%%%%%%%%%%%%%%%%%%%%%%%%%%%%%%%%%%%%%%%%%%%%%%%%%%%%%%%%%%%%%%%%%%%%%%%%%%%%
%
\subsubsection{Proof of Theorem \ref{RRR}}\label{SSEC10.6.1}
We recall the theorem. 

%%%%%%%%%%%%%%%%%%%%%%%%%%%%%%%%%%%%%%%%%%%%%%%%%%%%%%%%%%%%%%%%%%%%%%%%%%%%%%%%
%%%%%%%%%%                           THEOREM                           %%%%%%%%%
%%%%%                                                                       %%%%
\begin{theorem}[\ref{RRR}]
  \hfill\\
\jwf{this is supposed to be a restatement}
Let ${\mathcal D}$ be an ordered disk pair and $\boldsymbol{\bar{\gamma}}$
  a weighted periodic orbit in ${\mathbb R}^N$. 
The map $\boldsymbol{\bar{\oplus}}:{\mathcal V}\rightarrow
  Y^{3,\delta_0}_{{\mathcal D},\varphi}({\mathbb R}\times {\mathbb
  R}^N,\boldsymbol{\gamma})$ is a M-polyfold construction by the
  $\oplus$-method fitting into the commutative diagram
  $$
    \begin{CD}
    {\mathcal V} @> \bm{\bar{\oplus}}>> Y^{3,\delta_0}_{{\mathcal
    D},\varphi}({\mathbb R}\times {\mathbb R}^N,\boldsymbol{\gamma})\\
    @V pr_1 VV    @V \bar{r} VV\\
    [0,1)\ \ \ \  @= [0,1),
    \end{CD}
    $$
  where the vertical arrows are the obvious extractions of the
  $r$-parameter.
\end{theorem}
%%%%%                                                                       %%%%
%%%%%%%%%%                                                             %%%%%%%%%
%%%%%%%%%%%%%%%%%%%%%%%%%%%%%%%%%%%%%%%%%%%%%%%%%%%%%%%%%%%%%%%%%%%%%%%%%%%%%%%%
%
\begin{proof}
In view of our work so far we only need, in order to prove Theorem
  \ref{RRR}, to show that the quotient topology ${\mathcal T}$ is
  metrizable.
In view of the established facts we obtain the following lemma.

%%%%%%%%%%%%%%%%%%%%%%%%%%%%%%%%%%%%%%%%%%%%%%%%%%%%%%%%%%%%%%%%%%%%%%%%%%%%%%%%
%%%%%%%%%%                            LEMMA                            %%%%%%%%%
%%%%%                                                                       %%%%
\begin{lemma}\label{LEMMN3.23}
The sc-diffeomorphisms $\boldsymbol{\bar{\oplus}}:O_{\mathcal
  U}\rightarrow {\mathcal W}$ and $
  \boldsymbol{\bar{\oplus}}:O_{\dot{\mathcal V}}\rightarrow \mathfrak{X}$
  induce the same M-polyfold structure on the common open subset $
  {\mathcal W}\cap \mathfrak{X}$ of ${\mathcal W}$ and $\mathfrak{X}$.
Consequently the metrizable topologies ${\mathcal T}_{\mathfrak{X}}$ on
  ${\mathfrak{X}}$ and ${\mathcal T}_{\mathcal W}$  on ${\mathcal W}$
  coincide on ${\mathcal W}\cap \mathfrak{X}$.
As a consequence it also holds that the degeneracy index of a point
  $(r,\wt{u})$ with $r\in (0,1)$ is zero.
\end{lemma}
%%%%%                                                                       %%%%
%%%%%%%%%%                                                             %%%%%%%%%
%%%%%%%%%%%%%%%%%%%%%%%%%%%%%%%%%%%%%%%%%%%%%%%%%%%%%%%%%%%%%%%%%%%%%%%%%%%%%%%%
%
As a consequence of Lemma \ref{LEMMN3.23} the following definition makes
  sense.

%%%%%%%%%%%%%%%%%%%%%%%%%%%%%%%%%%%%%%%%%%%%%%%%%%%%%%%%%%%%%%%%%%%%%%%%%%%%%%%%
%%%%%%%%%%                          DEFINITION                         %%%%%%%%%
%%%%%                                                                       %%%%
\begin{definition}
The uniquely determined topology on $Y^{3,\delta_0}_{{\mathcal
  D},\varphi}$ inducing on the two sets $\mathfrak{X}$ and ${\mathcal W}$
  the existing metrizable  topologies is denoted by ${\mathcal T}'$.
\end{definition}
%%%%%                                                                       %%%%
%%%%%%%%%%                                                             %%%%%%%%%
%%%%%%%%%%%%%%%%%%%%%%%%%%%%%%%%%%%%%%%%%%%%%%%%%%%%%%%%%%%%%%%%%%%%%%%%%%%%%%%%
%
The topology ${\mathcal T}'$ is the finest topology on
  $Y^{3,\delta_0}_{{\mathcal D},\varphi}$ so that the obvious map
  ${\mathcal W}\coprod \mathfrak{X}\rightarrow Y^{3,\delta_0}_{{\mathcal
  D},\varphi}$ is continuous.
We note that at this point it is not clear(!) that ${\mathcal T}'$ is
  metrizable.
We recall that $\mathfrak{X}$, ${\mathcal W}$, and $\mathfrak{X}\cap
  {\mathcal W}$ belong to ${\mathcal T}'$.
Our strategy consists of the following two steps:
  \begin{itemize}
    \item[(a)] 
    Show that ${\mathcal T}'$ is metrizable.
    \item[(b)] 
    Show that ${\mathcal T}={\mathcal T}'$.
    \end{itemize}
From now on we equip $Y^{3,\delta_0}_{{\mathcal D},\varphi}$ with the
  topology ${\mathcal T}'$.
Equip ${\mathbb S}_x$ and ${\mathbb S}_y$ with metrics, and on ${\mathbb
  R}$ we take the usual metric.  
On the $H^{3,\delta_0}$-type spaces we take the usual norms.  
Recall the sc-smooth retraction $\tau:{\mathcal U}\rightarrow {\mathcal
  U}$ and its image $O_{\mathcal U}=\tau({\mathcal U})$.
Given 
  \begin{eqnarray}\label{EQNN20}
    \mathfrak{y}_0=(0,(\wt{u}_0^x,[\wh{x}_0,\wh{y}_0],\wt{u}^y_0))\in
    Y^{3,\delta_0}_{{\mathcal D},\varphi}
    \end{eqnarray}
  we denote by $B_{\varepsilon}(\mathfrak{y}_0)$ the collection of all
  $\mathfrak{y}=(r,(\wt{q}^x,[\wh{x},\wh{y}],\wt{q}^y)\dotplus(\wt{h}^x,
  \wt{h}^y))$ such that
  \begin{itemize}
    \item[$\bullet$] $0\leq r<\varepsilon$, $|c^x_0-c^x|<\varepsilon$,
    $|c^y_0-c^y|<\varepsilon$,
    $d(\wh{x},\wh{x}_0)<\varepsilon$, $d(\wh{y},\wh{y}_0)<\varepsilon$.
    \item[$\bullet$]
    $\norm{(\wt{h}^x-\wt{h}^x_0,\wt{h}^y-\wt{h}^y_0)}_{H^{3,\delta_0}}
    <\varepsilon.$
    \end{itemize}
Denote by $\Gamma_\varepsilon(\mathfrak{y}_0)$ the open subset of
  ${\mathcal V}$ defined by
  $$
    \Gamma_\varepsilon(\mathfrak{y}_0)=\{\mathfrak{y}\in {\mathcal V}\ |\
    \mathfrak{y}\in  B_\varepsilon(\mathfrak{y}_0),\
    \tau(\mathfrak{y})\in B_\varepsilon(\mathfrak{y}_0)\}
    $$
We note that $\mathfrak{y}_0\in \Gamma_\varepsilon(\mathfrak{y}_0)$ and
  that $\tau:\Gamma_\varepsilon(\mathfrak{y}_0)\rightarrow
  \Gamma_\varepsilon(\mathfrak{y}_0)$.  
Take an open neighborhood $U'$ of $\gamma(S^1)$ such that 
  \begin{eqnarray}\label{OOO41}
    \cl_{{\mathbb R}^N}(U')\subset U.
    \end{eqnarray}

%%%%%%%%%%%%%%%%%%%%%%%%%%%%%%%%%%%%%%%%%%%%%%%%%%%%%%%%%%%%%%%%%%%%%%%%%%%%%%%%
%%%%%%%%%%                            LEMMA                            %%%%%%%%%
%%%%%                                                                       %%%%
\begin{lemma}\label{LEMR3.24}
Given $\mathfrak{y}_0=(0,(\wt{u}_0^x,[\wh{x}_0,\wh{y}_0],\wt{u}^y_0))\in
  Y^{3,\delta_0}_{{\mathcal D},\varphi}$ there exists $\varepsilon>0$,
  which can be picked arbitrarily small, such that for a suitable
  $\varepsilon_1\in (0,\varepsilon)$
  $$
    \cl_{(Y^{3,\delta_0}_{{\mathcal D},\varphi},{\mathcal
    T}')}(\boldsymbol{\bar{\oplus}}(\Gamma_{\varepsilon_1}(\mathfrak{y}_0)))
    \subset \boldsymbol{\bar{\oplus}}(\Gamma_{\varepsilon}(\mathfrak{y}_0)).
    $$
\end{lemma}
%%%%%                                                                       %%%%
%%%%%%%%%%                                                             %%%%%%%%%
%%%%%%%%%%%%%%%%%%%%%%%%%%%%%%%%%%%%%%%%%%%%%%%%%%%%%%%%%%%%%%%%%%%%%%%%%%%%%%%%
%
\begin{proof} 
Given $\mathfrak{y}_0$ as in (\ref{EQNN20}) fix $\varepsilon>0$ such that
  $d(\wh{x},\wh{x_0})\leq \varepsilon$ implies if $\wh{x}=e^{2\pi
  i\theta}\cdot \wh{x}_0$ that $|\theta|<1/(2k)$, and similarly for
  $\wh{y}$.  
Moreover, we assume that $\varepsilon>0$ has been picked such that  for
  $\mathfrak{y}\in \Gamma_{\varepsilon}(\mathfrak{y}_0)$ with
  $\bar{r}(\mathfrak{y})\in (0,1)$ it holds that
  $\boldsymbol{\bar{\oplus}}(\mathfrak{y})(M^3_a)\subset {\mathbb R}\times
  U'$.
We already know that $\boldsymbol{\bar{\oplus}}(\mathfrak{y})\in {\mathcal
  W}$.
It holds with $O_\varepsilon= \tau(\Gamma_{\varepsilon}(\mathfrak{y}_0))$
  that
\begin{itemize}
  \item[(i)] 
  $\boldsymbol{\bar{\oplus}}(\Gamma_{\varepsilon}(\mathfrak{y}_0))
  =\boldsymbol{\bar{\oplus}}(O_{\varepsilon})$ is open in
  $(Y^{3,\delta_0}_{{\mathcal D},\varphi},{\mathcal T}')$.
  \item[(ii)] 
  $\boldsymbol{\bar{\oplus}}:O_\varepsilon\rightarrow\boldsymbol
  {\bar{\oplus}} (\Gamma_{\varepsilon}(\mathfrak{y}_0))$ is an
  sc-diffeomorphism.
  \end{itemize}
Define for $\varepsilon_1\in (0,\varepsilon)$ similarly
  $O_{\varepsilon_1}$ and it holds $\mathfrak{y}_0\in O_\varepsilon\cap
  O_{\varepsilon_1}$.
Let $(r,\wt{w})\in Y^{3,\delta_0}_{{\mathcal  D},\varphi}$
  belong to $\cl_{(Y^{3,\delta_0}_{{\mathcal D},\varphi},{\mathcal
  T}')}(\boldsymbol{\bar{\oplus}}(\Gamma_{\varepsilon_1}(\mathfrak{y}_0)))$.
We find $(\mathfrak{y}_i)\subset O_{\varepsilon_1}$ with
  $$
    \boldsymbol{\bar{\oplus}}(\mathfrak{y}_i)\rightarrow (r,\wt{w}),
    $$
  and this element has to belong to ${\mathcal W}$ using (\ref{OOO41}).  
Hence we see that $\mathfrak{y}_i=K\circ
  \boldsymbol{\bar{\oplus}}(\mathfrak{y}_i) \rightarrow K(r,\wt{w})$
  implying that the latter belongs $O_\varepsilon$.
Trivially $\boldsymbol{\bar{\oplus}}(K(r,\wt{w}))=(r,\wt{w})$.
This proves the desired result.
\end{proof}%%%%%%%%%%%%%%%%%%%%%%%%%%%%%%%%%%%%%%%%%%%%%%%%            END PROOF
With the help of Lemma \ref{LEMR3.24} we can prove that ${\mathcal T}'$ is
  metrizable.

%%%%%%%%%%%%%%%%%%%%%%%%%%%%%%%%%%%%%%%%%%%%%%%%%%%%%%%%%%%%%%%%%%%%%%%%%%%%%%%%
%%%%%%%%%%                          PROPOSITION                        %%%%%%%%%
%%%%%                                                                       %%%%
\begin{proposition}
The topology ${\mathcal T}'$ on $ Y^{3,\delta_0}_{{\mathcal
  D},\varphi}({\mathbb R}\times {\mathbb R}^N,\boldsymbol{\gamma})$ is
  metrizable.  
Consequently, as a corollary to Lemma \ref{LEMMN3.23},  the metrizable
  topological space $(Y^{3,\delta_0}_{{\mathcal D},\varphi}({\mathbb
  R}\times {\mathbb R}^N,\boldsymbol{\gamma}),{\mathcal T}')$ has a unique
  M-polyfold structure so that the maps
  $\boldsymbol{\bar{\oplus}}:O_{\mathcal U}\rightarrow {\mathcal W}$ and $
  \boldsymbol{\bar{\oplus}}:O_{\dot{\mathcal V}}\rightarrow \mathfrak{X}$
  are sc-diffeomorphisms.
Recall that $Y^{3,\delta_0}_{{\mathcal D},\varphi}({\mathbb R}\times
  {\mathbb R}^N,\boldsymbol{\gamma})={\mathcal W}\cup \mathfrak{X}$.
\end{proposition}
%%%%%                                                                       %%%%
%%%%%%%%%%                                                             %%%%%%%%%
%%%%%%%%%%%%%%%%%%%%%%%%%%%%%%%%%%%%%%%%%%%%%%%%%%%%%%%%%%%%%%%%%%%%%%%%%%%%%%%%
%
\begin{proof}
We first note that ${\mathcal T}'$ is second countable. 
Indeed ${\mathcal U}$ and $\dot{\mathcal V}$ and therefore $O_{\mathcal
  U}$ and $O_{\dot{\mathcal V}}$  are second countable so that
  $\mathfrak{X}$ and ${\mathcal W}$ are second countable, implying that
  ${\mathcal T}'$ is second countable.

Next we show that ${\mathcal T}'$ is Hausdorff.  
Let $(r_1,\wt{w}_1), (r_2,\wt{w}_2) \in Y^{3,\delta_0}_{{\mathcal
  D},\varphi}({\mathbb R}\times {\mathbb R}^N,\boldsymbol{\gamma})$ be two
  different points.
If both lie in $\mathfrak{X}$ or both lie in ${\mathcal W}$ they can be
  separated by open sets.
Hence we may assume that $(r_1,\wt{w}_1)\in \mathfrak{X}\setminus{\mathcal
  W}$ and $(r_2,\wt{w}_2)\in {\mathcal W}\setminus\mathfrak{X}$. 
This implies that $(r_1,\wt{w}_1)$ satisfies  $r_1\in (0,1)$ and
  $(r_2,\wt{w}_2)$ has the form $(0,(\wt{u}^x,[\wh{x},\wh{y}],\wt{u}^y))$.
At this point it suffices to note that the map
  $\bar{r}:(Y^{3,\delta_0}_{{\mathcal D},\varphi}({\mathbb R}\times {\mathbb
  R}^N,\boldsymbol{\gamma}),{\mathcal T}')\rightarrow [0,1)$ extracting $r$
  is continuous.

Next we show that ${\mathcal T}'$ is regular. For this take a closed
  subset $A$ and a point $y_0:=(r_0,\wt{w}_0)$ in $Y^{3,\delta_0}_{{\mathcal
  D},\varphi}({\mathbb R}\times {\mathbb R}^N,\boldsymbol{\gamma})$ with
  $y_0\not\in A$.  
Assume first that $y_0\in \mathfrak{X}$.
Let $U_1$ be the  set of all $y\in Y^{3,\delta_0}_{{\mathcal
  D},\varphi}({\mathbb R}\times {\mathbb R}^N,\boldsymbol{\gamma})$ with
  $\bar{r}(y)< \frac{3}{4}\cdot r_0$ and let $A'=\{y\in A\ |\
  \bar{r}(y)\geq 1/2\cdot r_0\}$. 
With $A'$ and $y_0$ belonging to $\mathfrak{X}$ we take an open
  neighborhood $U_2$ of $A'$ and an open neighborhood $U(y_0)$ such that
  $U(y_0)\cap U_2=\emptyset$ and $U(y_0)\subset \{y\in
  Y^{3,\delta_0}_{{\mathcal D},\varphi}({\mathbb R}\times {\mathbb
  R}^N,\boldsymbol{\gamma})\ |\ \ \bar{r}(y)> 3/4 \cdot r_0\}$.   
Then $U(y_0)\cap U_1=\emptyset$ and consequently $U_1\cup U_2$ is an open
  neighborhood of $A$ which is disjoint from $U(y_0)$.

Next assume that our point has the form
  $(0,(\wt{u}^x_0,[\wh{x}_0,\wh{y}_0],\wt{u}^y_0))$ not contained in the
  closed set $A$.
Then $U=Y^{3,\delta_0}_{{\mathcal D},\varphi}({\mathbb R}\times {\mathbb
  R}^N,\boldsymbol{\gamma})\setminus A$  is an open neighborhood of
  $(0,(\wt{u}^x_0,[\wh{x}_0,\wh{y}_0],\wt{u}^y_0))$ which is disjoint from
  $A$. 
In view of Lemma \ref{LEMR3.24} we find a closed subset $V$  for the
  topology ${\mathcal T}'$ with $V\subset U$
  such that $(0,(\wt{u}^x_0,[\wh{x}_0,\wh{y}_0],\wt{u}^y_0))$ is an
  interior point of $V$.
Hence we find $Q=Q(0,(\wt{u}^x_0,[\wh{x}_0,\wh{y}_0],\wt{u}^y_0))\in
    {\mathcal T}'$ contained in $V$.  
Consequently $Y^{3,\delta_0}_{{\mathcal D},\varphi}\setminus V$ is an open
  neighborhood of $A$ which is disjoint from the open neighborhood $Q$ of
  $(0,(\wt{u}^x_0,[\wh{x}_0,\wh{y}_0],\wt{u}^y_0))$.
This completes the regularity proof.

At this point we have established that ${\mathcal T}'$ is an Hausdorff,
  second countable, and regular topological space.
It follows from Urysohn's metrization theorem that
  $(Y^{3,\delta_0}_{{\mathcal D},\varphi},{\mathcal T}')$ is a metrizable
  space.
Since  the maps $\boldsymbol{\bar{\oplus}}:O_{\mathcal U}\rightarrow
  {\mathcal W}$ and $ \boldsymbol{\bar{\oplus}}:O_{\dot{\mathcal
  V}}\rightarrow \mathfrak{X}$ are homeomorphism and sc-smoothly
  compatible it follows that $(Y^{3,\delta_0}_{{\mathcal
  D},\varphi},{\mathcal T}')$ has the desired M-polyfold structure.
\end{proof}%%%%%%%%%%%%%%%%%%%%%%%%%%%%%%%%%%%%%%%%%%%%%%%%            END PROOF
We still have to show that the topology which we have defined is indeed
  the quotient topology.

%%%%%%%%%%%%%%%%%%%%%%%%%%%%%%%%%%%%%%%%%%%%%%%%%%%%%%%%%%%%%%%%%%%%%%%%%%%%%%%%
%%%%%%%%%%                          PROPOSITION                        %%%%%%%%%
%%%%%                                                                       %%%%
\begin{proposition}\label{PPROPN3.48}
The topology ${\mathcal T}'$  on $Y^{3,\delta_0}_{{\mathcal D},\varphi}$
  is the finest topology making the map $\boldsymbol{\bar{\oplus}}:{\mathcal
  V}\rightarrow Y^{3,\delta_0}_{{\mathcal D},\varphi}$ continuous.
With other words ${\mathcal T}={\mathcal T}'$.
\end{proposition}
%%%%%                                                                       %%%%
%%%%%%%%%%                                                             %%%%%%%%%
%%%%%%%%%%%%%%%%%%%%%%%%%%%%%%%%%%%%%%%%%%%%%%%%%%%%%%%%%%%%%%%%%%%%%%%%%%%%%%%%
%
\begin{proof}
We have the topology ${\mathcal T}'$ on $Y^{3,\delta_0}_{{\mathcal
  D},\varphi}$ as well as the finest topology ${\mathcal T}$ for which
  $\boldsymbol{\bar{\oplus}}$ is continuous. 
For the topology ${\mathcal T}'$ the map  $\boldsymbol{\bar{\oplus}}$ is
  continuous implying that ${\mathcal T}'\subset {\mathcal T}$.
Assume that $U\in {\mathcal T}$, which means by definition that
  $\boldsymbol{\bar{\oplus}}^{-1}(U)$ is an open subset of ${\mathcal V}$.
Assume first that $U$ contains a point
  $(0,(\wt{u}^x,[\wh{x},\wh{y}],\wt{u}^y))$.
We take the co-retraction $K$ and since it is continuous for ${\mathcal
  T}'$ we find $O(y)\in {\mathcal T}'$ such that $K(O(y))\in
  \boldsymbol{\bar{\oplus}}^{-1}(U)$.
Then $O(y)=\boldsymbol{\bar{\oplus}}\circ K(O(y))\subset U$.    
Next assume we are given a point $(r,\wt{w})\in U$ with $r\neq 0$.  
We take the co-retraction $H$ and find $O(y)$ with $H(O(y))\subset
  \boldsymbol{\bar{\oplus}}^{-1}(U)$.
As before we conclude that $O(y)\subset U$.  
Hence we have written the element $U\in {\mathcal T}$ as the union of
elements in ${\mathcal T}'$ and conclude that $U\in {\mathcal T}'$.  
The proof is complete.
\end{proof}%%%%%%%%%%%%%%%%%%%%%%%%%%%%%%%%%%%%%%%%%%%%%%%%            END PROOF

It remains to prove the tameness assertion as well as the assertion about
  the degeneracy index.
Since $\mathfrak{X}$ is an open subset of $Y^{3,\delta}_{{\mathcal
  D},\varphi}$ and it contains all the points $(r,\wt{u})$ with $r\in
  (0,1)$ it follows from the properties of $\mathfrak{X}$ which has
  $d_{\mathfrak{X}}\equiv 1$ that $d(r,\wt{u})=0$ provided $r\in (0,1)$.
Assume next a point of the form $y=(0,\wt{u})$. 
We find an open neighborhood $O(y)$ such that
  $\tau:=K\circ\bm{\bar{\oplus}}:O(y)\rightarrow O(y)$ is a sc-smooth
  retraction. 
Clearly $\tau(0,\wt{u})=(0,\wt{u})$ and from the properties of $K$ and
  $\bm{\bar{\oplus}}$ we see that in general $\tau$ preserves the
  $r$-coordinate. 
Hence it is a splicing and consequently tame, see \cite{HWZ2017} Lemma
  2.5.
Note that in the above reference one works in sectors of sc-Banach spaces.
Of course, the argument generalizes immediately if the ambient space is a
  ssc-manifold with boundary with corners.
At this point Theorem \ref{RRR} is proved.
\end{proof}%%%%%%%%%%%%%%%%%%%%%%%%%%%%%%%%%%%%%%%%%%%%%%%%            END PROOF
We also note that the above discussion implies Proposition
  \ref{PROPN3.15}.

%%%%%%%%%%%%%%%%%%%%%%%%%%%%%%%%%%%%%%%%%%%%%%%%%%%%%%%%%%%%%%%%%%%%%%%%%%%%%%%%
%%%%%%%%%%                          DEFINITION                         %%%%%%%%%
%%%%%                                                                       %%%%
\begin{definition}
  \hfill\\
Given a weighted periodic orbit $\boldsymbol{\bar{\gamma}}$ in ${\mathbb
  R}^N$ consider the previously defined set $Y^{3,\delta_0}_{{\mathcal
  D},\varphi}({\mathbb R}\times {\mathbb R}^N,\boldsymbol{\gamma})$
  equipped with the finest topology ${\mathcal T}$ for which
  $\boldsymbol{\bar{\oplus}}:{\mathcal V}\rightarrow
  Y^{3,\delta_0}_{{\mathcal D},\varphi}({\mathbb R}\times {\mathbb
  R}^N,\boldsymbol{\gamma})$ is continuous.  
We equip this metrizable topological spaces with the M-polyfold structure
  just constructed and denote this M-polyfold by $Y_{{\mathcal
  D},\varphi}({\mathbb R}\times {\mathbb R}^N,\boldsymbol{\bar{\gamma}})$.
\end{definition}
%%%%%                                                                       %%%%
%%%%%%%%%%                                                             %%%%%%%%%
%%%%%%%%%%%%%%%%%%%%%%%%%%%%%%%%%%%%%%%%%%%%%%%%%%%%%%%%%%%%%%%%%%%%%%%%%%%%%%%%
%

%%%%%%%%%%%%%%%%%%%%%%%%%%%%%%%%%%%%%%%%%%%%%%%%%%%%%%%%%%%%%%%%%%%%%%%%%%%%%%%%
%%%%%%%%%%                         SUB-SECTION                         %%%%%%%%%
%%%%%%%%%%%%%%%%%%%%%%%%%%%%%%%%%%%%%%%%%%%%%%%%%%%%%%%%%%%%%%%%%%%%%%%%%%%%%%%%
%
\subsubsection{Proof of Theorem \ref{RRRR}}
We recall the statement for the  convenience of the reader.

%%%%%%%%%%%%%%%%%%%%%%%%%%%%%%%%%%%%%%%%%%%%%%%%%%%%%%%%%%%%%%%%%%%%%%%%%%%%%%%%
%%%%%%%%%%                           THEOREM                           %%%%%%%%%
%%%%%                                                                       %%%%
\begin{theorem}[\ref{RRRR}]
\jwf{[restatement of another result]}
The $\bm{\bar{\oplus}}$-polyfold structure on $Y^{3,\delta_0}_{{\mathcal
  D},\varphi}({\mathbb R}\times {\mathbb R}^N,\boldsymbol{\gamma})$ has
  the following properties.
\begin{itemize}
  \item[(1)] 
  The topologies induced by ${\mathcal T}$  on $\mathfrak{X}$ and
  $\partial\mathfrak{Z}$ are the original topologies for the already
  existing sc-manifold structure on $\mathfrak{X}$
  and ssc-manifold structure on $\partial\mathfrak{Z}$.
  \item[(2)]
  The M-polyfold structure on $Y^{3,\delta_0}_{{\mathcal
  D},\varphi}({\mathbb R}\times {\mathbb R}^N,\boldsymbol{\gamma})$ does
  not(!) depend on the smooth $\beta:{\mathbb R}\rightarrow [0,1]$  which
  was taken in the definition of $\boldsymbol{\bar{\oplus}}$ as long as it
  satisfies the usual properties $\beta(s)+\beta(-s)=1$, $\beta(s)=1$ for
  $s\leq -1$, and $\beta'(s)<0$ for $s\in (-1,1)$.
  \end{itemize}
The set $Y^{3,\delta_0}_{{\mathcal D},\varphi}({\mathbb R}\times {\mathbb
  R}^N,\boldsymbol{\gamma})$ equipped with this M-polyfold structure is
  denoted by $Y_{{\mathcal D},\varphi}({\mathbb R}\times {\mathbb
  R}^N,\boldsymbol{\bar{\gamma}})$ or, more explicitly,  $Y_{{\mathcal
  D},\varphi}^{3,\delta}({\mathbb R}\times {\mathbb
  R}^N,\boldsymbol{\gamma})$.
\end{theorem}
%%%%%                                                                       %%%%
%%%%%%%%%%                                                             %%%%%%%%%
%%%%%%%%%%%%%%%%%%%%%%%%%%%%%%%%%%%%%%%%%%%%%%%%%%%%%%%%%%%%%%%%%%%%%%%%%%%%%%%%
%
\begin{proof}
\noindent (1) follows immediately from the fact that ${\mathcal
  T}={\mathcal T}'$.\\

\noindent In order to prove (2) we take $\beta$ and $\beta'$. The set
  $Y^{3,\delta_0}_{{\mathcal D},\varphi}$ is by definition independent of
  the choice of $\beta$.
We abbreviate this set by $Y$ and denote by $\bm{Y}$ the set $Y$ equipped
  with the M-polyfold structure defined via $\bm{\bar{\oplus}}$ utilizing
  $\beta$, and by $\bm{Y}'$  similarly utilizing $\beta'$. 
It suffices to show that the identity map $I:\bm{Y}\rightarrow \bm{Y}'$ is
  a sc-diffeomorphism. Interchanging the role of $\beta$ and $\beta'$ it
  suffices to prove that $I$ is sc-smooth.
The surjective map $\bm{\bar{\oplus}}:{\mathcal V}\rightarrow Y$ depends
  on $\beta$ and the same holds for the coretractions $H$ and $K$. 
We denote the corresponding maps obtained when using $\beta'$ by primed
  letter.
We have to show the sc-smoothness of the maps
  \begin{eqnarray}
    &H'\circ\bm{\bar{\oplus}}&\label{xrth1}\\
    &K'\circ\bm{\bar{\oplus}}.&\label{xrth2}
    \end{eqnarray}   
The sc-smoothness of (\ref{xrth1}) follows similarly as the proof
  of Lemma \ref{LEMMN3.16} (3), where we just note that the expressions 
  one gets still allow the application of the Fundamental Lemma.
Along the same lines a modification of Lemma \ref{K-lemmar} proves
  (\ref{xrth2}).
\end{proof}%%%%%%%%%%%%%%%%%%%%%%%%%%%%%%%%%%%%%%%%%%%%%%%%            END PROOF
%

%%%%%%%%%%%%%%%%%%%%%%%%%%%%%%%%%%%%%%%%%%%%%%%%%%%%%%%%%%%%%%%%%%%%%%%%%%%%%%%%
%%%%%%%%%%                         SUB-SECTION                         %%%%%%%%%
%%%%%%%%%%%%%%%%%%%%%%%%%%%%%%%%%%%%%%%%%%%%%%%%%%%%%%%%%%%%%%%%%%%%%%%%%%%%%%%%
%
\subsubsection{More Results and Proofs}
For the further considerations we abbreviate $Y:=Y_{{\mathcal
  D},\varphi}({\mathbb R}\times {\mathbb R}^N,\boldsymbol{\bar{\gamma}})$
  and consider $\boldsymbol{\bar{\oplus}}:{\mathcal V}\rightarrow Y$. 
This map is sc-smooth and for every $y_0\in Y$ there exists an open
  neighborhood $U(y_0)$ and an sc-smooth map $H_{y_0}:U(y_0)\rightarrow
  {\mathcal V}$ such that $\boldsymbol{\bar{\oplus}}\circ
  H_{y_0}=Id_{U(y_0)}$.
The following proposition also provides a proof of Proposition
  \ref{PROPRRR} (see {(2)} and {(4)} below).

%%%%%%%%%%%%%%%%%%%%%%%%%%%%%%%%%%%%%%%%%%%%%%%%%%%%%%%%%%%%%%%%%%%%%%%%%%%%%%%%
%%%%%%%%%%                          PROPOSITION                        %%%%%%%%%
%%%%%                                                                       %%%%
\begin{proposition}\label{PPROPRRR}
Let $Y=Y_{{\mathcal D},\varphi}({\mathbb R}\times {\mathbb
  R}^N,\boldsymbol{\bar{\gamma}})$ be  the previously defined M-polyfold.
Then the following maps are sc-smooth:
\begin{itemize}
  \item[(1)] 
  The ${\mathbb R}$-action ${\mathbb R}\times Y\rightarrow Y$ defined by
    $(c,y)\rightarrow c\ast y$. 
  Here, if $y=(0,(\wt{u}^x,[\wh{x},\wh{y}],\wt{u}^y))$ then $c\ast y=
    (0,(c\ast \wt{u}^x,[\wh{x},\wh{y}],c\ast \wt{u}^y))$ and if
    $y=(r,\wt{w})$ then $c\ast y=(r,c\ast\wt{w})$.
  On functions the action by $c$ is addition on the first ${\mathbb
    R}$-component.
  \item[(2)] 
  The maps $\bar{r},\bar{a}_d$ and $\bar{a}$ are sc-smooth.
  \item[(3)] 
  The automorphism group $G$ of the ordered disk-pair acts naturally on
    $Y$ and the action is sc-smooth.
  \item[(4)] 
  The map $\bar{r}$ is submersive.
  \end{itemize}

\end{proposition}
%%%%%                                                                       %%%%
%%%%%%%%%%                                                             %%%%%%%%%
%%%%%%%%%%%%%%%%%%%%%%%%%%%%%%%%%%%%%%%%%%%%%%%%%%%%%%%%%%%%%%%%%%%%%%%%%%%%%%%%
%
\begin{proof}
We begin with {(1)}. We observe that $\bar{\oplus}$ is ${\mathbb
  R}$-equivariant and the ${\mathbb R}$-action on ${\mathcal V}$ is
  ssc-smooth.
Pick $y_0\in Y$ and take an sc-smooth $H_{y_0}:U(y_0)\rightarrow {\mathcal
  V}$ such $\bm{\bar{\oplus}}\circ H_{y_0}=Id_{U(y_0)}$.
For $c$ and $y$ near $y_0$ we observe that 
  $$
    \bar{\oplus} ( c\ast H_{y_0}(y)) = c\ast \bm{\bar{\oplus}}\circ
    H_{y_0}(y)= c\ast y,
    $$
  which implies sc-smoothness.\\

\noindent {(2)} The map $\bar{r}:Y^{3,\delta}_{{\mathcal
  D},\varphi}\rightarrow [0,1)$ is sc-smooth since
  $\bar{r}\circ\bm{\bar{\oplus}}:{\mathcal V}\rightarrow [0,1)$ is
  sc-smooth.
The same argument holds for $\bar{a}_d$ in view of \cite{HWZ8.7}, Lemma
  4.4.
This also implies the assertion for $\bar{a}$.\\

\noindent{(3)} The automorphism group acts sc-smoothly on ${\mathcal V}$
  via $G\times {\mathcal V}\rightarrow {\mathcal V}$ with
  $$
    (g^x,g^y)\ast (r,(\wh{u}^x,[\wh{x},\wh{y}],\wh{u}^y))=(r,\wh{u}^x\circ
    g^{x,-1},[Tg^x\cdot \wh{x},Tg^y\cdot\wh{y}],u^y\circ g^{y,-1}).
    $$
Take a point $y_0$ and consider the sc-smooth $H_{y_0}:U(y_0)\rightarrow
  {\mathcal V}$. 
Then
  $$
    g\ast y = g\ast (\bar{\oplus}\circ H_{y_0}(y))=\bar{\oplus}(g\ast
    H_{y_0}(y))
    $$
  since $\bar{\oplus}$ is $G$-equivariant. 
This implies that the action is sc-smooth. \\

\noindent{(4)} The map $\rho: =H_y\circ \boldsymbol{\bar{\oplus}}:
  \boldsymbol{\bar{\oplus}}^{-1}(U(y))\rightarrow
  \boldsymbol{\bar{\oplus}}^{-1}(U(y))$ is an sc-smooth retraction and
  preserves the $\bar{r}$-fiber.
$\boldsymbol{\bar{\oplus}}^{-1}(U(y))$ is an open subset of the product
  $[0,1)\times Z^{3,\delta}_{\mathcal D}$ where the second factor is a
  ssc-manifold.
Since  $\rho$ preserves the $r$-component  the desired result follows.
\end{proof}%%%%%%%%%%%%%%%%%%%%%%%%%%%%%%%%%%%%%%%%%%%%%%%%            END PROOF

We consider the category with objects being the  pairs $({\mathbb
  R}^N,\boldsymbol{\bar{\gamma}})$, where $\boldsymbol{\bar{\gamma}}$ is a
  weighted periodic orbit in ${\mathbb R}^N$.
The morphisms are the smooth maps $\phi:{\mathbb R}^N\rightarrow {\mathbb
  R}^M$ such that $\phi\circ \gamma(t) =\gamma'(t)$ for suitable
  representatives and $(T,k,\delta)=(T',k',\delta')$.
Associated to a morphism $\phi$  we can define a  $\wt{\phi}=Id\times
  \phi$ which induces a map
  $$
    \wt{\phi}_\ast : Y_{{\mathcal D},\varphi}({\mathbb R}\times {\mathbb
    R}^N,\boldsymbol{\bar{\gamma}})\rightarrow
    Y_{{\mathcal D},\varphi}({\mathbb R}\times {\mathbb
    R}^M,\boldsymbol{\bar{\gamma}'})
    $$
  by $\wt{\phi}_\ast(0,(\wt{u}^x,[\wh{x},\wh{y}],\wt{u}^y))=(0,(
  \wt{\phi}\circ\wt{u}^x,[\wh{x},\wh{y}],\wt{\phi}\circ
  \wt{u}^y))$ and $\wt{\phi}_\ast(r,\wt{u})=(r,\wt{\phi}\circ\wt{u})$.
The following prop implies Proposition \ref{PPP3.22}  and using the ideas
  as in nodal case Theorem \ref{TTT3.24}.

%%%%%%%%%%%%%%%%%%%%%%%%%%%%%%%%%%%%%%%%%%%%%%%%%%%%%%%%%%%%%%%%%%%%%%%%%%%%%%%%
%%%%%%%%%%                          PROPOSITION                        %%%%%%%%%
%%%%%                                                                       %%%%
\begin{proposition}\label{PROX3.49}
For a morphism $\phi: ({\mathbb R}^N,\boldsymbol{\bar{\gamma}})\rightarrow
  ( {\mathbb R}^M,\boldsymbol{\bar{\gamma}'})$ the induced map
  $\wt{\phi}_\ast$ is sc-smooth.
\end{proposition}
%%%%%                                                                       %%%%
%%%%%%%%%%                                                             %%%%%%%%%
%%%%%%%%%%%%%%%%%%%%%%%%%%%%%%%%%%%%%%%%%%%%%%%%%%%%%%%%%%%%%%%%%%%%%%%%%%%%%%%%
%
\begin{proof}
We can take the $H_y$ which we constructed and use the fact that
  $\wt{\phi}$ induce ssc-smooth maps $Z_{{\mathcal D},\varphi}({\mathbb
  R}\times {\mathbb R}^N,\boldsymbol{\bar{\gamma}})\rightarrow Z_{{\mathcal
  D},\varphi}({\mathbb R}\times {\mathbb R}^M,\boldsymbol{\bar{\gamma}'})$.
Then we factor our map of interest as $\bar{\oplus}\circ (Id_{[0,1)}\times
  \wt{\phi}_\ast)\circ H_y$.
This is the same idea as used in Proposition
\ref{prop2}.
\end{proof}%%%%%%%%%%%%%%%%%%%%%%%%%%%%%%%%%%%%%%%%%%%%%%%%            END PROOF

Having established the above results we can show as in the nodal case that
  for a closed manifold $Q$ or more generally a manifold
  without boundary together with a weighted periodic orbit, there is a
  well-defined M-polyfold $Y_{{\mathcal D},\varphi}({\mathbb R}\times
  Q,\boldsymbol{\bar{\gamma}})$ and a morphism
  $h:(Q,\boldsymbol{\bar{\gamma}})\rightarrow
  (Q',\boldsymbol{\bar{\gamma}'})$ induces an sc-smooth map $\wt{h}_\ast$
  between the associated spaces.
This M-polyfold is constructed by the embedding method. 
These spaces have all the listed properties from the special ${\mathbb
  R}^N$-case.
We summarize the result in the following theorem.

%%%%%%%%%%%%%%%%%%%%%%%%%%%%%%%%%%%%%%%%%%%%%%%%%%%%%%%%%%%%%%%%%%%%%%%%%%%%%%%%
%%%%%%%%%%                           THEOREM                           %%%%%%%%%
%%%%%                                                                       %%%%
\begin{theorem}\label{THX8.58}
For every smooth manifold $Q$ without boundary equip\-ped with a weighted
  periodic orbit $\bm{\bar{\gamma}}$, giving the pair
  $(Q,\bm{\bar{\gamma}})$, and ordered disk pair ${\mathcal D}$, the set
  $Y^{3,\delta_0}_{{\mathcal D},\varphi}({\mathbb R}\times Q,\bm{\gamma})$
  has a natural M-polyfold structure.
The set equipped with this M-polyfold structure is denoted by
  $Y^3_{{\mathcal D},\varphi}({\mathbb R}\times Q,\bm{\bar{\gamma}})$.
Moreover, the natural map
  $$
    \bar{r}:Y_{{\mathcal D},\varphi}({\mathbb R}\times
    Q,\boldsymbol{\bar{\gamma}})\rightarrow [0,1)
    $$
  is sc-smooth and submersive, and the domain parameter extraction is
  sc-smooth
  $$
    \bar{a}:Y_{{\mathcal D},\varphi}({\mathbb R}\times
    Q,\boldsymbol{\bar{\gamma}})\rightarrow {\mathbb B}.
    $$
\end{theorem}
%%%%%                                                                       %%%%
%%%%%%%%%%                                                             %%%%%%%%%
%%%%%%%%%%%%%%%%%%%%%%%%%%%%%%%%%%%%%%%%%%%%%%%%%%%%%%%%%%%%%%%%%%%%%%%%%%%%%%%%
%
The fact that $\bar{r}$ is submersive is very important as we shall see
  next when we deal with several periodic orbits.

%%%%%%%%%%%%%%%%%%%%%%%%%%%%%%%%%%%%%%%%%%%%%%%%%%%%%%%%%%%%%%%%%%%%%%%%%%%%%%%%
%%%%%%%%%%                         SUB-SECTION                         %%%%%%%%%
%%%%%%%%%%%%%%%%%%%%%%%%%%%%%%%%%%%%%%%%%%%%%%%%%%%%%%%%%%%%%%%%%%%%%%%%%%%%%%%%
%
\subsubsection{Several Periodic Orbits}\label{several-x-orbits}
Using this discussion we can deal with the situation of several periodic
  orbits.
We assume we are given a finite family of ordered disk pairs $\bm{D}:=
  {({\mathcal D}_{(z,z')})}_{(z,z')\in \wt{\Gamma}}$, a compact manifold $V$
  without boundary $Q$, and for every $(z,z')$ a weighted periodic orbit.
By the previous construction we obtain for every ordered pair $(z,z')\in
  \wt{\Gamma}$
  $$
    \bar{r}_{(z,z')}:Y_{{\mathcal D}_{(z,z')},\varphi}({\mathbb R}\times
    Q,\boldsymbol{\bar{\gamma}})\rightarrow [0,1).
    $$
From this it follows that the product map
  \begin{eqnarray}\label{eqw10.34}
    \prod_{(z,z')\in\wt{\Gamma}}\bar{r}_{(z,z')}:\prod_{(z,z')\in
    \wt{\Gamma}} Y_{{\mathcal D}_{(z,z')},\varphi}({\mathbb R}\times
    Q,\boldsymbol{\bar{\gamma}})\rightarrow
    \prod_{(z,z')\in\wt{\Gamma}}[0,1).
    \end{eqnarray}
  is submersive. 
Denoting  by $\Delta:[0,1)\rightarrow [0,1)^{\wt{\Gamma}}:r\rightarrow
  (r,r,...,r)$ the diagonal map we can take the pull-back of the diagram
  (\ref{eqw10.34}).
We note that the argument even works if the manifolds $Q$ depend on
  $(z,z')$.
For the following we assume that for every $(z,z')$ we have the same $Q$
  and that the collection of  periodic orbits has the following properties.
We have the finite set $\wt{\Gamma}$ of ordered pairs and a map
  $\bar{\digamma}$ which associates to $(z,z')$ a weighted periodic orbit
  $\bm{\bar{\gamma}}_{(z,z')}$ in $Q$
If the underlying classes $[\gamma_{(z_1,z_1')}]$ and
  $[\gamma_{(z_2,z_2')}]$ intersect, i.e. $\gamma_{(z_1,z_1')}(S^1)\cap
  \gamma_{(z_2,z_2')}(S^1)\neq \emptyset$
  we require that $\bar{\digamma}(z_1,z_1')=\bar{\digamma}{(z_2,z_2')}$.
  Let us call $\bar{\digamma}$ a (weighted periodic orbit) selector. 
If we forget about the weights we obtain the (periodic orbit) selector
  $\digamma$.
We shall denote by $Y^{3,\delta_0}_{\bm{D},\varphi}({\mathbb R}\times
  Q,{\digamma})$ the associated pull-back of the product of single periodic
  orbit situations.
It has as natural M-polyfold structure for which the projection
  $\bar{r}_{\bm{D}}$ onto $[0,1)$ is submersive.  
Moreover denoting by ${\mathbb B}_{\bm{D}}$ the product of the ${\mathbb
  B}_{(z,z')}$ we obtain a natural map $\bar{a}_{\bm{D}}$ which extracts the
  total gluing parameter.
Hence we obtain the following result.

%%%%%%%%%%%%%%%%%%%%%%%%%%%%%%%%%%%%%%%%%%%%%%%%%%%%%%%%%%%%%%%%%%%%%%%%%%%%%%%%
%%%%%%%%%%                           THEOREM                           %%%%%%%%%
%%%%%                                                                       %%%%
\begin{theorem}\label{THX8.59}
For every smooth manifold $Q$ without boundary equip\-ped with a weighted
  periodic orbit selector $\bar{\digamma}$, defined on the set
  $\wt{\Gamma}$, which is the finite set of ordered nodal pairs
  associated to a finite collection $\bm{D}$ of ordered nodal disk pairs,
  the set  $Y^{3,\delta_0}_{\bm{D},\varphi}({\mathbb R}\times
  Q,{\digamma})$ has a natural M-polyfold structure. 
The set equipped with this M-polyfold structure is denoted by
  $Y_{\bm{D},\varphi}({\mathbb R}\times Q,\bar{\digamma})$. 
Moreover, the natural map
  $$
    \bar{r}_{\bm{D}}:Y_{\bm{D},\varphi}({\mathbb R}\times
    Q,{\bar{\digamma}})\rightarrow [0,1)
    $$
  is sc-smooth and submersive, and the domain parameter extraction is
  sc-smooth
  $$
    \bar{a}_{\bm{D}}:Y_{\bm{D},\varphi}({\mathbb R}\times
    Q,{\bar{\digamma}})\rightarrow {\mathbb B}_{\bm{D}}.
    $$
\end{theorem}
%%%%%                                                                       %%%%
%%%%%%%%%%                                                             %%%%%%%%%
%%%%%%%%%%%%%%%%%%%%%%%%%%%%%%%%%%%%%%%%%%%%%%%%%%%%%%%%%%%%%%%%%%%%%%%%%%%%%%%%
%
Sometimes we shall write $Y^{3,\delta}_{\bm{D},\varphi}({\mathbb R}\times
  Q,\digamma)$ if we want to make the underlying choices more explicit.

\subsection{Strong Bundles}
The construction is almost identical to the nodal case.   
\subsubsection{One Periodic Orbit}
We have the submersive situation
$$
p_{\mathbb B}:X^{2,\delta}_{{\mathcal D},\varphi,0}( {\mathbb C}^L)\rightarrow {\mathbb B}
$$
and consider the diagram
$$
\begin{CD}
@.     X^{2,\delta}_{{\mathcal D},\varphi,0}( {\mathbb C}^L)\\
@. @V p_{\mathbb B} VV\\
Y^{3,\delta}_{{\mathcal D},\varphi}({\mathbb R}\times {\mathbb R}^N,\boldsymbol{\gamma})@>\mathfrak{a}>> {\mathbb B}
\end{CD}
$$
Here $\mathfrak{a}$ is the extraction of the gluing parameter.  The fibered product is a strong bundle submersive over
$[0,1)$.  We denote it by 
$$
E^{3,2,\delta}_{{\mathcal D},\varphi,\boldsymbol{\gamma}}(({\mathbb R}\times {\mathbb R}^N)\times  {\mathbb C}^L)\rightarrow [0,1).
$$
We also have the sc-smooth maps
$$
E^{3,2,\delta}_{{\mathcal D},\varphi,\boldsymbol{\gamma}}(({\mathbb R}\times {\mathbb R}^N)\times  {\mathbb C}^L)\rightarrow{\mathbb B}
$$
and
$$
E^{3,2,\delta}_{{\mathcal D},\varphi,\boldsymbol{\gamma}}(({\mathbb R}\times {\mathbb R}^N)\times  {\mathbb C}^L)\rightarrow Y^{3,\delta}_{{\mathcal D},\varphi}({\mathbb R}\times {\mathbb R}^N,\boldsymbol{\gamma}).
$$
Given a closed odd-dimensional manifold $Q$ equipped with a periodic orbit $\boldsymbol{\gamma}=([\gamma],T,k)$ we assume 
that ${\mathbb R}\times Q$ is equipped with an almost complex structure $\wt{J}$ which has the following form.
There exists a co-dimension one sub-bundle and a smooth nowhere vanishing vector field $R$ on $Q$ transversal to $H$ defining the splitting
$$
T_{(a,q)} ({\mathbb R}\times Q) = {\mathbb R}\oplus {\mathbb R}R(q) \oplus H_q.
$$
In addition we assume that $\dot{\gamma}(t)=(T/k)\cdot R(\gamma(t))$.
Further we  assume that $\wt{J}(a,q)(T,0,0)=(0,-kR(q),0)$ and $H_q$ is invariant. 

Next we take a smooth embedding $\phi: Q\rightarrow  {\mathbb R}^N$ which will map the periodic orbit $\boldsymbol{\gamma}$
to a periodic orbit $\boldsymbol{\gamma}_0$. Associated to this embedding we obtain an embedding ${\mathbb R}\times Q\rightarrow
{\mathbb R}\times {\mathbb R}^N$ which is the identity in the first factor. This embedding can be lifted to an embedding of $(T({\mathbb R}\times Q),\wt{J})$
into the complex bundle $({\mathbb R}\times {\mathbb R}^N)\times {\mathbb C}^L$. This data fits into the commutative diagram
$$
\begin{CD}
(T({\mathbb R}\times Q),\wt{J})@>\Phi >>  ({\mathbb R}\times {\mathbb R}^N)\times {\mathbb C}^L\\
@VVV @VVV\\
{\mathbb R}\times Q  @> Id\times \phi >>  {\mathbb R}\times {\mathbb R}^N
\end{CD}
$$
This implies that we obtain a well-defined strong M-polyfold bundle
$$
E^{3,2,\delta}_{{\mathcal D},\varphi,\gamma}(T({\mathbb R}\times Q,\wt{J}))\rightarrow Y^{3,\delta}_{{\mathcal D},\varphi}({\mathbb R}\times Q,\boldsymbol{\gamma}).
$$
This space we can identify canonically with 
$$
\Omega^{3,2,\delta}_{{\mathcal D},\varphi,\gamma}({\mathbb R}\times Q,\wt{J})\rightarrow Y^{3,\delta}_{{\mathcal D},\varphi}({\mathbb R}\times Q,\boldsymbol{\gamma}),
$$
where this is as in the nodal case the bundle of suitable $(0,1)$-forms.  

\subsubsection{Several Periodic Orbits}
If we have several periodic orbit 
the same works as well. With the same notation as in Subsection \ref{several-x-orbits} this situation is obtained by taking products and pull-backs.
With $\bm{D}$ and $\bar{\digamma}$ as previously described we obtain for every $(z,z')\in\wt{\Gamma}$
$$
\Omega^{3,2,\delta}_{{\mathcal D}_{(z,z')},\varphi,\gamma_{(z,z')}}({\mathbb R}\times Q,\wt{J})\rightarrow Y^{3,\delta}_{{\mathcal D}_{(z,z')},\varphi}({\mathbb R}\times Q,\gamma_{(z,z')}).
$$
We take the product of these diagrams and pull-back by the natural embedding of  
$$
Y_{\bm{D},\varphi}({\mathbb R}\times Q,\bar{\digamma})\rightarrow \prod_{(z,z')\in\wt{\Gamma}} Y_{{\mathcal D}_{(z,z')},\varphi}({\mathbb R}\times Q,\boldsymbol{\bar{\gamma}})
$$
We denote the result of this construction by
\begin{eqnarray}
\Omega^{3,2,\delta}_{\bm{D},\varphi,\digamma}({\mathbb R}\times Q,\wt{J})\rightarrow Y^{3,\delta}_{\bm{D},\varphi}({\mathbb R}\times Q,\digamma),
\end{eqnarray}
or in a less explicit notation
\begin{eqnarray}
\Omega_{\bm{D},\varphi,\bar{\digamma}}({\mathbb R}\times Q,\wt{J})\rightarrow Y_{\bm{D},\varphi}({\mathbb R}\times Q,\bar{\digamma}).
\end{eqnarray}
\index{$\Omega_{\bm{D},\varphi,\bar{\digamma}}({\mathbb R}\times Q,\wt{J})\rightarrow Y_{\bm{D},\varphi}({\mathbb R}\times Q,\bar{\digamma})$}
\subsection{Sc-Smoothness of the CR-Operator}
Next we consider the nonlinear Cauchy Riemann operator. The questions we are concerned
with can be studied in the  context of one periodic orbit and the
results we shall obtain  immediately extend to the case of several periodic orbits.
\subsubsection{Sc-Smoothness}
We define the Cauchy-Riemann section of 
$$
\Omega^{3,2,\delta}_{{\mathcal D},\varphi,\gamma}({\mathbb R}\times Q,\wt{J})\rightarrow Y^{3,\delta}_{{\mathcal D},\varphi}(Q,\boldsymbol{\gamma})
$$
 in the usual way as 
\begin{eqnarray}\label{erq10.57}
\bar{\partial}_{\wt{J}}(r,\wt{u})=\left(\left(r,\wt{u}\right),\frac{1}{2}\cdot \left[T\wt{u} +\wt{J}(\wt{u})\circ T\wt{u}\circ j\right]\right)
\end{eqnarray}
and shall prove the following result.
\begin{prop}\label{PUPPY10.57}
The Cauchy-Riemann section is sc-smooth.
\end{prop}
\begin{proof}
We first note that we can  distinguish two cases. 
Namely if we work near a  $(r,\wt{u})$ which satisfies  $r\in (0,1)$, we are precisely working with maps on cylinders of finite length
and the result follows from the nodal case, see Proposition \ref{PROPI9.20}.
  Hence we need only to study the case 
near an element of the form $(0,\wt{u})$. The question of sc-smoothness is, of course, connected with the 
behavior inside of a long cylinder, far away from the boundary.  Indeed, near the boundary the behavior is as in the (more) classical  ssc-smooth case.  We shall use the discussion in Subsection \ref{USEFul0}
which equally well applies to the periodic orbit case.

We know that by construction we may assume that $Q\subset {\mathbb R}^{N}$ for some possibly large 
$N$. We have a smooth ${\mathbb R}$-invariant almost complex structure $\wt{J}$ defined on ${\mathbb R}\times Q$.
The periodic orbit $([\gamma],T,k)$ has the property that $\gamma(S^1)\subset Q$ and consequently defines
the $\wt{J}$-complex cylinder ${\mathbb R}\times \gamma(S^1)$. 
Since ${\mathbb R}\times Q\subset {\mathbb R}\times {\mathbb R}^{N}$ we can
can find a ${\mathbb R}$-invariant map
$$
(b,u)\rightarrow {\mathcal L}({\mathbb R}\times {\mathbb R}^{N}):(b,u)\rightarrow\mathsf{J}(b,u)
$$
such that for $(b,u)\in {\mathbb R}\times Q$ it holds that the restriction of $\mathsf{J}$ to $T_{(b,u)}({\mathbb R}\times Q)$ is just $\wt{J}(b,u)$.  
With this at hand we can view the map in (\ref{erq10.57}) as the restriction (after some identifications)
of the following map 
\begin{eqnarray}\label{WEQN10.51}
&\bar{\mathsf{L}}:Y^{3,\delta}_{{\mathcal D},\varphi}({\mathbb R}\times {\mathbb R}^{N},\bm{\gamma})\rightarrow X^{2,\delta}_{{\mathcal D},\varphi,0}({\mathbb R}\times {\mathbb R}^{N})&\\
& (r,\wt{u})\rightarrow  \frac{1}{2}\cdot \left[\wt{u}_s +\mathsf{J}(\wt{u})\wt{u}_t\right].&\nonumber
\end{eqnarray}
Here the partial derivatives are with respect the obvious holomorphic coordinates
coming from the standard cylinder model for the domains which was discussed in 
Subsection \ref{USEFul0}. 
Recall that the M-polyfold structure on $Y^{3,\delta}_{{\mathcal D},\varphi}({\mathbb R}\times {\mathbb R}^{N},\bm{\gamma})$ is defined by 
$$
\bm{\bar{\oplus}}:{\mathcal V}\rightarrow Y^{3,\delta}_{{\mathcal D},\varphi}({\mathbb R}\times {\mathbb R}^{N},\bm{\gamma})
$$
and that of $X^{2,\delta}_{{\mathcal D},\varphi,0}({\mathbb R}\times {\mathbb R}^{N})$ by
$$
\wh{\oplus}:{\mathbb B}\times H^{2,\delta}\rightarrow X^{2,\delta}_{{\mathcal D},\varphi,0}({\mathbb R}\times {\mathbb R}^{N}).
$$
For the above see Definitions \ref{DEF3.15} and  Theorem \ref{RRR} for the first statement, and 
 (\ref{WEQN3.18}) and Theorem \ref{THM1.7} for the second. The sc-smoothness of  $\bar{\mathsf{L}}$ is equivalent
 to the sc-smoothness of $\bar{\mathsf{L}}\circ \bm{\bar{\oplus}}$. In order to prove the latter
we shall   define an sc-smooth  map
$\mathsf{L}:{\mathcal V}\rightarrow {\mathbb B}_{\mathcal D}\times H^{2,\delta}$
such that 
\begin{eqnarray}\label{RRREQ}
\bar{\mathsf{L}}\circ \bm{\bar{\oplus}} =\wh{\oplus} \circ \mathsf{L}.
\end{eqnarray}
Since $\wh{\oplus}$ is sc-smooth the desired result will follow.

Starting with an element   $(r,\wt{q}\dotplus \wt{h})$ we first consider $(r,\wt{q})$.
Since the element belongs to ${\mathcal V}$ the following holds true, where we write
$\wt{q}$ as $(\wt{q}^x,[\wh{x},\wh{y}],\wt{q}^y)$ as in (\ref{PPP10.3}). This data gives the asymptotic constants 
$c^x$ and $c^y$. We define  if $r=0$ the gluing parameter $a=0$, and if $r>0$ 
$a=|a|\cdot [\wh{x},\wh{y}]$, with $T\cdot\varphi(|a|)=\varphi(r)+c^y-c^x$.  Recalling Lemma
\ref{LEM3.14} the following holds.
\begin{lem}
The map ${\mathcal V}\rightarrow {\mathbb B}_{{\mathcal D}}$ defined by
$$
(r,\wt{q}\dotplus\wt{h})\rightarrow a
$$
is ssc-smooth. 
\qed
\end{lem}
Given $(r,\wt{q}\dotplus\wt{h})\in {\mathcal V}$ we define a new element 
$(r,\wt{q}\dotplus\wt{k})\in {\mathcal V}$
where $\wt{k}=(\wt{k}^x,\wt{k}^y)$ is given as follows. If $r=0$ we put $\wt{k}:=\wt{h}$. 
In the case that $r>0$ we obtain the nonzero gluing parameter  $a$ and define with $R=\varphi(|a|)$
\begin{eqnarray*}
\wt{k}^x \circ\sigma^+_{\wh{x}}(s,t)
& =& \beta(s-R/2-2)\wt{h}^x\circ\sigma_{\wh{x}}^+(s,t)\\
&& + (\beta(s-R/2-2)-\beta(s-R/2))\wt{h}^y\circ\sigma^-_{\wh{y}}(s-R,t)
\end{eqnarray*}
Similarly
\begin{eqnarray*}
&&\wt{k}^y \circ\sigma_{\wh{y}}^-(s',t')\\
& =& \beta(-s'-R/2-2)\wt{h}^y\circ\sigma_{\wh{y}}^-(s',t')\\
&& + (\beta(-s'-R/2-2)-\beta(s'-R/2))\wt{h}^x\circ\sigma^+_{\wh{x}}(-s'-R,t')
\end{eqnarray*}
As a consequence of the fundamental lemma the map
$$
{\mathbb B}_{\mathcal D}\times H^{3,\delta}({\mathcal D},{\mathbb R}\times {\mathbb R}^N)\rightarrow 
H^{3,\delta}({\mathcal D},{\mathbb R}\times {\mathbb R}^N):(a,\wt{h})\rightarrow \wt{k}
$$
is sc-smooth. Using this we define the sc-smooth map
\begin{eqnarray}
&\mathsf{C}: {\mathcal V}\rightarrow {\mathcal V}&\\
&\mathsf{C}(r,\wt{q}\dotplus \wt{h})= (r,\wt{q} \dotplus \wt{k}).&\nonumber
\end{eqnarray}
Next we consider  following map belonging to the classical context 
\begin{eqnarray}
&\mathsf{D}:{\mathcal V}\rightarrow {\mathbb B}_{\mathcal D}\times H^{2,\delta}&\\
&\wt{q}\dotplus\wt{h}\rightarrow (a, (\bar{\mathsf{L}}^x(\wt{q}^x+\wt{h}^x),\bar{\mathsf{L}}^y(\wt{q}^y+\wt{h}^y))),&\nonumber\\
\end{eqnarray}
where $\bar{\mathsf{L}}^x$ and $\bar{\mathsf{L}}^y$ are defined similarly as in (\ref{WEQN10.51}) but on punctured disks.
\begin{lem}
The map $\mathsf{D}$ is ssc-smooth.
\end{lem}
\begin{proof}
This is classical and follows from the technology in \cite{El}.
\end{proof}
As a consequence of the two lemmata the composition map $\mathsf{L}= \mathsf{D}\circ \mathsf{C}$
$$
\mathsf{L}:{\mathcal V}\rightarrow {\mathbb B}_{\mathcal D}\times H^{2,\delta}({\mathbb R}\times {\mathbb R}^N)
$$
is sc-smooth.  It is a trivial exercise to verify the validity of the property (\ref{RRREQ})
and the proof is complete. 
\end{proof}

\section{Summary of Some Classical Constructions}\label{SECX3}
We shall recall some classical results which are well-known and can be proved using the tools provided by \cite{El}.
Then we add some additional structures which will will be important for the later constructions.
\subsection{An ssc-Manifold and Strong Bundle}
Consider  a  compact Riemann surface with smooth boundary which we denote by $(\Sigma,j)$. We also assume a finite group $H$
is acting on $\Sigma$ by biholomorphic map. 

\subsubsection{A Strong Bundle}
Given an almost complex smooth manifold $(Q,J)$  we can define $X^3_\Sigma(Q)$\index{$X^3_\Sigma(Q)$} to consist of all maps of Sobolev class $H^3$ defined on $\Sigma$ with image in $Q$.
Then $X^3_\Sigma(Q)$ has the structure of an ssc-manifold where level $m$ corresponds to Sobolev regularity $H^{m+3}$. 
  Next we assume that a smooth family of almost complex structures $\mathfrak{j}: v\rightarrow j(v)$
on $\Sigma$ is given  having the following properties
\begin{eqnarray}\label{bott}
\begin{array}{cc}
 j(0)=j & \\
j(v)=j  & \ \ \text{near}\  \partial \Sigma,\  \ \text{for}\ \  v\in V. 
\end{array}
\end{eqnarray}
Taking the product with $V$ we obtain the ssc-manifold $V\times X^3_\Sigma(Q)$, which we shall abbreviate
$X^3_{\Sigma,{\mathfrak{j}}}(Q)$.  There also exists a strong ssc-bundle $W^{3,2}_{\Sigma,\mathfrak{j}}(Q,J)\rightarrow X^3_{\Sigma,\mathfrak{j}}(Q)$, where the elements in our bundle are pairs
$((v,u),\xi)$ and $\xi$ is of class $H^2$ and $\xi(z):(T_z\Sigma,j(v))\rightarrow (T_{u(z)}Q,J)$  is complex anti-linear.
The bi-grading $(m,k)$, with $0\leq k\leq m+1$, corresponds to regularity $H^{m+3}$ of $u$,
and $H^{2+k}$ for $\xi$.

\begin{thm}\label{THM18}
With the family $\mathfrak{j}$ having the above properties 
there exists a  natural ssc-smooth strong bundle structure for $p: W^{3,2}_{\Sigma,\mathfrak{j}}(Q,J)\rightarrow X^3_{\Sigma,\mathfrak{j}}(Q)$.\index{$p: W^{3,2}_{\Sigma,\mathfrak{j}}(Q,J)\rightarrow X^3_{\Sigma,\mathfrak{j}}(Q)$}
\qed
\end{thm}   
\subsubsection{Constraints}
 Assume a finite collection $\Xi$ of points lying in $\Sigma\setminus \partial\Sigma$ is given and for every $z\in \Xi$
we are given a submanifold $H_z$ of $Q$ of co-dimension two (without boundary). Denote the collection
${(H_z)}_{z\in\Xi}$ by ${\mathcal H}$. Then define the subset $X^3_{\Sigma,{\mathcal H}}(Q)\subset X^3_{\Sigma}(Q)$ \index{$X^3_{\Sigma,{\mathcal H}}(Q)$} to consist of all
$u$ such that $u$ intersects $H_z$ transversally at $z$ for all $z\in \Xi$.  Similarly as before we define
$X^3_{\Sigma,\mathfrak{j},{\mathcal H}}(Q)$\index{$X^3_{\Sigma,\mathfrak{j},{\mathcal H}}(Q)$}. The following result follows
from classical methods.
\begin{prop}
The subset $X^3_{\Sigma,{\mathcal H}}(Q)$ of $X^3_\Sigma(Q)$ is an ssc-submanifold. The same holds 
for $ X^3_{\Sigma,\mathfrak{j}, {\mathcal H}}(Q)\subset  X^3_{\Sigma,\mathfrak{j}}(Q)$.  In particular the strong ssc-bundles
$W^{3,2}_{\Sigma,\mathfrak{j}}(Q,J)$  restricted to $V\times X^3_{\Sigma,{\mathcal H}}(Q)$ are strong bundles.
\end{prop}
\begin{proof}
The subset of $X^3_\Sigma(Q)$ consisting of maps mapping $z$ to $H_z$ (no transversality 
condition yet) is an ssc-submanifold which follows directly from the classical constructions. Since $H^3$ embeds continuously into $C^1$ it follows that $ X_{\Sigma,{\mathcal H}}(Q)$ is an open subset of the latter.
\end{proof}
Denote by
$W^{3,2}_{\Sigma,\mathfrak{j},{\mathcal H}}(Q,J)$ the restriction of $W^{3,2}_{\Sigma,\mathfrak{j}}(Q,J)$ to $X^3_{\Sigma,\mathfrak{j},{\mathcal H}}(Q)$.  Based on this, here is the first classical construction of a strong bundle
\begin{eqnarray}
W^{3,2}_{\Sigma,\mathfrak{j},{\mathcal H}}(Q,J)\rightarrow X^3_{\Sigma,\mathfrak{j},{\mathcal H}}(Q).
\end{eqnarray}
This will be used later on.
\subsection{A Special Case }
Consider  the special case,  where $Q$ is replaced by ${\mathbb R}\times Q$.
In this case we assume that $Q$ is odd-dimensional and we are given 
a smooth vector field $R$ on $Q$ and hyperplane distribution $\xi\subset TQ$ 
together with the structure of a complex vector bundle on $\xi\rightarrow Q$, so that
$TQ={\mathbb R}R \oplus \xi$. Then we obtain a complex multiplication by $i$ denoted by
$J:\xi\rightarrow \xi$, and extend this to a special almost complex structure $\wt{J}$ which maps
at $(a,q)\in {\mathbb R}\times Q$ the vector 
$(1,0)\in T_a{\mathbb R}\times T_qQ$ to $(0,R(q))$. Without the constraints we obtain
from the previous discussion the strong bundle 
\begin{eqnarray}
  W^{3,2}_{\Sigma,\mathfrak{j}}({\mathbb R}\times Q,\wt{J})\rightarrow
  X^3_{\Sigma,\mathfrak{j}}({\mathbb R}\times Q).
  \end{eqnarray}
Having the additional structures $R$ and $\xi\rightarrow Q$ and the
  ${\mathbb R}$-action on the first factor, we can define certain
  sub-M-polyfolds of $X^3_{\Sigma,\mathfrak{j}}({\mathbb R}\times Q)$ which
  will be useful later on.
Having a splitting of $T_{(a,q)}({\mathbb R}\times Q)=(T_a{\mathbb
  R}\oplus {\mathbb R}R(q)) \oplus \xi_q$ we can distinguish between
  different types of transversal constraints of codimension two.  
One type is ${\mathbb R}$-invariant and has the form ${\mathbb R}\times
  H$, where $H\subset Q$ has codimenson two.
The second type is not ${\mathbb R}$-invariant and has the form
  $\{a\}\times H$, where $H\subset Q$ has codimension one.
This will be discussed in more detail later. 

\subsection{Anchor Set}
There is also another ingredient which is called an anchor set.

%%%%%%%%%%%%%%%%%%%%%%%%%%%%%%%%%%%%%%%%%%%%%%%%%%%%%%%%%%%%%%%%%%%%%%%%%%%%%%%%
%%%%%%%%%%                          DEFINITION                         %%%%%%%%%
%%%%%                                                                       %%%%
\begin{definition}
Let $(\Sigma,j)$ be a compact Riemann surface with smooth boundary and $G$
  a finite group acting on $(\Sigma,j)$ by biholomorpic maps.  
An anchor set\index{anchor set} $\rup$ is a finite $G$-invariant subset of
  $\Sigma\setminus\partial\Sigma$.
The associated anchor average \index{anchor average} for a continuous map
  $\wt{u}=(b,u):\Sigma\rightarrow {\mathbb R}\times Q$ is the number
  $$
  \text{av}(\wt{u}) =\frac{1}{|\rup|}\cdot \sum_{z\in \rup} b(z)
  $$
%jwf--20180617 $$
%jwf--20180617 \text{av}(\wt{u}) =\frac{1}{|\rup|}\cdot \sum_{z\in \tiny\rup} b(z)
%jwf--20180617 $$
A continuous map $\wt{u}$ is said to be anchored provided
  $\text{av}(\wt{u})=0$.
\end{definition}
%%%%%                                                                       %%%%
%%%%%%%%%%                                                             %%%%%%%%%
%%%%%%%%%%%%%%%%%%%%%%%%%%%%%%%%%%%%%%%%%%%%%%%%%%%%%%%%%%%%%%%%%%%%%%%%%%%%%%%%
%

With $(\Sigma,j)$ as previously described and a given anchor set $\rup$ we
  can define the subset $X^3_{(\Sigma,\mathfrak{j}),\rup}({\mathbb R}\times
  Q)$
%jwf $X^3_{(\Sigma,\mathfrak{j}),\tiny\rup}({\mathbb R}\times Q)$ 
  of  $X^3_{(\Sigma,\mathfrak{j})}({\mathbb R}\times Q)$ to consist of all
  pairs $(v,\wt{u})$, where $v\in V$ and $\wt{u}$ is anchored.

%%%%%%%%%%%%%%%%%%%%%%%%%%%%%%%%%%%%%%%%%%%%%%%%%%%%%%%%%%%%%%%%%%%%%%%%%%%%%%%%
%%%%%%%%%%                          PROPOSITION                        %%%%%%%%%
%%%%%                                                                       %%%%
\begin{proposition}
  \hfill\\
$X^3_{(\Sigma,\mathfrak{j}),\rup}({\mathbb R}\times Q)$ 
  %jwf $X^3_{(\Sigma,\mathfrak{j}),\tiny\rup}({\mathbb R}\times Q)$ 
  is  a ssc-submanifold of $X^3_{(\Sigma,\mathfrak{j}),\rup}({\mathbb
  R}\times Q)$.
%jwf $X^3_{(\Sigma,\mathfrak{j}),\tiny\rup}({\mathbb R}\times Q)$.
\end{proposition}
%%%%%                                                                       %%%%
%%%%%%%%%%                                                             %%%%%%%%%
%%%%%%%%%%%%%%%%%%%%%%%%%%%%%%%%%%%%%%%%%%%%%%%%%%%%%%%%%%%%%%%%%%%%%%%%%%%%%%%%
%
\begin{proof}
The map $(v,(b,u))\rightarrow (v,(b-\text{av}(\wt{u}),u))$ is a global
  ssc-retraction of $X^3_{(\Sigma,\mathfrak{j})}({\mathbb R}\times Q)$
  onto $X^3_{(\Sigma,\mathfrak{j}),\rup}({\mathbb R}\times Q)$ .
%jwf $X^3_{(\Sigma,\mathfrak{j}),\tiny\rup}({\mathbb R}\times Q)$ .
\end{proof}%%%%%%%%%%%%%%%%%%%%%%%%%%%%%%%%%%%%%%%%%%%%%%%%            END PROOF
We immediately obtain by restriction the strong bundle
\begin{eqnarray}
  W^{3,2}_{(\Sigma,\mathfrak{j}),\rup}({\mathbb R}\times
  Q,\wt{J})\rightarrow X^3_{(\Sigma,\mathfrak{j}),\rup}({\mathbb R}\times
  Q).  
  \end{eqnarray}
%jwf \begin{eqnarray}
%jwf W^{3,2}_{(\Sigma,\mathfrak{j}),\tiny\rup}({\mathbb R}\times
%jwf Q,\wt{J})\rightarrow X^3_{(\Sigma,\mathfrak{j}),\tiny\rup}({\mathbb
%jwf R}\times Q).
%jwf \end{eqnarray}
%jwf 
We also could add constraints to obtain a further construction. 
Again this will be discussed later.

%%%%%%%%%%%%%%%%%%%%%%%%%%%%%%%%%%%%%%%%%%%%%%%%%%%%%%%%%%%%%%%%%%%%%%%%%%%%%%%%
%%%%%%%%%%                         SUB-SECTION                         %%%%%%%%%
%%%%%%%%%%%%%%%%%%%%%%%%%%%%%%%%%%%%%%%%%%%%%%%%%%%%%%%%%%%%%%%%%%%%%%%%%%%%%%%%
%
\subsection{The CR-Section}
The CR-operators is defined by 
  $$
  \bar{\partial}_J(v,u)=\left((v,u),\frac{1}{2}\cdot \left[ Tu+J\circ
  Tu\circ j(v)\right]\right).
  $$
It can be viewed as a section of $p:
  W^{3,2}_{\Sigma,\mathfrak{j}}(Q,J)\rightarrow
  X^3_{\Sigma,\mathfrak{j}}(Q)$ as well as
  by restriction of $p: W^{3,2}_{\Sigma,\mathfrak{j},{\mathcal
  H}}(Q,J)\rightarrow X^3_{\Sigma,\mathfrak{j},{\mathcal H}}(Q)$.
The first result is the ssc-smoothness.

%%%%%%%%%%%%%%%%%%%%%%%%%%%%%%%%%%%%%%%%%%%%%%%%%%%%%%%%%%%%%%%%%%%%%%%%%%%%%%%%
%%%%%%%%%%                          PROPOSITION                        %%%%%%%%%
%%%%%                                                                       %%%%
\begin{proposition}
Let $(Q,J)$ be a smooth almost complex manifold without boundary,
  $(\Sigma,j)$ a compact  Riemann surface with smooth boundary, and
  $\mathfrak{j}$ a smooth deformation of $j$ with the properties listed in
  (\ref{bott}).
The Cauchy-Riemann section $\bar{\partial}_J$ of $p:
  W^{3,2}_{\Sigma,\mathfrak{j}}(Q,J)\rightarrow
  X^3_{\Sigma,\mathfrak{j}}(Q)$ is ssc-smooth.
\end{proposition}
%%%%%                                                                       %%%%
%%%%%%%%%%                                                             %%%%%%%%%
%%%%%%%%%%%%%%%%%%%%%%%%%%%%%%%%%%%%%%%%%%%%%%%%%%%%%%%%%%%%%%%%%%%%%%%%%%%%%%%%
%
Due to elliptic regularity theory the section has a regularizing property.

%%%%%%%%%%%%%%%%%%%%%%%%%%%%%%%%%%%%%%%%%%%%%%%%%%%%%%%%%%%%%%%%%%%%%%%%%%%%%%%%
%%%%%%%%%%                          PROPOSITION                        %%%%%%%%%
%%%%%                                                                       %%%%
\begin{proposition}
With the CR-operator as defined above the following holds.  
If $\bar{\partial}_J(u)\in (W^{3,2}_{\Sigma,\mathfrak{j}}(Q,J))_{m,m+1}$
  and near $\partial \Sigma$ the map $u$ belongs to $H^{m+4}_{loc}$, then $u
  \in (X^3_{\Sigma,\mathfrak{j}}(Q))_{m+1}$.
\end{proposition}
%%%%%                                                                       %%%%
%%%%%%%%%%                                                             %%%%%%%%%
%%%%%%%%%%%%%%%%%%%%%%%%%%%%%%%%%%%%%%%%%%%%%%%%%%%%%%%%%%%%%%%%%%%%%%%%%%%%%%%%
%

\appendix

%%%%%%%%%%%%%%%%%%%%%%%%%%%%%%%%%%%%%%%%%%%%%%%%%%%%%%%%%%%%%%%%%%%%%%%%%%%%%%%%
%%%%%%%%%%                           SECTION                           %%%%%%%%%
%%%%%%%%%%%%%%%%%%%%%%%%%%%%%%%%%%%%%%%%%%%%%%%%%%%%%%%%%%%%%%%%%%%%%%%%%%%%%%%%
%
\section{A Primer on DM-Theory}\label{DMTHX}
It is possible to carry out the Deligne-Mumford theory using the polyfold
  technology and one could view such an approach as a toy example of SFT.
However, we shall assume Deligne-Mumford theory and derive the facts we
  need, for example the use of a different gluing profile.
Our primer will use some of the results in the approach by Robbin and
  Salamon in \cite{RS,RS-err} and combines it with a consideration in
  \cite{HWZ8.7}.

%%%%%%%%%%%%%%%%%%%%%%%%%%%%%%%%%%%%%%%%%%%%%%%%%%%%%%%%%%%%%%%%%%%%%%%%%%%%%%%%
%%%%%%%%%%                         SUB-SECTION                         %%%%%%%%%
%%%%%%%%%%%%%%%%%%%%%%%%%%%%%%%%%%%%%%%%%%%%%%%%%%%%%%%%%%%%%%%%%%%%%%%%%%%%%%%%
%
\subsection{Basic Concepts}
Given two complex manifolds $P$ and $A$ satisfying $\dim_{\mathbb C}P-\dim_{\mathbb C}(A)=1$ we consider an holomorphic map
$\pi:P\rightarrow A$. If $p\in P$ is a regular point for $\pi$ we set $a=\pi(p)$ and obtain using the the implicit function theorem 
germs of holomorphic diffeomorphisms $\Sigma:(P,p)\rightarrow ({\mathbb C}^{n+1},0)$ and $\sigma:(A,a)\rightarrow ({\mathbb C}^n,0)$
fitting into the commutative diagram
$$
\begin{CD}
(P,p) @>\pi >>    (A,a)\\
@V \Sigma VV   @V\sigma VV\\
({\mathbb C}\times {\mathbb C}^n,0)@>\text{pr}_2>> ({\mathbb C}^n,0),
\end{CD}
$$
where $n=\text{dim}_{\mathbb C}(A)$. If $p$ is not a regular point the situation can be very complicated.
However, one of the easy cases is the {\bf nodal case}.  We say $p\in P$ is a nodal point provided 
there exist germs of holomorphic diffeomorphisms fitting into the commutative diagram
$$
\begin{CD}
(P,p) @>\pi >>    (A,a)\\
@V \Sigma VV   @V\sigma VV\\
({\mathbb C}^2\times {\mathbb C}^{n-1},0)@>\wh{\pi}>> ({\mathbb C}^n,0),
\end{CD}
$$
where $\wh{\pi}(x,y,t_2,..,t_n)=(xy,t_2,...,t_n)$.
\subsubsection{Deformations and Kodaira-Spencer Deformation}
We start with a stable 
$$
\alpha=(S,j,M,D).
$$
 Important for our construction of uniformizers is the notion of a deformation $\mathfrak{j}$  of $j$.
Roughly speaking a {\bf deformation} $\mathfrak{j}:v\rightarrow j(v)$ of $j$ is a smooth family of almost complex structures 
on $S$ with $j(v_0)=j$ for a suitable $v_0$. We shall describe a class of deformations which is useful for our constructions and to single them out we need some preparation.
Part of the material is taken from \cite{H2014}.

Given the object $\alpha$ we denote by $\Gamma_0(\alpha)$ the complex vector space of smooth 
sections of $TS\rightarrow S$ which vanish at the points in $M\cup |D|$.  We shall denote by $\Omega^{0,1}(\alpha)$ the complex vector space of smooth sections $h$  of $\text{Hom}_{\mathbb R}(TS,TS)\rightarrow S$
so that $h(z):T_zS\rightarrow T_zS$ is complex anti-linear.  The Cauchy-Riemann operator is naturally defined
$$
\bar{\partial}:\Gamma_0(\alpha)\rightarrow \Omega^{0,1}(\alpha).
$$
Namely taking a section $u\in\Gamma_0(\alpha)$ the section $\bar{\partial}(u)$ is defined by
$$
\bar{\partial}(u) (z) =\frac{1}{2}\cdot \left[ Tu(z) + j\circ Tu(z)\circ j\right],\ z\in S.
$$
Define the {\bf arithmetic genus} $g_{\text{arith}}$ of $\alpha$ by the formula 
$$
g_{\text{arith}} =1 +\sharp D +\sum_C[g(C)-1],
$$
where the sum is taken over all connected components $C$ of $S$.   As a consequence of the Riemann-Roch theorem
the stability assumption implies that $\bar{\partial}$ is injective and the complex codimension of its image
is $3\cdot g_{\text{arith}} +\sharp M-\sharp D -3$.  
\begin{definition}
The  {\bf natural deformation space for fixed combinatorial type}
of $\alpha$ is the complex vector space $H^1(\alpha)$ defined by
$$
H^1(\alpha) := \Omega^{0,1}(\alpha)/(\bar{\partial}(\Gamma_0(\alpha))).
$$
\qed
\end{definition}
The construction $H^1$ is natural and defines a functor from ${\mathcal R}$ to the category of finite-dimensional complex vector spaces.  Indeed, if $\Phi:\alpha\rightarrow \alpha'$ is a morphism, then the underlying
biholomorphic map $\phi:(S,j,M,D)\rightarrow (S',j',M',D')$ defines
$$
H^1(\Phi): H^1(\alpha)\rightarrow H^1(\alpha'): [\tau]\rightarrow [T\phi\circ\tau\circ T\phi^{-1}].
$$
Consequently the automorphism group $G$ will act on $H^1(\alpha)$.

  Assume that $U$ is an open subset
of some finite-dimensional real or complex vector space and $v\rightarrow j(v)$ is a smooth family
of almost complex structures on $S$, where $v\in U$, so that $v\rightarrow \alpha_v:=(S,j(v),M,D)$ is a family of objects.
Since $j(v)^2=-Id$ it follows that 
$$
Dj(v)(\delta v)\circ j(v) + j(v)\circ Dj(v)(\delta v)=0.
$$
For every $\delta v$, 
$Dj(v)(\delta v)$ is  complex antilinear and defines an element in $\Omega^{0,1}(\alpha_v)$.
Passing to the quotient 
we obtain an element in $H^1(\alpha_v)$ denoted by $[Dj(v)(\delta v)]$. 
Hence we obtain for every $v\in U$ an underlying real linear map
\begin{eqnarray}
[Dj(v)]: V\rightarrow H^1(\alpha_v).
\end{eqnarray}
\begin{definition}
We shall call $U\ni v\rightarrow [Dj(v)]$ the {\bf Kodaira-Spencer differential} of the deformation $v\rightarrow j(v)$.
\qed
\end{definition}
\subsubsection{Good Deformations and Uniformizers}
The crucial definition is the following.
\begin{definition}
Let $\alpha$ be an object in ${\mathcal R}$ with automorphism group $G$ and assume that a small disk structure ${\bm{D}}$ has been fixed. A {\bf good} deformation $\mathfrak{j}$  for $\alpha$ with given ${\bm{D}}$ consists of a $G$-invariant open neighborhood
$U$ of $0\in H^1(\alpha)$ and a smooth family $U\ni v\rightarrow j(v)$ of almost complex structures on $S$
so that the following holds.
\begin{itemize}
\item[(i)] $ j(0)=j$ and $j(v)=j$ on all $D_x$ for $x\in |D|$ and $v\in U$.
\item[(ii)] $[Dj(v)]:H^1(\alpha)\rightarrow H^1(\alpha_v)$ is a complex linear isomorphism for every $v\in U$.
\item[(iii)] For every $g\in G$ and $v\in U$ the map $g:\alpha_v\rightarrow \alpha_{g\ast v}$ is biholomorphic.
\end{itemize}
\qed
\end{definition}
It is a known fact that for given $\alpha$ and small disk structure a good deformation always exists.
Fix a good deformation $\mathfrak{j}$ for $\alpha$ and ${\bm{D}}$. Then we can define for a natural gluing parameter $\mathfrak{a}$ and $v\in U$ the objects $\alpha_{v,\mathfrak{a}}$ given by
$$
\alpha_{v,\mathfrak{a}} =(S_\mathfrak{a},j(v)_\mathfrak{a},M_\mathfrak{a},D_\mathfrak{a}).
$$
On $U\times {\mathbb B}_\alpha$ we have the natural action of $G$ by diffeomorphisms
defining us the translation groupoid $G\ltimes (U\times {\mathbb B}_\alpha)$. We also obtain a functor
\begin{eqnarray}\label{DM-ppsi}
\Psi:G\ltimes (U\times {\mathbb B}_\alpha)\rightarrow {\mathcal R}
\end{eqnarray}
which on objects maps $(v,\mathfrak{a})$ to $\alpha_{v,\mathfrak{a}}$ and a morphism $(g,(v,\mathfrak{a}))$
to the morphism $(\alpha_{v,\mathfrak{a}},g_\mathfrak{a},\alpha_{g\ast v,g\ast\mathfrak{a}})$.
The following theorem is well-known and a consequence of standard DM-theory and holds for every gluing profile.
A proof is given in \cite{HWZ-DM}.
\begin{thm}\label{DM_MAIN}
The orbit space $|{\mathcal R}|$ has a natural metrizable topology for which every connected component is compact. 
Moreover, the following holds, which also characterizes the topology.
Let $\alpha$, ${\bm{D}}$, and $\mathfrak{j}$ as described above and $ \Psi$ the associated functor defined in (\ref{DM-ppsi}).
Then there exists a $G$-invariant open neighborhood $O$ of $(0,0)$ in $H^1(\alpha)\times {\mathbb B}_\alpha$
such that the following holds.
\begin{itemize}
\item[(1)] $\Psi:G\ltimes O\rightarrow {\mathcal R}$ is fully faithful and injective on objects.
\item[(2)] The map $|\Psi|: |G\ltimes O|\rightarrow |{\mathcal R}|$ induced between orbit spaces
defines a homeomorphism onto an open neighborhood of $|\alpha|$.
\item[(3)] For every $(v,\mathfrak{a})$ the Kodaira-Spencer differential associated to $\alpha_{v,\mathfrak{a}}$
is an isomorphism.
\item[(4)] For every $(v,\mathfrak{a})\in O$ there exists an open neighborhood $\Lambda$ having the following property. If $(v_k,\mathfrak{a}_k)$ is a sequence in $\Lambda$ for which $|\Psi(v_k,\mathfrak{a}_k)|$ converges
to some point in $|{\mathcal R}|$, then it has a subsequence converging to some point in $O$.
\end{itemize}
\qed
\end{thm}
In view of Theorem \ref{DM_MAIN} we can make the following definition.
\begin{definition}
Let $\alpha$ be an object in ${\mathcal R}$. We say that $\Psi$ is a {\bf good uniformizer} associated to $\alpha$ if
it is obtained after a choice of small disk structure ${\bm{D}}$ and deformation $\mathfrak{j}$ 
as a restriction satisfying the properties (1)--(4). Given $\alpha$ there is a set of choices we can make, resulting 
in a set of associated uniformizers denoted by $F(\alpha)$. 
\qed
\end{definition}
If $\Phi:\alpha\rightarrow \alpha'$ is a morphism it is clear that there is a 1-1 correspondence between the 
choices for $\alpha$ and $\alpha'$, respectively. Hence we obtain a natural bijection $F(\Phi):F(\alpha)\rightarrow F(\alpha')$. This defines a functor
$$
F:{\mathcal R}\rightarrow \text{SET}
$$
associating to an object $\alpha$ a set of good uniformizers having $\alpha$ in the image.
As was already said the above construction works for every gluing profile. However, if we would like 
to have a smooth transition between two uniformizers we have to be careful in the choice of the gluing profile.

Assume that we have fixed a gluing profile. For two uniformizers $\Psi\in F(\alpha)$ and $\Psi'\in F(\alpha')$
consider the transition set $\bm{M}(\Psi,\Psi')$ defined by
$$
\bm{M}(\Psi,\Psi') = \{(o,\Phi,o')\ |\ o\in O,\ o'\in O',\ \Phi\in \text{mor}(\Psi(o),\Psi'(o'))\}.
$$
We not that we have two important natural maps, namely the {\bf source map}
$$
s:\bm{M}(\Psi,\Psi') \rightarrow O:(o,\Phi,o')\rightarrow o,
$$
and the {\bf target map}
$$
t:\bm{M}(\Psi,\Psi')\rightarrow O':(o,\Phi,o')\rightarrow o'
$$
Besides this there is the {\bf  inversion map}
$$
\iota:\bm{M}(\Psi,\Psi') \rightarrow \bm{M}(\Psi,'\Psi) :(o,\Phi,o')\rightarrow (o',\Phi^{-1},o)
$$
the {\bf unit map}
$$
u:O\rightarrow\bm{M}(\Psi,\Psi) : u(o)=(\Psi(o),1_{\Psi(o)},\Psi(o))
$$
and the {\bf multiplication map}
\begin{eqnarray}
&m: \bm{M}(\Psi',\Psi'') {_{s}\times_t}\bm{M}(\Psi,\Psi') \rightarrow \bm{M}(\Psi,\Psi''): &\\
&((o',\Phi',o''),(o,\Phi,o'))\rightarrow (o,\Phi'\circ\Phi,o'').&\nonumber
\end{eqnarray}
Irrespectively of the gluing profile one can show the following.
\begin{prop}
Every $\bm{M}(\Psi,\Psi')$ carries a natural metrizable topology. For this topology all structure maps
are continuous and  $s$ and $t$ are local homeomorphisms.
\end{prop}
We can equip each $\bm{M}(\Psi,\Psi')$ with the smooth manifold structure making $s$ a local diffeomorphism.
Observe that it even defines a holomorphic manifold structure. 
It depends now on the gluing profile if for this structure $t$ is a smooth map or a local diffeomorphism.
For example for the DM-gluing profile $\varphi(r)=-\frac{1}{2\pi}\cdot \ln(r)$ this will be the case and 
$t$ will be a local biholomorphism. This is a repacking of the Deligne-Mumford theory.
In the case of the exponential gluing profile $t$ will still be a local diffeomorphism. Depending on the gluing profile
all structure maps are in the first case holomorphic and in the second case smooth. There is a more detailed discussion
of this in \cite{H2014}. The uniformizer construction for ${\mathcal R}$ can be viewed as an example of a categorical polyfold construction, see \cite{HWZ2017}.

\subsubsection{Compatibility of Uniformizers and Universal Property}
Part of the discussion is taken from \cite{HWZ5}.   \jwf{ \marginnote{MAY BE DESCRIBE AS IN SC-SMOOTHNESS}}
Let $\Psi:G\ltimes O\rightarrow {\mathcal R}$ be a good uniformizer and let  \jwf{\marginnote{Possible change explaining the behavior in the ends, see also end, see 23.5}}
 $(S', j', M', D')$  be  another stable connected nodal  Riemann surface with  unordered marked points, and let  ${\bm{D}}'$  be a small disk structure  on $S'$. This time we allow the marked points in $M'$ to belong to $\abs{{\bm{D}}'}$.  Here 
 $$
 \abs{{\bm{D}}'} = \bigcup_{z\in |D'|} D_z'.
 $$
 Of course it holds that $M'\cap |D'|=\emptyset$.
In order to define a deformation $M'(\sigma)$ of $M'$ we choose around every point $z\in M'$ an open neighborhood $V(z)$ whose closure does not intersect $\abs{D'}$.  {Moreover, these sets $V(z)$ are chosen to be mutually disjoint. 
If $M'=\{z_1,\ldots, z_{m'}\}$, we introduce the product space 
$$V(M'):=V(z_1)\times \cdots V(z_{m'}).$$
A point $\sigma=(\sigma_1,\ldots ,\sigma_{m'})\in V(M')$ where $\sigma_j\in V(z_j)$, defines the set $M(\sigma')$ of marked points on $S'$, in the following called a 
{\bf deformation of the marked points $M'$}.} \index{deformation of marked points} 
We now take a $C^{\infty}$ neighborhood $U$ of the complex structure $j'$ and denote by $U_{{\bm{D}}'}$ the collection of complex structures $k$ in $U$ which coincide with $j'$ on the union of discs $\abs{{\bm{D}}'}$ of the small disk structure. Denote by ${\mathbb B}'$ the set of natural gluing parameters associated to $\alpha'$. We obtain 
a map associating to $(k,\mathfrak{b},\sigma)\in U_{{\bm{D}}'}\times{\mathbb B}' \times  V(M')$   the nodal Riemann surface 
$$
\beta_{(k,\mathfrak{b}, \sigma)}=(S'_{\mathfrak{b}}, k_{\mathfrak{b}}, M' (\sigma)_{\mathfrak{b}}, D'_\mathfrak{b}).
$$
 and study  
the  surfaces for  $\abs{\mathfrak{b}}$  small, $\sigma$  close to $\sigma_0$ satisfying  $M(\sigma_0)=M'$ and $k$  close to $j'$. Taking  a finite-dimensional family 
$$w\mapsto k(w)$$
of complex structures $k$,  which depend  smoothly  on a parameter $w$ varying  in a finite dimensional Euclidean space, we obtain  the map 
$$
(w,\mathfrak{b}, \sigma)\mapsto  \beta_{( k(w),\mathfrak{b}, \sigma)}=(S'_\mathfrak{b}, k(w)_\mathfrak{b}, M'(\sigma)_\mathfrak{b}, D'_\mathfrak{b})
$$
which we shall refer to as a {\bf smooth family}.  A proof of the following  result can be found in   \cite{HWZ-DM}.

\begin{thm}[{\bf Universal property}] \label{smoothfamily}
Let $\Psi:G\ltimes O\rightarrow {\mathcal R}$ be  good uniformizer 
giving us in particular the  family $(v,\mathfrak{a})\mapsto \alpha_{( v,\mathfrak{a})}$. Assume we are given a smooth family $(w,v,\mathfrak{b})\rightarrow  \beta_{( k(w),\mathfrak{b}, \sigma)}$ and there exists 
 an isomorphism $\Psi_0:\beta_{(k(w_0),\mathfrak{b}_0,\sigma_0)}\rightarrow \alpha_{(v_0,\mathfrak{a}_0)}$
 associated to a biholomorphic map $\psi_0$.
 Then  there exists a uniquely determined smooth germ
$$
(w,\mathfrak{b},\sigma)\mapsto  \mu (w, \mathfrak{b}, \sigma)=(v(w,\mathfrak{b},\sigma),a(w,\mathfrak{b},\sigma))
$$
near $(w_0, \mathfrak{b}_0, \sigma_0)$ satisfying 
$\mu (w_0, \mathfrak{b}_0, \sigma_0) =(v_0,\mathfrak{a}_0)$ and an associated core-smooth (defined below) germ
$(w,\mathfrak{b}, \sigma)\to \psi (w,\mathfrak{b}, \sigma)$ defining  isomorphisms 
$$
\Psi (w,\mathfrak{b}, \sigma):\beta_{(w,\mathfrak{b}, \sigma)}\to \alpha_{\mu (w,\mathfrak{b}, \sigma)}
$$
between the nodal Riemann surfaces satisfying $\phi (0, j', \sigma_0)=\psi$.
\qed
\end{thm}
 Core-smoothness is defined as follows, where 
 $$
\phi (k(w), \mathfrak{b},\sigma):\beta_{(k(w),\mathfrak{b},\sigma)}\rightarrow \alpha_{\mu ( k(w),\mathfrak{b}, \sigma)}
$$
is the family of biholomorphic maps between nodal Riemann surfaces.  If the point $z$ in $S'_{b_0}$ is not a nodal point, then its image $\zeta$ under 
$\phi (b_0, k(w_0),\sigma_0)$ in $S_{a(b_0, k(w_0),\sigma_0)}$ is also not a nodal point of the target surface $\alpha_{\mu (b_0, k(w_0),\sigma_0)}.$ If $z$ is in the core, a neighborhood of $z$ can canonically be identified with a neighborhood of $z$ viewed as a point in $S'$. If $z$ does not belong to a core, we find a nodal pair $\{x', y'\}\in D'$ and can identify $z$ with a point in $D_{x'}\setminus \{x'\}$ or with a point in $D_{y'}\setminus \{y'\}$ where $D_{x'}$ and $D_{y'}$ are discs belonging to the chosen small disk structure of the nodal surface $(S', j', M', D')$. The same alternatives hold for the image $\zeta$ of $z$ under the isomorphism $\phi (b_0, k(w_0),\sigma_0)$. Via these identifications, the family $\phi (b, k(w),\sigma)$ of isomorphisms gives rise to a  family of diffeomorphisms defined on a neighborhood of $z$ in $S'$ into a neighborhood of $\zeta$ in $S$. Being   core-continuous respectively core-smooth requires that all these germs of families  of isomorphisms are continuous respectively  smooth families of  local diffeomorphisms in the familiar sense,  between the fixed Riemann surfaces $S'$ and $S$.

We can define for two uniformizers the transition set $\bm{M}(\Psi,\Psi')$ similarly as discussed in Appendix 
\jwf{[broken reference]}
%\ref{APPX-I}. 
The universal property allows 
to equip this set with a smooth manifold structure characterized by the property that the source map $s:\bm{M}(\Psi,\Psi')\rightarrow O$ is a local diffeomorphism.
It will turn out that the universal property implies that the target map will be a local diffeomorphism as well. All other structure maps will be smooth. Hence we see that the uniformizer construction for ${\mathcal R}$ is a very special case of a polyfold construction.

 \subsection{Smoothness Properties of Maps between Buildings}
 Assume that $(S,j,,D)$ is a nodal compact Riemann surface without boundary.   Fix a small disk structure $\bm{D}$ and 
 a smooth deformation $\mathfrak{j}$ of $j$.
We obtain  the family of glued surfaces $(\mathfrak{a},v)\rightarrow \alpha_{(\mathfrak{a},v)}$.  
Define 
$$
\bm{S}= \left\{(\mathfrak{a},v,z)\ |\ v\in {\mathcal V},\ \mathfrak{a}\in {\mathbb B}_D,\ z\in S_{\mathfrak{a}}\right\}/\sim
$$
 Here we identify two points of the form $(\mathfrak{a},v,z)=(\mathfrak{a},v,z')$ provided $\{z,z'\}\in D_{\mathfrak{a}}$.
Assume that $(\mathfrak{a}_0,v_0,z_0)\in \bm{S}$ is given and $z_0$ is not a nodal point on $S_{\mathfrak{a}_0}$. 
We would like to define the notion of an sc-smooth local section.  Let us consider the following cases.
Assume first that $z_0$ belongs to the interior of a neck which is glued by a nontrivial gluing parameter
$a^0_{\{x,y\}}=|a^0_{\{x,y\}}|\cdot [\wh{x}_0,\wh{y}_0]$. Set  $\varphi(|a^0_{\{x,y\}}|)=R_0$ and denote by $\wh{x}_0$ and $\wh{y}_0$ 
representatives of $[\wh{x}_0,\wh{y}_0]$.  Hence we can write $z_0=\sigma_{\wh{x}_0}(s_0,t_0)$ with $s_0\in (0,R_0)$
or as $z_0=\sigma^-_{\wh{y}_0}(s_0',t_0')$ with $s_0'\in (-R_0,0)$. A local section has the form 
$(\mathfrak{a},v,z(\mathfrak{a},v))$ is said to be sc-smooth if represented by $\sigma^+_{\wh{x}_0}$ or 
$\sigma^-_{\wh{y}_0}$ is smooth, i.e. if we write
$$
z(\mathfrak{a},v) =\{\sigma^+_{\wh{x}_0}(s,t),\sigma^-_{\wh{y}_0}(s',t')\}
$$
the maps $(s,t)$ or $(s',t')$ are smooth.\jwf{\marginnote{Have to add a comment to the definition of the glued surface concerning the representation of neighboring surfaces  with the same choice $(\wh{x}_0,\wh{y}_0)$}}

\jwf{ \textcolor{red}{MAY BE a nice explanation about the cylinder stuff from sc-smoothness or GW-paper
 and extrapolation to DM.   Also action on maps. 
 We go somewhat over the middle and cut-off the basemap and the map above with or without asymmpotic limit
 we cut off a little bit before that}}

In this appendix we collect known facts and introduce notation and notions which are used throughout this paper.

\section{Structures Associated to Riemann Surfaces}
The main prurpose of this appendix is to recall basic facts about Riemann surfaces and most importantly fix notions
and notation which will be used throughout this paper as well as research papers, most  notably\cite{FH0} \jwf{and \cite{FH1}},
which use the results in this lecture note.
\subsection{Basic Notions}\label{APP5.1}\label{SEC_basic_notions}
We recall some facts feeding into the DM-theory and refer the reader to \cite{HWZ5} for more detail, particularly with respect to the modified version using the exponential gluing profile.
For the latter there are also important details in \cite{HWZ8.7}.

\subsubsection*{Disk Pairs}
We are interested in disk-like Riemann surfaces $D_x$ with smooth boundary and interior point $x$. We refer to $x$ as a {\bf nodal point}.  An {\bf un-ordered nodal disk pair}
${\mathcal D}$ has the form $(D_x\sqcup D_y,\{x,y\})$, where $D_x$ and $D_y$ are as just described. An {\bf ordered nodal disk pair}
has the form $(D_x\sqcup D_y,(x,y))$. The ordered pair $(x,y)$ is called an {\bf ordered nodal pair} and $\{x,y\}$ is called an {\bf un-ordered nodal pair}.
In the case $(x,y)$ we shall refer to $x$ as the {\bf lower nodal point} and $y$ as the {\bf upper nodal point}.

Given $(D_x,x)$, a {\bf decoration}  $\wh{x}$ of the nodal point $x$ is a an oriented real line $\wh{x}\subset T_xD_x$.
The circle 
 $S^1={\mathbb R}/{\mathbb Z}$ acts naturally on the tangent spaces using their complex structures and therefore
 it acts also on the possible decorations for $x$ by 
\begin{eqnarray}
 (\theta,\wh{x})\rightarrow \theta\ast \wh{x}:=e^{2\pi i\theta}\cdot \wh{x}.
\end{eqnarray}
 Next we consider unordered pairs $\{\wh{x},\wh{y}\}$ which we call a {\bf decorated unordered nodal pair} or a {\bf decoration}
 of the nodal pair $\{x,y\}$.
 We declare $\{\wh{x},\wh{y}\}$  to be {\bf equivalent} to
$\{\theta\ast \wh{x},\theta^{-1}\ast \wh{y}\}$ where $\theta\in S^1$. The symbol $[\wh{x},\wh{y}]$
denotes the equivalence class associated to $\{\wh{x},\wh{y}\}$
\begin{eqnarray}
[\wh{x},\wh{y}] =\left\{ \{\theta\cdot\wh{x},\theta^{-1}\cdot \wh{y}\}\ |\ \theta\in S^1\right\}.
\end{eqnarray}
We call $[\wh{x},\wh{y}]$ a  {\bf natural angle} or {\bf argument} associated to $\{x,y\}$.
We denote by ${\mathbb S}_{\{x,y\}}$ the collection of all $[\wh{x},\wh{y}]$  associated to $\{x,y\}$ 
and call it  the {\bf set of  arguments or angles} associated to $\{x,y\}$.
Denote by ${\mathbb S}^1$ the standard unit circle in ${\mathbb C}$.
Fixing $z=[\wh{x}_0,\wh{y}_0]$ the map 
\begin{eqnarray}
\text{ar}_z: {\mathbb S}_{\{x,y\}}\rightarrow {\mathbb S}^1: [\theta\ast \wh{x}_0,\theta'\ast \wh{y}_0]\rightarrow e^{2\pi i(\theta+\theta')}
\end{eqnarray}
 is a bijection and
any two such maps, say $\text{ar}_z$ and $\text{ar}_{z'}$  have a transition map $\text{ar}_{z'}\circ \text{ar}_z^{-1}$ which is a rotation on ${\mathbb S}^1$.
Hence ${\mathbb S}_{\{x,y\}}$ has a natural smooth structure. It also has a natural orientation by requiring that $\text{ar}_z$ is orientation preserving,
where ${\mathbb S}^1$ is equipped with the orientation as a boundary of the unit disk.

Consider formal expressions $r\cdot [\wh{x},\wh{y}]$, where $r\in [0,1/4)$ (the choice of 1/4 has no deeper meaning other
than that certain constructions need a bound on the choice of $r$ and in our case $1/4$ is always a good bound). We say that $r\cdot[\wh{x},\wh{y}]=r'\cdot [\wh{x}',\wh{y}']$ provided either  $r=r'=0$,  or $r=r'>0$ and $[\wh{x},\wh{y}]=[\wh{x}',\wh{y}']$. In the following we shall call $r\cdot [\wh{x},\wh{y}]$ a {\bf natural gluing parameter} associated  to $\{x,y\}$. The collection ${\mathbb B}_{\{x,y\}}$  of these formal gluing parameters $r\cdot [\wh{x},\wh{y}]$ has a natural one-dimensional holomorphic 
manifold structure, so that fixing any $\{\wh{x}_0,\wh{y}_0\}$ the map
$$
r\cdot [\theta\cdot \wh{x}_0,\theta'\cdot \wh{y}_0]\rightarrow r\cdot e^{2\pi i(\theta +\theta')}
$$
 onto the standard
open disk in ${\mathbb R}$ of radius $1/4$ is a biholomorphic map.

When we deal with a finite number of disk pairs we can take their unordered or ordered 
nodal pair as an index set. We shall for example write $D$ for the whole collection of all occurring $\{x,y\}$
and we shall write ${\bm{D}}$ for the collection, i.e.
\begin{eqnarray}
{\bm{D}}=\left\{ (D_x\sqcup D_y,\{x,y\})\ |\ \{x,y\}\in D\right \}.
\end{eqnarray}
Sometimes, always clear from the context, we also view $\bm{D}$ as defining the disjoint union of all $D_x$, where $x$ varies
over $|D|=\cup_{\{x,y\}\in D} \{x,y\}$, together with the collection $D$ of nodal pairs
\begin{eqnarray}
\bm{D}=  \left(\coprod_{z\in |D|} D_z,\ D\right).
\end{eqnarray}
This is a specific compact nodal Riemann surface with smooth boundary.
Associated to every $\{x,y\}$ we have the set of natural gluing parameter ${\mathbb B}_{\{x,y\}}$
and we shall write ${\mathbb B}_D$ for the set of total gluing parameters, which are maps
$\mathfrak{a}$ associating to $\{x,y\}$ an element $a_{\{x,y\}}\in {\mathbb B}_{\{x,y\}}$. We can view ${\mathbb B}_D$ as sections of  a bundle over the finite set $D$, namely
$$
\coprod_{\{x,y\}\in D}{\mathbb B}_{\{x,y\}}   \rightarrow D: a_{\{x,y\}}\rightarrow \{x,y\}.
$$
 From this viewpoint a natural gluing parameter is a section.

\subsubsection*{Gluing Disks}

Consider an un-ordered nodal disk pair ${\mathcal D}:=(D_x\sqcup D_y,\{x,y\})$, consisting  of disk-like Riemann surfaces  $D_x$ and $D_y$, with smooth boundaries
containing the interior points $x$ and $y$, respectively, so that $(D_x,x)$ and $(D_y,y)$ are biholomorphic 
to $({\mathbb D},0)$, where ${\mathbb D}\subset {\mathbb C}$ is the closed unit disk. These biholomorphic maps are not unique but any two of them differ by a rotation which is biholomorphic. Denote for $0\in {\mathbb D}$
by $\wh{0}$ the standard decoration given by ${\mathbb R}\subset T_0{\mathbb D}$ with the standard orientation of the real numbers. If $\wh{x}$ is a decoration of $x$ there exists a unique biholomorphic map
$$
h_{\wh{x}}:(D_x,\wh{x})\rightarrow ({\mathbb D},\wh{0}).
$$
In the following we need the {\bf exponential gluing profile} $\varphi:(0,1]\rightarrow [0,\infty)$ defined by 
$$
\varphi(r)=e^{\frac{1}{r}}-e.
$$
\begin{definition}
Given an unordered disk pair $(D_x\sqcup D_y,\{x,y\})$ and a non-zero gluing parameter $a_{\{x,y\}}=r\cdot [\wh{x},\wh{y}]$
define  the set $Z_{a_{\{x,y\}}}$ by
\begin{eqnarray}\label{KOPTY20.1}
Z_{a_{\{x,y\}}}&=&\left\{ \{z,z'\}\ |\  z\in D_x,\ z'\in D_y,\ \right.\\
&&\left.\phantom{XXXX} h_{\wh{x}}(z)\cdot h_{\wh{y}}(z') = e^{-2\pi \varphi(r)}\right\}.\nonumber
\end{eqnarray}
Here  $\{\wh{x},\wh{y}\}$  is a representative of $[\wh{x},\wh{y}]$, but the 
 definition of the set does not depend on its choice.  If the gluing parameter vanishes, i.e. if $a_{\{x,y\}}=0$
 we define $Z_0=(D_x\sqcup D_y,\{x,y\})$.   We note that $Z_a=Z_b$ if and only if $a=b$. 
 $Z_{a_{\{x,y\}}}$ is said to be obtained from $(D_x\sqcup D_y,\{x,y\})$ by gluing with gluing parameter
 $a_{\{x,y\}}$.
 \qed
 \end{definition}
 \begin{remark}
 We use this special gluing profile $\varphi$ in order to have compatibility with the sc-Freholm theory.
We obtain the classical Deligne-Mumford theory when we use the gluing profile $r\rightarrow -\frac{1}{2\pi}\cdot \ln(r)$.
 \qed
 \end{remark}
 Given a non-zero gluing parameter $a_{\{x,y\}} =r_{\{x,y\}}\cdot [\wh{x},\wh{y}]$ put $R=\varphi(r_{\{x,y\}})$
 and define the closed annuli $A_x(R)\subset D_x$ and $A_y(R)\subset D_y$ of modulus $2\pi R$  by 
\begin{eqnarray*}
 A_x(R)&=&\left\{ z\in D_x\setminus\{x\}\ |\ |h_{\wh{x}}(z)|\geq e^{-2\pi R}\right\}\\
  A_y(R)&=&\left\{ z'\in D_y\setminus\{y\}\ |\ |h_{\wh{y}}(z')|\geq e^{-2\pi R}\right\}.
\end{eqnarray*}
 The set $Z_{a_{\{x,y\}}}$ defined in (\ref{KOPTY20.1}) for non-zero gluing parameter  has a natural holomorphic manifold structure making it biholomorphic to a closed annulus of modulus $2\pi \cdot \varphi(r_{\{x,y\}})$ so that in addition the maps
\begin{eqnarray}\label{EQNA1.1}
 A_x(R) \xleftarrow{\pi_x^{a_{\{x,y\}}}} Z_{a_{\{x,y\}}}\xrightarrow{\pi_y^{a_{\{x,y\}}} }A_y(R)
\end{eqnarray}
 defined by $\pi_x(z,z')=z$ and $\pi_y(z,z')=z'$  are biholomorphic. Hence
 \begin{lem}
 $Z_{a_{\{x,y\}}}$ has a natural structure as a Riemann surface.\qed
 \end{lem}
 Assume that $a_{\{x,y\}}$ and $a'_{\{x,y\}}$ are two nonzero gluing parameters in ${\mathbb B}_{\{x,y\}}$ with the same modulus. We abbreviate
 $$
 R:=\varphi(|a_{\{x,y\}}|)=\varphi(a_{\{x',y'\}}|).
 $$
 In this case we obtain two copies of the diagram (\ref{EQNA1.1}), say, with $a=a_{\{x,y\}}$ and $a'=a'_{\{x,y\}}$
\begin{eqnarray}
&A_x(R) \xleftarrow{\pi^a_x} Z_{a}\xrightarrow {\pi^a_y} A_y(R)&\\
&A_x(R) \xleftarrow{\pi^{a'}_x} Z_{a'}\xrightarrow {\pi^{a'}_y} A_y(R).&\nonumber
\end{eqnarray}
We can compare the following two maps $A_x(R)\rightarrow A_y(R)$
$$
\pi_y^a\circ {(\pi^a_x)}^{-1}\ \ \text{and}\ \ \pi_y^{a'}\circ {(\pi^{a'}_x)}^{-1}.
$$
Given $(D_x,x)$ there is a well-define notion of a {\bf rotation} by $\theta\in S^1$. Namely take any biholomorphic map
$h: (D_x,x)\rightarrow ({\mathbb D},0)$ and define $R^x_\theta(z) = h^{-1}( e^{2\pi i\theta}\cdot h(z))$. This definition does not depend on the choice of $h$, and it follows immediately that the following identity holds.
\begin{eqnarray}
h_{\wh{x}} \circ  R^x_{\theta}  =   h_{e^{2\pi i\theta}\cdot \wh{x}}  = e^{2\pi i\theta}\cdot h_{\wh{x}}.
 \end{eqnarray}
We obtain the following lemma, which can be verified by a straight forward calculation.
 \begin{lem}
 Wring $a=|a|\cdot [\wh{x},\wh{y}]$ and $a'=|a'|\cdot [\wh{x}, e^{2\pi i\theta}\cdot \wh{y}]$, where  $|a|=|a'|\neq 0$, it holds that 
 $$
 \pi_y^a\circ {(\pi^a_x)}^{-1} = R_{\theta}^y  \circ\pi_y^{a'}\circ {(\pi^{a'}_x)}^{-1}.
$$
\end{lem}
\begin{proof}
By definition $Z_a=\{\{z,z'\}\ |\ h_{\wh{x}}(z)\cdot h_{\wh{y}}(z')=e^{-2\pi \varphi(r)}\}$ and with $a'=|a| \cdot [\wh{x},e^{2\pi i\theta}\cdot \wh{y}]$ we see that
\begin{eqnarray}
Z_{a'} &= &\{\{z,z'\}\ |\ h_{\wh{x}}(z)\cdot e^{2\pi i\theta}\cdot h_{\wh{y}}(z')= e^{-2\pi \varphi(r)}\}\\
&=& \{\{z,z'\}\ |\ h_{\wh{x}}(z) \cdot h_{\wh{y}}(R_{\theta}(z'))= e^{-2\pi \varphi(r)}\}.\nonumber\\
&=& \{\{z,R_{-\theta}(z')\}\ |\ h_{\wh{x}}(z) \cdot h_{\wh{y}}(z')= e^{-2\pi \varphi(r)}\}.\nonumber
\end{eqnarray}
From this it follows that if  $  \pi_y^a\circ {(\pi^a_x)}^{-1} (z)= z'$ then  $\pi_y^{a'}\circ {(\pi^{a'}_x)}^{-1}(z) = R_{-\theta}(z')$, and consequently 
$$
R_{\theta}\circ \pi_y^{a'}\circ {(\pi^{a'}_x)}^{-1}(z) =z' =  \pi_y^a\circ {(\pi^a_x)}^{-1} (z).
$$
\end{proof}

 \begin{definition}
If $a=a_{\{x,y\}}\neq 0$ with $R=\varphi(|a_{\{x,y\}}|)$ we denote by $M^p_{{\mathcal D},a}$ the collection of all $\{z,z'\}\in Z_a$
 such that $-\frac{1}{2\pi}\cdot \ln(|h_{\wh{x}}(z)|\in (R/2 -p/2,R/2+p/2)$. Note that this is only well-defined if $|a_{\{x,y\}}|$ is sufficiently small given $p$. 
 We call $M^p_{{\mathcal D},a}$ the {\bf middle annulus of width} $2p$ of $Z_a$.
 \qed
 \end{definition}
 For example with $a\in {\mathbb B}_{\{x,y\}}$ this is well-defined as long as $0<p<25$.
 When we are given a finite family of unordered disk pairs ${{\bm{D}}}$ and a gluing parameter
 $\mathfrak{a}\in {\mathbb B}_D$,  we denote by ${{\bm{Z}}}_{\mathfrak{a}}$ the disjoint union of all
$Z^{\{x,y\}}_{a_{\{x,y\}}}$. In the case that $a_{\{x,y\}}=0$ we have that $Z^{\{x,y\}}_{a_{\{x,y\}}}=(D_x\sqcup D_y,\{x,y\})$ and consequently, if $\mathfrak{a}\equiv 0$
we see that ${{\bm{Z}}}_{\mathfrak{a}}=\bm{D}$.

 In the case of  ordered disk pairs we use a similar formalism. It will be clear from the context in which situation we are.
\subsubsection*{Holomorphic Polar Coordinates}\label{SEC_holo_polar_coords}
Given $(D_x,x)$  let $\wh{x}$ be a decoration.
 Take the associated $h_{\wh{x}}:(D_x,x)\rightarrow ({\mathbb D},0)$ satisfying $Th_x(\wh{x})=\wh{0}$.
 We introduce the biholomorphic maps 
 $$
 \sigma_{\wh{x}}^+:[0,\infty)\times S^1\rightarrow D_x\setminus\{x\} :(s,t)\rightarrow h_{\wh{x}}^{-1}\left(e^{-2\pi (s+it)}\right)
 $$
 and
 $$
 \sigma^-_{\wh{x}}:(-\infty,0]\times \rightarrow D_x\setminus\{x\}:(s',t')\rightarrow h_{\wh{x}}^{-1}\left( e^{2\pi (s'+it')}\right).
 $$
 We shall call $\sigma_{\wh{x}}^\pm$ {\bf positive and negative holomorphic polar coordinates} on $D_x$ around $x$ 
 associated to the decoration $\wh{x}$.
 
If $(D_x\sqcup D_y,\{x,y\})$ is nodal disk pair and $a\in {\mathbb B}_{\{x,y\}}$ we obtain through gluing the space
 $Z_a$. With $a=|a|\cdot [\wh{x},\wh{y}]$  fix  a representative  $\{\wh{x},\wh{y}\}$.  We have the special biholomorphic maps,
 where $R=\varphi(|a|)$
\begin{eqnarray}\label{ERT31}
&&\\
&\sigma_{\wh{x}}^{+,a}:  [0,R]\times S^1\rightarrow Z_a: (s,t) \rightarrow \{\sigma^+_{\wh{x}}(s,t), \sigma^-_{\wh{y}}(s-R,t)\}&\nonumber\\
&\sigma_{\wh{y}}^{a,-}:[-R,0]\times S^1\rightarrow Z_a: (s',t')\rightarrow \{\sigma^+_{\wh{x}}(s'+R,t'),\sigma^-_{\wh{y}}(s',t')\}&\nonumber
\end{eqnarray}
 There are also maps $\sigma_{\wh{x}}^{-,a}$ and $\sigma_{\wh{y}}^{+,a}$ obtained by interchanging the roles of 
 $x$ and $y$.  In the case of an ordered disk pair the maps in (\ref{ERT31}) are the relevant ones, i.e. we take poistive holomorphic polar coordinates
 for the lower disk and negative one for the upper disk.
 
As part of the constructions we shall consider maps $u:Z_a\rightarrow {\mathbb R}^N$ and  sometimes need to evaluate the average 
 over the loop in the middle. For this we can pick a nodal point in $\{x,y\}$, say $x$,  and take $\sigma^{+,a}_{\wh{x}}$ and calculate with $R=\varphi(|a|)$
 $$
 \int_{S^1} u\circ\sigma^{+,a}_{\wh{x}}(R/2,t) dt.
 $$
 The integral does not depend on the choice of $x$ or $y$ in $\{x,y\}$, and after the choice of $x$, it does not depend
 on the decoration $\wh{x}$.  We call the integral
 the {\bf middle loop average}.
 We call any of the maps $t\rightarrow \sigma^{\pm,a}_{\wh{x}}(\pm R/2,t)$ or $t\rightarrow \sigma^{\pm,a}_{\wh{y}}(\pm R/2,t)$  a  {\bf middle-loop map}. 
 If $a=0$ and $u$ is defined on the disk pair, being continuous over the nodal value (i.e. $u(x)=u(y)$),  then we can define the associated 
 middle loop average as $u(x)$ or $u(y)$, which are the same. 

A related concept is that of an {\bf $a$-loop}.  Assume we have a map defined on $D_x$ or the punctured $D_x\setminus\{x\}$.
Pick positive holomorphic polar coordinates centered at $x$, i.e.
$$
\sigma^+:{\mathbb R}^+\times S^1\rightarrow D_x\setminus\{x\}.
$$
Assume that $u:D_x\rightarrow {\mathbb R}^N$ is a continuous map and $a\in {\mathbb B}$ a nonzero gluing parameter.
We define an {\bf $a$-loop} as the map
$$
S^1\rightarrow D_x\setminus\{x\}:  t\rightarrow \sigma^+_a(t):= \sigma^+(R/2,t)
$$
where $R=\varphi(|a|)$. There is a whole $S^1$-family of $a$-loops. However the integral
$$
\int_{S^1}  u\circ\sigma_a^+(t)\cdot dt
$$
does not depend on the choice of the specific $a$-loop.

\subsection{Riemann Surfaces}\label{APPENDIX-RS}
After some preparation we shall  describe the category of stable Riemann surfaces and introduce auxiliary structures
for the DM-theory.

\subsubsection{Nodal Riemann Surfaces}
We shall consider tuples $(S,j,M,D)$ consisting of a compact  Riemann surface $(S,j)$ without boundary,
but possibly disconnected, where  $M$ is a finite (unordered) subset of $S$ called {\bf (unordered) marked points}, and $D$ is a finite collection  of unordered pairs $\{x,y\}$, where $x,y\in S$ are different points. We require that $D$ has the property that $\{x,y\}\cap \{x',y'\}\neq \emptyset$ implies that
$\{x,y\}=\{x',y'\}$. We shall write $|D|$ for the union of all the $\{x,y\}$, i.e.
$$
|D| =\bigcup_{\{x,y\}\in D} \{x,y\},
$$
 and require that $|D|\cap M=\emptyset$. We refer to the elements $\{x,y\}$ as {\bf nodal pairs} and to $x$ and $y$ as {\bf nodal points}.  We can view the tuples as objects of a category.
For the following discussion we assume $M$ and the elements of $D$ to be unordered.
A morphism $\Phi:\alpha\rightarrow \alpha'$ is given by a tuple $\Phi=(\alpha,\phi,\alpha')$,
where $\phi:(S,j)\rightarrow (S',j')$ is a biholomorphic map such that $\phi(M)=M'$  and $\phi_\ast(D)=D'$,
where
$$
\phi_\ast (D) =\{\{\phi(x),\phi(y)\}\ |\ \{x,y\}\in D\}.
$$
We denote the category with objects $(S,j,M,D)$ and morphisms $\Phi$ by $\bar{\mathcal R}$.
 \begin{remark}
We shall sometimes consider modifications, namely we may allow $M$ to be ordered and in this case referred to as the set of {\bf ordered}
marked points. We also sometimes allow some of the nodal pairs to be ordered, i.e. the object are $(x,y)$ rather than $\{x,y\}$ and we refer
to an {\bf ordered} nodal pair. 
\qed
\end{remark}
We shall discuss later on in more detail objects in $\bar{\mathcal R}$ together with a finite group 
$G$ acting on it by biholomorphic maps preserving the additional structure. 

\subsubsection{The Category of Stable Riemann Surfaces}\label{RRRR---}
Denote by ${\mathcal R}$ the  full subcategory  of $\bar{\mathcal R}$ associated to objects, which satisfy 
an additional condition. Namely we impose the {\bf stability condition} (\ref{STABBB}) that for every connected component $C$ of $S$ its genus
$g(C)$ and the number $\sharp C:= C\cap (M\cup |D|)$ satisfies
\begin{eqnarray}\label{STABBB}
2g(C) + \sharp C\geq 3.
\end{eqnarray}
 From the classical Deligne-Mumford theory, see \cite{DM,RS} and also \cite{HWZ-DM}, it follows that
${\mathcal R}$ is what we shall call a  {\bf groupoidal category}. Namely every morphism is an isomorphism, between two objects are at most finitely morphisms (a consequence of the stability condition),
and the collection of isomorphism classes $|{\mathcal R}|$ is a set. It is also an important fact that $|{\mathcal R}|$  has a natural metrizable topology for which the connected components are compact. We shall call ${\mathcal R}$ the {\bf category of stable Riemann surfaces} with unordered marked points and nodal pairs.

\subsubsection{Glued Riemann Surfaces} 
Let $\alpha$ be an object in $\bar{\mathcal R}$ and denote by $G$ a finite 
group  acting by automorphisms of $\alpha$.   
\begin{definition}
A pair $(\alpha,G)$, where $\alpha$ is an object in $\bar{\mathcal R}$ and $G$ a finite 
group acting by automorphisms will be called a {\bf Riemann surface with a finite group action}.
\qed
\end{definition}
\begin{remark}
Note that the biholomorphic automorphism group might be infinite. However, $G$ utilizes only a finite part of the existing 
symmetries. An obvious example is the Riemann sphere, where we can take a finite subgroup of the rational transformations.
\qed
\end{remark}
Assume $(\alpha,G)$ is given, where $\alpha$ is an object in $\bar{\mathcal R}$.
Define ${\mathbb B}_\alpha$
by
$$
{\mathbb B}_\alpha =\prod_{\{x,y\}\in D} {\mathbb B}_{\{x,y\}},
$$
which as a product of one-dimensional complex manifolds is a complex manifold.  There is a natural projection
$\pi:{\mathbb B}_\alpha\rightarrow D$ and 
a {\bf gluing parameter} for an object $\alpha$ is a section $\mathfrak{a}$ of $\pi$, i.e. it  associates to $\{x,y\}\in D$
a symbol $a_{\{x,y\}}\in {\mathbb B}_{\{x,y\}}$
$$
\mathfrak{a}: D\rightarrow {\mathbb B}_\alpha: \{x,y\}\rightarrow a_{\{x,y\}}.
$$
The natural action of $G$ on $D$ by $g\ast \{x,y\}=\{g(x),g(y)\}$ lifts to a {\bf natural action}
of $G$ on the complex manifold of natural gluing parameters
$$
G\times {\mathbb B}_\alpha\rightarrow {\mathbb B}_\alpha,
$$
by $g\ast \mathfrak{a}=\mathfrak{b}$, where, writing $a_{\{x,y\}}=r_{\{x,y\}}\cdot [\wh{x},\wh{y}]$ we have
$$
b_{\{g(x),g(y)\}} = r_{\{x,y\}}\cdot [(Tg)\wh{x},(Tg) \wh{y}]\ \ \text{for}\ \ \{x,y\}\in D.
$$
We fix for every $z\in |D|$
a closed disk-like neighborhood $D_z$ with smooth boundary and $z$ an interior point
so that the union of these $D_z$ is invariant under $G$. We also require that $M\cap D_z=\emptyset$
for all $z\in |D|$. This choice gives for every $\{x,y\}\in D$ an unordered nodal disk pair 
${\mathcal D}_{\{x,y\}}=(D_x\sqcup D_y,\{x,y\})$. 
\begin{definition}
The collection ${{\bm{D}}}$ of these disk pairs,
having the properties stated above is called a {\bf small disk structure} for $\alpha$. 
\end{definition}
In case we have an ordered nodal pair $(x,y)$ we obtain an ordered disk pair written as $(D_x\sqcup D_y,(x,y))$.
We shall refer to $D_x$ as the {\bf lower disk} and $D_y$ as the {\bf upper disk}. We usually would also assume that
the action of $G$ would map an ordered nodal pair to an ordered nodal pair and also preserve the ordering. 

Given an object $\alpha=(S,j,M,D)$ in $\bar{\mathcal R}$ and a small disk structure 
${{\bm{D}}}$ it is convenient to note that given ${\bm{D}}$ we can recover $D$.  Hence, we introduce the
objects $(S,j,M,{\bm{D}})$ which are compact Riemann surfaces with small disk structure 
as well as $((S,j,M,{\bm{D}}),G)$, which is $((S,j,M,D),G)$ equipped with a small disk structure 
so that the union of the disks is invariant.

Assume that $\alpha=(S,j,M,D)$ is a nodal stable Riemann surface with unordered marked points and nodal points and $G$ is a finite group acting on $\alpha$ as previously described.  Fix a small disk structure ${\bm{D}}$, which for every $\{x,y\}$ gives us an unordered nodal disk pair $(D_x\sqcup D_y,\{x,y\})$.  Hence we consider $((S,j,M,{\bm{D}}),G)$.  Given a gluing parameter $\mathfrak{a}$ for $\alpha$ we obtain the  $a_{\{x,y\}}$ and obtain by disk-gluing
 $$
 Z^{\{x,y\}}_\mathfrak{a} := Z_{a_{\{x,y\}}},
 $$
 which is obtained from $(D_x\sqcup D_y,\{x,y\})$ by gluing with $a_{\{x,y\}}$. If $a_{\{x,y\}}=0$ we recover the nodal
 disk pair. Remove from $S$ for every nonzero $a_{\{x,y\}}$ the complement in $D_x\sqcup D_y$ 
 of $A_x(R)\sqcup A_y(R)$, where $R=\varphi(|a_{\{x,y\}}|)$.
 We define a new surface $S_{\mathfrak{a}}$  by using for every $\{x,y\}$ with non-zero gluing parameter
   (\ref{EQNA1.1}) to define a holomorphic equivalence relation on $A_x(R)\sqcup A_y(R)$
   identifying $z\in A_x(R)$ with $z'\in A_y(R)$ provided $\{z,z'\}\in Z_\mathfrak{a}^{\{x,y\}}$.
 If the gluing parameter for some $\{x,y\}$ vanishes we do not do anything.
Having carried out this for every nodal pair  we obtain a nodal Riemann surface  surface $S_\mathfrak{a}$
with associated almost complex structure $j_\mathfrak{a}$.  We denote by $D_\mathfrak{a}$ the collection of all
 $\{x,y\}\in D$ with $a_{\{x,y\}}=0$. We shall write $M_\mathfrak{a}$ for the set $M$ viewed as a subset of $S_\mathfrak{a}$. Finally we set
 $$
 \alpha_\mathfrak{a}=(S_\mathfrak{a},j_\mathfrak{a},M_\mathfrak{a},D_\mathfrak{a}).
 $$
\begin{definition}
 We call $\alpha_\mathfrak{a}$ the stable Riemann surface {\bf obtained from $\alpha$ by gluing with $\mathfrak{a}$}.
 \qed
 \end{definition}
For $g\in G$ one easily verifies that the construction of $S_\mathfrak{a}$ allows to construct in a natural way a biholomorphic map
 $g_\mathfrak{a}:\alpha_\mathfrak{a}\rightarrow \alpha_{g\ast\mathfrak{a}}$. Given $h,g\in G$ we have the functorial properties
 $$
 h_{g\ast \mathfrak{a}}\circ g_\mathfrak{a} = (h\circ g)_{\mathfrak{a}}
 $$
 and $1_\mathfrak{a} = Id_{\alpha_\mathfrak{a}}$.

\subsection{Riemann Surface Buildings}
The building blocks are tuples $(\Gamma^-,S,j,D,\Gamma^+)$, where $(S,j)$ is a not necessarily connected compact Riemann surface, $D$ is a nodal set,
and $\Gamma^\pm$ are a finite set of positive and negative punctures.  The sets $|D|$, $\Gamma^+$, and $\Gamma^-$ are mutually disjoint. Let us denote such an object
by $\alpha$.  We allow finite groups $G$ acting on such a $\alpha$ by biholomorphic maps, where we impose the restriction that $G$ preserves the sets $\Gamma^+$, $\Gamma^-$,
and the nodal pairs $D$. A small disk structure ${\bm{D}}$ for $\alpha$ consists of a small disk structure associated to the nodal pairs in $D$ so that the union of the 
disks is invariant under $G$. The disks are assumed to be mutually disjoint and not to contain the points in $\Gamma^\pm$.  
Denoting by ${\mathbb B}_{\alpha}$ the set of natural gluing parameters we obtain through gluing $\alpha_{\mathfrak{a}}=(\Gamma^-,S_\mathfrak{a},j_\mathfrak{a},D_\mathfrak{a},\Gamma^+) $, where we identify $\Gamma^\pm$ naturally as a subset of $S_\mathfrak{a}$.  We shall call $\alpha_\mathfrak{a}$ a ${\bm{D}}$-descendent of $\alpha$.
It is convenient to consider the smooth manifold of gluing parameters ${\mathbb B}_D$ together with the $G$-action as a translation groupoid
$G\ltimes {\mathbb B}_D$.  We can also consider the groupoid whose objects are the glued $\alpha_\mathfrak{a}$ and the morphisms 
are the $(\alpha_{\mathfrak{a}},g_{\mathfrak{a}},\alpha_{g\ast\mathfrak{a}})$.  Obviously the two groupoids are isomorphic via $\mathfrak{a}\rightarrow \alpha_\mathfrak{a}$
and $(g,\mathfrak{a})\rightarrow (\alpha_\mathfrak{a},g_\mathfrak{a},\alpha_{g\ast\mathfrak{a}})$
We also note that 
$G$ defines actions $G\times \Gamma^\pm\rightarrow \Gamma^\pm$.  Denote by $G\ltimes \Gamma^\pm$ the associated translation groupoids, so hat we have the equivariant
diagram of inclusions
$$
\Gamma^+ \rightarrow \alpha_{\mathfrak{a}}\leftarrow \Gamma^-
$$
We generalize this now as follows.
We first consider tuples 
$$
(\alpha_0,b_1,...,b_{k-1},\alpha_{k-1}),
$$
 where $\alpha_i=(\Gamma^-_i,S_i,j_i,D_i,\Gamma^+_i)$ is as just described and $b_i:\Gamma^+_{i-1}\rightarrow \Gamma^-_{i}$
is a bijection.  We assume $G$ is a finite group acting acting on each  $\alpha_0,...,\alpha_{k-1}$ by biholomorphic maps  as previously described. Moreover, 
$G$ defines actions on the $\Gamma^\pm_i$ and  we assume that these actions are such that every $b_i:\Gamma^+_{i-1}\rightarrow \Gamma^-_i$ is equivariant.
We fix for every $(z,b_i(z))$ an ordered disk pair ${\mathcal D}_{(z,b_i(z))}=(D_z\sqcup D_{b_i(z)}, (z,b_i(z)))$ and assume that the union of these 
disks is invariant under $G$. Of course, the disks are mutually disjoint and do not intersect the floor disks.
An interface gluing parameter for the $(i-1,i)$-interface, $i=1,..,k-1$, is a map $\mathfrak{a}_{i-1,i}$  which assigns to $(z,b_i(z))$ a gluing parameter ${\mathfrak{a}}_{i-1,i}(z)$,
which has the additional property that either all its components are zero or all of its components are non-zero.
A total gluing parameter is given
as $(2k-1)$-tuple  $(\mathfrak{a}_0,{\mathfrak{a}}_{0,1},\mathfrak{a}_2,...,{\mathfrak{a}}_{k-2,k-1},\mathfrak{a}_{k-1})$. Denote by $1\leq i_1<..<i_\ell\leq k-1$

the ordered sequence of indices such that ${\mathfrak{a}}_{i-1,i}=0$.  We say that we have nontrivial interfaces $(i_1-1,i_1)$,..,$(i_{\ell}-1,i_{\ell})$.  Given $\alpha$ and the small disk structure applying the total gluing parameter we obtain $\alpha_\mathfrak{a}$ which again is a Riemann surface building.
The collection of all  such Riemann surface buildings we shall refer to as the set of ${\bm{D}}$-decendents of $\alpha$.  For example 
if $\mathfrak{a}$ has the nontrivial interfaces $1\leq i_1<..<i_\ell\leq k-1$ let us introduce $i_0=0$ and $i_{\ell+1}=k$. Then define $S_e= \coprod_{i\in \{i_e,i_{e+1}-1\}} S_i$,
$D_e =\coprod_{i\in \{i_e,i_{e+1}-1\}} D_i \coprod_{i\in \{i_e+1, i_{e+1}-1\}} D_{i-1,i}$.  As punctures we take $\Gamma_{i_e}^-$ and $\Gamma_{i_{e+1}-1}^+$ so that we obtain
$(\Gamma_{i_e}^-,Se,j_e,D_e,\Gamma^+_{i_{e+1}-1})$. The gluing parameters $(\mathfrak{a}_{i_e},\mathfrak{a}_{i_e,i_e+1},...,\mathfrak{a}_{i_{e+1}-1})$ allow us
to glue this surface and obtain $\alpha_{\mathfrak{a},e}$.
Together with the $b_{i_e}$ we obtain the Riemann surface building
$$
(\alpha_{\mathfrak{a},0},b_{i_1},\alpha_{\mathfrak{a},1},b_{i_2},...,b_{i_\ell},\alpha_{\mathfrak{a},\ell}).
$$

\section{Results Around Sc-Smoothness}\label{APP12}
We shall recall the sc-smoothness results used in the paper from two sources, \cite{HWZ8.7} and \cite{HWZ5}.
We also list some other results which are more classical and should be well-known.

\subsection{Classes of Maps}\label{APP12.1}

\subsection{Classical Results}
The following result is well-known and can be proved as in \cite{El}. 
\begin{lem}
Assume that $\Psi:S^1\times {\mathbb R}^N\rightarrow {\mathbb R}^M$ is a smooth map satisfying 
$\Psi(t,0)=0$. Then for every $m\geq 2$ and $\tau>0$ the map
$$
\Psi_\ast: H^{m,\tau}_{0}([0,\infty)\times S^1,{\mathbb R}^N)\rightarrow H^{m,\tau}_0({\mathbb R}^+\times S^1,{\mathbb R}^M):
h\rightarrow  [(s,t)\rightarrow \Psi(t,h(s,t))]
$$
is (classically)  $C^\infty$.  In particular, with the standard sc-structure on $H^{3,\delta_0}_c({\mathbb R}^+\times S^1,{\mathbb R}^N)$
where level $m$ corresponds to regularity $(m+3,\delta_m)$  the map $\Psi_\ast$ is sc-smooth.
\end{lem}

\subsection{The Fundamental Lemma} \label{FUND0b} %jwf multiply defined
We continue with our study of sc-smoothnees.  We denote by $\varphi$ the exponential gluing profile 
$$\varphi (r)=e^{\frac{1}{r}}-e,  \quad \text{$r>0$. }$$
With the  nonzero complex number $a$ (gluing parameter)  we associate the gluing angle 
 $\vartheta\in S^1$ and the gluing length $R$ via  the formulae
$$a=\abs{a}e^{-2\pi i \vartheta}\quad \text{and}\quad R=\varphi (\abs{a}).$$
Note that $R\to \infty$ as $\abs{a}\to 0$.

Denote by $L$ the Hilbert sc-space $L^2(\R\times S^1,\R^N)$ equipped  the sc-structure $(L_m)_{m\in \N_0}$ defined by 
$L_m=H^{m,\delta_m}(\R\times S^1,\R^N)$, where $(\delta_m)$ is a strictly increasing sequence starting at $\delta_0=0$. Let us also introduce
the sc-Hilbert spaces $F=H^{2,\delta_0}({\mathbb R}\times S^1,{\mathbb R}^N)$ with sc-structure, where level $m$ corresponds to regularity
$(m+2,\delta_m)$, with $\delta_0>0$ and $(\delta_m)$ is a strictly increasing sequence. Finally we introduce $E=H^{3,\delta_0}({\mathbb R}\times S^1,{\mathbb R}^N)$ with level $m$ corresponding to regularity $(m+3,\delta_m)$ and $(\delta_m)$ is as in the $F$-case.

With these data fixed we prove the following prop. The  prop has many applications. In particular, it will be used in Section 
\jwf{[broken reference]}
%\ref{subsub} 
in  order to prove that  the transition maps   between local M-polyfolds are sc-smooth. The following is Proposition 2.8 from \cite{HWZ8.7} combined with Proposition 2.17 from \cite{HWZ8.7}. 
\begin{lem}[Fundamental Lemma]\label{FUNDAMENTAL}
The following four  maps
$$
\Gamma_i:B_\frac{1}{2}\oplus G\to  G,\ \ i=1,\ldots, 4
$$
where $G=L$, $G=F$ or $G=E$ are sc-smooth.
\begin{description}
\item[(1)]  Let $f_1:\R\to  \R$ be a smooth  function which is constant outside of a compact interval so that $f_1(+\infty)=0$. Define
$$
\Gamma_1(a,h)(s, t)=f_1\left(s -\frac{R}{2}\right)h (s, t)
$$
if $a\neq 0$ and $\Gamma_1(0,h)=f(-\infty)h$ if $a=0$. \\
\item[(2)]  Let $f_2:\R\to  \R$ be a compactly supported  smooth function. Define  
$$
\Gamma_2(a,h)(s, t)=f_2\left(s -\frac{R}{2}\right)h(s-R,t-\vartheta)
$$
if  $a\neq 0$ and $\Gamma_2(0,h)=0$ if $a=0$.
\item[(3)]  Let $f_3:\R\to  \R$ be a smooth  function which is constant outside of a compact interval and satisfying $f_3(\infty)=0$. Define 
$$
\Gamma_3(a,h)(s', t')=f_3\left(-s' -\frac{R}{2}\right)h(s',t')
$$
if $a\neq 0$ and $\Gamma_3(0,h)=f_3(-\infty)h$
if $a=0$.
\item[(4)]  Let $f_4:\R\to  \R$ be a smooth function of compact support and define  
$$
\Gamma_4(a,h)(s', t')=f_4\left(-s' -\frac{R}{2}\right)h(s'+R,t'+\vartheta)
$$
if $a\neq 0$ and $\Gamma_4(0,h)=0$  if  $a=0$.
\end{description}
\end{lem}

We shall study  the following five mappings:
\begin{itemize}
\item[{\bf M1.}]  The map 
$$H^{3,\delta_0}_c(\R^+\times S^1,\R^{N})\to  \R^{N}, \quad \xi^+\mapsto c$$
 which associates with  $\xi^+$  its  asymptotic constant $c$.
 \item[{\bf M2.}]  The map 
 $$B_\frac{1}{2}\times H^{3,\delta_0}(\R^+\times S^1,\R^{N})\to  \R^{N},\quad (a,r^+)\mapsto  [r^+]_R.$$ 
\item[{\bf M3.}]   The map 
$$B_\frac{1}{2}\times H^{3,\delta_0}(\R^+\times S^1,\R^{N})\to  H^{3,\delta_0}(\R^+\times S^1,\R^{N}),  \quad 
(a,r^+)\mapsto  \frac{\beta_a}{\gamma_a}(\cdot )[r^+]_{R}.$$
\item[{\bf M4.}]  The map $B_\frac{1}{2}\oplus H^{3,\delta_0}(\R^+\times S^1,\R^{N})\to  H^{3,\delta_0}(\R^+\times S^1,\R^{N}),$
$$(a,r^+)\mapsto \frac{\beta_a^2}{\gamma_a}\cdot r^+.$$
 \item[{\bf M5.}]  The map $B_\frac{1}{2}\oplus H^{3,\delta_0}(\R^-\times S^1,\R^{N})\to  H^{3,\delta_0}(\R^+\times S^1,\R^{N})$,
$$ (a,r^-)\mapsto   \frac{\beta_a(1-\beta_a)}{\gamma_a} r^-(\cdot -R,\cdot -\vartheta).$$
\end{itemize}
 In view of the formula for the projection map $\pi_a$  the  sc-smoothness of the map $(a,(\xi^+,\xi^-))\mapsto  \pi_a(\eta^+,\eta^-)$  in Theorem 
\jwf{[broken reference]}
%\ref{first-splice},
 is a consequence of the following prop.
 \begin{prop}\label{klopx}
 The maps {\bf M1}-{\bf M5} listed above (and suitably defined at the parameter value $a=0$) are sc-smooth in a neighborhood of $a=0$.
 \end{prop}

\jwf{\subsection{Some Corollaries}}

\section{Maps, Nodes, and  Abstract Periodic Orbits}\label{Appendix-periodicx}
For the M-polyfold constructions we need certain classes of maps.
Two situations of interest are maps defined on a nodal disk pair which map the nodal points to the same nodal value.
Here the behavior near the nodal points is important. Note that in this case the map will have a continuous extension of the nodes.
A more complicated situation is that we have a map defined on a punctured nodal disk pair, where the map approaches a periodic orbit or  a cylinder over a periodic
orbit when we come close to to the puncture.
\subsection{Some Classes of Maps and Nodes}\label{SSSEXC16}
Given a Riemann surface $S$,  the notion of a map $u:S\rightarrow {\mathbb R}^N$  locally belonging to the Sobolev class $H^m$
is well defined. We shall say in this case that $u\in H^m_{loc}$. Here $H^m$ refers to maps which have partial derivatives up to order $m$ 
belonging to $L^2$. Of course, we can make the same definition for a map $u:S\rightarrow Q$, where $Q$ is a smooth manifold, provided
the map is at least continuous, i.e. $m\geq 2$ by the Sobolev embedding theorem.

For the M-polyfold constructions we need certain classes of maps defined on punctured Riemann surfaces.
Let $S$ be a Riemann surface and $x\in S$ an interior point, i.e. not a boundary point. We denote by ${\mathbb D}\subset {\mathbb C}$
the closed unit disk. After fixing a compact disk $D_x\subset S$ with smooth boundary around $x$ we can pick
a biholomorphic map $h:({\mathbb D},0)\rightarrow (D_x,x)$ and define the associated positive holomorphic polar coordinates by
$$
[0,\infty)\times S^1\rightarrow S:\sigma^+(s,t)=h(e^{-2\pi(s+it)}).
$$
We shall call $\sigma^+$ positive holomorphic polar coordinates around $x$. We say a map $u:S\rightarrow {\mathbb R}^N$
is of {\bf up to constant of class $(m,\varepsilon)$ on $D_x$} provided the map 
$$
[0,\infty)\times S^1\rightarrow {\mathbb R}^N: v(s,t):=u\circ \sigma^+(s,t)
$$
belongs to $H^m_{loc}$ and there is a constant $c\in {\mathbb R}^N$ so that  partial derivatives up to order $m$ of $v-c$, weighted by $e^{\varepsilon\cdot s}$ belong to
$L^2([0,\infty)\times S^1,{\mathbb R}^N$.  Note that in this case the behavior of $u$ near the point $x$ is different than with respect to other points.
The constant $c\in {\mathbb R}^N$ is called the {\bf nodal value.}

Given a Riemann surface $S$, possibly with smooth boundary, and a finite number of punctures in $S\setminus\partial S$ 
we say the map is of {\bf class $(m,\varepsilon)$} if it is of class $H^m_{loc}$ away from the punctures and of class $(m,\varepsilon)$ on suitable disks 
around the punctures. This definition does not depend on the choice of disks $D_x$.   In the case that $m\geq 2$ we can replace the target ${\mathbb R}^N$ by a smooth manifold $Q$.

The Sobolev spaces which are important for the class of maps we described above are
\begin{itemize}
\item[(1)] $H^m({\mathbb D},{\mathbb R}^N)$  arising when we talk about $H^m_{loc}$,
\end{itemize}
with norm defined by 
$$
|u|_{H^m}^2 =\sum_{|\alpha|\leq m} \int_{\mathbb D} |D^\alpha u(s,t)|^2 ds dt.
$$
\begin{itemize}
\item[(2)] $H^{m,\varepsilon}({\mathbb R}^\pm\times S^1,{\mathbb R}^N)$ arising when we talk about being of class $(m,\varepsilon)$,
\end{itemize}
with norm defined by
$$
|u|_{H^{m,\varepsilon}}^2=\sum_{|\alpha|\leq m} \int_{{\mathbb R}^\pm\times S^1} |D^\alpha u(s,t)|^2\cdot e^{2\varepsilon |s|} ds dt.
$$
\begin{itemize}
\item[(3)] $H^{m,\varepsilon}_c({\mathbb R}^\pm\times S^1,{\mathbb R}^N)$ arising when we talk about class $(m,\varepsilon)$ up to constant,
\end{itemize}
with norm defined by 
$$
|u|_{H^{m,\varepsilon}_c}^2=|c|^2+ \sum_{|\alpha|\leq m} \int_{{\mathbb R}^\pm\times S^1} |D^\alpha r(s,t)|^2\cdot e^{2\varepsilon |s|} ds dt,
$$
where we decompose $u=c+r$ where $r$ has vanishing asymptotic constant. 

By abuse of notation we very often say that a map is of class $(m,\varepsilon)$ even in  a case,where the map is of class $(m,\varepsilon)$ up to constant, if
this is obvious from the context and no confusion is possible.  Moreover given a Riemann surface we can use charts to define the appropriate mapping spaces
for maps defined on the Riemann surface or a punctured version.

\subsection{Convergence to a Periodic Orbit}\label{APP_C}
Given a positive real number $\tau\geq 0$ and an integer $m\geq 0$ we consider the Hilbert space $H^{m,\varepsilon}({\mathbb R}^\pm \times S^1,{\mathbb R}^N)$. It consists of all maps $u:{\mathbb R}^\pm\times S^1\rightarrow {\mathbb  R}^N$ such that the distributional partial derivatives $D^\alpha u$
up to order $m$ belong to $L^2_{loc}$ so that the norm $\norm{u}_{H^{m,\varepsilon}({\mathbb R}^{\pm}\times S^1)}$ is finite, where 
$$
\norm{u}_{H^{m,\varepsilon}({\mathbb R}^{\pm}\times S^1)}^2 =\sum_{|\alpha|\leq m } \int_{{\mathbb R}^\pm\times S^1} |D^\alpha u(s,t)|^2 e^{2\varepsilon |s|} ds dt <\infty.
$$
We also define the spaces $H^{m,\varepsilon}_c({\mathbb R}^\pm\times S^1,{\mathbb R}^n)$, where an element has the form $u=r+d$, where $d\in {\mathbb R}^N$
and $r\in H^{m,\varepsilon}_0({\mathbb R}^{\pm}\times S^1, {\mathbb R}^N)$.  The norm in this case is defined by 
$$
\norm{r+d}_{H^{m,\varepsilon}_c({\mathbb R}^{\pm}\times S^1)}^2 =|d|^2 +\sum_{|\alpha|\leq m } \int_{{\mathbb R}^\pm\times S^1} |D^\alpha r(s,t)|^2 e^{2\varepsilon |s|} ds dt.
$$
Often, when the domain is clear we shall just write $H^{m,\varepsilon}$ or $H^{m,\varepsilon}_c$.  Instead of ${\mathbb R}^\pm \times S^1$
we shall also consider domains of the form $[R_0,\infty)\times S^1$ for large $R_0$ or $(-\infty,-R_0]\times S^1$.

Next we consider periodic orbits which will play an important role. We start with a definition.
\begin{definition}\label{APP_E1b} %jwf multiply defined
Let $Q$ be a smooth (connected) manifold without boundary. A {\bf periodic orbit} in $Q$ is a tuple $\boldsymbol{\gamma}=([\gamma],T,k)$
where $\gamma:S^1\rightarrow Q$ is a smooth embedding, $T>0$ and $k\geq 1$ is an integer. Moreover, $[\gamma]$
is the collection of all parameterizations $t\rightarrow \gamma(t+c)$, where $c\in S^1$.  We call $T$ the {\bf period}, $k$ the {\bf covering number}, and 
$T/k$ the {\bf minimal period}.
A {\bf weighted periodic orbit}
is a pair $\boldsymbol{\bar{\gamma}}=(\boldsymbol{\gamma},\delta)$, where $\boldsymbol{\gamma}$ is a periodic orbit
and $\delta$ is a weight sequence, i.e. $\delta={(\delta_i)}_{i=0}^\infty$, $0<\delta_0<...<\delta_i<\delta_{i+1}<..$
\qed
\end{definition}
The model case is  $Q=S^1\times {\mathbb R}^{N-1}$ and the periodic orbits of interest are the   $\boldsymbol{\gamma^0}=([\gamma^0],T,k)$, with $\gamma^0(t)=(t,0)\in S^1\times {\mathbb R}^{N-1}$.  If $\boldsymbol{\gamma}$ is a periodic orbit in ${\mathbb R}^N$ we find an open neighborhood
$U$ of $\gamma(S^1)$ and a diffeomorphism $\Phi:U\rightarrow S^1\times {\mathbb R}^{N-1}$ such that $\Phi\circ\gamma(t)=\gamma^0(t)$ 
for a suitable $\gamma\in [\gamma]$.  

Next we shall give two equivalent definitions for the convergence of a map to a periodic orbit.
First we restrict ourselves to the case, where the domain is ${\mathbb R}^+\times S^1$. The first definition is the following.
\begin{definition}\label{APP_E1}
Consider a map $\wt{w}:{\mathbb R}^\pm\times S^1\rightarrow {\mathbb R}\times {\mathbb R}^{N}$ and a periodic orbit
$\boldsymbol{\gamma}$ in ${\mathbb R}^N$.  
We say $\wt{w}$ is of class $H^{m,\tau}$ with respect to $\boldsymbol{\gamma}$, or $H^{m,\tau}_{\boldsymbol{\gamma}}$ for short,
where $\tau\in (0,\infty)$ and $m\geq 2$ is an integer, provided
the following holds, where we use a $\Phi:U\rightarrow S^1\times {\mathbb R}^{N-1}$ as previously introduced.
\begin{itemize}
\item[(1)]  The map $\wt{w}$ is of class $H^{m}_{loc}$ and with $\wt{w}=(a,w)$ it holds that $w(s,t)\in U$ for all $(s,t)$ provided $|s|$ is large.
\item[(2)]  There exists a  constant $c\in {\mathbb R}$ such that the map $(s,t)\rightarrow Ts +c -a(s,t)$ belongs to $H^{m,\tau}({\mathbb R}^\pm\times S^1,{\mathbb R})$.
\end{itemize}
With $v(s,t)=\Phi\circ u(s,t)$ for $s\geq R_0$ and $\Phi$ is as previously introduced the following holds, where we write $v=(\theta,z)\in S^1\times {\mathbb R}^{N-1}$.
\begin{itemize}
\item[(3)] There
 exists  $d\in S^1$ and $r\in H^{m,\tau}({\mathbb R}^\pm\times S^1,{\mathbb R})$ such that $\theta(s,t) =kt +d +r(s,t)$, where $\theta$ is the $S^1$-factor.
\item[(4)] The ${\mathbb R}^{N-1}$-factor $z$ belongs to $H^{m,\tau}({\mathbb R}^\pm\times S^1,{\mathbb R}^{N-1})$.
\end{itemize}
\qed
\end{definition}
Here is the second definition.
\begin{definition}\label{APP_E2}
Consider a map $\wt{w}:{\mathbb R}^\pm\times S^1\rightarrow {\mathbb R}\times {\mathbb R}^{N}$ and a periodic orbit
$\boldsymbol{\gamma}$ in ${\mathbb R}^N$.  We say that $\wt{w}$ is of class $H^{m,\tau}_{\boldsymbol{\gamma}}$
provided for a suitable choice $\gamma\in [\gamma]$ and $c\in {\mathbb R}$ the map defined by 
$$
\wt{v}(s,t) =\wt{w}(s,t) - (Ts+c,\gamma(kt))
$$
belongs to $H^{m,\tau}({\mathbb R}^\pm\times S^1,{\mathbb R}^N)$.
\qed
\end{definition}
\begin{prop}
The Definition \ref{APP_E1} and the Definition \ref{APP_E2} are equivalent.
\end{prop}
\begin{proof}
The easy proof is left to the reader.
\end{proof}
\begin{prop}
Assume that $\boldsymbol{\gamma}=([\gamma],T,k)$ is a periodic orbit in ${\mathbb R}^N$ and $f:{\mathbb R}^N\rightarrow {\mathbb R}^M$ a smooth map
such that $\gamma':=f\circ\gamma:S^1\rightarrow {\mathbb R}^M$ is a smooth embedding. Consider the periodic orbit $\boldsymbol{\gamma}'=([\gamma'],T,k)$
in ${\mathbb R}^M$. If $\wt{w}:{\mathbb R}^\pm\times S^1\rightarrow {\mathbb R}\times {\mathbb R}^N$ is of class $H^{m,\tau}_{\boldsymbol{\gamma}}$ then
$(Id_{\mathbb R}\times f)\circ \wt{u}$ is of class $H^{m,\tau}_{\boldsymbol{\gamma'}}$.
\end{prop}
\begin{proof}
The proof is left to the reader.
\end{proof}
In view of the previous discussion we can  make the following definition, which will be the relevant definition for us in this paper.
\begin{definition}\label{DEF-E1}
Let $Q$ be a smooth (connected) manifold without boundary and $\boldsymbol{\gamma}=([\gamma],T,k)$ a periodic orbit in $Q$.
Assume that $(D_x,x)$ is a compact disk-like Riemann surface with smooth boundary and interior point $x$. 
\begin{itemize}
\item[(1)] A map $\wt{u}:D_x\setminus\{x\}\rightarrow {\mathbb R}\times Q$ is said to be of class $H^{m,\tau}_{\boldsymbol{\gamma}+}$ 
provided for positive holomorphic polar coordinates $\sigma_{\wh{x}}^+:{\mathbb R}^+\times S^1 \rightarrow D_x\setminus\{x\}$ and a proper embedding
$\phi:Q\rightarrow {\mathbb R}^N$ with $\boldsymbol{\gamma}'=([\phi\circ\gamma],T,k)$ the map $\wt{w}=(Id_{\mathbb R}\times \phi)\circ\wt{u}\circ \sigma^+_{\wh{x}}$ is of class $H^{m,\delta}_{\boldsymbol{\gamma'}}({\mathbb R}^+\times S^1)$.
\item[(2)] A map $\wt{u}:D_x\setminus\{x\}\rightarrow {\mathbb R}\times Q$ is said to be of class $H^{m,\tau}_{\boldsymbol{\gamma}-}$ 
provided for negative holomorphic polar coordinates $\sigma_{\wh{x}}^+:{\mathbb R}^-\times S^1 \rightarrow D_x\setminus\{x\}$ and a proper embedding
$\phi:Q\rightarrow {\mathbb R}^N$ with $\boldsymbol{\gamma}'=([\phi\circ\gamma],T,k)$ the map $\wt{w}=(Id_{\mathbb R}\times \phi)\circ\wt{u}\circ \sigma^+_{\wh{x}}$ is of class $H^{m,\delta}_{\boldsymbol{\gamma'}}({\mathbb R}^-\times S^1)$.
\end{itemize}
\qed
\end{definition}
The  notion is well-defined independent of the choice of ${\mathbb R}^N$ and  the proper embedding.
Next we define a matching condition.
\begin{definition}\label{DEF-E2}
Assume that $(D_x\sqcup D_y,(x,y))$ is an ordered disk pair and $\wt{w}^x$ is defined on the punctured $D_x$ positively asymptotic to $\boldsymbol{\gamma}$
of class $(m,\tau)$ and $\wt{w}^y$ defined on the punctured $D_y$, negatively asymptotic to $\boldsymbol{\gamma}$ of class $(m,\tau)$. 
We say that $\wt{w}^x$ and $\wt{w}^y$ are  {\bf $[\wh{x},\wh{y}]$-directionally matching} provided with $\wt{w}^x\circ\sigma_{\wh{x}}^+$ and $\wt{w}^y\circ\sigma_{\wh{y}}^-$ satisfy
$$
\lim_{s\rightarrow\infty}w^x\circ\sigma_{\wh{x}}^+(s,t) =
\lim_{s'\rightarrow-\infty}w^y\circ\sigma_{\wh{y}}^-(s',t) 
$$
for all $t\in S^1$.
\qed
\end{definition}

\section{Stable Hamiltonian Structures and  Periodic Orbits}\label{APP11}
We recall the notion of a stable Hamiltonian structure and derive useful results which are needed in the book and
quite well-known.
\subsection{Stable Hamiltonian Structures}

One of the important objects is that of a stable Hamiltonian structure. A detailed study of these structures can be found in 
\cite{CV1}.
\subsubsection{Basic Definition}
We begin with the definitions of a stable Hamiltonian structure.
\begin{definition} \label{28.1QQ}
Let $Q$ be a closed odd-dimensional manifold of dimension $\dim(Q)=2n-1$. A stable Hamiltonian structure on $Q$ is given by a pair
$(\lambda,\omega)$, where $\omega$ is a closed two-form of maximal rank on $Q$ and $\lambda$ a one-form such that
\begin{itemize}
\item[(1)] $\lambda\wedge \omega^{n-1}$ is a volume-form.
\item[(2)] The vector field $R$, called the {\bf Reeb vector field},  defined by
$$
i_R\lambda =1\ \ \hbox{and}\ \  {i_R} \omega=0
$$
 satisfies
$$
L_R\lambda =0.
$$
\end{itemize}
\qed
\end{definition}
The latter condition implies by Cartan's formula
$$
0=L_R\lambda= di_R\lambda +i_Rd\lambda = i_Rd\lambda.
$$
Since $R$ spans the kernel of $\omega$ this implies
\begin{eqnarray}\label{EQP28}
\text{ker}(\omega)\subset \text{ker}(d\lambda).
\end{eqnarray}
 The fact that $i_R\omega=0$ implies again by the Cartan formula that $L_R\omega=0$.
\begin{remark}
The standard example for a stable Hamiltonian structure is $(\lambda,d\lambda)$, where $\lambda$ is a contact form on $Q$.
\qed
\end{remark}
Stable Hamiltonian structures are an interesting object to study, see \cite{CV1}.  It is important to have such structures when studying pseudoholomorphic curves,
see \cite{EGH}. These structures allow to control certain area-based energy functionals, which is important in obtaining a priori estimates.

Associated to a stable Hamiltonian structure $(\lambda,\omega)$ on $Q$
we have the distribution $\xi=\ker(\lambda)$ and the natural splitting
of the tangent bundle
$$
TQ={\mathbb R}R\oplus \xi.
$$
We observe that the line bundle ${\mathbb R}R$ has a distinguished section $R$ and $\xi$ is in a natural way a symplectic vector bundle with symplectic structure being  $\omega|\xi\oplus \xi$. Let us observe that the flow $\phi_t$ associated to $R$ maps a vector in $\xi_x$ to a vector in $\xi_{\phi_t(x)}$
\begin{eqnarray}
T\phi_t:\xi\rightarrow \xi_{\phi_t(x)}.
\end{eqnarray}
\subsubsection{Symplectic Forms Associated to $(Q,\lambda,\omega)$}
We discuss stable Hamiltonian manifolds $(Q,\lambda,\omega)$ in somewhat more detail.
Denote by $p:{\mathbb R}\times Q \rightarrow Q $ the obvious projection and given a smooth map $\phi:{\mathbb R}\rightarrow {\mathbb R}$ 
we denote by $\wh{\phi}$ the map ${\mathbb R}\times Q\rightarrow {\mathbb R}$ defined by $\wh{\phi}(s,q)=\phi(s)$.
Given $(Q,\lambda,\omega)$ and $\phi$ we  can consider the two-form
$\Omega_{\phi}$  on ${\mathbb R}\times Q$ defined by
$$
\Omega_{\phi} = p^\ast \omega + d(\wh{\phi}\cdot p^\ast\lambda)
$$
which we shall write sloppily as $\omega+d(\phi\lambda)$. 
We observe that $\Omega_{\phi}(s,q)= \omega_q +\phi(s) \cdot d\lambda_q + \phi'(s) ds\wedge \lambda$. 
If $|\phi(s)|$ is small enough we see that $\omega_q+\phi(s) d\lambda_q$ as a two-form on $Q$ is maximally non-degenerate
and $\text{ker}(\omega+\phi(s) d\lambda)=\text{ker}(\omega)$ in view of (\ref{EQP28}).
The maximal non-degeneracy implies that $\omega+\phi(s) d\lambda$ restricted to 
$\{0\}\times\xi_q \subset T_{(s,q)}({\mathbb R}\times Q)$ is non-degenerate, i.e. a symplectic form.
 If $\phi$ satisfies this smallness condition 
and in addition $\phi'(s)>0$ for all $s$ then $\Omega_{\phi}$ is a symplectic form.
Hence we have obtain.
\begin{lem}\label{LEMMA-epsilon}
Given a smooth manifold $Q$ equipped with a stable Hamiltonian structure $(Q,\lambda,\omega)$ 
there exists $\varepsilon>0$ such that for every smooth $\phi:{\mathbb R}\rightarrow [-\varepsilon,\varepsilon]$ 
with $\phi'(s)>0$ for all $s\in {\mathbb R}$ the two-form $\Omega_{\phi}$ is symplectic.
\qed
\end{lem}
In view of this lemma we make the following definition.
\begin{definition}\label{DEFQQ00}
We denote for $\varepsilon>0$ by $\Sigma_{\varepsilon}$ the set of all smooth maps $\phi:{\mathbb R}\rightarrow [-\varepsilon,\varepsilon]
$ satisfying $\phi'(s)\geq 0$ for all $s\in {\mathbb R}$.\qed
\end{definition}
In the case of $Q$ equipped with a contact form we obtain the stable Hamiltonian structure $(Q,\lambda,d\lambda)$.
In this case we can take any smooth map $\phi:{\mathbb R}\rightarrow [-1,\infty)$ and assuming that $\phi'(s)>0$ 
it follows that
$$
\Omega_{\phi}= (1+\phi)d\lambda + \phi'(s)ds\wedge \lambda= d((1+\phi)\lambda)
$$
is symplectic.  A possible example is the map $\phi(s)=e^s-1$ which gives $\Omega_{\phi} = d(e^t\lambda)$ which is the usual symplectization form. We note that in the general case of stable Hamiltonian structures the upper bound on $\phi$ is important to get symplectic forms.

\subsubsection{Compatible Almost Complex Structures}
Starting with a stable Hamiltonian structure $(Q,\lambda,\omega)$ we take the manifold ${\mathbb R}\times Q$
and fix $\varepsilon>0$ with the properties guaranteed by Lemma \ref{LEMMA-epsilon}.
We consider the set of all 2-forms $\Omega_{\phi}$ on ${\mathbb R}\times Q$ with $\phi\in\Sigma_{\varepsilon}$.
Then this collection is invariant under the ${\mathbb R}$-action on ${\mathbb R}\times Q$ via addition on the first factor.
With $\xi=\ker(\lambda)$ we obtain the symplectic vector bundle $(\xi,\omega)\rightarrow Q$ and fix a complex structure 
for this vector bundle, i.e. a smooth fiber-preserving map $J:\xi\rightarrow \xi$ with the following two properties
\begin{itemize}
\item[(1)] $J^2=-Id$.
\item[(2)] $g_J:\xi\oplus\xi\rightarrow {\mathbb R}$ defined by $g_J(q)(h,k)=\omega_q(h,J(q)k)$ is fiber-wise a positive definite inner product.
\end{itemize}
With $R$ being the Reeb vectorfield we define a smooth ${\mathbb R}$-invariant  almost complex structure $\wt{J}$ for ${\mathbb R}\times Q$ by
\begin{eqnarray}\label{REDF28.3}
\wt{J}(a,q)(h,kR(q)+\Delta)= (-k, hR(q)+J(q)\Delta),
\end{eqnarray}
where $h,k\in {\mathbb R}$ and $\Delta\in\xi$.  Consider for $\phi\in\Sigma_{\varepsilon}$ the fiber-wise bilinear form
$$
\Omega_{\phi}\circ (Id\oplus \wt{J}).
$$
We compute with a vector $(h,kR(q)+\Delta)\in T_{(s,q)}({\mathbb R}\times Q)$
\begin{eqnarray}
&&\Omega_{\phi}\circ (Id\oplus \wt{J})((h,kR(q)+\Delta),(h,kR(q)+\Delta))\\
&=&(\omega+\phi(s)d\lambda)(\Delta,J(q)\Delta) + \phi'(s)(h^2 +k^2)\nonumber\\
&\geq & 0.\nonumber
\end{eqnarray}
Since $\omega+\phi(s)d\lambda$ is non-degenerate on $\xi$ we see that in case $\phi'(s)>0$ the expression is a positive definite
quadratic form.  
\begin{lem}
Given a smooth manifold with stable Hamiltonian structure $(Q,\lambda,\omega)$ and a ${\mathbb R}$-invariant 
almost complex structure $\wt{J}$ as described in (\ref{REDF28.3}) pick an admissible $\varepsilon>0$ as in Definition 
\ref{DEFQQ00}.  Then for every $\phi\in\Sigma_{\varepsilon}$ the (fiber-wise) symmetric quadratic form 
$\mathsf{Q}^{\phi}$ defined with $(h,kR(q)+\Delta)\in T_{(s,q)}({\mathbb R}\times Q)$ by 
$$
\mathsf{Q}^{\phi}_{(s,q)}(h,kR(q)+\Delta):= \Omega_{\phi}((h,kR(q)+\Delta,\wt{J}(s,q)(h,kR(q)+\Delta))
$$
satisfies $\mathsf{Q}^{\phi}_{(s,q)}\geq 0$.  If $\phi\in\Sigma_{\varepsilon}$ has the property $\phi'(s)>0$ for all $s\in {\mathbb R}$
it satisfies $\mathsf{Q}_{(s,q)}^{\phi}>0$.
\qed
\end{lem}

\subsection{Periodic Orbits}
In the Appendix \ref{Appendix-periodicx} we have introduced in Definition \ref{APP_E1} the abstract notion of a periodic orbit in a smooth manifold
$Q$ and this quite general notion was sufficient for our space constructions in the main body of the book.

It is an important fact that a given Hamiltonian structure $(\lambda,\omega)$ on an odd-dimensional manifold $Q$
produces automatically a set of periodic orbits. We shall explain this next, where we give a formulation compatible with the general definition.
\begin{definition}
Assume $Q$ is equipped with $(\lambda,\omega)$ and $R$ is the associated Reeb vector field. A periodic orbit for $(V,\lambda,\omega)$ is a tuple
$([\gamma],k,T)$ with $k$ being a nonzero positive integer, $T$ a positive number, $\gamma:S^1\rightarrow Q$ a smooth embedding 
and $[\gamma]$ the set of reparameterisations of $\gamma$, i.e. $t\rightarrow \gamma(t+c)$, where $c\in S^1$,
such that the following property holds
$$
\frac{d\gamma}{dt}(t) = \frac{T}{k}\cdot R(\gamma(t)).
$$
We call $T$ the {\bf period} and $k$ the {\bf covering number}, and $T/k$ the {\bf minimal period}.
\qed
\end{definition} 
\begin{remark}
The way to think about a periodic orbit for $(Q,\lambda,\omega)$ is as follows.  Take the Reeb vector field $R$ 
and solve $\dot{x}=R(x)$. Assume we have a periodic orbit $(x,T)$, i.e. there exists a $T>0$ such that $x(0)=x(T)$.
Then there exists an integer $k\geq 1$ such that $x(t)\neq x(0)$ for $0<t<T/k$, $x(0)=x(T/k)$, and $T/k$ is called the minimal period.
We can define  the embedding $\gamma:{\mathbb R}/{\mathbb Z}\rightarrow Q$ by 
$$
\gamma(t) = x(tT/k).
$$
Then $\frac{d\gamma}{dt}(t) = (T/k)\cdot R(\gamma(t))$. If we take another point on $x({\mathbb R})$ and solve the differential equation
with this as starting point we obtain a map $y$, which again can be viewed as a $T$-periodic solution $(y,T)$. Applying the same procedure 
we obtain another element in $[\gamma]$. Hence our notation keeps track of the period,  the minimal period, and the set $\gamma(S^1)$ with
a preferred class of parameterizations.
\qed
\end{remark}

\begin{definition}
Consider a periodic orbit $([\gamma],k,T)$ associated to $(Q,\lambda,\omega)$.
We say that $([\gamma],k,T)$ is {\bf non-degenerate} if the symplectic map $T\phi_{kT}(p):\xi_p\rightarrow \xi_p$ for some fixed 
$p\in \gamma(S^1)$  does not have $1$ in its spectrum. The definition does not depend on the choice of $p$. If for a Hamiltonian structure $(\lambda,\omega)$ all periodic orbits are non-degenerate we say that $(\lambda,\omega)$ is a {\bf non-degenerate stable Hamiltonian structure}.
\qed
\end{definition}

\begin{remark}
It is always possible to perturb a contact form in $C^\infty$, by keeping the associated contact structure, so that   the new form is non-degenerate.
The situation for stable Hamiltonian structure is more subtle, see \cite{CFP,CV1}. It iseems to be  possible to always perturb them to a Morse-Bott situation.
Our discussion of polyfold structures can be generalized to this case, but we shall not do it here and concentrate on the non-degenerate case.
\qed
\end{remark}

For the sc-Fredholm Theory in this book we are interested in the situation where the periodic orbits come from a 
non-degenerate stable Hamiltonian structure $(\lambda,\omega)$ on $Q$. 
\begin{definition}
Given a smooth compact manifold equipped with a non-degenerate stable Hamiltonian structure $(Q,\lambda,\omega)$ 
we denote by ${\mathcal P}(Q,\lambda,\omega)$ the collection of all
periodic orbits $([\gamma],k,T)$. We denote by  ${\mathcal P}^\ast(Q,\lambda,\omega)$ the union
$$
{\mathcal P}^\ast(Q,\lambda,\omega) := {\mathcal P}(Q,\lambda,\omega)\bigcup \{\emptyset\}.
$$
\qed
\end{definition}
For the sc-Fredholm theory it will be important  to associate to the elements in ${\mathcal P}^\ast(Q,\-\lambda,\omega) $
weight sequences. However, care has to be taken that these choices are compatible with spectral gaps coming from
a certain class of self-adjoint operators which occur naturally after a choice of almost complex structure compatible 
with $(\lambda,\omega)$ has been made. We shall discuss this in end of the next subsection, after introducing
the before-mentioned class of self-adjoint operators, called asymptotic operators.

\subsection{Special Coordinates and Asymptotic Operators}
Consider  $(Q,\lambda,\omega)$,  a smooth manifold with a stable Hamiltonian structure. With $\xi:=\text{ker}(\lambda)$
we equip the symplectic vector bundle $(\xi,\omega)\rightarrow Q$ with a compatible complex structure $J:\xi\rightarrow \xi$, $J^2=-Id$,
so that $\omega\circ (Id\oplus J)$ equips each fiber with a positive definite inner product.  We equip we can equip $Q$ with the Riemannian metric
\begin{eqnarray}
g_J:= \lambda\otimes \lambda +\omega\circ(Id\oplus J).
\end{eqnarray}
  As we already explained before, the data  $(\lambda,\omega,J)$ will determine a ${\mathbb R}$-invariant 
almost complex structure $\wt{J}$ on ${\mathbb R}\times Q$, see (\ref{REDF28.3}) and this structure  will be important in
the sc-Fredholm theory.  However, not every $\wt{J}$ will work for a pseudoholomorphic curve theory
in ${\mathbb R}\times Q$.  It will be important that the underlying $(Q,\lambda,\omega)$ is non-degenerate and that weight sequences
are picked appropriately.  One can express this in various ways. The choice here is to pick suitable special coordinates and to bring 
the study of a periodic orbit into the context of a special model. Of course, it will be important to verify that 
different choices of special coordinates lead to the same conclusion.
\subsubsection{Special Coordinates}
We are interested in the geometry near
a given periodic orbit 
$$
([\gamma],T,k).
$$
The idea is to transfer the general problem to the model case $Q_0:=S^1\times {\mathbb R}^{2n-2}$
with periodic orbit $([\gamma_0],T,k)$, where $\gamma_0(t)=(t,0)$. The structure on $Q_0$ is given by $\lambda_0=dt$ and $\omega_0=\sum_{i=1}^{n-1} dx_i\wedge dy_i$
and $J_0$ is the standard structure on ${\mathbb R}^{2n-2}={\mathbb R}^2\oplus..\oplus {\mathbb R}^2$, where on the ${\mathbb R}^2$-factors 
$(1,0)$ is mapped to $(0,1)$. We define $\xi^0=\text{ker}(dt)$ and as inner product $dt\otimes dt + \langle.,.\rangle_{{\mathbb R}^{2n-2}}$.  Note that 
$\omega_0\circ (Id\oplus J_0)=\langle.,.\rangle_{{\mathbb R}^{2n-2}}$.  

What will be of interest to us are structures 
$(\lambda',\omega',J')$ defined near $\gamma_0(S^1)$, which coincide with $(\lambda_0,\omega_0,J_0)$ on $\gamma_0(S^1)$.
This data $(\lambda',\omega',J')$ will be obtained as the push forward of a restriction of $(\lambda,\omega,J)$ on $Q$ to a small 
open neighborhood of $\gamma(S^1)$ by a special choice of coordinates.
\begin{definition}
Let $(Q,\lambda,\omega)$ be a smooth manifold with a stable Hamiltonian structure.  Consider a periodic orbit
$([\gamma],T,k)$.
A {\bf special coordinate transformation} is a smooth diffeomorphism  $\phi:U(\gamma(S^1))\rightarrow U(\gamma_0(S^1))$ which has the following properties.
\begin{itemize}
\item[(1)] There exists a representative $\gamma^{\phi}\in[\gamma]$ such that $\phi\circ\gamma^{\phi}(t)=\gamma_0(t)$.
\item[(2)] For every $t\in S^1$ the tangent map 
$$
(T\phi)(\gamma^{\phi}(t)):T_{\gamma^{\phi}(t)}Q\rightarrow T_{\gamma_0(t)}Q_0
$$
induces  a linear isomorphism $\wh{\phi}_t:\xi_{\gamma^{\phi}(t)}\rightarrow \xi^0_{\gamma_0(t)}$
which is complex linear and isometric for the distinguished inner products.
\end{itemize}
\qed
\end{definition}
The following is left as an exercise.
\begin{lem}
Given $(Q,\lambda,\omega,J)$ and a periodic orbit $([\gamma],k,T)$ there exist for given representative $\gamma^{\phi}$
suitable open neighborhoods
$U(\gamma(S^1))$ and $U(\gamma_0(S^1))$ and a special coordinate transformation $\phi:U(\gamma(S^1))\rightarrow
U(\gamma_0(S^1))$ with the property $\phi\circ \gamma^{\phi}(t)=\gamma_0(t)$.
\qed
\end{lem}
\subsubsection{Asymptotic Operator}\label{ASYMOperator}
We start with $(Q,\lambda,\omega)$ and an associated periodic orbit $([\gamma],T_0,k_0)$
Consider the map which associates to an element $y$  in  $C^1(S^1,Q)$ the loop $\frac{d}{dt}y - T_0R(y)$.
The latter can be viewed as a $C^0$-section of $y^\ast TQ$. The section vanishes at $x=\gamma$ and consequently 
has a linearization $L_{\gamma}$ at every representative $\gamma$ of $[\gamma]$. 

Having fixed $\gamma:=\gamma^{\phi}\in [\gamma]$ take a special coordinate 
$\phi$ such that $\phi(\gamma(t))=\gamma_0(t)$. 
With this  choice of $\phi$ and given a loop $y$ near $\gamma$ we obtain a loop $\phi\circ y$ near $\gamma_0$.
We consider the following which associates to a smooth loop $z$ near $\gamma_0$ the smooth loop $\eta=\eta(z)$ in ${\mathbb R}^{2n-1}$
defined by
\begin{eqnarray}
t\rightarrow \text{pr}_2\circ T\phi\left(\frac{d}{dt}(\phi^{-1}(z(t)))- T_0\cdot R(\phi^{-1}(z(t)))\right).
\end{eqnarray}
We note that we can rewrite this as
$$
t\rightarrow \frac{d}{dt}z(t) - T_0\cdot [pr_2\circ T\phi \circ R\circ \phi^{-1}(z(t))]
$$
We differentiate this expression at $z=\gamma_0$ in the direction $h$ to obtain $L^{\phi}(h)$, which is a loop $S^1\rightarrow {\mathbb R}\times {\mathbb R}^{2n-2}$.
With the linear map 
$$
\wh{B}^{\phi}(t):{\mathbb R}\times {\mathbb R}^{2n-2}\rightarrow {\mathbb R}\times {\mathbb R}^{2n-2}
$$
 being obtained by differentiating $y\rightarrow T_0\cdot [pr_2\circ T\phi\circ  R\circ \phi^{-1}(y)$ at $\gamma_0(t)$ we see that
$L^{\phi}$ has the form 
$$
L^{\phi}(h)=\frac{d}{dt}h - \wh{B}(t)h.
$$
\begin{lem}
The following holds true.
\begin{itemize}
\item[(1)]  The map $\wh{B}^{\phi}$ has the form 
$$
\wh{B}^{\phi}(t)(h_1,\Delta)=(0,B^{\phi}(t)\Delta),
$$
where $h=(h_1,\Delta)\in {\mathbb R}\times {\mathbb R}^{2n-2}$.
\item[(2)] $-J_0B^{\phi}(t)$ is symmetric for the standard structure on ${\mathbb R}^{2n-2}$.
\item[(3)] The unbounded operator defined in $L^2(S^1,{\mathbb R}^{2n-2})$ with domain $H^1(S^1,{\mathbb R}^{2n-2})$
by $h\rightarrow -J_0[\frac{d}{dt}h -B^{\phi}(t)h]$ is self-adjoint and has a compact resolvent.
\end{itemize}
\end{lem}
\begin{proof}
These are known results and we refer for a discussion  to \cite{BEHWZ,Dragnev,Wendl}.
\end{proof}
With the above discussion in mind we make the following definition.
\begin{definition}
Given $(Q,\lambda,\omega,J)$ and a periodic orbit $([\gamma],T_0,k_0)$ we denote for given special coordinates 
by $\bm{L}^{\phi}$ the linear unbounded  self-adjoint operator
\begin{eqnarray}
&\bm{L}^{\phi}: L^2(S^1,{\mathbb R}^{2n-2})\supset H^1(S^1,{\mathbb R}^{2n-2})\rightarrow L^2(S^1,{\mathbb R}^{2n-2}):&\\
&h\rightarrow 
-J_0[\frac{d}{dt}h -B^{\phi}(t)h].&\nonumber
\end{eqnarray}
\qed
\end{definition}
It is important to know the relationship between $\bm{L}^{\phi}$ and $\bm{L}^{\psi}$ given two different choices 
of special coordinates.   Given an ordered pair $(\phi,\psi)$ we have that 
$$
\phi\circ \gamma^{\phi}(t)=\gamma_0(t)=\psi\circ\gamma^{\psi}(t),\ t\in S^1.
$$
There exists a well-define element $c=c(\phi,\psi)\in S^1$, called {\bf phase},  such that $\gamma^{\psi}(t)=\gamma^{\phi}(t+c)$ for $t\in S^1$.
Next we study a pair $(\phi,\psi)$ with phase $c=c(\phi,\psi)$,  We consider the transition map
$\sigma:=\psi\circ \phi^{-1}$ which is defined on an open neighborhood of $\gamma_0(S^1)$ in $S^1\times {\mathbb R}^{2n-2}$.
\begin{lem}
The following identity holds  for a pair $(\phi,\psi)$ with phase $c$
$$
\psi\circ\phi^{-1}(\gamma_0(t+c))=\gamma_0(t)
$$
\end{lem}
\begin{proof}
We compute
$$
 \gamma_0(t) =\psi\circ \gamma^{\psi}(t)=\psi\circ \gamma^{\phi}(t+c) = \psi\circ\phi^{-1} \circ\gamma_0(t+c).
$$
\end{proof}
Abbreviate $\sigma=\psi\circ\phi^{-1}$.  From the properties of $\phi$ and $\psi$ we know that for $(t,z)\in S^1\times {\mathbb R}^{2n-2}$
it holds that $\sigma(t+c,0)=(t,0)$ and moreover $D\sigma(t,0)(\{0\}\times {\mathbb R}^{2n-2})\subset \{0\}\times {\mathbb R}^{2n-2}$.
Further the induced map ${\mathbb R}^{2n-2}\rightarrow {\mathbb R}^{2n-2}$ is unitary for using the complex structure coming from $J_0$.
Denoting this map by $U_{(\phi,\psi)}(t)$ we obtain a loop of unitary matrices. We note that 
$$
D\sigma(t,0)(h,k) =(h,U(t)k),\ (h,k)\in {\mathbb R}\times {\mathbb R}^{2n-2}.
$$
Unitary here means that the operators commute with $J_0$ and are isometric.
Given the unitary loop  
$U_{(\phi,\psi)}$ we shall write $U_{(\phi,\psi)}^{-1}$  for the point-wise inverted loop.
\begin{lem}
For given $(\phi,\psi,\sigma)$
we have the identity 
\begin{eqnarray}\label{Gldfg}
U_{\phi,\sigma}(t) = U_{\psi,\sigma}(t-c(\phi,\psi))\circ U_{\phi,\psi}(t).
\end{eqnarray}
\end{lem}
\begin{proof}
We shall write $\sigma\circ\phi^{-1} =(\sigma\circ\psi^{-1})\circ (\psi\circ\phi^{-1})$ and 
 recall  that $\psi\circ\phi^{-1}(\gamma_0(t)=\gamma_0(t-c(\phi,\psi))$.  Differentiating the first expression along $\gamma_0(t)$
and taking the ${\mathbb R}^{2n-2}$-part we obtain at $\gamma_0(t)$
$$
U_{(\phi,\sigma)}(t) = U_{(\psi,\sigma)}(t-c(\phi,\psi))\circ U_{(\phi,\psi)}(t).
$$
\end{proof}
Finally we show the following.
\begin{prop}
For given $(Q,\lambda,\omega,J)$ and a choice of periodic orbit $([\gamma],T_0,k_0)$ consider 
for associated special coordinates $\phi,\psi$ the asymptotic operators $\bm{L}^{\phi}$ and $\bm{L}^{\psi}$.
Then the following equality holds
$$
\bm{L}^{\psi} = \mathsf{U}_{(\phi,\psi)}\circ \bm{L}^{\phi}\circ \mathsf{U}^{-1}_{(\phi,\psi)}.
$$
\end{prop}
\begin{proof}
We have the following two expressions  from which $\bm{L}^{\phi}$ and $\bm{L}^{\psi}$ are being derived.
\begin{eqnarray}
& z\rightarrow \frac{d}{dt} z -T_0\cdot [\text{pr}_2\circ T\phi\circ R\circ\phi^{-1}(z)]&\\
& y\rightarrow \frac{d}{dt} y -T_0\cdot [\text{pr}_2\circ T\psi\circ R\circ\psi^{-1}(y)]&\nonumber
\end{eqnarray}
Let us define $\sigma=\psi\circ\phi^{-1}$. We note that $\sigma\circ\gamma_0(t) = \gamma_0(t-c)$
where $c=c(\phi,\psi)$. We define the map $z\rightarrow y$ by $y(t-c)=\sigma\circ z(t)$.

We compute
\begin{eqnarray*}
 &&\left(\frac{d}{dt} y -T_0\cdot [\text{pr}_2\circ T\psi\circ R\circ\psi^{-1}(y)]\right)(t-c)\\
&=& \frac{d}{dt} y(t-c) -T_0\cdot [\text{pr}_2\circ T\psi\circ R\circ\psi^{-1}(y(t-c))]\\
&=& \frac{d}{dt} (\sigma\circ z(t)) - T_0 \cdot  [\text{pr}_2\circ T\sigma\circ T\phi\circ R\circ\phi^{-1}(z(t))]\\
&=& D\sigma(z(t))[ \frac{d}{dt}z(t) -T_0 \cdot T\phi\circ R\circ\phi^{-1}(z(t))].
\end{eqnarray*}
Differentiating the relationship $y(t-c)=\sigma\circ z(t)$ between the input loop $z$
and the output loop $y$ with respect to $z$ at $\gamma_0$ in the direction of $h$ gives
$$
k(t-c) = D\sigma(t,0)h(t).
$$
Hence we obtain
\begin{eqnarray*}
&&  \frac{d}{dt}k(t-c) - \wh{B}^{\psi}(t-c) k(t-c)\\
&=& D\sigma(t,0)[\frac{d}{dt}h(t) - \wh{B}^{\phi}(t) h(t)]
\end{eqnarray*}
Specializing we obtain from this for $h\in H^1(S^1,{\mathbb R}^{2n-2})$ also the relation ship
$$
\frac{d}{dt}k(t-c) -B^{\psi}(t-c)k(t-c)= U_{(\phi,\psi)}(t) [\frac{d}{dt}h(t) - B^{\phi}(t) h(t)]
$$
This means that
$$
(\frac{d}{dt} -B^{\psi})\mathsf{U}_{(\phi,\psi)}h = \mathsf{U}_{(\phi,\psi)} (\frac{d}{dt}h- B^{\phi}h),
$$
and after multiplying by $-J_0$ we obtain the desired result.
\end{proof}
As a consequence of the previous discussion we can define the $J$-spectral interval around $0$ 
associated to a periodic orbit $([\gamma],T,k)$ associated to $(Q,\lambda,\omega)$ and $J$.
\begin{definition}\label{SPECTrum}
Let $(Q,\lambda,\omega)$ be a closed manifold equipped with a stable Hamiltonian structure 
and $J$ an admissible complex multiplication for $\text{ker}(\lambda)\rightarrow Q$. 
The $J$-{\bf spectral interval} associated to a periodic orbit $\bm{\gamma}=([\gamma],T_0,k_0)$
is the largest interval $(a,b)\subset {\mathbb R}$,  $a\leq 0\leq b$, such that $\sigma(\bm{L}^{\phi})\cap (a,b)=\emptyset$.
Here $\sigma{\bm{L}}^{\phi})$ is the spectrum associated to this self-adjoint operator. 
We also define  $\sigma(\bm{\gamma},J):= \sigma(\bm{L}^{\phi})$, which, of course, does not depend 
on $\phi$, and call it the $J$-{\bf spectrum} of $\bm{\gamma}$.
\qed
\end{definition}
\begin{remark}
Since for $\phi,\psi$ the associated operators are unitarily conjugated the $J$-interval does not depend 
on the choice of special coordinate.  In the case that $\bm{\gamma}$ is non-degenerate the spectral interval will be 
nonempty, containing $0$ in the interior.
\qed
\end{remark}
As just mentioned the non-degeneracy assumption implies that $0\not\in \sigma (\bm{\gamma},J)$.
Since the operator $\bm{L}^{\phi}$  has a compact resolvent we have a spectral gap around
$0$. We call a positive number $\delta$ associated to $\bm{\gamma}=([\gamma],T,k)$  admissible, provided
$$
\sigma(\bm{\gamma},J)\cap [-\delta,\delta]=\emptyset.
$$
Note that the admissibility of $\delta$ depends on the original choice of $J$, which most of the time is fixed
from the beginning. If we want to stress the dependence on $J$ we call $\delta$ admissible for $(\bm{\gamma},J)$.
There are several notions and results associated to periodic orbits, which are frequently used in constructions.
\begin{definition}\label{DEFSpectrum}
Given a non-degenerate stable Hamiltonian structure $(\lambda,\omega)$ on the closed
manifold $	Q$ and a compatible $J$ we call a map
$$
\delta_o:{\mathcal P^\ast}\rightarrow (0,\infty)
$$
a {\bf weight selector} associated to $J$ provided for every $\bm{\gamma}=([\gamma],k,T) \in {\mathcal P}$,  the  number $\delta(\bm{\gamma},J)$ associated to $(\bm{\gamma},J)$ is admissible and bounded strictly by $2\pi$.  In addition we require that $\delta(\emptyset)\in (0,2\pi)$. 
A {\bf weight sequence} $(\delta_i)$ for $(\lambda,\omega,J)$ is a sequence
of weight functions $\delta_m$ so that for every periodic orbit $\gamma$ or $\gamma=\emptyset $ we have
$$
0<\delta_0(\gamma)<\delta_1(\gamma)<....
$$
\qed
\end{definition}

\subsection{Conley-Zehnder and Maslov Index}\label{CZMaslov}
In the (Fredholm) index theory for the CR-oper\-ator the Conley-Zehnder index plays an important role. We follow 
\cite{Dragnev}, which is based on \cite{FloerH1,FloerH2,HWZ-Em,SZ}. We view ${\mathbb R}^2$ as a symplectic vector space with coordinates 
$(x,y)$ and symplectic form $dx\wedge dy$.  Then ${\mathbb R}^{2n}$ is identified with the direct sum ${\mathbb R}^2\oplus..\oplus {\mathbb R}^2$,
coordinates $(x_1,y_1,...,x_n,y_n)$, and symplectic form $\omega=\sum_{i=1}^n dx_i\wedge dy_i$. 
\subsubsection{Conley-Zehnder Index} Denoting by $\text{Sp}(n)$ the group of linear symplectic maps 
${\mathbb R}^{2n}\rightarrow {\mathbb R}^{2n}$ we consider the space of continuous arcs $\Phi:[0,1]\rightarrow \text{Sp}(n)$ starting at the identity $\text{Id}_{2n}$ 
at $t=0$ and ending at $\Phi(1)$ which is a symplectic map not having $1$ in the spectrum. We denote by $ \Sigma(n)$ the maps 
$\alpha:[0,1]\rightarrow \text{Sp}(n)$ starting and ending at $\text{Id}_{2n}$. The map
$$
\text{G}(n)\times\Sigma(n)\rightarrow \Sigma(n): (\alpha,\Phi)\rightarrow \alpha\cdot\Phi, \ (\alpha\cdot\Phi)(t)=\alpha(t)\circ\Phi(t)
$$
is well-defined. We also have the inversion map
$$
\Sigma(n)\rightarrow \Sigma(n): \Phi\rightarrow \Phi^{-1},\ \Phi^{-1}(t):= (\Phi(t))^{-1}.
$$
Finally there is the obvious map
$$
\Sigma(n)\times\Sigma(m)\rightarrow\Sigma(n+m): (\Phi,\Psi)\rightarrow \Phi\oplus\Psi.
$$
A classical map $\mu^n_M:\text{G}(n)\rightarrow {\mathbb Z}$ is the Maslov index which is characterized by the following theorem.
\begin{thm}
The maps $\mu^n_M$ for $n\in \{1,....\}$ are characterized by the following requirements.
\begin{itemize}
\item[(1)] Two loops  $\alpha_1,\alpha_2\in \text{G}(n)$ are homotopic in $\text{G}(n)$ if and only if $\mu^n_M(\alpha_1)=\mu^n_M(\alpha_2)$.
\item[(2)] The map induced on $\pi_1(\text{Sp}(n),Id)$ by $\mu^n_M$ is a group isomorphism to ${\mathbb Z}$, i.e in this particular case equivalently
$$
\mu^n_M(\alpha_1\cdot \alpha_2) =\mu^n_M(\alpha_1)+\mu^n_M(\alpha_1).
$$
Here $(\alpha_1\cdot \alpha_2)(t)=\alpha_1(t)\circ\alpha_2(t)$.
\item[(3)] It holds 
$$
\mu^M_n\left(\left[t\rightarrow\left(e^{2\pi it}Id_2 \oplus Id_{2n-2}\right)\right]\right)=1.
$$
\end{itemize}
\qed
\end{thm}
Having characterized the Maslov index we state the main result about the Conley-Zehnder index.
The Conley-Zehnder index refers to a family of maps $\mu_{CZ}^n:\Sigma(n)\rightarrow {\mathbb Z}$, $n\in \{1,2,..\}$.

\begin{thm}
There exists a unique family $\mu_{CZ}^n: \Sigma(n)\rightarrow {\mathbb Z}$ for $n\in\{1,2,..\}$
characterized by the following properties:
\begin{itemize}
\item[(1)] Homotopic maps in $\text{G}(n)$ have the same index $\mu_{CZ}^n$.
\item[(2)] For $\alpha\in \text{G}(n)$ and $\Phi\in \Sigma(n)$ the identity
$$
\mu_{CZ}^n(\alpha\cdot\Phi) =\mu_{CZ}^n(\Phi) + 2\cdot \mu_M^n(\alpha).
$$
\item[(3)] $\mu_{CZ}^n(\Phi^{-1})+\mu_{CZ}^n(\Phi)=0$.
\item[(4)] $\mu_{CZ}^1(\gamma)=1$, where $\gamma(t)=e^{\pi it}Id_{{\mathbb R}^2}$.
\item[(5)] $\mu^{n+m}_{CZ}(\Phi\oplus \Psi)=\mu^n_{CZ}(\Phi) +\mu^m_{CZ}(\Psi)$.
\end{itemize}
\qed
\end{thm}

\subsubsection{Trivializations and CZ-Index}
\jwf{\marginnote{ADD HERE}}
\section{Stable Hamiltonian Symplectic Cobordisms   }\label{APPENDIX28.3}
We discuss a variety of issues around compact symplectic manifolds with boundaries admitting compatible stable Hamiltonian structures.
\subsection{Basic Concepts}
Ina first step we define the objects of interest and present them in a way which fits with our larger goals.
\begin{definition}
Let $(W,\Omega)$ be a compact symplectic manifold with smooth boundary. We say the boundary is of {\bf stable Hamiltonian type}
provided there exists a  one-form $\lambda$ on $\partial W$ such that with $\omega$ being the pull-back of $\Omega$ to $\partial W$
the pair $(\lambda,\omega)$ is a stable Hamiltonian structure on $\partial W$ in the sense of Definition \ref{28.1QQ}
\qed
\end{definition}
We note that if $(\lambda,\omega)$ is a stable Hamiltonian structure so is $(-\lambda,\omega)$.
If $\partial W$ has different boundary components we obviously can make  sign changes on the different components.
Given $(W,\Omega)$ with $W$ of dimension $2n$ the $n$-fold exterior power $\Omega^n$ is a volume form and 
defines a {\bf natural orientation }on $W$.  Taking an outward pointing vector field $X$ on $\partial W$ the volume form
$\Omega^n(X,.)$ on $\partial W$ orients $\partial W$ and this is called the standard orientation for the boundary. 
Given $(W,\Omega)$ take $\lambda$ as described so that $(\lambda,\omega)$ is a stable Hamiltonian structure.
Then $\lambda\wedge \omega^{n-1}$ is a volume form on $\partial X$ and induces an orientation. 
Given a connected component $Q$ of $\partial X$ we say that $(\lambda,\omega)$ is {\bf positive} or {\bf negative} over $Q$,
provided the stable Hamiltonian orientation is the same or the opposite of the natural orientation.  For  the applications 
it is important to designate a union of some boundary components as positive and the rest as negative.
With such an additional choice one might view a compact symplectic manifold with stable Hamiltonian boundary 
as a directed symplectic cobordism. We make this precise next.

\begin{definition}
A {\bf stable Hamiltonian cobordism}  is given by a tuple $(W,\Omega,\upsilon)$, where $(W,\Omega)$ is a compact symplectic manifold
with smooth boundary of stable Hamiltonian type, and $\upsilon:\partial W\rightarrow \{-1,1\}$ is a locally constant map.
We call $\partial_\pm W=\{w\in\partial W\ |\ \upsilon(w)=\pm 1\}$ the positive and negative boundary, respectively.
\qed
\end{definition}
These orientations can be understood as follows. Given $w\in \partial W$ we consider the full tangent space $T_wW$ with with $\Omega_w$ is a symplectic vector space. 
Pick a vector $R\neq 0$ in $T_w\partial W$ lying in the kernel of $\omega$. We can take a complement $C$ in $T_w\partial W$ to the line generated by $R$. 
For $(C,\omega)$ we take a symplectic basis $a_1,b_1,...,a_{n-1},b_{n-1}$. Finally let $X\in T_wW$ be outward pointing so that it is $\Omega$-orthogonal
to $C$ and by changing the sign of $R$ 
we may assume that $\Omega_w(X,R)>0$. Then $X,R,a_1,b_1,...,a_{n-1},b_{n-1}$ is a symplectic basis for $T_wW$, which also defines the standard orientation,
and $R,a_1,b_1,...,a_{n-1},b_{n-1}$ defines the natural orientation for $\partial W$ at $w$ by definition. 
A one-form $\lambda$ so that $(\lambda,\omega)$ is a stable Hamiltonian structure is then positive provided $\lambda(R)>0$. 

Given a stable Hamiltonian cobordism $(W,\Omega,\upsilon)$ we can pick a one form $\lambda$ on $\partial W$
which is positive and negative on $\partial_\pm W$ so that $(\lambda,\omega)$ is a stable Hamiltonian structure.  Let us abbreviate $Q^\pm:=\partial_{\pm}W$
and $\lambda_{\pm}=\lambda|Q^\pm$ and similarly $\omega_{\pm}$.
If we pick $\lambda$ small enough we assume the following.
\begin{itemize}
\item[(1)] $\bar{\Omega}_{\pm}:=\omega_{\pm} + d(a\lambda_{\pm})$ is a symplectic form on $(-3,3)\times Q^\pm$. 
\item[(2)]  There exist neighborhoods $U(Q^\pm)$ with disjoint closures in $W$, and 
 diffeomorphisms $\Psi_+:(-3,0]\times Q^+\rightarrow U(Q^+)$, $\Psi_-:[0,3)\times Q^-\rightarrow U(Q^-)$
such that $\Psi_{\pm}(0,q)=q$ and 
$$
\Psi^\ast _{\pm}\Omega = \Omega_{\pm}
$$
\end{itemize}
Here $\Omega_-$ is a symplectic form on $[0,3)\times Q^-$ obtained by restricting $\bar{\Omega}_-$  and $\Omega_+$ similarly for $(-3,0]\times Q^+$. 

We shall study stable maps with image in $W$ and shall also define a Cauchy-Riemann section.
Stable maps of height $1$ will actually take its image in $W\setminus \partial W$ and there will be concatenations with buildings in 
${\mathbb R}\times Q^-$ and ${\mathbb R}\times Q^+$.  The concatenated objects will appear as limits of height-$1$-sequences.
In the present situation the choice os suitable almost complex structures on $W\setminus\partial W$ will be important and the fact these
structures which degenerate towards the boundary are compatible with whole families of symplectic structures which is used 
in the compactness results, \cite{BEHWZ}, to control the areas of pseudoholomorphic curves near $\partial W$.

\subsection{A Family of Symplectic Structures}
We shall introduce a particular family ${\mathcal F}$ of non-negative two-forms and symplectic structures on $W$ derived in some sense from $\Omega$
and the choices of $\lambda_{\pm}$ and $\Psi_{\pm}$. The family will turn out to be compatible with an important class of almost complex structures
on $W\setminus\partial W$.  We consider pairs $\phi=(\phi_-,\phi_+)$ of smooth maps satisfying 
\begin{itemize}
\item[(1)] $\phi_+:(-3,0]\rightarrow (-3,3)$ satisfies
\begin{eqnarray}
&\phi_+(s)=s,\ s\in (-3,-1]&\\
&\phi'_+(s)\geq 0,\ s\in (-1,0]\nonumber
\end{eqnarray}
\item[(2)]  $\phi_-:[0,3)\rightarrow (-3,3)$ satisfies
\begin{eqnarray}
&\phi_-(s)=s,\ s\in [2,3)&\\
&\phi'_-(s)\geq 0,\ s\in  [0,1)\nonumber
\end{eqnarray}
\end{itemize}
We denote by $\Sigma$ the collection of all $\phi$ having the above properties. For given $\phi\in \Sigma$ we define 
the two-form $\Omega_{\phi}$ on $W$ as follows. If $w\in W\setminus(U(Q^-)\cup U(Q^-))$ we put
\begin{eqnarray}
\Omega_{\phi}(w)=\Omega(w).
\end{eqnarray}
Further we define implicitly on $U(Q^+)$ and $U(Q^-)$.
\begin{eqnarray}
&\Psi_-^\ast\Omega_{\phi} = \omega_- + d(\phi_- \lambda_-)\ \text{on}\ [0,3)\times Q^-&\\
&\Psi_+^\ast\Omega_{\phi}=\omega_+ +d(\phi_+ \lambda_+)\ \text{on}\ (-3,0]\times Q^+&\nonumber
\end{eqnarray}
The family $\{\Omega_{\phi}\ |\ \phi\in\Sigma\}$ is called the {\bf family of compatible two-forms}. We note that if $\phi_{\pm}$ have positive 
derivatives the associated $\Omega_{\phi}$ is symplectic.  In this case we say that $\Omega_{\phi}$ is an {\bf admissible symplectic form}.
Admissibility refers to the data $(W,\Omega,\upsilon)$, the choice of a small enough $\lambda$, $(\lambda,\omega)$ being a stable Hamiltonian structure
 compatible with $\upsilon$, 
and $\Psi_{\pm}$.  We shall use the abbreviation 
\begin{eqnarray}
\mathscr{D} =((W,\Omega,\upsilon),\lambda_-,\lambda_+,\Psi_-,\Psi_+),
\end{eqnarray}
where $\mathscr{D}$ stands for `data'.
\subsection{Admissible Almost Complex Structures}
Our aim in this subsection is to define a suitable class of smooth almost complex structures $\wt{J}$ on $\dot{W}:=W\setminus\partial W$ compatible in some sense, to be made precise,  with the data $\mathscr{D}$. We introduce the abbreviations 
$$
\dot{U}(Q^\pm) = U(Q^\pm)\setminus Q^\pm.
$$

Let $\varphi_+:(-3,0)\rightarrow (-3,+\infty)$ be the map having the following properties
\begin{eqnarray}
&\varphi_+(s)=s\ \text{for}\ s\in (-3,-2].&\\
&\varphi'_+(s)>0\ \text{for}\ s\in (-2,-1).&\nonumber\\
&\varphi_+(s)= e^{-\frac{1}{s}}-e\ \text{for}\ s\in[-1,0).&\nonumber
\end{eqnarray}
We define the diffeomorphism $\bm{\Psi_+}$ through its inverse
$$
\bm{\Psi_+}^{-1}: \dot{U}(Q^+)\rightarrow (-3,\infty)\times Q^+: w\rightarrow (\phi_+\times Id_{Q^+})\circ \Psi_+^{-1}(w).
$$
\jwf{
\begin{definition}
A smooth almost complex structure $\wt{J}$ on $W\setminus\partial W$ is said to be {\bf admissible} provided for every $\phi\in\Sigma$ the following holds.
Over $ W_1$ the composition  $\Omega_{\phi}\circ (Id\oplus \wt{J})$ is a Riemannian metric (note that this expression does not depend on $\phi$ oever $W_1$.)
Over $\dot{U}_1(Q^+)$ the structure $\wt{J}$ is the pull-back of 
$$
.....
$$
\end{definition} 
\textcolor{red}{Variable symplectic structures models etc}
}

Consider a closed manifold equipped with a stable Hamiltonian structure $(Q,\lambda,\omega)$ as introduced in  Definition \ref{28.1QQ}.
Given a smooth map $\phi:{\mathbb R}\rightarrow [0,1]$ satisfying $\phi'(s)\geq 0$ consider the two form
on ${\mathbb R}\times Q$ defined with $(b,h),(b',h')\in T_{(a,q)}({\mathbb R}\times Q)$ by
\begin{eqnarray}
\ \ \ \ \Omega_{\phi}(a,q)((b,h),(b',h'))&=& \omega(q)(h,h') + \phi(a)d\lambda(q)(h,h')\\
&+& \phi'(a)(b\cdot \lambda(q)(h')-b'\cdot \lambda(q)(h)).\nonumber
\end{eqnarray}
Sloppily we may write this as $\omega+d(\phi\lambda)$. With $R$ being the Reeb vector field we know that $L_R\lambda=0$.
In view of the Cartan formula this implies that $0=d(\lambda(R))+ d\lambda(R,.)=d\lambda(R,.)$ using that $\lambda(R)\equiv 1$.
With other words $\text{ker}(\omega)\subset \text{ker}(D\lambda)$. Using the splitting $TQ={\mathbb R}R\oplus \xi$
we write $h= sR(q) +\Delta$ and similarly $h'= s'R(q)+\Delta'$.  Continuing the previous calculation 
we see that 
\begin{eqnarray}
&\omega(q)(h,h')+\phi(a)d\lambda(q)(h,h')= \omega(q)(\Delta,\Delta')+\phi(a)d\lambda(q)(\Delta,\Delta')&
\end{eqnarray}
There exists an $\varepsilon\in (0,1]$ such that the above expression on the right-hand side is non-degenerate for 
every $\phi:{\mathbb R}\rightarrow [0,\varepsilon]$. Moreover
\begin{eqnarray}
\phi'(a)(b\cdot\lambda(q)(h')-b'\cdot \lambda(q)(h))&=& b\cdot s' - b'\cdot s,
\end{eqnarray}
which a non-degenerate expression on the right-hand side in case $\phi'(a)>0$.  Given $J:\xi\rightarrow \xi$
with $J^2=-Id$ so that $\omega\circ (Id_\xi\oplus J)$ defines a positive definition inner product on the fibers 
of $\xi\rightarrow Q$ the associated $\wt{J}$ has the following property noting that
$$
\wt{J}(a,q)(b,sR(q)+\Delta)= (-s,bR(q)+J(q)\Delta).
$$
Namely 
\begin{eqnarray}
&& \Omega_{\phi}(a,q)((b,h),\wt{J}(a,q)(b,h))\\
&=& (\omega(q)(\Delta,\Delta')+\phi(a) d\lambda(q)(\Delta,\Delta')) + \phi'(a)(s^2 + b^2).\nonumber
\end{eqnarray}
Hence if $\phi:{\mathbb R}\rightarrow [0,\varepsilon]$ with $\phi'(a)\geq 0$ the above expression 
is always non-negative and if $\phi'(a)$ even symplectic.  Denote by $\Sigma$ the set of all 
smooth maps ${\mathbb R}\rightarrow [0,\varepsilon]$  such that $\phi'(a)\geq 0$. 
As we have seen $\Omega_\phi$ is a symplectic form on ${\mathbb R}\times Q$ provided $\phi'(a)>0$ for all $a$.
The family of all $\Omega_\phi$ on ${\mathbb R}\times Q$ is invariant under the ${\mathbb R}$-action 
and as a result of the previous discussion contains an ${\mathbb R}$-invariant family of symplectic forms.

Another important fact is that given a closed hypersurface $Q$ of dimension $2n-1$ in a symplectic 
manifold $(W,\Omega)$ so that there exists a nowhere vanishing one-form $\lambda$ on $Q$ for which the 
pull-back $\omega$ of $\Omega$ to $Q$ defines a stable Hamiltonian structure on $Q$ there exists an embedding
 $\Psi: (-\varepsilon,\varepsilon)\times Q\rightarrow W$ such that $\Psi^\ast\Omega = \omega+ d(a\lambda)$.
 We can apply this fact in the following situation.
 
\subsection{A Convenient Representation}
\jwf{\textcolor{red}{With ends glued in}

\textcolor{red}{NEED THE FORM FOR $\wt{J}$ which can be achieved in the $S^1$-model}}

\section{Estimates for  CR-Operators}\label{CRTYPOP}
We shall discuss a variety of estimates for Cauchy-Riemann type operators which are being used in establishing 
the properties of the CR-sections occurring in the polyfold theory.

\subsection{CR-Type Operators on the Riemann Sphere}
Consider the Riemann sphere $S^2$ and consider  for some $m\geq 1$ the Sobolev space
$H^m(S^2,{\mathbb C}^n)$.  We denote by $\Omega\rightarrow S^2$ the complex vector bundle 
whose fiber over $z\in S^2$ consists of all complex anti-linear maps $T_zS^2\rightarrow {\mathbb C}^n$. 
We denote by $H^{m-1}(\Omega)$ the Sobolev space of sections of $\Omega\rightarrow S^2$ of Sobolev class
$H^{m-1}$.  We can define the standard Cauchy-Riemann opertor 
$$
\bar{\partial}: H^m(S^2,{\mathbb C}^n)\rightarrow H^{m-1}(\Omega):u\rightarrow \frac{1}{2}\cdot [Du+i\circ Du\circ i],
$$
where $Du=\text{pr}_2\circ Tu$. The basic result about this operator is the following.
\begin{thm}\label{ATHEOREM1}
For every $m\geq 1$ the CR-operator $\bar{\partial}$ is a complex linear and surjective Fredholm operator
of complex Fredholm index $n$. The kernel is spanned by the constant maps.
\qed
\end{thm}
A proof can be found in the appendix of \cite{HZ}. It turns out this result can be used as one of the basic results 
for deriving  all other results relevant to us.

If we remove $0$ and $\infty$ from $S^2$ we can identify $S^2\setminus\{0,\infty\}\equiv {\mathbb C}\setminus\{0\}$
which we can parametrize using holomorphic polar coordinates 
$$
{\mathbb R}\times S^1\rightarrow {\mathbb C}\setminus\{0\}:(s,t) \rightarrow e^{2\pi(s+it)}.
$$
In SFT certain Cauchy-Riemann type operators on punctured Riemann surfaces will be important and the description 
near the punctures it best done using positive or negative holomorphic polar coordinates. 

We give next version of Theorem \ref{ATHEOREM1} in this spirit. 
Fix a negative weight $-\delta\in (-2\pi, 0)$ and define for $m\geq 1$ the Sobolev space $H^{m,-\delta}({\mathbb R}\times S^1,{\mathbb C}^n)$
to consist of all maps $u:{\mathbb R}\times S^1\rightarrow {\mathbb C}^n$ which belong to $H^m_{loc}$ so that in addtion
$|u|_{H^{m,-\delta}}<\infty$, where 
$$
|u|^2_{H^{m,-\delta}}=\sum_{|\alpha|\leq m} \int_{{\mathbb R}\times S^1} |D^{\alpha}u(s,t)|^2 e^{-2\delta |s|} ds dt.
$$
We can define 
\begin{eqnarray}\label{CREQ1}
&\ \ \ \ \ \ \ \ \ \bar{\partial}:H^{m,-\delta}({\mathbb R}\times S^1,{\mathbb C}^n)\rightarrow H^{m-1,-\delta}({\mathbb R}\times S^1,{\mathbb C}^n):u\rightarrow u_s+iu_t.&
\end{eqnarray}
\begin{thm}\label{ACXZ1}
Let $m\geq 1$ and $-\delta\in (-2\pi,0)$. Then the Cauchy-Riemann operator in (\ref{CREQ1}) is a complex linear surjective 
Fredholm operator whose kernel consists of the constant maps.  In particular the complex Fredholm index is $n$.
\qed
\end{thm}
This is essential the same as Theorem \ref{ATHEOREM1}, but using different function spaces. 
This can be proved using Fourier series.   If we take $\delta\in (0,2\pi)$ as a weight and work with $H^{m,\delta}$ it will turn out that 
$\bar{\partial}$ will stay Fredholm, but will change its Fredholm index to $-n$ and  will be injective.  We can obtain an isomorphism by slightly enlarging the domain.
Define for $\delta\in (0,2\pi)$ and $m\geq 1$ the Sobolev space $H^{m,\delta}_{ap}({\mathbb R}\times S^1,{\mathbb C}^n)$ as follows.
It consists of all $u\in H^m_{loc}$ so that there exists a smooth map $v:{\mathbb R}\times S^1\rightarrow {\mathbb C}^n$ (depending on $u$)  which is constant outside of a compact set
with $v(+\infty)+v(-\infty)=0$ so that
$$
|u-v|_{H^{m,\delta}}<\infty.
$$
The definition does not depend on the choice of $v=v(u)$. We can define a norm as follows. Take a smooth cut-off model $\beta:{\mathbb R}\rightarrow [0,1]$ and define 
$v_c:{\mathbb R}\times S^1\rightarrow {\mathbb C}^n$ as follows.
$$
v_c(s,t) =(1-2\beta(s))c,
$$
where $c\in {\mathbb C}$. An element $u\in H^{m,\delta}_{ap}({\mathbb R}\times S^1,{\mathbb C}^n)$ can be uniquely written as 
$$
u = v_c +w
$$
for a suitable $c\in {\mathbb C}$ and $w\in H^{m,\delta}({\mathbb R}\times S^1,{\mathbb C}^n)$.  We define the norm on $H^{m,\delta}_{ap}({\mathbb R}\times S^1,{\mathbb C}^n)$
by
$$
|w|^2_{H^{m,\delta}_{ap}}({\mathbb R}\times S^1,{\mathbb C}^n) =|c|^2 + |w|_{H^{m,\delta}}^2.
$$
\begin{thm}\label{ACXZ2}
For $m\geq 1$ and $\delta\in (0,2\pi)$ consider the Cauchy-Riemann operator 
$$
\bar{\partial}:H^{m,\delta}_{ap}({\mathbb R}\times S^1,{\mathbb C}^n)\rightarrow H^{m-1,\delta}({\mathbb R}\times S^1,{\mathbb C}^n):u\rightarrow u_s+iu_t.
$$
Then $\bar{\partial}$ is a complex linear topological isomorphism.
\qed
\end{thm}

The next class of operators of interest are of the following form. Given $u:{\mathbb R}\times S^1\rightarrow {\mathbb C}^n$
we consider the operator 
$$
u\rightarrow u_s+iu_t +A(t) u,
$$
where $S^1\rightarrow L_{\mathbb R}({\mathbb C}^n)$ is a smooth arc of real linear maps, which are self-adjoint for the real inner product
on ${\mathbb C}^n\equiv {\mathbb R}^2\oplus..\oplus{\mathbb R}^2$. Note that identifying ${\mathbb C}\equiv {\mathbb R}^2$ we identify $i$ with the $2\times 2$ matrix $J_0$.
First of all we recall that 
$$
\bm{L}_A: L^2(S^1,{\mathbb C}^n)\supset H^1(S^1,{\mathbb C}^n)\rightarrow L^2(S^1,{\mathbb C}^n): x\rightarrow -i\frac{dx}{dt}-A(t)x
$$
is a real self-adjoint operator with compact resolvent and has a pure point spectrum with real multiplicities bounded by $2n$.
\begin{thm}
Let $\delta\in {\mathbb R}$ such that $[-|\delta|,|\delta|]$ is not contained in the spectrum $\sigma(\bm{L}_A)$ of $\bm{L}_A$.  
Then for $m\geq 1$ the real linear operator
\begin{eqnarray*}
&H^{m,\delta}({\mathbb R}\times S^1,{\mathbb C}^n)\rightarrow H^{m-1,\delta}({\mathbb R}\times S^1,{\mathbb C}^n)&\\
& u\rightarrow u_s +i u_t +A(t) u&
\end{eqnarray*}
is topological linear isomorphism.
\qed
\end{thm}
For more results along these lines and particularly the relationship between the spectral flow and the Maslov/ Conley-Zehnder index, see  \cite{RS2}.

\jwf{
\subsection{Classical Case Estimates}
Consider a smooth almost complex structure $J$  on ${\mathbb C}^n$ such that $J(0)\equiv i$, i.e. complex multiplication by $i$. 
Our aim is derive some estimate for the map $(v,u)\rightarrow Tu +J(u)\circ Tu\circ j(v)$, where $v\rightarrow j(v)$, $j(0)=i$,  is a finite-dimensional deformation 
of $i$ and $u(0)=0$.  We are only interested in estimates for $u$  restricted to a small neighborhood of $0$ and $v$ near $0$. Global studies can be be brought by partitions
of unity arguments into this special case. We write the expression as 
\begin{eqnarray*}
&&Tu +J(u)\circ Tu\circ j(v)\\
&=& Tu + J(0) \circ Tu\circ i  \\
&& + J(0)\circ Tu\circ (j(v)-i)  +(J(u)-J(0))\circ Tu\circ j(v) \\
&=& Tu + i \circ Tu\circ i  \\
&& + J(0)\circ Tu\circ (j(v)-i)  +(J(u)-J(0))\circ Tu\circ j(v) \\
\end{eqnarray*}
We assume next that $u\in H^m(D,{\mathbb C}^N)$

\subsection{Nodal Case Estimates}
Our Cauchy Riemann operator is an sc-smooth section

\subsection{Periodic Orbit Case Estimates}

\subsection{Reflexive Compactness Estimates}
}

\section{A Calculus Lemma}\label{CALCAPP}
The result discussed here is  a generalization for Lemma 4.4. in \cite{HWZ8.7}.
Denote by $\varphi:(0,1]\rightarrow [0,\infty)$ the exponential gluing profile defined by 
$$
\varphi(r) =e^{\frac{1}{r}}-e
$$
We start by defining a map. For this let $T>0$ be a positive number and define 
$$
\mathsf{B_T}:[0,1]\rightarrow [0,1]
$$
by setting  $\mathsf{B}(0)=0$ and for $x>0$ we  put
$$
\mathsf{B}_T(x)=\varphi^{-1}(T\cdot\varphi(x)).
$$
We note that $\mathsf{B}(1)= \varphi^{-1}(0)=1$. Clearly $\mathsf{B}_T$ is continuous on $[0,1]$
and restricted to $(0,1)$ it is a smooth diffeomorphism $(0,1)\rightarrow (0,1)$.
Define $y=\varphi(x)$. Then 
$$
\varphi^{-1}(y)=\frac{1}{\ln(y+e)}
$$
Using this identity we compute if $x\in (0,1)$
\begin{eqnarray*}
& \varphi^{-1}\left(T \cdot \varphi(x)\right)=\frac{1}{\ln \left( e +T\cdot \left( e^{\frac{1}{x}}-e\right)\right)}
=: \frac{1}{\ln \left( T\cdot e^{\frac{1}{x}} + D \right)}, &
\end{eqnarray*}
where $D=e-Te$. 
Continuing we find 
\begin{eqnarray*}
 \varphi^{-1}\left(T \cdot \varphi(x)\right)
&=& x\cdot \frac{1}{1+g(x)},
\end{eqnarray*}
where
$
g(x)= x\cdot \ln(T + D\cdot e^{-\frac{1}{x}}).
$
Hence,
\begin{eqnarray*}
\mathsf{B}_T(x)= \frac{x}{1+g(x)}.
\end{eqnarray*}
We note that $\lim_{x\rightarrow g(x)}=0$ from which we deduce that 
$$
\lim_{x\rightarrow 0}\frac{\mathsf{B}(x)}{x}=1.
$$
Hence the derivative $\mathsf{B}_T'(0)$ exists.
Also, for $x\in (0,1)$ we see that 
$$
\mathsf{B}'(x)= \frac{1}{1+g(x)} -\frac{x}{(1+g(x))^2}\cdot g'(x).
$$
Since $g'(x)=\ln\left( T+D\cdot e^{-\frac{1}{x}}\right)+\frac{x}{T+D\cdot e^{-\frac{1}{x}}}\cdot e^{-\frac{1}{x}}\cdot\frac{1}{x^2}$ we see that $\mathsf{B}'_T(x)\rightarrow 1$ as $x\rightarrow 0$. Consequently
$\mathsf{B}_T:[0,1]\rightarrow[0,1]$ is a $C^1$-diffeomorphism which we know is smooth on $(0,1]$.
In order to establish the smoothness of $\mathsf{B}$ near $0$, using that $g(x)$ for $x$ near $0$ is close to $0$ and
$$
\mathsf{B}_T(x)=\frac{1}{1+g(x)}\cdot x
$$
 it suffices  to prove the smoothness of $g$ near $0$.  We define $g(0)=0$ and compute 
 for $x>0$ near $0$
\begin{eqnarray*}
 g'(x)&=&\ln(T+ D\cdot e^{-\frac{1}{x}}) +\frac{x}{T+D\cdot e^{-\frac{1}{x}}} \cdot D\cdot e^{-\frac{1}{x}}\cdot \frac{1}{x^2}\\
 &=& \ln(T+ D\cdot e^{-\frac{1}{x}}) +\frac{D}{T+D\cdot e^{-\frac{1}{x}}} \cdot e^{-\frac{1}{x}}\cdot \frac{1}{x}\\
 &=:& \ln(T+ D\cdot e^{-\frac{1}{x}}) +h(x) \cdot e^{-\frac{1}{x}}\cdot \frac{1}{x}.
\end{eqnarray*}
 As $x\rightarrow 0$ we see that $g'(x)\rightarrow \ln(T)$. We also see that
 $$
 \lim_{x\rightarrow 0} \frac{g(x)}{x}=\ln(T).
 $$
 This shows that $g$ is $C^1$ near $x$. It is a trivial exercise that for every
 order of derivative we have 
 that 
 $$
 \lim_{x\rightarrow 0} \frac{d^n}{dx^n}\left( h(x)\cdot \frac{1}{x}\cdot e^{-\frac{1}{x}}\right)=0.
 $$
 Hence
 $$
 \lim_{x\rightarrow 0} g''(x) = \lim_{x\rightarrow 0} \frac{D}{T+D\cdot e^{-\frac{1}{x}}}\cdot e^{-\frac{1}{x}}\cdot \frac{1}{x^2}=\lim_{x\rightarrow } h(x)\cdot e^{-\frac{1}{x}}\cdot \frac{1}{x^2} =0
 $$
 Since it is easily established that $lim_{x\rightarrow 0} \frac{d^n}{dx^n} h(x)\cdot e^{-\frac{1}{x}}\cdot \frac{1}{x^2}=0$ we deduce from the previous discussion that
 $$
 \lim_{x\rightarrow 0} \frac{d^n}{dx^n} g(x)=0\ \ \text{for all} \ n\geq 2.
 $$
Hence we have proved the following
\begin{lem}
The map $g$ defined near $0$ by $g(0)=0$ and for $x>0$ by 
$$
x\cdot \ln(T + D\cdot e^{-\frac{1}{x}})
$$
is smooth. It has the properties $g(0)=0$, $g'(0)=\ln(T)$ and $g^{(n)}(0)=0$ for $n\geq 2$.
\qed
\end{lem}
Using $\mathsf{B}_T(x)= \frac{x}{1+g(x)}$ it follows that $\mathsf{B}_T:[0,1]\rightarrow [0,1]$
is a smooth diffeomorphism with $\mathsf{B}(0)=0$ and $\mathsf{B}'(0)=1$. Hence we established
the result which we need.
\begin{prop}
The map $\mathsf{B}_T:[0,1)\rightarrow [0,1)$ defined by $\mathsf{B}_T(0)=0$ and 
$\mathsf{B}_T(x)=\varphi^{-1}(T\cdot \varphi(x))$ for $x\in (0,1)$ is a smooth diffeomorphism
with $\mathsf{B}'(0)=1$.
\qed
\end{prop}

Define the open subset  $\Omega\subset [0,1)\times {\mathbb R}$ to consist
of all $(r,c)$ with either $r=0$ or in the case $r\in (0,1)$ with $\varphi(r)+c>0$. Define 
$$
B:\Omega\rightarrow [0,1)
$$
by
$$
B(r,c)=\left[\begin{array}{cc}
0& r=0\\
\varphi^{-1}(\varphi(r)+c)& r\in (0,1)
\end{array}
\right.
$$
The following result was established in \cite{HWZ8.7}, Lemma 4.4.
\begin{prop}
The map $B:\Omega\rightarrow [0,1)$ is smooth and  satisfies $B(0,c)=0$, $\partial_rB(0,c)=1$,
$\partial_cB(0,c)=0$
$\partial_r^2B(0,c)=0$ and for $n\geq 2$ and $m\geq 0$ 
it holds $\partial_r^n\partial^m_cB(0,c)=0$.
\qed
\end{prop}
Finally we are ready to state the result we are interested in. For $T>0$ define $\Omega_T\subset [0,1)\times {\mathbb R}$ to consist of all $(x,c)$ with either $x=0$ or if $x\in (0,1)$ with $T\cdot \varphi(x)+c>0$.
Then define $\mathsf{C}_T: \Omega_T\rightarrow [0,1)$ by
$$
\mathsf{C}_T(x,c)=\left[\begin{array}{cc}
0& x=0\\
\varphi^{-1}(T\cdot \varphi(x)+c)& x\in (0,1)
\end{array}
\right.
$$
\begin{thm}\label{thm-h1}
The map $\mathsf{C}_T:\Omega_T\rightarrow [0,1)$ is smooth and 
$\partial_x\mathsf{C}_T(0,c)=1$ for all $c$.
\end{thm}
\begin{proof}
The map is clearly smooth as long as $x\neq 0$. 
We compute for $(x,c)$ with $x$ close to $0$
\begin{eqnarray*}
\mathsf{C}_T(x,c)&=&\varphi^{-1}(T\cdot\varphi(x)+c)\\
&=&\varphi^{-1}\circ (T\cdot\varphi)\circ \varphi^{-1}(\varphi(x) +c/T)\\
&=& \mathsf{B}_T\circ B(x,c/T)
\end{eqnarray*}
This is a composition of smooth maps and the result follows from the previously established facts.
\end{proof}

%-----------------------------------------------------------------------------
% Beginning of biblio.tex
%-----------------------------------------------------------------------------
\bibliographystyle{alpha}

%-----------------------------------------------------------------------------
% End of biblio.tex
%-----------------------------------------------------------------------------

\printindex
\end{document}